\documentclass[11pt]{amsart}

\usepackage{mathtools}
\usepackage{amsmath}
\usepackage{amssymb, stmaryrd, bm}
\usepackage{amsthm}
\usepackage{graphicx}
\usepackage{tikz}
\usepackage{xcolor}
\usepackage{enumitem}
\usepackage{algorithm}
\usepackage[foot]{amsaddr}

\mathtoolsset{showonlyrefs}

\allowdisplaybreaks

\usepackage[margin=1.2in]{geometry}
\usepackage{bookmark}

\def\pr{\mathbb{P}}
\def\E{\mathbb{E}}
\def\var{\mathbb{Var}}
\def\cov{\textrm{Cov}}

\def\bin{{\rm Bin}}

\def\N{\mathbb{N}}
\def\C{\mathbb{C}}
\def\E{\mathbb{E}}
\def\R{\mathbb{R}}
\def\P{\mathbb{P}}
\def\Z{\mathbb{Z}}

\def\eps{\varepsilon}
\def\del{\delta}

\def\cA{\mathcal {A}}
\def\cB{\mathcal {B}}
\def\cC{\mathcal {C}}
\def\cD{\mathcal{D}}
\def\cE{\mathcal {E}}
\def\cF{\mathcal {F}}
\def\cG{\mathcal {G}}
\def\cH{\mathcal {H}}
\def\cL{\mathcal{L}}

\def\cR{\mathcal {R}}

\def\cT{\mathcal {T}}
\def\cI{\mathcal {I}}

\def\bG{\mathbf {G}}
\def\bH{\mathbf {H}}

\def\Ec{E_{\mathrm{cr}}}

\def\1{\mathbf{1}}

\def\lam {\lambda}

\def\tce{t_c + \eps}
\def\tce2{t_c + \frac{\eps}{2}}
\def\ER{Erd\H{o}s-R\'{e}nyi }

\def\var{\textup{var}}
\def\cov{\textup{cov}}

\newtheorem*{theorem*}{Theorem}
\newtheorem{theorem}{Theorem}
\numberwithin{theorem}{section}
\newtheorem{lemma}[theorem]{Lemma}
\newtheorem{cor}[theorem]{Corollary}
\newtheorem{defn}[theorem]{Definition}
\newtheorem*{defn*}{Definition}
\newtheorem{prop}[theorem]{Proposition}
\newtheorem*{prop*}{Proposition}
\newtheorem{conj}{Conjecture}
\newtheorem*{conj*}{Conjecture}
\newtheorem{claim}[theorem]{Claim}
\newtheorem{question}{Question}
\newtheorem{problem}{Problem}
\newtheorem*{fact*}{Fact}

\newtheorem{remark}[theorem]{Remark}

\numberwithin{equation}{section}

\subjclass[2020]{05C30}

\begin{document}
\title{On the evolution of structure in triangle-free graphs}

\author{Matthew Jenssen}
\author{Will Perkins}
\author{Aditya Potukuchi}

\address{King's College London, Department of Mathematics}
\email{matthew.jenssen@kcl.ac.uk}

\address{Georgia Institute of Technology, School of Computer Science}
\email{math@willperkins.org}

\address{York University, Department of Electrical Engineering and Computer Science}
\email{apotu@yorku.ca}

\date{\today}

\begin{abstract}  
We study the typical structure and the number of triangle-free graphs with $n$ vertices and $m$ edges where $m$ is large enough so that a typical triangle-free graph has a cut containing nearly all of its edges, but may not be bipartite.

Erd\H{o}s, Kleitman, and  Rothschild showed that almost every triangle-free graph is bipartite, which leads to an asymptotic formula for the number of triangle-free graphs on $n$ vertices.  Osthus, Pr\"omel, and Taraz later showed that for $m \ge (1+\eps) \frac{\sqrt{3}}{4} n^{3/2}\sqrt{\log n}$, almost every triangle-free graph on $n$ vertices and $m$ edges is bipartite, which likewise leads to an asymptotic formula for their number.    Here we give a precise characterization of the distribution of edges within each part of the max cut of a uniformly chosen triangle-free graph $G$ on $n$ vertices and $m$ edges, for a larger range of densities with $m= \Theta(n^{3/2} \sqrt{\log n})$.  Using this characterization, we describe the evolution of the structure of typical triangle-free graphs as the density changes.  We show that as the number of edges decreases below $\frac{\sqrt{3}}{4} n^{3/2}\sqrt{\log n}$, the following structural changes occur in  $G$:
\begin{itemize}
\item Isolated edges, then trees, then more complex subgraphs emerge as `defect edges', the edges within the parts of a max cut of $G$. In fact, the distribution of defect edges is first that of independent \ER random graphs inside the parts, then that of independent exponential random graphs, conditioned on a small maximum degree and no triangles.
\item There is a sharp threshold for $3$-colorability at $m \sim \frac{\sqrt{2}}{4} n^{3/2}\sqrt{\log n}$ and a sharp threshold between $4$-colorability and unbounded chromatic number at $m \sim \frac{1}{4} n^{3/2}\sqrt{\log n}$.
\item Giant components emerge in the defect edges at $m \sim \frac{1}{4} n^{3/2}\sqrt{\log n}$.
\end{itemize}

We further use this structural characterization to prove asymptotic formulas for the number of triangle-free graphs with $n$ vertices and $m$ edges in this range of densities.  The asymptotic formula exhibits a change in form around the threshold $m \sim \frac{1}{4} n^{3/2}\sqrt{\log n}$ at which  giant components emerge among the defect edges.

We likewise prove the analogous results for the random graph $G(n,p)$ conditioned on triangle-freeness.
\end{abstract}

\maketitle

\pagebreak

\setcounter{tocdepth}{1}
\tableofcontents

\section{Introduction}
\label{secIntro}

Three central topics in combinatorics and graph theory are extremal problems, asymptotic enumeration, and structural questions about typical combinatorial objects. These three topics and their connections are nicely illustrated by the case of triangle-free graphs.   

Mantel's Theorem solves an extremal problem by characterizing the triangle-free graphs on $n$ vertices with the most edges: they are the complete, balanced bipartite graphs. 
\begin{theorem}[Mantel~\cite{mantel1907problem}]
A triangle-free graph on $n$ vertices has at most $\lfloor n^2/4 \rfloor$ edges, and the graphs achieving this bound  are the complete bipartite graphs with part sizes $\lfloor n/2 \rfloor, \lceil n/2 \rceil$.  
\end{theorem}

Let $\cT(n)$ be the set of (labelled) triangle-free graphs on $n$ vertices and $\cB(n)$ be the set of bipartite graphs on $n$ vertices.  The following theorem of Erd\H{o}s, Kleitman, and  Rothschild answers the asymptotic enumeration problem and also describes the typical structure of a triangle-free graph. 

\begin{theorem}[Erd\H{o}s, Kleitman, Rothschild~\cite{erdosasymptotic}]
\label{thmEKR}
Almost all triangle-free graphs are bipartite.  That is,
\[ |\cT(n) | \sim |\cB(n)|   \sim \binom{n}{\lfloor n/2 \rfloor} 2^{ \lfloor n^2/4 \rfloor  -1  }    \sqrt{\frac{\pi   }  {\log 2}  }     \,.\]
\end{theorem}
Here the notation $f(n) \sim g(n)$ means that $\lim_{n \to \infty} \frac{f(n)}{g(n)}= 1$, or equivalently $f(n) = (1+o(1)) g(n)$, and `almost all' means a fraction $1-o(1)$. 
In particular, Theorem~\ref{thmEKR} shows that a typical triangle-free graph is a subgraph of a nearly balanced complete bipartite graph on $n$ vertices; or in other words, typical triangle-free graphs exhibit the same rigid global structure as the extremal example, even though their number of edges is roughly half as many.   

To phrase it  differently, recall that for two probability distributions $\mu, \nu$ on a common  sample space $\Omega$, their total variation distance is defined as
\[
\|\mu-\nu\|_{\text{TV}}= \sup_{A\subseteq \Omega} |\mu(A) - \nu(A)|\, .
\] 
An equivalent formulation of Theorem~\ref{thmEKR} is that the uniform distribution on $\cT(n)$ is within total variation distance $o(1)$  of the uniform distribution on $\cB(n)$. We remark that total variation distance $o(1)$ is a very strong notion of closeness of probability distributions, much stronger than other notions such as asymptotic contiguity.

 How far does this structural behavior persist?  Let $\cT(n,m)$ be the set of triangle-free graphs  on $n$ vertices and $m$ edges and let $\cB(n,m)$ be the subset of bipartite graphs.  Osthus, Pr{\"o}mel, and Taraz -- building on work of  Pr{\"o}mel and Steger~\cite{promel1996asymptotic} -- proved a sharp threshold result in $m$ for a typical triangle-free graph on $n$ vertices to be bipartite with high probability.    

\begin{theorem}[Osthus, Pr\"omel, and Taraz~\cite{osthus2003densities}]
\label{thmOst}
For every $\eps >0$, 
\begin{enumerate}
\item If $m  \ge (1+\eps) \frac{\sqrt{3}}{4} n^{3/2}\sqrt{\log n}$, then almost every graph in $\cT(n,m)$ is bipartite; that is,
\[ |  \cT(n,m) | \sim |  \cB(n,m) | \, . \]
\item If $n/2 \le m \le (1-\eps) \frac{\sqrt{3}}{4} n^{3/2}\sqrt{\log n}$, then almost every graph in $\cT(n,m)$ is not bipartite; that is,
\[ |  \cB(n,m) | =  o \left(  |  \cT(n,m) |  \right) \, . \]
\end{enumerate}
\end{theorem}

One can again rephrase this result in terms of total variation distance: part (1) states that the uniform distributions on $\cT(n,m)$ and $\cB(n,m)$ respectively are within  total variation distance $o(1)$ whereas (2) states that these distributions are asymptotically singular: they have total variation distance $1-o(1)$.  In other words, the rigid  structural property of a typical triangle-free graph being bipartite persists, as the edge density is lowered, until $m \approx  \frac{\sqrt{3}}{4} n^{3/2}\sqrt{\log n}$, and thus in this range of densities the asymptotic enumeration and typical structure problems reduce to the much simpler problem of understanding bipartite graphs.

Far enough below $n^{3/2}$ edges, the asymptotic enumeration problem has also been solved through  entirely different methods.  When $m \le n^{3/2-\eps}$, the asymptotics of $|  \cT(n,m) |$ have been determined in a series of papers~\cite{erdos1960evolution,janson1987uczak,promel1996counting,wormald1996perturbation,stark2018probability,mousset2020probability}.  The first step was the result of Erd\H{o}s and R\'{e}nyi showing that with $m = \Theta(n)$ the distribution of the  number of triangles in the random graph $G(n,m)$ is asymptotically Poisson, and thus the proportion of all graphs on $n$ vertices with $m$ edges that are triangle-free is $\sim \exp (- \mu )$, where $\mu$ is the expected number of triangles in $G(n,m)$.  Using what is now known as `Janson's Inequality', Janson, \L uczak and Ruci\'nski~\cite{janson1987uczak}  then showed that for $m =o (n^{6/5})$, the Poisson behavior persists and  the probability in $G(n,m)$ of seeing no triangles is still asymptotic to $\exp(-\mu)$.   This approach was pushed further, to $m\le n^{3/2-\eps}$ for any fixed $\eps>0$ by Wormald~\cite{wormald1996perturbation} and Stark and Wormald~\cite{stark2018probability} (see also~\cite{mousset2020probability}), and here the asymptotic formula for the probability of triangle-freeness is the exponential of a sum whose number of terms grows as $\eps$ gets smaller.     Unlike  in Theorems~\ref{thmEKR} and~\ref{thmOst}, the asymptotics in this regime are not driven by a rigid global structure like bipartiteness, but rather by a lack of global structure.

Adopting the terminology of statistical physics, we call the dense regime, in which typical triangle-free graphs align with a bipartition and have all (or nearly all) their edges in  a max cut, the \textit{ordered regime}; and the sparse regime, in which graphs lack this global structure, the \textit{disordered regime}.  Below we describe how  intuition and tools from the study of order--disorder phase transitions in statistical physics are useful in studying triangle-free graphs.

Our goal in this paper is to  understand the number and typical structure of triangle-free graphs in the intermediate range of densities not covered by the two sets of results described above.  In particular,  we will delve further into the ordered regime, and solve these problems for a range of edge densities  at which typical triangle-free graphs are not bipartite but are still very structured: they have a unique max cut $(A,B)$ with only a small number of `defect edges' within $A$ and $B$. We further characterize precisely the distribution of the number and structure of these defect edges.  

In particular, we will prove the following asymptotic enumeration and structural results.
\begin{itemize}
\item We give an asymptotic formula for $|\cT(n,m)|$ when $m \ge (1-\eps) \frac{1}{4} n^{3/2} \sqrt{\log n}$ for constant but suitably small $\eps$ (Theorems~\ref{thmNumberM} and \ref{thmNumberMsuper}).

\item We determine the precise structure of a uniformly random graph from $\cT(n,m)$ in this regime.  Almost all such graphs have a  unique max cut $(A,B)$ with almost all edges crossing the cut.  The distribution of defect edges inside $A$ and $B$ is as follows.
\begin{itemize}
\item  when $m \ge  (1+\eps)\frac{1}{4}  n^{3/2} \sqrt{\log n}$ the graphs inside $A$ and $B$ are independent \ER random graphs with  edge probability $q(m,n)$ which we determine; 
\item when $m$ is smaller, the distribution of defect edges are independent copies of a conditioned exponential random graph  with parameter values  we determine.
\end{itemize}

\item As corollaries, we determine  sharp thresholds and scaling windows for several structural properties:
\begin{itemize}
\item we determine the limiting distribution of the smallest number of edges one needs to remove to make a typical graph in $\cT(n,m)$ bipartite  and identify  the scaling window for a random triangle-free graph in $\cT(n,m)$ to be bipartite   (Theorem~\ref{thmDistToBipartite}).
\item we identify the sharp threshold for  a random triangle-free graph in $\cT(n,m)$ to be $3$-colorable (Theorem~\ref{thm3color}).
\item we identify the sharp threshold for the property of a random triangle-free graph in $\cT(n,m)$ to be $4$-colorable (Theorem~\ref{thm4color}).
\end{itemize}

\item We likewise prove analogues of all the results above for the \ER random graph $G(n,p)$ conditioned on being triangle-free; for example,  we determine the first-order asymptotics of the probability of being triangle-free when $p \ge (1-\eps) \sqrt{\frac{\log n } {n}}$, and characterize the typical structure of graphs drawn from this conditional distribution.
\end{itemize}

To prove these results we  use intuition from the  study of order--disorder phase transitions in statistical physics; we use tools such as the cluster expansion as well as develop  new techniques to work with cumulant generating functions in certain exponential random graph models.  We expect these techniques to be widely applicable to other combinatorial enumeration problems.

\subsection{Main results}
\label{secIntroMainResults}

In the ordered regime, almost every triangle-free graph $G$ has a unique max cut, whose partition we will denote by $(A,B)$, and this max cut contains almost all of the edges of $G$.  We will denote by $G_A, G_B$ the subgraphs induced by $A$ and $B$ respectively, and let   $S \subseteq \binom{A}{2}$ and $T \subseteq \binom{B}{2}$ denote the \textit{defect edges}, so that $G_A = (A,S)$ and $G_B = (B,T)$.    The \textit{crossing edges} $\Ec$ of $G$ are those with one endpoint in $A$ and the other in $B$.   

To describe typical structure, we will determine the distribution of the max cut $(A,B)$ (in particular the distribution of their respective sizes), the distribution of the defect edges $S$ and $T$ given $(A,B)$, and the distribution of $\Ec$ given $S$ and $T$ to high enough accuracy and in a simple enough form that we can do explicit calculations of asymptotics.  We say a few words about each of these distributions in reverse order.

Conditioned on $(A,B)$ and $S,T$, the distribution of $\Ec$ is essentially that of a uniformly random subset of $m - |S| -|T|$ edges from $A \times B$ conditioned on the event that these edges form no triangles with the edges from $S,T$.   Equivalently, it is a uniformly random independent set of size   $m - |S| -|T|$ from the graph $S \boxempty T$,   the Cartesian product of the graphs $(A, S)$, $(B, T)$;  that is, the graph with vertex set $V(S \boxempty T) = A \times B$ and edge set $$E(S \boxempty T) = \{ \{(a,b), (a,b')\}: \{b,b'\}\in T  \} \cup \{ \{(a,b), (a',b)\}: \{a,a'\}\in S  \} \,.$$  We will show below that this random independent set model is very nicely behaved: it is `subcritical' in the sense that we can write an explicit asymptotic formula for the number of such independent sets using the cluster expansion--one of the main tools from statistical physics we use in this work.

Deriving the distribution of the defect edges $S, T$ conditioned on the cut $(A,B)$  is at the heart of this paper.  We will show that the distribution of $S,T$ is asymptotically that of independent copies of an exponential random graph conditioned on a  maximum degree bound and triangle-freeness.  The parameters of this  random graph depend on the  edge density of the triangle-free graph $G$.

Finally, the distribution of the cut $(A,B)$ will follow fairly easily from an understanding of the other distributions. In particular, we show that the imbalance $|A|- \lfloor n/2 \rfloor$ follows a discrete Gaussian distribution. 

We state our results in three different regimes, corresponding to distinct behavior of the distribution of the defect edges. We highlight that as $m$ decreases, the typical number of defect edges in a sample from $\cT(n,m)$ \emph{increases}, i.e., the graph becomes less bipartite.

\begin{itemize}
\item The \textbf{subcritical defect regime}, $m \ge \left (1+\eps \right) \frac{1}{4} n^{3/2} \sqrt{\log n}$, for $\eps>0$ fixed.  We will see that in this regime, whp all defect edges are in small components, and  (up to $o(1)$ total variation distance)  the distribution of defect edges within $A$ and within $B$ is that of independent  \ER random graphs.

\item The \textbf{supercritical defect regime}, $ m \le (1-\eps) \frac{1}{4} n^{3/2} \sqrt{\log n}$, for $\eps\in(0,1/14]$ fixed.  In this regime,  the defect edges form connected graphs on both $A$ and $B$ and the distribution of defect edges is given by a conditioned exponential random graph with weights that are a function of the number of edges and paths of length $2$ in the graph.
\item The \textbf{critical defect regime}, $(1-\eps) \frac{1}{4} n^{3/2} \sqrt{\log n}\le m \le (1+\eps) \frac{1}{4} n^{3/2} \sqrt{\log n}$ for $\eps = o(1)$.  In this critical regime,  giant components emerge among the defect edges in $A$ and $B$, the chromatic number becomes unbounded, and the number of paths of length $2$ becomes significant in the distribution of defect edges. 
\end{itemize}
A visual overview of the structural changes (leaving out the critical regime) is shown in Figure~\ref{fig:timeline}. Note that the upper bound of $1/14$ on $\eps$ is due to technical limitations of our methods and not a significant qualitative change in the problem.  We discus potential extensions of the range of densities treated in Section~\ref{secOutlook} below. 

\begin{figure}
    \centering
    \includegraphics[width=1.0\linewidth]{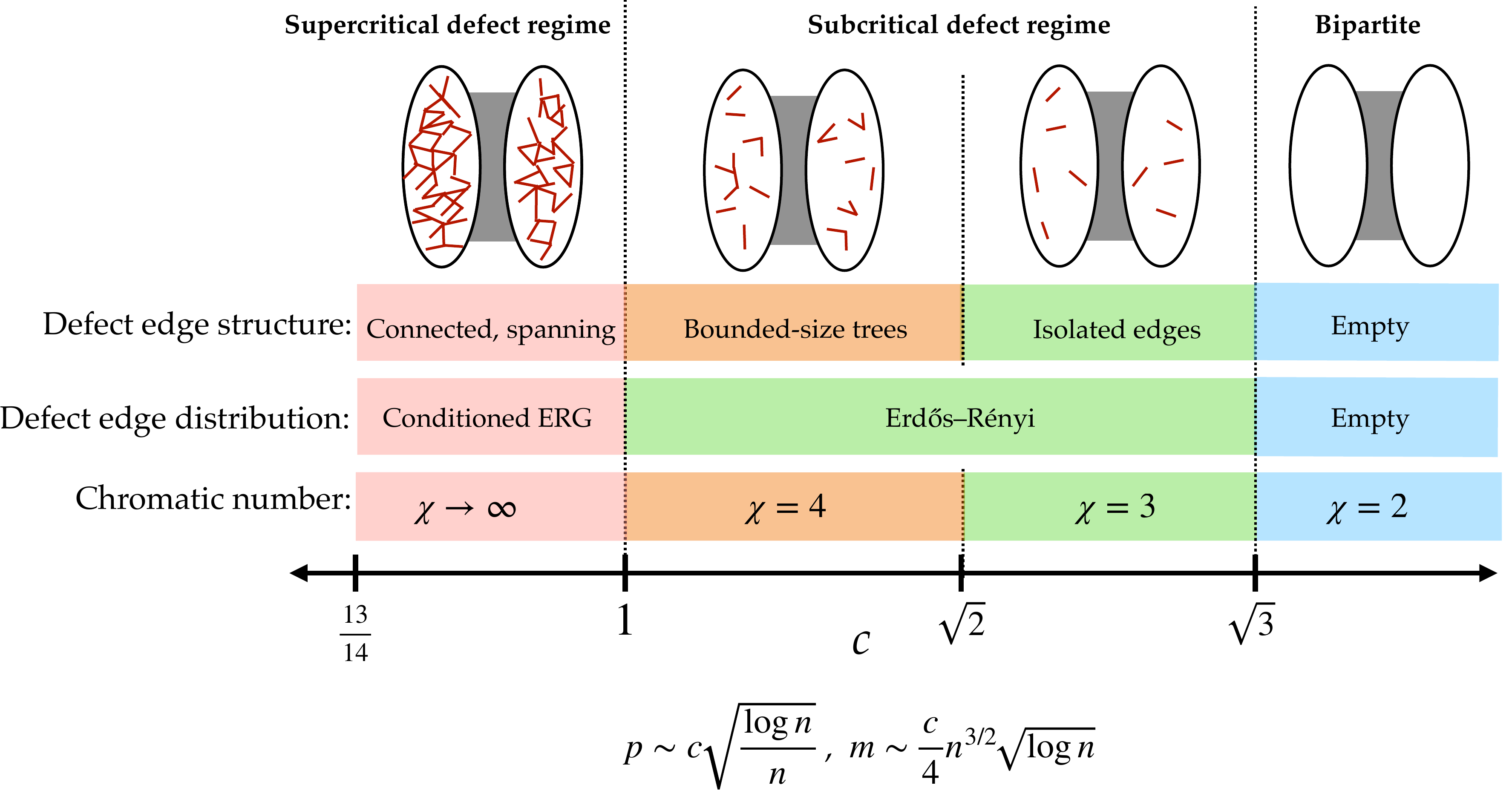}
    \caption{Summary of structural changes as the edge density changes.}
    \label{fig:timeline}
\end{figure}

\subsubsection{Subcritical defect regime}
\label{secIntroSubCritical}

To state precise results, we first define some parameters. Let 
\begin{align}
\label{eqlam0def}
\lam_0:=\frac{4m}{n^2}\, ,
\end{align}
and 
\begin{align}\label{eqpdef}
\lam:= \lam_0+\lam_0^2 + (\lam_0^2n-1)\lam_0e^{-\lam_0^2n/2}\, .
\end{align}
Note that in our setting $\lam \sim \lam_0$, but the lower order corrections will be consequential.

Define $q_0=q_0(n,m)$  so that
\begin{align}\label{eqqdef}
\frac{q_0}{1-q_0} =  \lam e^{-\lam^2n/2}\, .
\end{align}
The parameter $q_0$ will represent the approximate density of defect edges. When $m \sim \frac{c}{4} \cdot n^{3/2} \sqrt{\log n}$ we have $q_0 = n^{- \frac{1}{2} - \frac{c^2}{2} + o(1)}$.

We write $G(V,q)$ to denote the \ER random graph on a vertex set $V$ with edge probability $q$.  We also define a distribution on partitions of $[n]=\{1,\ldots,n\}$. 
\begin{defn}\label{defPartitionDist}
Given $\lam>0$, let $\xi_\lam$ denote the following discrete Gaussian distribution on $\Z$:  
\begin{align*}
\xi_\lam(t)\propto \left(1+\lam\right)^{-t^2}\, .
\end{align*}
We define a probability distribution $\theta_\lam$ on partitions $(A,B)$ of $[n]$ as follows:
\begin{enumerate}
\item Sample $t\in \Z$ according to $\xi_\lam$.  Let $t' = \min \{ t, \lceil n/2 \rceil \}$.
\item Sample a partition $(A,B)$ of $[n]$ with $|A| = \lfloor n/2\rfloor+t', |B|=\lceil n/2\rceil -t'$ uniformly at random. 
\end{enumerate}
\end{defn}

\begin{defn}\label{defstrongbalance}
Call a partition $(A,B)$ of $[n]$ \emph{strongly balanced} if $\big ||A|-|B| \big |\leq 10(n \log n)^{1/4}$. Let $\Pi_{\textup{strong}}$ denote the set of all strongly balanced partitions of $[n]$.
\end{defn}

In what follows `whp' stands for `with high probability', meaning with probability $1-o(1)$ as $n \to \infty$.
\begin{theorem}\label{thmSubcriticalfixedEdgeDist}
\label{ThmStructure}
Fix $\eps>0$.  Suppose $m \ge \frac{1+\eps}{4} n^{3/2} \sqrt{\log n}$ and let $\lam$ be as in~\eqref{eqpdef}.  
Choose $G$ from the uniform distribution on $\cT(n,m)$.  Then
\begin{enumerate}
\item Whp $G$ has a unique, strongly balanced max cut $(A,B)$ of size  $m-o(m)$.  The cut $(A,B)$ is distributed according to $\theta_\lam$ up to $o(1)$ total variation distance. 
\item Whp over the random max cut  $(A,B)$,  the distribution of the subgraphs $G_A$ and $ G_B$ induced by $A$ and $B$ respectively is that of independent samples from $G(A,q_0)$ and $G(B,q_0)$ up to $o(1)$ total variation distance. 
\end{enumerate}
\end{theorem}

Alternatively, we can rephrase Theorem~\ref{thmSubcriticalfixedEdgeDist} in terms of an algorithm that generates an approximately uniform sample from $\cT(n,m)$.   We describe the process below and denote the resulting distribution on $\cT(n,m)$ by $\mu_{m,1}$.

\begin{algorithm}[ht]
  \caption{The distribution $\mu_{m,1}$\label{algMum1}}
\begin{enumerate}[leftmargin=*]
\item Choose a random partition $(A,B)$ according to $\theta_\lam$. 
\item Choose  defect edges $S \subseteq \binom{A}{2}$, $T\subseteq \binom{B}{2}$ according to independent realizations of $G(A,q_0)$ and $G(B,q_0)$ respectively. If $S\cup T$ contains a triangle or if $|S| + |T| >m$, output an arbitrary graph $G_0 \in \cT(n,m)$. Otherwise proceed to the next step. 
\item Choose crossing edges $\Ec \subset A \times B$ uniformly from all subsets of size $m - |S| - |T|$ so that $S \cup T \cup \Ec$ contains no triangles.
\item Output $S \cup T \cup \Ec$.
\end{enumerate}
\end{algorithm}

\begin{theorem}
\label{thmGnmSubRestated}
Fix $\eps>0$ and suppose $m \ge \frac{1+\eps}{4} n^{3/2} \sqrt{\log n}$.  The distribution $\mu_{m,1}$ is at total variation distance  $o(1)$ to the uniform distribution on $\cT(n,m)$.
\end{theorem}

From Theorem~\ref{ThmStructure} or Theorem~\ref{thmGnmSubRestated}, one can immediately deduce essentially any  desired structural information about the defect graphs in the relevant regime from an understanding of the structure of a very sparse \ER random graph.  Using this and an understanding of the independent set model that generates the crossing edges,  we can describe the evolution of the structure of a uniformly random $G$ from $\cT(n,m)$ as $m$ decreases from $\frac{\sqrt{3}}{4} n^{3/2}  \sqrt{\log n}$ to  $\frac{1}{4} n^{3/2}  \sqrt{\log n}$.  At a high level, as the overall edge density decreases within the ordered regime, the density of the defect edges and the complexity of the structure of typical graphs increase.    From~\eqref{eqqdef}, we  see that  as $m$ drops from $\frac{\sqrt{3}}{4} n^{3/2}  \sqrt{\log n}$ to  $\frac{1}{4} n^{3/2}  \sqrt{\log n}$, $q_0$ increases from around $n^{-2}$ to $n^{-1}$; this gives some intuition for the constant $\frac{\sqrt{3}}{4}$ in Theorem~\ref{thmOst} and hints at more complex structure emerging when the leading constant passes $\frac{1}{4}$.

Now we describe some of the structural changes precisely.  Osthus, Pr{\"o}mel, and Taraz (in~\cite{osthus2003densities}) show that  $\frac{\sqrt{3}}{4} n^{3/2}  \sqrt{\log n}$ is the sharp threshold for $G$ sampled uniformly from $\cT(n,m)$ to be bipartite; that is, for the defect graphs to be empty.  We refine this by characterizing the distribution of the  distance from bipartiteness.  This also gives the precise scaling window for the property of $G$ being bipartite.

\begin{theorem}
\label{thmDistToBipartite}
Fix $\eps>0$ and suppose $m \ge \frac{1+\eps}{4} n^{3/2} \sqrt{\log n}$.  Suppose $G$ is drawn uniformly  from $\cT(n,m)$, and let $X(G)$ be the minimum number of edges whose removal makes $G$ bipartite.     Let $\hat X \sim \text{Bin}(\lfloor n^2/4 \rfloor, q_0)$ where $q_0$ is as in~\eqref{eqqdef}.  Then  $\| X(G)- \hat X\|_{TV} \to 0$ as $n \to \infty$. 

In particular, if $m =  \frac{ \sqrt{3+ \frac{\log \log n}{\log n} - \frac{t}{\log n}  }}{ 4    } n^{3/2} \sqrt{\log n}$ for $t \in \R$, then
\[ \lim_{n\to \infty} \P ( G \in \cB(n,m) ) =  \exp \left(  - \frac{\sqrt{3}}{4} e^{t/2} \right ) \,.   \]
\end{theorem}

In their survey on random triangle-free graphs~\cite{promel2001random}, Pr{\"o}mel and Taraz  ask: what is the typical chromatic number of a graph $G$ drawn uniformly from $\cT(n,m)$ in the regime where $G$ is not bipartite whp?      
Our next result identifies a sharp transition from $3$-colorability to $4$-colorability at $m \sim \frac{\sqrt{2}}{4} n^{3/2} \sqrt{\log n}$.

\begin{theorem}
\label{thm3color} 
Let $\eps>0$ be fixed and let $G$ be drawn uniformly from $\cT(n,m)$. Then
\begin{itemize}
\item If $ (1+\eps)\frac{\sqrt{2}}{4} n^{3/2} \sqrt{\log n} \le m  \le  (1-\eps)\frac{\sqrt{3}}{4} n^{3/2} \sqrt{\log n}$ then $\chi(G) = 3 $ whp.
\item If $ (1+\eps) \frac{1}{4} n^{3/2} \sqrt{\log n} \le m  \le (1-\eps)\frac{\sqrt{2}}{4} n^{3/2} \sqrt{\log n}$ then $\chi(G) = 4 $ whp.
\end{itemize}
\end{theorem}
In Theorem~\ref{thm4color} below, we show that the chromatic number becomes unbounded when $m$ is just below  $\frac{1}{4} n^{3/2} \sqrt{\log n}$ (the supercritical defect regime).

Finally, we can use Theorem~\ref{ThmStructure} to asymptotically enumerate triangle-free graphs in this range of densities. 

\begin{theorem}
\label{thmNumberM}
Fix  $\eps>0$ and suppose $m \ge (1+\eps)\frac{1}{4} n^{3/2} \sqrt{\log n}$. Then
\begin{equation}
\label{eqAsymptoticsM}
| \cT(n,m) | \sim 
\frac{1}{\sqrt{2}\lam^{m+1}n }\binom{n}{\lfloor n/2\rfloor}(1+\lam)^{n^2/4}\exp\left\{\lam e^{-\lam^2n/2+\lam^3n} \frac{n^2}{4} + \lambda^5e^{-\lambda^2 n}\frac{n^4}{8} \right\}\, ,
\end{equation}
where $\lam=\lam(n,m)$ is as in~\eqref{eqpdef}.
\end{theorem}
We note that an asymptotic formula for $|\cB(n,m)| $ is straightforward to compute (see, e.g.,~\cite[Theorem 4]{osthus2003densities}), and so Theorem~\ref{thmNumberM} can be used to give an asymptotic formula for the probability that a uniformly chosen $G\in\cT(n,m)$ is bipartite.

\subsubsection{Supercritical defect regime}\label{subsecerg}

Next, postponing discussion of the critical regime to Section~\ref{subsecCrit},  we characterize the typical structure of triangle-free graphs at lower densities, when the defect edges are denser.  In contrast to the subcritical defect regime, the distribution of the defect edges will not be that of  \ER random graphs on  $A$ and $B$.  Instead,  the distribution of defect edges in $A$ and in $B$ will be asymptotically identical  to independent exponential random graphs with energy functions depending on the count of edges and copies of $P_2$ (the path on $2$ edges), conditioned  on triangle-freeness and on a bound on the max degree. An exponential random graph is a log linear probability distribution on the set of  graphs on $n$ vertices, where the log of the probability mass function is an energy function that is a  linear combination of subgraph counts of the graph   \cite{frank1986markov,wasserman1996logit,robins2007introduction,bhamidi2008mixing,chatterjee2013estimating,radin2013phase}.

Let $|G|$ denote the number of edges of a graph $G$ and $P_2(G)$ the number of copies of $P_2$ (as a subgraph) in $G$.
Then given parameters $q \in (0,1)$, $\psi \in \R$ and a vertex set $V\subseteq [n]$, let $G(V,q,\psi)$ denote the random graph on $V$ with distribution 
\begin{align}\label{eqGVqpsiDef}
  \nu_{q,\psi} (G) \propto \left(  \frac{q}{1-q} \right)^{|G|}  e^{\psi P_2(G)} \, ,
  \end{align}
 \emph{conditioned} on the event that $\Delta(G) \le 50\max\{qn, \log n\}$ and $G$ is triangle-free. {In what follows $q$ will be roughly $n^{-c}$ for $c\in(1/2, 2]$ and $\psi$ will be roughly $n^{-1/2}$. Since $\psi$ is positive, this gives a boost to graphs with more copies of $P_2$, and since $\psi$ is small one might think the boost would be mild.  With the conditioning on the max degree the boost is indeed mild; the average degree remains $\sim q n$ (as it would be with $\psi=0$).  Without conditioning on the max degree  the average degree would jump significantly (by a factor polynomial in $n$) on account of the preference for $P_2$'s. Thus the conditioning gives $\nu_{q,\psi}$ distinct properties from those of exponential random graphs considered in the literature, and it is an essential component of our results and techniques. }

We now define the specific parameters $q=q(n,m), \psi=\psi(n,m)$ that arise in the defect distribution.
First we let $\lam$ be as in ~\eqref{eqpdef} and let
\begin{align}\label{eqq0q1Def}
 \frac{q_1}{1-q_1}:=\lam e^{-\lam^2n/2+\lam^3n-7\lam^4n/4}\, .
\end{align}
Next, let
 \begin{align}\label{eqMuRhoDef}
\mu= \binom{n/2}{2}q_1e^{\lam^3n^2q_0}, 
\end{align}
and 
\begin{align}\label{eqq2Def}
\frac{q_2}{1-q_2} := \frac{q_1}{1-q_1}e^{4\lam^3\mu}\, .
\end{align}

Finally define
\begin{align}\label{eqPsiDef}
\psi:= \lam^3n/2\, .
\end{align}
\begin{theorem}
\label{thmP2ERG}
Fix $\eps \in (0, 1/14]$.  Suppose  $m \ge  (1-\eps) \frac{1}{4} n^{3/2} \sqrt{\log n}$. Choose $G$ from the uniform distribution on $\cT(n,m)$.  Then
\begin{enumerate}
\item Whp $G$ has a unique, strongly balanced max cut $(A,B)$ of size  $m-o(m)$.  The cut $(A,B)$ is distributed according to $\theta_\lam$ up to $o(1)$ total variation distance. 

\item  Whp over the random cut $(A,B)$, the distribution of the subgraphs $G_A$ and $ G_B$  are independent samples from $G(A,q_2, \psi)$ and $G(B,q_2,\psi)$ respectively up to $o(1)$ total variation distance. 
\end{enumerate}  
\end{theorem}
As mentioned above, the upper bound on $\eps$ (and lower bound on $m$) in our results is due to technical limitations.  For $m \sim c n^{3/2} \sqrt{\log n}$ for smaller $c$, we expect the distributions of $G_A$ and $G_B$ to be exponential random graphs conditioned on triangle-freeness and a bound on the maximum degree, with a probability mass function that depends on more and more subgraphs (beyond just $P_2$) as $c$ decreases.  See further discussion in Section~\ref{secOutlook}.

Like Theorem~\ref{ThmStructure}, Theorem~\ref{thmP2ERG} can be rephrased algorithmically. We do this in Section~\ref{secFixedM} where we introduce the measure $\mu_{m,2}$.
Theorem~\ref{thmP2ERG} provides a very precise description of the distribution of defect edges.   As in the subcritical regime, understanding the distribution to this level of detail will allow us to understand the evolution of the structure of the graph and give precise asymptotics for the number of triangle-free graphs.
\begin{theorem}
\label{thmNumberMsuper}
Fix $\eps \in (0, 1/14]$.  Suppose  $m \ge  (1-\eps) \frac{1}{4} n^{3/2} \sqrt{\log n}$. Then
\begin{align*}
| \cT(n,m) | \sim& 
\frac{1}{\sqrt{2}\lam^{m+1}n }\binom{n}{\lfloor n/2 \rfloor}(1+\lam)^{n^2/4}(1-q_2)^{-n^2/4+n/2}\times\\ 
& \exp\left\{  \frac{1}{64} \lam^6 n^5 q_0^2 - \frac{1}{64}\lam^6 n^6 q_0^3 -\frac{1}{24}n^3q_0^3 + \frac{1}{64}\lam^4n^4q_0^2-\frac{1}{6}\lam^4n^5q_0^3-\frac{1}{2}\lam^4n^4q_0^2 \right\}
\, .
\end{align*}
\end{theorem}
One can check that the formula in Theorem~\ref{thmNumberMsuper} does indeed reduce to that of Theorem~\ref{thmNumberM} when $m \ge (1+\eps)\frac{1}{4} n^{3/2} \sqrt{\log n}$.

We next show that a random triangle-free graph makes a sharp transition from almost surely being $4$-colorable  to having unbounded chromatic number as $m$ decreases past $\frac{1}{4}n^{3/2} \sqrt{\log n}$.   In fact we  prove an upper bound on the independence number below this threshold.  Recall that Theorem~\ref{thm3color} states that for $ (1+\eps) \frac{1}{4} n^{3/2} \sqrt{\log n} \le m  \le (1-\eps)\frac{\sqrt{2}}{4} n^{3/2} \sqrt{\log n}$, $\chi(G) = 4 $ whp.

\begin{theorem}
\label{thm4color}
Fix $\eps \in (0,1/14)$ and let $G$ be drawn uniformly from $\cT(n,m)$.   If 
$$ m  \sim (1-\eps)  \frac{1}{4}n^{3/2} \sqrt{\log n}$$
 then the independence number of $G$ satisfies $\alpha(G)  =o(n)$ whp. In particular, the chromatic number of $G$ satisfies $\chi(G)=\omega(n)$ whp.
\end{theorem}

We can in fact say  more about how quickly the chromatic number increases as the edge density of $G$ decreases, but we postpone this to Section~\ref{secSuperCrit}.

\subsubsection{Critical defect regime}\label{subsecCrit}

We now discuss the critical defect regime, with $(1-\eps) n^{3/2} \sqrt{\log n}  /4 \le  m \le (1+\eps) n^{3/2} \sqrt{\log n} /4$ and $\eps =o(1)$.  In this window  giant components emerge among the defect edges in $A$ and $B$, and their distribution begins to depend on the $P_2$ count.  
(Note that the asymptotic enumeration result in this regime is already covered by Theorem~\ref{thmNumberMsuper}).

We  first determine the scaling windows for the emergence of giant components and connectivity among the defect edges. We use the notation $f(n)\ll g(n)$ to denote that $\lim_{n\to\infty}f(n)/g(n)=0$. 
\begin{theorem}
\label{thmGiantComponent} 
Let $G$ be drawn uniformly from $\cT(n,m)$,  conditioned on $G$ having a strongly balanced max cut $(A,B)$.  
\begin{enumerate}
\item If $m$ is such that $q_0 = \frac{2}{n} - \frac{\omega(n)}{n^{4/3}}$, with $1 \ll \omega(n) \ll n^{1/3}$,  then whp the largest connected component  of $G_A$ is of size $\Theta(n^{2/3} \omega^{-2} \log (\omega))$.  
\item If $m$ is such that $q_0 = \frac{2}{n} + \frac{\omega}{n^{4/3}}$, with $ \omega \in \mathbb R$ constant,  then whp the  largest connected component  of $G_A$ is of size  $\Theta(n^{2/3})$.
\item If $m$ is such that $q_0 = \frac{2}{n} + \frac{\omega(n)}{n^{4/3}}$, with $1 \ll \omega(n) \ll n^{1/3}$,  then whp the largest connected component of $G_A$ is of size $(2+o(1))  \cdot \omega \cdot (n/2)^{2/3}$. 
\item If $m$ is such that $q_0 = \frac{c}{n}$ with $c>2$ fixed, then whp the largest connected component of $G_A$ is of size $\Theta(n)$.
\item  If $m$ is such that $q_0= (1+\eps)\frac{2 \log n}{n}$ for $\eps >0$ constant, then whp $G_A$ is connected, while if $m$ is such that $q_0 \le (1-\eps)\frac{2 \log n}{n}$ for $\eps>0$ constant, then whp $G_A$ is not connected.
\end{enumerate}
Moreover, these results also hold for the graph $G_B$.  
\end{theorem}

Note that when $q_0 = \tilde \Theta(n^{-1})$, $m \sim \frac{1}{4} n^{3/2} \sqrt{\log n}$, and so giant defect components and connectivity in the defect graphs emerge in rapid succession at this threshold;  the results of Theorem~\ref{thmGiantComponent} give a  description of the scaling window of this emergence.

\subsection{Methods}
\label{subsecMethods}

Here we give an overview of the proofs of the structural and enumeration results above.  The intuition for the approach comes from both statistical physics and algorithms.  We use and extend tools from~\cite{luczak2000triangle,osthus2003densities,balogh2016typical,jenssen2020homomorphisms,JKP2,jenssen2020independent,jenssen2022independent}, and then develop some new tools here. 

Our starting point is to write a statistical physics partition function for triangle-free graphs:
\begin{equation}
\label{eqZndef}
 Z(\lam) = \sum_{G \in \cT(n)} \lam^{|G|} \,,
 \end{equation}
 along with the corresponding probability distribution on $\cT(n)$, the Gibbs measure
 \begin{equation}
 \label{eqmudef0}
 \mu_{\lam}(G) = \frac{\lam^{|G|}}{Z(\lam)} \,,
 \end{equation}
where $|G|$ is the number of edges of $G$. As we see below in Section~\ref{secERresults}, up to scaling $Z(\lam)$ is exactly the probability that $G(n,p)$ is triangle-free, with $p = \frac{\lam}{1+\lam}$ and $\mu_{\lam}$ is the corresponding conditional probability measure. 

Classical statistical physics is concerned with understanding partition functions like this and their associated Gibbs measures.   Many statistical physics models (such as Ising, Potts, hard-core) on lattices like $\Z^d$ undergo order/disorder phase transitions as the strength of interaction or density of particles increases. For instance,  consider the ferromagnetic Ising model on a torus $(\Z/n\Z)^d$.  In the low-temperature, strong interaction regime typical samples from the model look like small perturbations from either the all $+$ or all $-$ configurations (the \textit{ground states}).  Quantifying the contribution to the partition function $Z$ from all such small perturbations (and understanding their probabilistic properties) is a delicate and difficult task and the subject of a huge amount of work in statistical physics dating all the way back to Peierls' work on the Ising model~\cite{peierls1936ising}.  Some of the powerful methods developed to address this problem include the cluster expansion~\cite{ruelle1963correlation,penrose1963convergence,brydges1984short,kotecky1986cluster,FernandezProcacci} and Pirogov--Sinai theory~\cite{pirogov1975phase,borgs1989unified,laanait1991interfaces,borgs2012tight}.  Recently, these types of tools have been applied outside the context of classical statistical physics, to algorithmic and combinatorial problems of enumeration and sampling~\cite{barvinok2015computing,regts2015zero,patel2016deterministic,barvinok2017combinatorics,HelmuthAlgorithmic2,jenssen2020independent,jenssen2020homomorphisms,jenssen2022independent}.

We can take this perspective on the partition function defined in~\eqref{eqZndef}. The (near) ground states of the triangle-free graph model are the (nearly) balanced complete bipartite graphs on $n$ vertices; this is the content of Mantel's Theorem.  The first step in understanding the `ordered' phase of a statistical physics model is to prove that configurations far from any ground state have a negligible contribution to the partition function.  For classical lattice systems this is often accomplished by the Peierls' argument; in the triangle-free graph setting the result of  {\L}uczak~\cite{luczak2000triangle} (Theorem~\ref{thm:luczak} below) accomplishes this coarse separation of ground states: he proves that almost all triangle-free graphs with $m$ edges have a max cut containing almost all the edges when $m \gg  n^{3/2}$.  More refined estimates at larger densities are given by Osthus, Pr\"omel, and Taraz~\cite{osthus2003densities} which allow them to prove Theorem~\ref{thmOst}.  In particular, their estimates show that when $m \ge (1+\eps) \frac{\sqrt{3}}{4} n^{3/2}\sqrt{\log n}$, a typical triangle-free graph has a max cut $(A,B)$ such that the max degree within $A$ and within $B$ is $0$ (that is, the graph is bipartite).

This marks the starting point of our approach.  We extend the estimates from~\cite{osthus2003densities} (by adapting a more general approach of Balogh, Morris, Samotij, and Warnke~\cite{balogh2016typical} who analyzed the structure of $K_r$-free graphs) to show that when  $\lam \ge C n^{-1/2}$ and $C$ is sufficiently large, a typical sample from $\mu_\lam$ has a max cut $(A,B)$ such that the max degree within $A$ and within $B$ is at most $\alpha/\lam$ for some small constant $\alpha$.  This is a more refined but still coarse structural result, and our proof uses the very coarse result of {\L}uczak as a key ingredient (see Section~\ref{secMorrisOPT}).  Crucially, the degree bound we obtain allows us to apply the cluster expansion to understand the contribution to $Z(\lam)$ from graphs with a given max cut $(A,B)$ and given sets $S,T$ of defect edges within $A$ and $B$ respectively. The conclusion of this analysis is that conditioned on $(A,B)$ the distribution of $S,T$ is that of an exponential random graph conditioned on a max degree bound.

From here we move on to a yet simpler approximation.  Using local probability estimates, we show that with negligible error we can impose a much stricter max degree condition on the exponential random graph, and this in turn allows us to truncate the cluster expansion after a small number of terms, thus significantly simplifying the parameters of the exponential random graph.  How simple we can make this model depends on the density: the lower the density the more complex the model, as more complex structure emerges among the defect edges.

From there, we then need to analyze this exponential random graph model conditioned on a max degree bound.  The conditioning is in fact essential -- without it the model would `blow up' and defects would proliferate, taking us out of the neighborhood of the $(A,B)$ ground state -- and this marks a departure with other exponential random graph models analyzed in probability theory and algorithms~\cite{bhamidi2008mixing,chatterjee2013estimating,radin2013phase}.   The tools we develop to understand this model are based primarily in approximating cumulant generating functions; this allows us to understand both partition functions and approximate probability distributions via Pinsker's Inequality.   The bounding of higher-order cumulants in the conditioned ERG model is the most technically involved part of the proof.  On top of this, we need to prove a version of Janson's Inequality~\cite{janson1988exponential} for the probability of triangle-freeness in the conditioned exponential random graph model (Lemma~\ref{lemJansonERG} below). This hints at a kind of approximate duality in the problem: in the ordered regime ($m \gg n^{3/2}$) the distribution of defect edges $S,T$ has much in common with the distribution of the full set of edges of a triangle-free graph  in the disordered regime $m \ll n^{3/2}$.  We discuss below how extensions of tools like those in~\cite{janson1988exponential,wormald1996perturbation,stark2018probability,mousset2020probability} might be used in the future to extend our results further into the ordered regime.

Nearly all of our work is done in analyzing $Z (\lam)$ and $\mu_{\lam}$, and this analysis leads directly to the results on $G(n,p)$ which we present below in Section~\ref{secERresults}.  To prove our results for $\cT(n,m)$, we transfer the results for $Z (\lam)$ and $\mu_{\lam}$ by using local central limit theorems for hard-core models.  Such results are relatively straightforward to prove when the underlying model has a convergent cluster expansion; this technique  has been used recently in both combinatorics~\cite{jenssen2022independent} and algorithms~\cite{jain2021approximate}.

\subsection{Outlook}
\label{secOutlook}

The ultimate goal in the study initiated in~\cite{erdos1960evolution,erdosasymptotic}, continued in~\cite{janson1988exponential,promel1996asymptotic,wormald1996perturbation,osthus2001almost,promel2001random,osthus2003densities,steger2005evolution,stark2018probability,mousset2020probability}, and pursued here would be to determine the asymptotic number and typical structure of triangle-free graphs at \textit{any} density.  This will require progress on two fronts.   

On the \textit{disordered} side, there are asymptotic formulas for $Z(\lam)$ and $|\cT(n,m)|$ when $\lam\le n^{-1/2 -\eps}$ and $m \le n^{3/2 -\eps}$  for any fixed $\eps>0$~\cite{stark2018probability,mousset2020probability}; these formulas refine Janson's inequality and involve the exponential of a sum of terms, with the number of terms growing as $\eps$ decreases.  It is tempting to believe that there is a convergent infinite series, of which these sums are finite truncations, that give an asymptotic formula for $\lam\le cn^{-1/2}$ and $m \le cn^{3/2}$ for some $c >0$.  Such a result would follow if one could show the cluster expansion for a hypergraph independent set model  (see Section~\ref{secStatPhysFormulation}) converges in a particular range of parameters or if $\log Z(\lam)$ has a different convergent expansion; however this is a difficult question and the hypergraph cluster expansion may not be a convergent series through the full range of densities; see the discussion in~\cite{galvin2022zeroes,zhang2023hypergraph}.

On the \textit{ordered} side, one would need to extend the results of this paper to sparser regimes. 
Theorem~\ref{thmGenDef} below gives a rough structural description of a typical sample from $\mu_\lam$ all the way down to $\lam \ge Cn^{-1/2}$ for some $C>0$. The use of the cluster expansion to measure the contribution of triangle-free graphs with a given max cut and given sets of defect edges also works down to this density (and in fact these tools can be used to give efficient algorithms to approximately sample from $\mu_\lam$~\cite{jenssen2024sampling}).   However,  in order obtain an asymptotic formula for $Z(\lam)$ and $|\cT(n,m)|$ from this, one would have to greatly extend the already very involved computations and estimates of Sections~\ref{secSubCritRegime} and~\ref{secSuperCrit} below used to measure the sum of these contributions over all possible sets of defect edges. The two main technical steps to proving such a result using the methods of this paper would be to find an efficient method for bounding higher cumulants in the conditioned ERG model (it is this task that currently limits our results the most, see Section~\ref{secSuperCrit}) and to prove an analogue of the refinements of Janson's Inequality in~\cite{stark2018probability,mousset2020probability} to infinitely many terms and in the conditioned ERG model. Note the similarity in technical bottleneck to the disordered regime, again reflecting the approximate duality in the problem.

Ultimately, the disordered and ordered regimes must meet to cover all densities, and we have some predictions and questions about how this might happen.  We first conjecture that there is a sharp order--disorder phase transition, marked by a non-analyticity in an `order parameter' defined by the fraction of edges in the max cut of a triangle-free graph (see~\cite{bollobas1984evolution,bollobas2001scaling,borgs2001birth} for the use and discussion of order parameters in combinatorial problems).
\begin{conj}
\label{conjPhaseTransition}
There exists $c^* >0$ and a continuous function $\del: (c^*, \infty) \to (0,1/2]$ so that the following holds.
\begin{enumerate}
\item If $c < c^*$ and $m \sim c n^{3/2}$, then whp a graph $G$ drawn uniformly from $\cT(n,m)$ has a max cut of size $( 1/2 +o(1)) m$.  
\item If $c > c^*$ and $m \sim c n^{3/2}$, then whp a graph $G$ drawn uniformly from $\cT(n,m)$ has a max cut of size $( 1/2 + \del(c)+o(1)) m$. 
\end{enumerate}
\end{conj}

After the first version of this paper, the current authors  showed in~\cite{jenssen2024lower} the existence of a phase transition in the sense of a non-analyticity of the typical max-cut fraction and a non-analyticity in the large deviation rate function; the location, order, and uniqueness of the phase transition remain open.

We also ask about the nature of the phase transition, and whether this order parameter is continuous or discontinuous at the transition point. 
\begin{question}
\label{questPhaseTransitionOrder}
Assuming Conjecture~\ref{conjPhaseTransition}, is the order/disorder phase transition in triangle-free graphs first order or second order?  More precisely, we ask
\begin{enumerate}
\item  Is $\lim_{c \to c^* +}  \del(c)= 0$ (reflecting a second-order phase transition) or $\lim_{c \to c^* +}  \del(c)> 0$ (first-order)? 
\item With $G$ drawn uniformly from $\cT(n,m)$, $m \sim c^* n^{3/2}$, does the random variable $\frac{\textrm{MAX-CUT}(G)}{m}$ converge in probability to $\frac{1}{2}$, a constant random variable different than $\frac{1}{2}$, or to a random variable supported on two distinct values?
\end{enumerate}
\end{question}
Convergence of $\frac{\textrm{MAX-CUT}(G)}{m}$ to a random variable with support of size more than $2$ is not ruled out, but we conjecture this does not occur. 
 In fact, we conjecture it converges to a constant greater than $1/2$, in analogy to the first-order phase transition in the ferromagnetic Potts model~\cite{laanait1991interfaces,duminil2017lectures,helmuth2023finite}, but here the number of ordered ground states (roughly the number of balanced partitions) is exponential in $n$, presumably dominating the single disordered state at the critical point.

Moving beyond triangle-free graphs, the same kind of questions can be asked for many other combinatorial enumeration problems, and the methods introduced here can likely be used to provide some answers.

Mantel's Theorem is one of the first and most central results in extremal graph theory, and a first example of the more general class of Tur{\'a}n-type problems which includes Tur{\'a}n's generalization of Mantel's Theorem to $K_{r+1}$-free graphs~\cite{turan1941external}.  The extensions of Theorems~\ref{thmEKR} and~\ref{thmOst} to the $K_{r+1}$-free case were proved by  Kolaitis, Pr{\"o}mel, and Rothschild~\cite{kolaitis1987k} and Balogh,  Morris,  Samotij, and Warnke~\cite{balogh2016typical} respectively.  In particular it is shown in~\cite{balogh2016typical} that for $m \ge ( \theta_r+ \eps) n^{2 - \frac{2} {r+2 } } (\log n)^{1/ [\binom{r+1}{2} -1]  }$ (for some explicit constant $\theta_r$) almost all $K_{r+1}$-free graphs on $n$ vertices with $m$ edges are $r$-partite, and this immediately yields an asymptotic formula for the number of such graphs.    We ask if variations on the methods employed here can give asymptotic formulas at lower densities.

\begin{problem}
\label{probKrFree}
For $r \ge 3$ fixed and $c< \theta_r$, determine the asymptotic number of $K_{r+1}$-free graphs on $n$ vertices with $m$ edges when $m \sim c n^{2 - \frac{2} {r+2 } } (\log n)^{1/ [\binom{r+1}{2} -1]  } $.
\end{problem}

We note that one major obstacle to directly employing the methods of this paper to the $K_{r+1}$-free problem is that our use of cluster expansion in the triangle-free problem is in understanding independent sets in a certain graph created by fixing the defect edges; for $r \ge 3$, fixing the defect edges yields a hypergraph and one must understand independent sets in this hypergraph. However, as mentioned above convergence of the relevant cluster expansion is not known (and very likely fails).

Eventually one would like to understand the possible order--disorder phase transition in $K_{r+1}$-free graphs as well, and one can pose analogs of Conjecture~\ref{conjPhaseTransition} and Question~\ref{questPhaseTransitionOrder} in this setting. 

More generally, there are a large number of classes of combinatorial objects that can be represented by independent sets in hypergraphs, with hyperedges encoding forbidden substructures.  Tur{\'a}n-type problems fall into this class, along with problems about sum-free sets, $k$-AP-free sets, and Sidon sets in the integers/Abelian groups (see, e.g.,~\cite{balogh2018method} for discussion).  A common phenomenon in this kind of problem is that  substructures of extremal objects account for almost all or a constant fraction of such structures (analogous to Theorem~\ref{thmEKR}). See, e.g.,~\cite{lev2002cameron,green2004cameron,sapozhenko2008cameron,balogh2015independent,saxton2015hypergraph,morris2016number,balogh2017number} for examples.  This can continue to hold for sparser objects (as in Theorem~\ref{thmOst}), for example the results of~\cite{balogh2016typical} for $K_{r+1}$-free graphs and of~\cite{alon2014counting,alon2014refinement} for sum-free sets in Abelian groups and in the integers $[n]$.  What can one say about even sparser objects?

\begin{question}
\label{quesERGuniversal}
Can one characterize  defect distributions for combinatorial enumeration problems in the ordered regime more generally?
\end{question}
In particular one could ask if distributions analogous to conditioned exponential random graphs arise universally in such problems.

\subsection{A note on asymptotic notation}
All asymptotic notation is to be understood with respect to the limit $n \to \infty$. All implicit constants in the asymptotic notation $O, \Omega$ etc. will be absolute constants unless specified otherwise. Moreover, for two functions $f,g:\N\to \R$ we understand $f(n)\leq g(n)$ to mean that the inequality holds for $n$ sufficiently large. We write $f(n)\sim g(n)$ to denote that $\lim_{n \to \infty} f(n)/g(n)=1$.  We write $f(n)\ll g(n)$ to denote that $\lim_{n\to\infty}f(n)/g(n)=0$. We say a sequence of events $A_n$ holds `with high probability' (abbreviated `whp') if $\mathbb P(A_n) = 1-o(1)$.

\section{Results for \ER random graphs}
\label{secERresults}

In this section we state our results for the \ER random graph $G(n,p)$ conditioned on the event that $G$ is triangle-free.  The results mirror those in the Section~\ref{secIntro}, and in fact in our proofs we will first address $G(n,p)$ before translating  the results to the uniform distribution on $\cT(n,m)$ in Section~\ref{secFixedM}. 

To state the results we will make use of the \textit{hard-core model} of a random independent set from a graph.  Given a graph $G$ let $\cI(G)$ denote the set of all independent sets of $G$.  For an activity parameter $\lam\ge0$, the hard-core model $\mu_{G,\lam}$ is the distribution (or Gibbs measure) on $\cI(G)$ given by
\[ \mu_{G,\lam}(I) = \frac{ \lam^{|I|}}{Z_G(\lam)} \, ,\]
where $Z_G(\lam)  = \sum_{I \in \cI(G)} \lam^{|I|}$ is the hard-core partition function or the independence polynomial.

Let $\mathbb{P}_{n,p}$ be the measure associated to the Erd\H{o}s-R\'{e}nyi random graph $G(n,p)$, i.e.,
\[
\mathbb{P}_{n,p}(G) = p^{|G|}(1-p)^{\binom{n}{2}-|G|}\, .
\]

We are interested in determining the asymptotics of $\mathbb{P}_{n,p}(\cT)$, the probability that $G(n,p)$ is triangle-free, and in understanding the conditional measure $\mathbb{P}_{n,p}( \, \cdot \, |\cT)$. 

In fact both this probability and the conditional measure can be represented as a partition function and its associated Gibbs measure. Recall $Z(\lam) $ and $\mu_{\lam}$ from~\eqref{eqZndef},\eqref{eqmudef0}.  In fact  $Z(\lam)$ and $\mu_{\lam}$ define  a hard-core model on a $3$-uniform hypergraph   $\mathcal H_n$ with vertex set $\binom{[n]}{2}$ (representing edges of the complete graph $K_n$ on $n$ vertices) and $3$-uniform hyperedges $\{x,y,z\}$ when the edges $\{x,y,z\}$ form a triangle in $K_n$.  Independent sets (sets of vertices containing no edge) in $\mathcal I(\mathcal H_n)$ are exactly  triangle-free graphs in $\mathcal T(n)$ (see, e.g., the discussion in~\cite{balogh2018method}).  The independence polynomial of  $\mathcal H_n$  is
\begin{align*}
 \sum _{I \in \cI(\mathcal H_n)} \lam^{|I|}  = \sum_{G \in \cT(n)} \lam ^{|G|} = Z(\lam) \, .
\end{align*}

To see the connection to the distribution of $G(n,p)$ conditioned on triangle-freeness, we set $\lam = \frac{p}{1-p}$ and note the identity
\begin{align}\label{eqpviaZ}
\mathbb{P}_{n,p} (\cT) &= \sum_{G \in \cT(n)} p^{|G|} (1-p) ^{\binom{n}{2} -|G|}  
=(1-p)^{\binom{n}{2}} Z(\lam) \,.
\end{align}
Therefore to determine the asymptotics of $\mathbb{P}_{n,p} (\cT)$, it suffices to determine the asymptotics of $Z(\lam)$.  Moreover with  $\lam = \frac{p}{1-p}$,  $\mu_{\lam}$ is identical to the  distribution of $G(n,p)$ conditioned on $\cT$.  As described in Section~\ref{subsecMethods}, our main task is to understand this partition function and its Gibbs measure.

\subsection{Subcritical defect regime}
\label{secGnpSub}

Given $p$, let $\lam = \frac{p}{1-p}$ and define $q_0=q_0(\lam)$ so that 
\begin{equation}
\label{eqqdefGnp}
\frac{q_0}{1-q_0}= \lam e^{-\lam^2n/2} \, . 
\end{equation}
This is the same $q_0$ as in~\eqref{eqqdef} only now defined as a function of $p$ rather than $m$.
 Recall from Definition~\ref{defPartitionDist} that $\theta_\lam$ is a distribution on partitions of $[n]$.

Consider the distribution $\mu_{\lam,1}$ on $\cT(n)$ defined by the following algorithm.

\begin{algorithm}[ht]
  \caption{The distribution $\mu_{\lam,1}$\label{algMulam1}}
\begin{enumerate}[leftmargin=*]
\item Choose $(A,B)$ from the distribution $\theta_\lam$.
\item Choose edges $S \subseteq \binom{A}{2}$ and  $ T \subseteq \binom{B}{2}$ according to independent \ER random graphs on $A$ and $B$ with edge probability $q_0$.  If $S$ or $T$ contains a triangle, output the empty graph. Otherwise proceed to the next step.
\item Given $S,T$, choose $\Ec \subset A \times B$ according to the hard-core model on the graph $S \boxempty T$ at activity $\lam$.
\item Output the graph $G = S \cup T \cup \Ec$.
\end{enumerate}
\end{algorithm}

We have the following analogues of Theorems~\ref{thmGnmSubRestated}-\ref{thmNumberM}.
\begin{theorem}
\label{thmGnpSubCritLDprob}
Fix  $\eps>0$ and suppose $p \ge (1+\eps) \sqrt{\frac{\log n } {n}}$.  Then, with $\lam=p/(1-p)$,
\begin{align}
\mathbb{P}_{n,p}( \cT)
&\sim \frac{1}{2}\sqrt{\frac{\pi}{\lam}} \binom{n}{\lfloor n/2 \rfloor} (1+\lam)^{-n^2/4+n/2} \exp \left\{ \lam e^{-\lam^2n/2+\lam^3n} \frac{n^2}{4} + \lambda^5e^{-\lambda^2 n}\frac{n^4}{8} \right \} \, .
\end{align}
Moreover
\[ \| \mu_{\lam} - \mu_{\lam,1} \| _{TV} = o(1) \,.\]
\end{theorem}

\begin{theorem}
\label{thmGnpBipartite}
Fix $\eps>0$ and let $p \ge (1+\eps) \sqrt{ \frac{\log n}{n}} $.  Let $\hat X \sim \text{Bin}(\lfloor n^2/4 \rfloor, q_0)$ where $q_0$ is as in~\eqref{eqqdefGnp}. 
Let $X=X(G)$ be the minimum number of edges whose removal makes $G$ bipartite.  Suppose $G$ is drawn   from $G(n,p)$ conditioned on triangle-freeness.    Then  $\| X- \hat X\|_{TV} = o(1)$. 

In particular, if $p =   \sqrt{3+ \frac{\log \log n}{\log n} - \frac{t}{\log n}  } \sqrt{ \frac{\log n}{n}}$ for $t \in \R$, then
\[ \lim_{n\to \infty} \P_{n,p} ( G \in \cB\mid G \in \cT ) =  \exp \left(  - \frac{\sqrt{3}}{4} e^{t/2} \right ) \,.   \]
\end{theorem}

\begin{theorem}
\label{thmGnpColoring}
Fix $\eps>0$. 
\begin{itemize}
\item If $ (\sqrt{2}+\eps) \sqrt{ \frac{\log n} {n}}\le p  \le  (\sqrt{3}-\eps)\sqrt{ \frac{\log n} {n}}$ then $\mathbb{P}_{n,p}[\chi(G) = 3  | \cT] = 1-o(1)$.
\item If $ (1+\eps)\sqrt{ \frac{\log n} {n}} \le p  \le (\sqrt{2}-\eps)\sqrt{ \frac{\log n} {n}}$ then $\mathbb{P}_{n,p}[\chi(G) = 4  | \cT] = 1-o(1)$.
\end{itemize}
\end{theorem}

\subsection{Supercritical defect regime}
\label{secGnpSuper}
Recall the definitions of $q_1, q_2, \psi$ from~\eqref{eqq0q1Def}-\eqref{eqPsiDef} and consider them now as functions of $p$ via $\lam = p/(1-p)$.

Consider the following distribution $\mu_{\lam,2}$ on $\cT(n)$. 
\begin{algorithm}[ht]
  \caption{The distribution $\mu_{\lam,2}$\label{algMulam2}}
\begin{enumerate}[leftmargin=*]
\item Choose $(A,B)$ from the distribution $\theta_\lam$.
\item Sample $S \subseteq \binom{A}{2}$ according to $G(A,q_2,\psi)$ and sample  $ T \subseteq \binom{B}{2}$ according to $G(B,q_2,\psi)$ with $S,T$ independent. 
\item Given $S,T$, choose $\Ec \subset A \times B$ according to the hard-core model on the graph $S \boxempty T$ at activity $\lam$.
\item Output $G = S \cup T \cup \Ec$. 
\end{enumerate}
\end{algorithm}

We have the following analogues of Theorems~\ref{thmP2ERG}-\ref{thm4color}.

\begin{theorem}
\label{thmGnpStructureSuper}
Fix $\eps \in (0, 1/14]$ and suppose $p \ge (1-\eps) \sqrt{ \frac{\log n}{n}}$.  Let $\lam=p/(1-p)$, then
\begin{align}
\mathbb{P}_{n,p}( \cT) \sim & \frac{1}{2}\sqrt{\frac{\pi}{\lam}} \binom{n}{\lfloor n/2 \rfloor}\left[(1+\lam)(1-q_2) \right]^{-n^2/4+n/2}\times\\ 
& \exp\left\{  \frac{1}{64} \lam^6 n^5 q_0^2 - \frac{1}{64}\lam^6 n^6 q_0^3 -\frac{1}{24}n^3q_0^3+ \frac{1}{64}\lam^4n^4q_0^2-\frac{1}{6}\lam^4n^5q_0^3-\frac{1}{2}\lam^4n^4q_0^2 \right\}\, .
\end{align}
Moreover,
\[ \| \mu_{\lam} - \mu_{\lam,2} \| _{TV} = o(1) \,.\]
\end{theorem}

\begin{theorem}
\label{thm4colorGnp}
Fix $\eps \in (0,1/14]$ and let  
$$ p \sim (1-\eps)\sqrt{\frac{\log n}{n}}\, ,$$
then $\mathbb{P}_{n,p}[\alpha(G) = o(n)  | \cT] = 1-o(1)$ where $\alpha(G)$ denotes 
  the independence number of $G$.
\end{theorem}

\subsection{Critical defect regime}
\label{secGnpCrit}

Theorem~\ref{thmGnpStructureSuper}  covers the entire regime $p \ge  (1-\eps) \sqrt{ \frac{\log n}{n}} $, reducing to the results of Theorem~\ref{thmGnpSubCritLDprob}  when $p \ge  (1+\eps) \sqrt{ \frac{\log n}{n}}$.  Here we examine the implications on the structure in the critical regime in which $(1-\eps) \sqrt{ \frac{\log n}{n}} \le p \le  (1+\eps) \sqrt{ \frac{\log n}{n}}$ with $\eps =o(1)$.

We have the following analogue of Theorem~\ref{thmGiantComponent}.  

\begin{theorem}
\label{thmGnpGiant} 
Let $G$ be drawn from $\mu_\lam$ conditioned on $G$ having a strongly balanced max cut $(A,B)$.  
\begin{enumerate}
\item If $\lam$ is such that $q_0 = \frac{2}{n} - \frac{\omega(n)}{n^{4/3}}$, with $1 \ll \omega(n) \ll n^{1/3}$,  then whp the largest connected component  of $G_A$ is of size $\Theta(n^{2/3} \omega^{-2} \log (\omega))$.  
\item If $\lam$ is such that $q_0 = \frac{2}{n} + \frac{\omega}{n^{4/3}}$, with $ \omega \in \mathbb R$  constant,  then whp the  largest connected component  of $G_A$ is of size  $\Theta(n^{2/3})$.
\item If $\lam$ is such that $q_0 = \frac{2}{n} + \frac{\omega(n)}{n^{4/3}}$, with $1 \ll \omega(n) \ll n^{1/3}$,  then whp the largest connected component of $G_A$ is of size $(2+o(1))  \cdot \omega \cdot (n/2)^{2/3}$. 
\item If $\lam$ is such that $q_0 = \frac{c}{n}$ with $c>2$ fixed, then whp the largest connected component of $G_A$ is of size $\Theta(n)$.
\item  If $\lam$ is such that $q_0= (1+\eps)\frac{2 \log n}{n} $ for $\eps>0$ constant, then whp $G_A$ is connected, while if $\lam$ is such that $q_0 \le (1-\eps)\frac{2 \log n}{n} $ for $\eps>0$ constant, then whp $G_A$ is not connected.
\end{enumerate}
Moreover, these results also hold for the graph $G_B$.  
\end{theorem}

\subsection{General defect distribution}
\label{secGnpDefect}

Here we give a rough structural description of $\mu_{\lam}$ for a wider range of parameters: when $\lam \ge \omega/\sqrt{n}$ for some large constant $\omega$.  This rough description will be the starting point for proving the much more detailed results above, but gives some interesting information   on its own. 

The rough description requires a few definitions which will be useful later as well.
\begin{defn}
Given a graph $G$ and a partition $(A,B)$ of $V(G)$, we call the graph $G_A\cup G_B$ the \emph{defect graph} of $G$ (wrt $(A,B)$). In an abuse of terminology we will sometimes refer to the pair $(G_A, G_B)$ as the defect graph. 
\end{defn} 

Let $\alpha=1/(96e^3)$. For a partition $(A,B)$ of $[n]$ let
\begin{align}\label{eqTABwDef}
\cT_{A,B,\lam}^{\textup{w}}=\{G\in \cT(n) : \Delta(G_A\cup G_B)\leq \alpha/ \lam\}\, .
\end{align}
That is, $\cT_{A,B,\lam}^{\textup{w}}$ is the set of graphs $G$ whose defect graph wrt $(A,B)$ has maximum degree at most $\alpha/\lam$. We will eventually restrict our attention to defect graphs with a much stronger degree bound and so the superscript `w' in the notation refers to the fact that the defect graph is `weakly sparse'.  We also introduce notation for the set of weakly sparse defect graphs and the weakly sparse restricted partition function.
\begin{align}\label{eqDABwDef}
 \cD_{A,B,\lam}^\textup{w}&:= \left\{(S,T): S\subseteq\binom{A}{2}, T\subseteq\binom{B}{2}, \Delta(S\cup T)\leq \alpha/\lam, S\cup T \text{ triangle-free} \right\} \\
 &= 
 \{(G_A, G_B): G\in  \cT_{A,B,\lam}^{\textup{w}}\}\, .
\end{align}
\begin{align}\label{ZABwDef}
Z_{A,B}^{\textup{w}}(\lam) = \sum_{ G\in \cT_{A,B,\lam}^{\textup{w}}} \lam^{|G|}\, .
\end{align}

 We also define a set of `weakly balanced' partitions of $[n]$.
 \begin{defn}\label{defWeakBalance}
Call a partition $(A,B)$ of $[n]$ \emph{weakly balanced} if $\big ||A|-|B| \big |\leq n/10$. Let $\Pi$ denote the set of all partitions of $[n]$ and let $\Pi_{\textup{weak}}\subseteq \Pi$ denote the set of all weakly balanced partitions.
\end{defn}

Given these definitions, Algorithm~\ref{algGendefect} defines a distribution $\mu_{\textup{weak},\lam}$ on $\mathcal T(n)$.

\begin{algorithm}[ht]
  \caption{The distribution $\mu_{\textup{weak},\lam}$\label{algGendefect}}
\begin{enumerate}[leftmargin=*]
\item Pick $(A,B)\in \Pi_{\textup{weak}}$ with probability proportional to $Z_{A,B}^{\textup{w}}(\lam)$.
\item Sample $(S,T)\in  \cD_{A,B,\lam}^\textup{w}$ with probability proportional to $\lam^{|S|+|T|}Z_{S\boxempty T}(\lam)$.
\item Select $\Ec\subseteq A\times B$ according to the hard-core model on $S\boxempty T$ at activity $\lam$. 
\item Output $S\cup T\cup \Ec$.
\end{enumerate}
\end{algorithm}

\begin{theorem}\label{thmGenDef}
There exists $\omega>0$ such that if $\lam\geq \omega/\sqrt{n}$ then
\[
\|\mu_\lam - \mu_{\textup{weak},\lam}\|_{TV}=o(1)\, .
\]
Moreover, a graph $G$ drawn according to $\mu_{\textup{weak},\lam}$ has a unique weakly balanced max cut whose defect graph has maximum degree at most $\alpha/\lam$.
\end{theorem}

The main content of Theorem~\ref{thmGenDef} is that the defect graph selected at Step 2 of Algorithm~\ref{algGendefect} has small maximum degree which allows us to understand the partition function $Z_{S\boxempty T}(\lam)$  (and hence also the measure at Step 2) in detail via  cluster expansion (see Section~\ref{secTools} for details on cluster expansion).

\section{Proof roadmap}
\label{secStatPhysFormulation}
In this section we provide a roadmap for the proofs to come and for the remainder of the paper.  In Section~\ref{secNotationReference} we collect definitions and notation in one place to serve as a reference sheet for the reader.

Recall that in Section~\ref{secERresults} we reformulated the main problems for $G(n,p)$ in terms of   a statistical physics partition function $Z(\lam)$ and  its associated Gibbs measure $\mu_\lam$.  Similarly, the main problems for $\cT(n,m)$ are about  the coefficient of $\lam^m$ in the polynomial $Z(\lam)$ and the distribution of $\mu_\lam$ conditioned on the event $\{|G| =m \}$.

To compute the asymptotics of $Z(\lam)$  we will make a number of successive  approximations, culminating in an approximation by a sum of partition functions of models that we can analyze.  We now describe this sequence of approximations, giving a roadmap for the rest of the paper.   We will always assume that
\begin{equation}
    \label{lamAssumption}
    \lam \le 2 \sqrt{\frac{\log n}{n}} \quad  \text{and} \quad  m \le \frac{1}{2} n^{3/2} \sqrt{\log n} \, ,
\end{equation} 
since all results for  larger values are covered by~\cite{osthus2001almost}.

\subsection{Reduction to graphs with a dense cut}
We begin with a classical theorem of \L uczak~\cite{luczak2000triangle} which states that almost all triangle-free graphs on $n$ vertices and $m\geq Cn^{3/2}$ edges admit a dense cut. 

\begin{theorem}[\cite{luczak2000triangle}]\label{thm:luczak}
For all $\delta>0$ there exists $C = C(\delta)>0$ so that if $m \geq Cn^{3/2}$, then almost all $G\in\cT(n,m)$ admit a cut of size at least $(1-\delta) m$.
\end{theorem}

We will  need a slight refinement of  Theorem~\ref{thm:luczak} (Theorem~\ref{thm:luczakrefined} below) which follows from combining~\cite[Proposition 6.1]{balogh2016typical} and~\cite[Claim 6.2]{balogh2016typical}. To state the result we need a  definition.

\begin{defn}\label{def:dominating}
Let $G$ be a graph and $(A,B)$ a partition of its vertex set. We say $(A,B)$ is a \emph{dominating} cut of $G$ if
\[
d_G(v,B)\geq d_G(v,A) \text{ for all } v\in A
\]
and similarly with $A,B$ swapped.
\end{defn}
Recall also the definition of a weakly balanced partition from Definition~\ref{defWeakBalance}.

\begin{theorem}\label{thm:luczakrefined}
For all $\delta>0$ there exists $C = C(\delta)>0$ so that if $m \geq Cn^{3/2}$, then almost all $G\in\cT(n,m)$ admit a dominating, weakly balanced cut of size at least $(1-\delta) m$.
\end{theorem}

We fix a sufficiently small  constant $\delta>0$ (to be specified later) and consider the set 
\begin{align}\label{eqcLdef}\\
{\small
\mathcal{L}(n,\lam):= \left\{G\in \cT(n): G \text{ admits a weakly balanced, dominating cut of size $\geq |G| - 2\delta \lambda n^2$}\right\}\, .}
\end{align}

For a subset $\mathcal{R}\subseteq \cT(n)$ and $\lam>0$ we let
\[
Z(\mathcal{R}, \lam):=  \sum_{G\in \mathcal{R}} \lam^{|G|}\, 
\]
be the restriction of the partition function $Z(\lam)$ to $\mathcal R$. 
We also define  the restricted Gibbs  measure $\mu_{\mathcal{R},\lam}$ on $\mathcal{R}$ by 
\begin{align}\label{eqmuRdef}
\mu_{\mathcal{R},\lam}(G)=\frac{\lam^{|G|}}{Z(\mathcal{R}, \lam)}\, .
\end{align}

Our first step is to approximate $Z(\lam)$ by $ Z(\cL, \lam)$ where $\cL=\cL(n,\lam)$.
The following result is a simple consequence of Theorem~\ref{thm:luczakrefined} which we prove in Section~\ref{secMorrisOPT}.
\begin{prop}\label{lemZerothApprox}
There exists a constant $\omega>0$ such that if $\lam\geq \omega/\sqrt{n}$ then, letting $\cL=\cL(n,\lam)$,
\[
Z(\lam)\sim Z(\cL, \lam)\, ,
\]
and 
\[
\|\mu_{\lam} - \mu_{\cL,\lam}\|_{TV}=o(1)\, .
\]
\end{prop}

\subsection{Clustering by ground states}
We next  approximately partition the set of triangle-free graphs according to  which partition $(A,B)$ of $[n]$ they align best with; that is,  according to their max-cut partition.  This partitioning will only be  approximate because some graphs (e.g., the empty graph) have multiple partitions achieving their max cut, but we will show that the contribution of these graphs to $Z(\lam)$ is negligible.

To view the problem from the statistical physics perspective, we identify a collection of `ground states' and ground-state graphs, whose contribution to $Z(\lam)$ is easy to calculate and for large enough $\lam$ makes up almost all of $Z(\lam)$.  In this case, each partition $(A,B)$ of $[n]$ gives rise to a collection of ground-state graphs consisting of all graphs $G$ that are bipartite with bipartition $(A,B)$.   The contribution of this collection is $(1+\lam)^{|A| \cdot |B|}$ since each of the $|A|\cdot |B|$ possible crossing edges can be included or not included in such a bipartite graph.  It is not hard to show that for $\lam \gg \log n / n$,   graphs obtained in this manner from more than one partition $(A,B)$ have a negligible contribution, and so we can write
\[ Z(\lam) \ge (1+o(1)) \sum_{(A,B)} (1+\lam)^{|A| \cdot |B|} \,. \] 

An easy modification of the proof of Osthus, Pr\"omel, and Taraz in~\cite{osthus2003densities} shows that when $\lam \ge (\sqrt{3}+\eps)\sqrt{\frac{\log n}{n}}$, then this lower bound is tight: $ Z(\lam) = (1+o(1)) \sum_{(A,B)} (1+\lam)^{|A| \cdot |B|}$, and moreover, when  $\lam = (\sqrt{3}- \eps)\sqrt{\frac{\log n}{n}}$, then $ Z(\lam) \gg   \sum_{(A,B)} (1+\lam)^{|A| \cdot |B|}$, and so to understand $Z(\lam)$ at these smaller densities we must take into account graphs that do not arise from any ground state.
On the other hand Proposition~\ref{lemZerothApprox} shows that one only needs to consider graphs that are `close' to a ground state.

This marks our point of departure from the previous literature, and the main contribution of this paper: how to account precisely for these near-ground-state graphs.

We first show that a typical sample from $\mu_{\cL,\lam}$ (and hence also $\mu_{\lam}$) has a unique weakly balanced max cut such that the defect edges form a graph of maximum degree at most $\alpha/\lambda$ where we recall that  $\alpha = 1/(96e^3)$. Recall from~\eqref{ZABwDef} that we let $Z_{A,B}^{\textup{w}}(\lam)=Z( \cT_{A,B,\lam}^{\textup{w}}, \lam)$.

Define
\[
Z_{\textup{weak}}(\lam) = \sum_{(A,B)\in\Pi_{\textup{weak}}} Z_{A,B}^{\textup{w}}(\lam)\, .
\]

Taken in conjunction with Proposition~\ref{lemZerothApprox} the following proposition is a refinement of Theorem~\ref{thmGenDef}.
\begin{prop}
\label{propGroundStateRefinedAlt}
There exists a constant $\omega>0$ such that if $\lam\geq \frac{\omega}{\sqrt{n}}$ then, letting $\cL=\cL(n,\lam)$,
\begin{align}\label{eqGroundStateRefined}
 Z(\cL, \lam) = \left(1+O\left(e^{-\sqrt{n}} \right) \right) Z_{\textup{weak}}(\lam)\, ,
\end{align}
and 
\[
\|\mu_{\cL,\lam} - \mu_{\textup{weak},\lam}\|_{\textup{TV}}= O\left(e^{-\sqrt{n}}\right)\, .
\]
Moreover, whp a graph $G$ drawn according to $\mu_{\textup{weak},\lam}$ has a unique weakly balanced max cut whose defect graph has maximum degree at most $\alpha/\lam$.
\end{prop} 

The proof of Proposition~\ref{propGroundStateRefinedAlt} is a modification of the strategy of~\cite{balogh2016typical} (specialized to triangle-free graphs) and is carried out in Section~\ref{secMorrisOPT}.  We will soon see that the maximum degree bound in~Proposition~\ref{propGroundStateRefinedAlt} is crucial for our approach.

Given Proposition~\ref{propGroundStateRefinedAlt}, our next goal is to understand the partition function $Z_{A,B}^{\textup{w}}(\lam)$ for  any weakly balanced partition $(A,B)$.
 To this end, it will be useful to consider the contribution to $Z_{A,B}^{\textup{w}}(\lam)$ from all $G$ with a fixed defect graph. Indeed, let us fix  edge sets $S\subseteq\binom{A}{2}$, $T\subseteq\binom{B}{2}$.
Recall from Section~\ref{secIntroMainResults} that $S \boxempty T$ denotes the Cartesian product of the graphs $(A, S)$, $(B, T)$.
 Whenever we write  $S \boxempty T$ below, the underlying partition $(A,B)$ will be clear from the context. Moreover, in an abuse of notation we will often identify the edge set $S$ with the graph $(A,S)$.

If we fix $(A,B)$, sample $G$ from $\mu_{\lam}$, and condition on the event that $G_A= S$ and $G_B= T$, then $E(G)\cap (A\times B)$ is distributed according to the hard-core model on the graph $S\boxempty T$ at activity $\lambda$.   This observation  is formalized in the following lemma which motivates the appearance of $Z_{S\boxempty T}(\lam)$ in Step 2 of Algorithm~\ref{algGendefect}. 

\begin{lemma}
\label{lemSTgraphLemma}
Let $(A, B)$ be a partition of $[n]$ and suppose $S\subseteq \binom{A}{2}, T \subseteq \binom{B}{2}$ such that $S\cup T$ is triangle-free.  Let $\cG(S,T)$ be the set of triangle-free graphs $G$ so that $G_A= S$ and $G_B= T$.  Then
\[ \sum_{G \in \cG(S,T)} \lam^{|G|} =  \lam^{|S| + |T|} Z_{S\boxempty T} (\lam) \,,\]
where $Z_{S\boxempty T} (\lam)$ is the hard-core partition function on the graph $S\boxempty T$.
\end{lemma}
\begin{proof}
In what follows we identify the vertex $(u,v)\in V(S\boxempty T)=A \times B$ with the edge $\{u,v\}\in \binom{[n]}{2}$.
The proof follows from the observation that if $I$ is an independent set in the graph $S\boxempty T$, then  the graph on $[n]$ with edge set $S \cup T \cup I$ is a triangle-free graph in $\cG(S,T)$, and likewise for any $G \in \cG(S,T)$, $E(G) \cap (A\times B)$ forms an independent set in $S\boxempty T$, giving a  one-to-one correspondence.
\end{proof}

With this observation in hand, we may write the identity
\[
Z_{A,B}^{\textup{w}}(\lam) = \sum_{(S,T)\in \cD_{A,B,\lam}^\textup{w}} \lam^{|S|+|T|} Z_{S\boxempty T} (\lam)\, ,
\]
where we recall from~\eqref{eqDABwDef} that $ \cD_{A,B,\lam}^\textup{w}=\{(G_A, G_B): G\in  \cT_{A,B,\lam}^{\textup{w}}\}$.

We can now see the crucial role of Proposition~\ref{propGroundStateRefinedAlt}. 
Observe that if $(S,T)\in  \cD_{A,B,\lam}^\textup{w}$, then $\Delta(S\cup T)\leq \alpha/\lam$ by definition, and so  $\Delta(S\boxempty T)\leq \Delta(S) + \Delta(T)\leq 2\alpha/\lam$. With the choice $\alpha=1/(96e^3)$,  $\lam\leq 1/(4e\Delta(S\boxempty T))$ and so the hard-core model on $S\boxempty T$ at activity $\lam$ is \emph{subcritical} in a sense to be made precise below in Lemma~\ref{lemClusterTail}. This allows us to understand $Z_{S\boxempty T}(\lam)$ via the cluster expansion and obtain a sequence of refinements of Proposition~\ref{propGroundStateRefinedAlt}. 

\subsection{Reducing to sparser defect graphs}
The first refinement of Proposition~\ref{propGroundStateRefinedAlt} is a strengthening of the sparsity condition on the defect graph. 

Throughout the paper, given a partition $(A,B)$, we let $a,b$ denote $|A|,|B|$ respectively. 
\begin{defn}\label{defsparse} 
Let $(A,B)\in \Pi$,
 $\lam>0$, $q_A/(1-q_A)=\lam e^{-\lam^2b}$, $q_B/(1-q_B)=\lam e^{-\lam^2a}$. Moreover let $q=\max\{q_A, q_B\}$, $\Delta_{A,B,\lam}=50\max\{qn ,\log n\}$ and $K_{A,B,\lam}=50\max\{qn^2, \log n\}$. We call a graph $H\subseteq \binom{A}{2}\cup\binom{B}{2}$ \emph{$\lam$-sparse} if
 \begin{enumerate}
     \item $H$ is triangle-free,
     \item $\Delta(H)\leq \Delta_{A,B,\lam}$,
     \item $\max\{|H_A|, |H_B|\}\leq K_{A,B,\lam}$.
 \end{enumerate}
\end{defn}

\begin{remark}
Throughout the paper, we identify the pair $(S,T)\in 2^{\binom{A}{2}}\times 2^{\binom{B}{2}}$ with the graph $S \cup T$ and use them interchangeably. As such, we call a pair $(S,T)$ $\lam$-sparse if $S\cup T$ is $\lam$-sparse. 
\end{remark}

Let 
\[
\cT_{A,B,\lam}:=\{G\in \cT(n) : (G_A , G_B) \text{ is $\lam$-sparse}\}\, ,
\] 
the set of $G$ whose defect graph wrt $(A,B)$ is $\lam$-sparse, and let
\[
 \cD_{A,B,\lam}:=\{(G_A, G_B): G\in  \cT_{A,B,\lam}\}\, ,
\]
the set of $\lam$-sparse defect graphs. Define also
\[
Z_{A,B}(\lam):= Z(\cT_{A,B,\lam},\lam)= \sum_{G\in \cT_{A,B,\lam}}\lam^{|G|}\, .
\]

We note that for $\lam\geq \frac{\omega}{\sqrt{n}}$ and $(A,B)\in \Pi_{\textup{weak}}$ we have $\Delta_{A,B,\lam}=o_\omega(1)\cdot \alpha/\lam$ and so when $\omega$ is large, the restriction on the sparsity of the defect graph in the definition of $\cT_{A,B,\lam}$ is stronger than that of $\cT_{A,B,\lam}^{\textup{w}}$. The following lemma allows us to refine the approximation of $Z(\cL, \lam)$ in Proposition~\ref{propGroundStateRefinedAlt}.
\begin{prop}
\label{propMaxDegBootstrap}
There exists $\omega>0$ such that if $\lam\geq \frac{\omega}{\sqrt{n}}$, and $(A,B)\in \Pi_{\textup{weak}}$, then
 \begin{align}\label{eqMaxDegBootstrap}
Z_{A,B}(\lam)=\left(1+O(n^2e^{-\Delta/2})\right) Z^{\textup{w}}_{A,B}(\lam)\, ,
\end{align}
where $\Delta=\Delta_{A,B,\lam}$, is as in Definition~\ref{defsparse}.
\end{prop} 

\subsection{Reducing to strongly balanced partitions}
Our next refinement of Proposition~\ref{propGroundStateRefinedAlt} comes from showing that we only need to consider strongly balanced partitions $(A,B)$ (see Definition~\ref{defstrongbalance}) provided $\lam$ is sufficiently large. For this we need an intermediate notion of balancedness. 
\begin{defn}\label{defModBalanced}
We call a partition $(A,B)\in \Pi$ \emph{$\lam$-moderately balanced} if
\[
\big ||A|-|B| \big|\leq M_\lam:=\max\{ne^{-\lam^2n/2}, n^{1/2} \}(\log n)^2  \, , 
\]
and we let $\Pi_{\textup{mod},\lam}$ denote the set of all $\lam$-moderately balanced partitions. 
\end{defn}
Define the measures $\mu_{\textup{mod},\lam}$ and  $\mu_{\textup{strong},\lam}$ on $\cT(n)$ via the following processes.

\begin{algorithm}[ht]
  \caption{The distribution $\mu_{\textup{mod},\lam}$ (resp. $\mu_{\textup{strong},\lam}$) \label{algMuModStrong}}
\begin{enumerate}[leftmargin=*]
\item Pick $(A,B)\in \Pi_{\textup{mod},\lam}$ (resp. $\Pi_{\textup{strong}}$) with probability proportional to $Z_{A,B}(\lam)$.
\item Pick $(S,T)\in\cD_{A,B,\lam}$  with probability proportional to $\lam^{|S|+|T|}Z_{S\boxempty T}(\lam)$.
\item Select $\Ec\subseteq A\times B$ according to the hard-core model on $S\boxempty T$ at activity $\lam$. 
\item Output $S\cup T\cup \Ec$.
\end{enumerate}
\end{algorithm}

Finally let
\[
Z_{\textup{mod}}(\lam) = \sum_{(A,B)\in\Pi_{\textup{mod},\lam}} Z_{A,B}(\lam) \quad\text{and}\quad Z_{\textup{strong}}(\lam) = \sum_{(A,B)\in\Pi_{\textup{strong}}} Z_{A,B}(\lam)\, .
\]
We prove the following two propositions covering overlapping ranges of $\lam$.

\begin{prop}\label{lemZmodsimsum}
Fix $c>0$ and let $\lam\geq c \sqrt{\frac{\log n }{n}}$. Then
\begin{align}\label{eqZsimZmod} 
Z_{\textup{mod}}(\lam)=\left(1 + O\left(n^{-3}\right)\right) Z_{\textup{weak}}(\lam)\, ,
\end{align}
and
\begin{align}\label{eqmusimmumod}
\|\mu_{\textup{mod},\lam} - \mu_{\textup{weak},\lam}\| _{TV}=O\left(n^{-3/2}\right)\, .
\end{align}
\end{prop}

\begin{prop}\label{propGroundStateStrong}
For $\lam\geq \frac{13}{14} \sqrt{\frac{\log n }{n}}$,
\begin{align}\label{eqStrongtoMod}
Z_{\textup{strong}}(\lam)=\left(1 + O\left(n^{-3}\right)\right) Z_{\textup{mod}}(\lam)\, ,
\end{align}
and
\begin{align}\label{eqStrongtoModTV}
\|\mu_{\textup{strong},\lam} - \mu_{\textup{mod},\lam}\| _{TV}=O\left(n^{-3/2}\right)\, .
\end{align}
\end{prop}
It will be convenient to record the following immediate corollary of Propositions~\ref{lemZerothApprox},~\ref{propGroundStateRefinedAlt},~\ref{lemZmodsimsum} and~\ref{propGroundStateStrong}.

\begin{cor}\label{corZmodsimsum}
For $\lam\geq \frac{13}{14} \sqrt{\frac{\log n }{n}}$,
\begin{align*}
Z(\lam)\sim Z_{\textup{strong}}(\lam)\, ,
\end{align*}
and
\begin{align*}
\|\mu_{\lam} - \mu_{\textup{strong},\lam}\| _{TV}=o(1)\, .
\end{align*}
\end{cor}

Proving these propositions and understanding $Z_{\textup{strong}}(\lam)$ will come down to computing asymptotics of $Z_{A,B}(\lam)$ for $\lam$-moderately balanced $(A,B)$, which we will do with the use of the cluster expansion, described in the next section.

The transfer of these results to results on $\cT(n,m)$ is done in Section~\ref{secFixedM}  using the following identity, valid for any $\lam>0$:
\begin{align*}
|\cT(n,m)| =   \frac{ Z(\lam)}{\lam^m} \cdot \mu_{\lam} ( \{|G| = m \}). 
\end{align*}
To use this to determine the asymptotics of $|\cT(n,m)|$, we will choose a value of $\lam$ so that the mean number of edges in a sample from $\mu_\lam$ is close to $m$.
Computing asymptotics of $\mu_{\lam} ( \{|G| = m \})$ will be made possible again by the fact that the hard-core measures obtained from Proposition~\ref{propGroundStateRefinedAlt}  are subcritical, which allows the use of a local central limit theorem to estimate this probability.

\subsection{Notation and definitions}
\label{secNotationReference}

Here we collect some key notation and definitions in one section which the reader can use as a reference sheet.  

As our argument evolves, we will consider three increasingly strict notions of a balanced  partition of the vertices of a graph $G$.

\begin{defn}\label{defVarBalanced}
We call a partition $(A,B)$ of $[n]$:
\begin{enumerate}
\item \emph{Weakly balanced} if $\big ||A|-|B| \big |\leq n/10$.
\item \emph{$\lam$-moderately balanced} if
\[
\big||A|-|B| \big|\leq M_\lam=\max\left\{ne^{-\lam^2n/2}, n^{1/2} \right\}(\log n)^2  \, ,
\]
\item \emph{Strongly balanced} if $\big ||A|-|B| \big |\leq 10(n \log n)^{1/4}$.
\end{enumerate}
Let $\Pi_\textup{weak}, \Pi_{\textup{mod},\lam}, \Pi_\textup{strong}$ and $\Pi$ denote the set of weakly balanced, $\lam$-moderately balanced, strongly balanced, and all partitions of $[n]$ respectively. 
\end{defn}

Given $(A,B)\in\Pi$, we use the convention $a=|A|, b=|B|$. For $\lam>0$, we define $q_A=q_A(\lam)$ and $q_B=q_B(\lam)$ via
\begin{align*}
\frac{q_A}{1-q_A}= \lam e^{-b\lam^2} \quad\text{and}\quad \frac{q_B}{1-q_B}= \lam e^{-a\lam^2}\, .
\end{align*}
Throughout the paper we let $q=q_{A,B,\lam}=\max \{q_A, q_B\}$. When we use the notation $q$, the partition $(A,B)$ and $\lam$ will be clear from the context.

Recall that for a subset $\mathcal{R}\subseteq \cT(n)$ and $\lam>0$ we let
\[
Z(\mathcal{R}, \lam)=  \sum_{G\in \mathcal{R}} \lam^{|G|}\, ,
\]
and let $\mu_{\mathcal{R},\lam}$ denote the measure on the set $\mathcal{R}$ defined by 
\begin{align}
\mu_{\mathcal{R},\lam}(G)=\frac{\lam^{|G|}}{Z(\mathcal{R}, \lam)}\, .
\end{align}
We denote $Z(\mathcal{T}(n), \lam),  \mu_{\mathcal{T}(n), \lam}$ simply by $Z(\lam), \mu_{\lam}$ respectively. 

Recall that for $(A,B)\in \Pi$, and graph $G$, we let $G_A, G_B$ denote the respective subgraphs of $G$ induced by vertex sets $A$,$B$. Throughout the paper we fix $\alpha=1/(96e^3)$. Let 
\begin{align}\label{eqTABwDefAgain}
\cT_{A,B,\lam}^{\textup{w}}=\{G\in \cT(n) : \Delta(G_A\cup G_B)\leq \alpha/ \lam\}\, ,
\end{align}
and 
\begin{align*}
\cT_{A,B,\lam}=\{G\in \cT(n) : \Delta(G_A\cup G_B)\leq \Delta_{A,B,\lam}, |G_A|, |G_B|\leq K_{A,B,\lam}\}\, ,
\end{align*}
where 
\[
 \Delta_{A,B,\lam}= 50\max\{qn, \log n\} \quad\text{and}\quad K_{A,B,\lam}= 50\max\{qn^2, \log n\}  \,.
\]
We also let
\begin{align}\label{eqDABwDefAgain}
 \cD_{A,B,\lam}:=\{(G_A, G_B): G\in  \cT_{A,B,\lam}\}\quad\text{and}\quad\cD_{A,B,\lam}^{\textup{w}}:=\{(G_A, G_B): G\in  \cT_{A,B,\lam}^{\textup{w}}\}\, ,
\end{align}
and refer to the edges of $G_A\cup G_B$ as the defect edges of $G$ (with respect to $(A,B)$). We will at times abuse notation and identify the pair $(G_A, G_B)$ with the graph $G_A\cup G_B$.

For $G\in \cT$, we let $c_{\textup{weak},\lam}(G)$ denote the number of weakly balanced partitions $(A,B)$ such that $(G_A, G_B)\in \cD^{\textup{w}}_{A,B,\lam}$. We let $c_{\textup{mod},\lam}(G), c_{\textup{strong},\lam}(G)$ denote the number of $\lam$-moderately/ strongly balanced partitions $(A,B)$ such that $(G_A, G_B)\in \cD_{A,B,\lam}$ respectively.

For ease of notation we let 
\begin{align}\label{eqZABdef}
Z_{A,B}(\lam)=Z(\cT_{A,B,\lam}, \lam) \quad{\text{and}}\quad Z_{A,B}^{\text{w}}(\lam)=Z(\cT^{\text{w}}_{A,B,\lam}, \lam)\, ,
\end{align}
and 
\begin{align}\label{eqmuABdef}
\mu_{A,B,\lam}=\mu_{\cT_{A,B,\lam},\lam} \quad{\text{and}}\quad \mu^{\text{w}}_{A,B,\lam}=\mu_{\cT^{\text{w}}_{A,B,\lam},\lam}  \, .
\end{align}

We define
\begin{align}\label{eq:ZweakDef}
Z_{\textup{weak}}(\lam) = \sum_{(A,B)\in\Pi_{\textup{weak}}} Z_{A,B}^{\textup{w}}(\lam)\, ,
\end{align}
and let
\[
Z_{\textup{mod}}(\lam) = \sum_{(A,B)\in\Pi_{\textup{mod},\lam}} Z_{A,B}(\lam) \quad\text{and}\quad Z_{\textup{strong}}(\lam) = \sum_{(A,B)\in\Pi_{\textup{strong}}} Z_{A,B}(\lam)\, .
\]

Recall (from Section~\ref{secIntroMainResults}) that for $S\subseteq \binom{A}{2}$ and $T\subseteq \binom{B}{2}$, we use $S\boxempty T$ to denote the Cartesian product of the graphs $(A,S)$ and $(B,T)$. In particular, $Z_{S\boxempty T}(\lam)  = \sum_{I \in \cI(S\boxempty T)} \lam^{|I|}$ denotes the hard-core partition function of $S\boxempty T$.
 Let $\nu_{A,B,\lam}$ denote the measure on $ \cD_{A,B,\lam}$ given by
\begin{align}\label{eqnuABdef}
\nu_{A,B,\lam}(S,T)=\frac{\lam^{|S|+|T|}Z_{S\boxempty T}(\lam)}{Z_{A,B}(\lam)}\, ,
\end{align}
and let $\nu^{\text{w}}_{A,B,\lam}$ denote the measure on $ \cD^{\text{w}}_{A,B,\lam}$ given by
\begin{align}\label{eqnuABwdef}
\nu^{\text{w}}_{A,B,\lam}(S,T)=\frac{\lam^{|S|+|T|}Z_{S\boxempty T}(\lam)}{Z^{\text{w}}_{A,B}(\lam)}\, .
\end{align}

We note that $\mu_{A,B,\lam}$ can be described as the measure given by the following two-step process:
\begin{enumerate}
\item Sample $(S,T)\in \cD_{A,B,\lam}$ according to $\nu_{A,B,\lam}(S,T)$.
\item Sample $\Ec \subseteq A\times B$ according to the hard-core model on $S\boxempty T$ at activity $\lam$. 
\end{enumerate}

The measure $\nu_{A,B,\lam}$ is therefore the distribution of defect edges in a sample from $\mu_{A,B,\lam}$ and similarly for $\mu^{\text{w}}_{A,B,\lam}, \nu^{\text{w}}_{A,B,\lam}$. If a partition $(A,B)$ of $[n]$ is clear from the context, we use $G_\boxempty$ as shorthand for the Cartesian product $G_A \boxempty G_B$.

As we will soon see, the quantities $q_A, q_B$ will serve as a good approximations to the edge densities within $A$ and $B$ respectively in a sample from $\nu_{A,B,\lam}$.  It will be useful to note that if $\lam= c\sqrt{\frac{\log n}{n}}$ for some $c>0$ and $(A,B)$ is $\lam$-moderately balanced, then 
\[
q_A, q_B =(1+o(1))  \lam e^{-n\lam^2/2} = \tilde \Theta \left(n^{-1/2-c^2/2} \right)\, . 
\] 

As we refine our analysis of the defect edges, we require more refined estimates of these densities. In Section~\ref{secSubCritRegime}, we define $q_A'$ and $q_B'$ via
\begin{align}\label{eqqPrimeDef}
\frac{q_A'}{1-q_A'}:= \lam e^{-\lam^2b+2\lam^3b} \quad\text{and}\quad \frac{q_B'}{1-q_B'}:= \lam e^{-\lam^2a+2\lam^3a}\, ,
\end{align}
while in Section~\ref{secSuperCrit} we require more precision and use the definitions
\begin{align}\label{eqqprimesuperdef}
\frac{q_A'}{1-q_A'}:= \lam e^{-\lam^2b+2\lam^3b-7\lam^4b/2} \quad\text{and}\quad \frac{q_B'}{1-q_B'}:= \lam e^{-\lam^2a+2\lam^3a-7\lam^4b/2}\, .
\end{align}

In Section~\ref{secSuperCrit} we also define the following parameters.
\begin{align}
\mu_A= \binom{a}{2}q'_Ae^{2\lam^3b(aq_A+bq_B)} \text{ and }
\mu_B= \binom{b}{2}q'_Be^{2\lam^3a(aq_A+bq_B)}\, ,
\end{align}
and 
\begin{align}
\frac{q_A''}{1-q_A''}= \frac{q_A'}{1-q_A'}e^{4\mu_B\lam^3} \text{ and }\frac{q_B''}{1-q_B''}= \frac{q_B'}{1-q_B'}e^{4\mu_A\lam^3}\, .
\end{align}

It will be useful to note that $q_A\sim q_A'\sim q_A''$ and similarly for $q_B$.
We note also that $q_0, q_1, q_2, \mu$ defined in the introduction correspond to the parameters $q_A, q_A', q_A'', \mu_A$ in the special case where $a=b=n/2$ (a perfectly balanced partition).

Given graphs $H, G$, we let $H(G)$ denote the number of (not necessarily induced) copies of $H$ in $G$. Given graphs $H,G_1, G_2$ let $H(G_1, G_2)$ denote the number of copies of $H$ in $G_1\cup G_2$ with at least one edge in $G_1$.

For $V\subseteq [n]$ $q\in(0,1)$, we write $G(V,q)$ to denote the \ER random graph on a vertex set $V$ with edge probability $q$. For $\psi \in \R$, we let $G(V,q,\psi)$ denote the random graph on $V$ with distribution 
\begin{align}
  \nu_{q,\psi} (G) \propto \left(  \frac{q}{1-q} \right)^{|G|}  e^{\psi P_2(G)} \, ,
  \end{align}
 \emph{conditioned} on the event that $\Delta(G) \le 50\max\{qn, \log n\}$ and $G$ is triangle-free.

\section{Tools and preliminaries}
\label{secTools}

Some of our main tools for estimating partitions functions will be the cluster expansion and bounds on cumulants in conditioned exponential random graph models.   We begin with some background and basic facts about these tools.

\subsection{Cluster expansion and the hard-core model}\label{subseccluster}

We defined the hard-core model in Section~\ref{secERresults}; the following is a multivariate generalization.  
Let $G$ be a graph and let $\cI(G)$ be the set of all independent sets of $G$. Let $\boldsymbol \lambda : V(G)\to\C$ be an assignment of complex weights to the vertices of $G$. The (multivariate) hard-core model partition function of $G$ is
\[ Z_G(\boldsymbol \lambda) = \sum_{I \in \cI(G)} \prod_{v\in I}\boldsymbol \lambda(v) \, . \]

When $\boldsymbol \lambda(v)=\lambda$ for all $v\in V(G)$ (the univariate case) we write $Z_G(\lambda)$ instead of $Z_G(\boldsymbol \lambda)$.

The cluster expansion is a formal power series for $\log Z_G(\boldsymbol\lam)$; in fact, it is the Taylor series around $\boldsymbol \lam =0$.  Conveniently, the terms of the cluster expansion have a nice combinatorial interpretation (see, e.g.,~\cite{scott2005repulsive,faris2010combinatorics}).   A cluster $\Gamma=(v_1, \ldots, v_k)$ is a tuple of vertices from $G$ such that the induced graph $G[\{v_1, \ldots, v_k\}]$ is connected. We let $\cC(G)$ denote the set of all clusters of $G$. We call $k$ the size of the cluster and denote it by $|\Gamma|$. Given a cluster $\Gamma$, the \emph{incompatibility graph} $H_\Gamma$, is the graph on vertex set $\Gamma$ (considered as a multiset) with an edge between $v_i, v_j$ if either $v_i, v_j$ are adjacent on $G$ or $i \ne j$ and $v_i, v_j$ correspond to the same vertex in $G$.   In particular, by the definition of a cluster, the incompatibility graph $H_{\Gamma}$ is connected.

As a formal power series, the cluster expansion is the infinite series
\begin{align}\label{eqclusterexp}
\log Z_G(\boldsymbol \lam) = \sum_{\Gamma\in \cC(G)} \phi(\Gamma) \prod_{v\in\Gamma}\boldsymbol \lam(v) \,,
\end{align}
where the product is over all coordinates of $\Gamma$ and
\begin{align}
\label{eqUrsell}
\phi(\Gamma) &= \frac{1}{|\Gamma|!} \sum_{\substack{A \subseteq E(H_\Gamma)\\ \text{spanning, connected}}}  (-1)^{|A|} \, .
\end{align}

The cluster expansion converges absolutely if $\boldsymbol\lam$ lies inside a polydisk $D \subset \mathbb C^{V(G)}$ so that $Z_G(\boldsymbol\xi) \ne 0$ for all $\boldsymbol\xi \in D$.  We will need the following lemma which gives a sufficient condition for convergence and  bounds the error in truncating the cluster expansion. In fact, we will require a slightly stronger statement that will allow us to truncate \emph{pinned} cluster expansions, i.e., restrictions of the cluster expansion to clusters that contain a fixed set of vertices. Given a set $\{u_1,\ldots, u_\ell\}$ of vertices of $G$, we write $\{u_1,\ldots, u_\ell\}\subseteq \Gamma$ to mean that each vertex $u_i$ appears in the tuple $\Gamma$.

The next lemma follows from a classical approach to  cluster expansion convergence based on a combinatorial inequality due to Penrose~\cite{penrose1967convergence} (see also~\cite{brydges1984short, FernandezProcacci}). We defer its proof to Appendix~\ref{secPinnedCluster}. For $\boldsymbol\lam: V(G)\to\C$, we let $\lam_{\textup{max}}:=\max_{v\in V(G)}|\boldsymbol \lam(v)|$.

\begin{lemma}\label{lemClusterTail}
Suppose $G$ is a graph on $n$ vertices with maximum degree $\Delta$, and suppose $ \lam_{\textup{max}}\le \frac{1}{4e\Delta}$. Then the cluster expansion converges absolutely. Moreover, for any non-empty vertex set $S\subseteq V(G)$, $k\geq |S|$, and $t\geq 0$,
\[
\left|  \sum_ {\substack{\Gamma: \Gamma\supseteq S, \\ |\Gamma|\geq k}} |\Gamma|^t\phi(\Gamma) \prod_{v\in\Gamma} \boldsymbol{\lam}(v)\right| 
=O_{k,t}\left( \Delta^{k-|S|}\lam_{\textup{max}}^k \right)
\, . 
\]
If $|S|\in\{1,2\}$ then we have the explicit upper bound
  \begin{align}\label{eqpinnedexplicittail}
\left|  \sum_ {\substack{\Gamma: \Gamma\supseteq S, \\ |\Gamma|\geq k}} \phi(\Gamma) \prod_{v\in\Gamma} \boldsymbol{\lam}(v)\right| 
\leq 
(2e)^k\Delta^{k-|S|}\lam_{\textup{max}}^k\, .
\end{align}
\end{lemma}

Throughout the paper, it will be useful to work with the following slightly modified form of the cluster expansion.
\begin{lemma}\label{lem:ClusterExpMod}
Let $G$ be a graph with $n$ vertices and maximum degree $\Delta$. Then for $\lam\leq \frac{1}{4e\Delta}$,
\begin{align}\label{eqClusterVsEmpty}
\log\left( \frac{Z_G(\lam)}{(1+\lam)^n} \right) =  \sum_{\Gamma\in \cC'(G)} \phi(\Gamma) \lam^{|\Gamma|}\, .
\end{align}
where $\cC'(G)\subseteq \cC(G)$ is the set of \emph{non-constant} clusters (those not of the form $(v,v,\ldots, v)$ for $v\in V(G)$).
\end{lemma}
\begin{proof}
Let $F=(V(G), \emptyset)$ denote the empty graph on the same vertex set as $G$. Note that $Z_F(\lam)=(1+\lam)^n$. 
Note that 
\[
\cC'(G)=\cC(G)\backslash\cC(F)\, ,
\]
since the only clusters of $F$ are precisely those of the form $\Gamma=(v, v,\ldots, v)$ for $v\in V(G)$ and these are also clusters of $G$. 
If $\lam\leq \frac{1}{4e\Delta}$, then by Lemma~\ref{lemClusterTail}, the cluster expansions of $\log Z_G(\lam)$ and $\log Z_F(\lam)$ converge absolutely and 
\begin{align}\label{eqClusterVsEmpty}
\log\left( \frac{Z_G(\lam)}{(1+\lam)^n} \right) =  \sum_{\Gamma\in \cC'(G)} \phi(\Gamma) \lam^{|\Gamma|}\, .
\end{align}
\end{proof}

It will be convenient to think of the cluster expansion of $\log Z_G(\lam)$ as an expansion in terms of subgraph counts in $G$.  By truncating the cluster expansion at clusters of size $3$ we obtain the following corollary for the hard-core partition function of triangle-free graphs. Recall that for a graph $G$, we let $P_2(G)$ denote the number of paths of length $2$ in $G$. 

\begin{cor}\label{corclustersimple}
Let $G$ be a triangle-free graph with $n$ vertices, at most $n'$ non-isolated vertices, and maximum degree $\Delta$. Then for $\lam\leq \frac{1}{4e\Delta}$,
\begin{align}\label{eqNonConstExpand}
\log\left( \frac{Z_G(\lam)}{(1+\lam)^n} \right)= -|G|\lam^2 + \left(P_2(G)+2|G|\right)\lam^3 + O(n'\Delta^3\lam^4)\, .
\end{align}
Moreover, if $\mathbf I$ is a random sample from the hard-core model on $G$ at activity $\lam$, then
\begin{align}\label{eqExpectExpand}
\E |\mathbf I|= \frac{\lam}{1+\lam}n - 2|G|\lam^2 +  3\left(P_2(G)+2|G|\right)\lam^3 + O(n'\Delta^3\lam^4)\, ,
\end{align}
and 
\begin{align}\label{eqVarExpand}
\var |\mathbf I|= \frac{\lam}{(1+\lam)^2}n - 4|G|\lam^2 + O(n'\Delta^2\lam^3)\, .
\end{align}
\end{cor}
\begin{proof}
The proof is a routine calculation from the definitions using Lemma~\ref{lemClusterTail} to bound the truncation error.  We include the details as they will be instructive for later calculations.

Let $\cC'=\cC'(G)$ and let $\cC'_k$ denote the set of clusters in $\cC'$ of size $k$. Then 
\(
\cC'_1=\emptyset,
\)
\[
\cC'_2=\{(v_1,v_2): \{v_1,v_2\}\in E\}\, ,
\]
and
\[
\cC'_3=\{(v_1,v_2,v_3): G[\{v_1,v_2,v_3\}]\cong K_2 \text{ or }P_2\}\, ,
\]
(here $K_2$ denotes the complete graph on 2 vertices, i.e., an edge). 

If $\Gamma\in \cC'_2$, then $H_\Gamma \cong K_2$ and so $\phi(\Gamma) = -1/2$. We note that $|\cC'_2|= 2|E|$, accounting for the orderings.

If $\Gamma=(v_1,v_2,v_3)$ such that $G[\{v_1,v_2,v_3\}]\cong P_2$, then $H_\Gamma \cong P_2$ and so $\phi(\Gamma) = 1/6$. The number of such clusters is $6P_2(G)$. 

If $\Gamma=(v_1,v_2,v_3)$ such that $G[\{v_1,v_2,v_3\}]\cong K_2$, then $H_\Gamma$ is isomorphic to a triangle and so $\phi(\Gamma) = 1/3$. The number of such clusters is $6|E|$. We conclude from~\eqref{eqClusterVsEmpty} that
\begin{align}\label{eqClusterTruncate4}
\log\left( \frac{Z_G(\lam)}{(1+\lam)^n} \right)= -|G|\lam^2 + \left(P_2(G)+2|G|\right)\lam^3 +  \sum_{\Gamma\in \cC' : |\Gamma|\geq 4} \phi(\Gamma) \lam^{|\Gamma|}\, .
\end{align}
Next we observe that if $\Gamma\in \cC'$ then $\Gamma$ cannot contain an isolated vertex of $G$. By applying Lemma~\ref{lemClusterTail} with $S=\{v\}$, $k=4$, $t=0$, for each non-isolated vertex $v$ of $G$ and summing the resulting bounds we deduce that
\begin{align}\label{eqTailBound4}
\left|   \sum_{\Gamma\in \cC' : |\Gamma|\geq 4} \phi(\Gamma) \lam^{|\Gamma|}  \right| 
=O\left(n' \Delta^{3}\lam^4 \right)\, .
\end{align}
To prove~\eqref{eqExpectExpand} we note 
\[
\E |\mathbf I|= \sum_{I\in \cI(G)} |I| \frac{\lam^{|I|}}{Z_G(\lam)} =\lam \frac{\partial}{\partial \lam} \log Z_G(\lam)\, .
\]
By Lemma~\ref{lemClusterTail}, the cluster expansion for $\log Z_G(\lam)$ converges uniformly on $[0,1/(4e\Delta)]$, and so we may differentiate termwise,  yielding, by~\eqref{eqClusterTruncate4},
\[
\E |\mathbf I| = n\frac{\lam}{1+\lam}-2|G|\lam^2 + 3\left(P_2(G)+2|G|\right)\lam^3 +  \sum_{\Gamma\in \cC' : |\Gamma|\geq 4} |\Gamma| \phi(\Gamma) \lam^{|\Gamma|}\, .
\]
Statement~\eqref{eqExpectExpand} now follows by bounding the sum on the RHS  exactly as we did for~\eqref{eqTailBound4} (applying Lemma~\ref{lemClusterTail} now with $t=1$).

Statement~\eqref{eqVarExpand} follows similarly from the observation that
\[
\var |\mathbf I| =  \lam \frac{\partial}{\partial \lam} \E |\mathbf I| \, .
 \qedhere \]
\end{proof}

We remark that the ratio $\frac{Z_G(\lam)}{(1+\lam)^n}$ from Corollary~\ref{corclustersimple} is the probability that a subset $S \subseteq V(G)$ is an independent set of $G$, when $S$ is chosen by including each vertex independently with probability $\frac{\lam}{1+\lam}$.

The following quasirandomness condition for the hard-core model will be useful.
\begin{lemma}\label{lemhcquasirandom}
Let $G$ be a graph of maximum degree $\Delta$ and let $U\subseteq V(G)$. Let $\lam \le \frac{1}{ 16 e^2\Delta}$ and let $\mathbf I$ be a random sample from the hard-core model on $G$ at activity $\lam$. Then
\[
\P \left (|\mathbf I \cap U|\geq 5 \lam |U| \right)\leq e^{-\lam|U|}\, ,
\]
and 
\[
\P \left(|\mathbf I \cap U|\leq \lam |U|/10 \right)\leq e^{-\lam|U|/8}\, .
\]
\end{lemma}
We note that the bounds of Lemma~\ref{lemhcquasirandom} and the range of $\lam$ for which they hold are not optimal, but they will suffice for our purposes. We defer the proof of Lemma~\ref{lemhcquasirandom} to Appendix~\ref{secHCQuasi}.

We will also make use of a \emph{local CLT} (LCLT) for the low-density hard-core model.  We say that a sequence of integer-valued random variables $X_n$ with mean $\mu_n$ and variance $\sigma_n^2$ satisfies a LCLT if for all integers $k$
\[
\P(X_n=k)=\frac{1}{\sqrt{2\pi}\sigma_n} e^{-(k-\mu_n)^2/(2\sigma_n^2)}+o(\sigma_n^{-1})\, ,
\]
as $n\to\infty$.
For graphs of  maximum degree at most $\Delta$, with $\Delta$ constant, the second author, Jain, Sah and Sawhney proved a sharp  LCLT for the hard-core model~\cite{jain2021approximate}.  The following is an analogue for sequences of graphs $G_n$ of maximum degree at most $\Delta_n$, with $\Delta_n \to \infty$.  We do not attempt to optimize  the bound on $\lam$ here.

\begin{prop}
\label{thmhcLCLT}
Let $G_n$ be a sequence of graphs on $n$ vertices of maximum degree at most $\Delta_n $. Let $X_n$ be the size of an independent set drawn from the hard-core model on $G_n$ at activity $\lam_n$.  Suppose $\lam_n \Delta_n \to 0$, $n \lam_n \to \infty$, and $\Delta_n \to \infty$ as $n \to \infty$.   Then $\var(X_n) \sim \lam n$ and $X_n$ obeys a local central limit theorem.  That is, for every integer $k \ge 0$,
\[  \P[X_n =k]  =   \frac{1}{\sqrt{  2 \pi \lam n}} \exp \left( - \frac{(k - \E X_n)^2 }{ 2 \lam n}  \right)   +   o \left(  \frac{1}{\sqrt{\lam n}}  \right)     \, .\]
\end{prop}

We prove Proposition~\ref{thmhcLCLT} in Appendix~\ref{secLCLTproof}.

\subsection{Cumulants and the cumulant generating function}
  Let $X$ be a bounded  random variable.  The \textit{cumulant generating function} of $X$ is
\[
K_X(t):= \log \E e^{tX}\, .
\]
The \textit{cumulants} of $X$ are defined as coefficients of the Taylor series for $K_X(t)$ around $0$:
\begin{align}\label{eqcumulantdef}
\kappa_{k}(X):= \frac{\partial^{k}K_X(t)}{\partial t^{k}}\Bigg|_{t=0}\, .
\end{align}
In particular, $\kappa_1(X) =\E X$ and $\kappa_2(X) = \var(X)$.

\begin{lemma}\label{lemtiltedcumulant}
Let $\mu$ be a  probability measure on a finite set $\Omega$, and let $X: \Omega\to \R$ be a random variable. 
Given $s\in\R$, let $\mu_s$ denote the tilted measure
\[
\mu_s(x)\propto \mu(x)e^{s X(x)} \text{ for $x\in \Omega$}\, .
\]
For $k\in \N$, let $\kappa_k^s(X)$ denote the $k$th cumulant of $X$ with respect to $\mu_s$ so that $\kappa_k(X)=\kappa_k^0(X)$.
Let $t>0$ and $\ell\in \N$. We have
\[
\log \E_\mu\left(e^{t X}\right)= \sum_{k=1}^{\ell-1}\kappa_k(X) \frac{t^k}{k!}  +\kappa^s_\ell(X) \frac{t^\ell}{\ell!}  
\]
for some $s\in[0,t]$.
\end{lemma}
\begin{proof}
Let $f(t):=\log \E_\mu\left(e^{t X}\right)$. Since $f$ is $\ell$ times differentiable, Taylor's Theorem (with the Lagrange form of the remainder) shows that there exists $s\in[0,t]$ such that
\[
f(t)=  \sum_{k=1}^{\ell-1}f^{(k)}(0) \frac{t^k}{k!}  +f^{(\ell)}(s) \frac{t^\ell}{\ell!} 
=\sum_{k=1}^{\ell-1}\kappa_k(X) \frac{t^k}{k!}  + f^{(\ell)}(s) \frac{t^\ell}{\ell!}\, ,
\]
where for the second equality we used~\eqref{eqcumulantdef}, the definition of cumulants.
Finally note that the cumulant generating function of $X$ with respect to $\mu_s$ is
\[
\log \E_{\mu_s}\left(e^{tX} \right)= f(t+s)- f(s)  \, ,
\]
and so
\[
\kappa^s_\ell(X) = \frac{\partial ^{\ell}}{\partial t ^{\ell}}\left({f(t+s)}-{f(s)} \right)\Bigg|_{t=0}=  f^{(\ell)}(s) \, . \qedhere
\]
\end{proof}

\subsection{Other probabilistic tools} 

We will use Pinsker's inequality  to bound the total variation distance between two probability measures.
Recall that for discrete probability distributions $\nu, \mu$ defined on the discrete same sample space $\Omega$, the Kullback-Leibler (KL) divergence of $\nu$ from $\mu$ is defined to be
\[
D_\text{KL}(\nu \parallel \mu) = \sum_{x\in\Omega}\nu(x) \log\left(\frac{\nu(x)}{\mu(x)} \right)\, ,
\]
provided $\nu(x)=0$ whenever $\mu(x)=0$, else we define $D_\text{KL}(\nu \parallel \mu) = + \infty$. (Note that we interpret $0/0$ and $0 \log 0$ as $0$.)
Their total variation distance is defined as
\[
\|\mu-\nu\|_{\text{TV}}= \sup_{A\subseteq \Omega} |\mu(A) - \nu(A)|= \frac{1}{2} \sum_{x\in\Omega}|\mu(x)-\nu(x)|\, .
\] 
Pinsker's inequality (see, e.g.,~\cite{csiszar2011information}) allows us to bound the total variation distance between measures in terms of their (KL) divergence which is often more convenient to compute.
\begin{lemma}[Pinsker's inequality]\label{lemPinsker}
If $\mu, \nu$ are two discrete probability distributions on a common sample space $\Omega$, then
\[
\| \nu- \mu \|_{TV}\leq \sqrt{\frac{1}{2}D_\text{KL}(\nu \parallel \mu)}\, .
\]
\end{lemma}

The total variation distance between two discrete random variables $X,Y$, denoted $\|X-Y\|_{\text{TV}}$, is the total variation distance between the law of $X$ and the law of $Y$, i.e., $\|X-Y\|_{\text{TV}}=\frac{1}{2}\sum_{x\in \R}|\P(X=x)-\P(Y=x)|$. We record the following elementary, yet powerful lemma for bounding the total variation distance between random variables (see, e.g.,~\cite{levin2017markov}). 

\begin{lemma}[Coupling inequality]\label{lemCoupling}
If $X,Y$ are random variables with a coupling $(X',Y')$, then
\[
\|X-Y\|_{TV}\leq \P(X'\neq Y')\, .
\]
\end{lemma}

Finally we note the following form of Chernoff's inequality (see, e.g.,~\cite[Theorem 4.4]{mitzenmacher2017probability}).
\begin{lemma}\label{lemChernoff}
Let $X_1,\ldots, X_n$ be independent Bernoulli random variables, let $X=\sum_i X_i$ and let $\mu=\E[X]$.
For any $\delta>0$,
\[
\P(X>(1+\delta)\mu) <\left(\frac{e^\delta}{(1+\delta)^{1+\delta}} \right)^{\mu}\, .
\]
In particular, if $1+\delta\geq e^2$, then
$\P(X>(1+\delta)\mu)<e^{-(1+\delta)\mu}$.
\end{lemma}

\section{Uniqueness of partitions}
\label{secUniquenessPartitions}

The goal of this section is to prove the following lemma, which is a key step towards the uniqueness statement of Proposition~\ref{propGroundStateRefinedAlt}. Recall the definitions of $\cT_{A,B,\lam}^{\textup{w}}, \cD_{A,B,\lam}^{\textup{w}}$ defined at~\eqref{eqTABwDefAgain},~\eqref{eqDABwDefAgain} respectively. It will also be useful to keep in mind the description of the measure $ \mu_{\cR,\lam}$ with $\cR=\cT_{A,B,\lam}^{\textup{w}}$ given after~\eqref{eqnuABwdef}. 

Throughout this section we assume that $\lam \geq \omega/\sqrt{n}$ where $\omega>0$ is a sufficiently large absolute constant.

\begin{lemma}\label{lemWeakUnique}
Let $(A,B)\in \Pi_{\textup{weak}}$ and sample $G$ according to $\mu^{\textup{w}}_{A,B,\lam}$.  With probability at least 
\(
1-e^{-\lam n/25}\, ,
\)
$(A,B)$ is the unique weakly balanced partition satisfying
 \(
\Delta(G_A \cup G_B)\leq \alpha/\lam
\) and is the unique max cut of $G$.
\end{lemma}

Lemma~\ref{lemWeakUnique} will be a consequence of Lemma~\ref{lemhcquasirandom}, the quasirandomness statement for the hard-core model. Before we turn to the proof, we begin with a definition. 
 \begin{defn}\label{defABExpand}
 Given a graph $G$ and a partition $(A,B)$, we call $G$ an $(A,B)$-$\lam$-expander if 
\[
\text{$d_G(v, B)\geq \lam n/30$ for all $v\in A$}\, ,
\] 
and 
\begin{align}\label{eqExpansionCond}
X\subseteq A,\,  |X|\geq \lam n/100,\,  Y\subseteq B,\,  |Y|\geq n/6 \implies |E(X,Y)|\geq \lam |X||Y|/10\, ,
\end{align}
and both statements hold also with $A,B$ swapped. Moreover if $\lam \geq \sqrt{\frac{\log n}{n}}$, then we require in addition that 
\begin{align}\label{eqExpansionCond2}
X\subseteq A,\, Y\subseteq B,\,  |X|,|Y|\geq 10\lam n \implies |E(X,Y)|\geq \lam |X||Y|/10\, .
\end{align}
 \end{defn}

\begin{lemma}\label{lemExpanderwhp}
Let $(A,B)\in \Pi_{\textup{weak}}$, and $(S,T)\in \cD_{A,B,\lam}^{\textup{w}}$. Sample $\Ec\subseteq A\times B$ according to the hard-core model on $S\boxempty T$ at activity $\lam$. Let $G$ be the graph $([n], \Ec)$. Then
\[
\P(G \textup{ is an $(A,B)$-$\lam$-expander} )\geq 1-  e^{-\lam n/25}\, .
\]
\end{lemma}
\begin{proof}
Note that $\Delta(S\cup T)\leq \alpha/\lam$ so that $\Delta:=\Delta(S\boxempty T)\leq 2\alpha/\lam$.  Fix $v\in A$ and note that 
$d_G(v, B) = |\Ec \cap (\{v\}\times B)|$.
By Lemma~\ref{lemhcquasirandom}\footnote{Noting that $\lam<\frac{1}{4e^2\Delta}$ since $\Delta\leq \frac{2\alpha}{\lam}$ and $\alpha=\frac{1}{96e^3}$.}, we have 
\[
\P(d_G(v, B)  \leq \lam |B|/10 ) \leq e^{-\lam |B|/8} \leq e^{-\lam n/24}\, ,
\]
where we used that $|B|\geq n/3$ (since $(A,B)$ is weakly balanced).
By a union bound over $v\in A$ we have
\begin{align}\label{eqCrossDegCond}
\P(d_G(v, B)  \geq \lam n/30 \text{ for all } v\in A )\geq 1-  ne^{-\lam n/24}\, .
\end{align}
We now turn to the second condition in the definition of an $(A,B)$-$\lam$-expander. 

Fix $X$ and $Y$ as in~\eqref{eqExpansionCond}. Note that
\[
|E(X,Y)|=|\Ec \cap (X\times Y)|\, . 
\]
By Lemma~\ref{lemhcquasirandom}, we then have 
\[
\P(|E(X,Y)|  \leq \lam |X||Y|/10 ) \leq e^{-\lam |X||Y|/8} \leq e^{-\lam^2 n^2/5000}\, .
\]
By a union bound over all choices of $X,Y$ we conclude that 
\begin{align}\label{eqExpCondWhp}
\P(\text{$G$ satisfies~\eqref{eqExpansionCond}})\geq 
1-2^{n} e^{-\lam^2 n^2/5000} \geq 1- e^{-n}\, ,
\end{align}
since $\lam\geq \omega/\sqrt{n}$.
Finally if $\lam \geq \sqrt{\frac{\log n}{n}}$ then we fix $X,Y$ as in~\eqref{eqExpansionCond2}. As above we have 
\[
\P(|E(X,Y)|  \leq \lam |X||Y|/10 ) \leq e^{-\lam |X||Y|/8}\leq e^{-12\lam^3 n^2}
\]
By a union bound over all choices of $X,Y$ we conclude that 
\begin{align}\label{eqExpCondWhp2}
\P(\text{$G$ satisfies~\eqref{eqExpansionCond2}})\geq 
1-\binom{n}{10\lam n}^2 e^{-12\lam^3 n^2}\geq 1-e^{-\lam^3 n^2}\, .
\end{align}
The result follows by combining this with~\eqref{eqExpCondWhp},~\eqref{eqCrossDegCond}, and the analogous statements with $A,B$ swapped.
\end{proof}

Lemma~\ref{lemWeakUnique} follows immediately from Lemma \ref{lemExpanderwhp} and the following consequence of expansion.
\begin{lemma}\label{lemExpansionCor} 
Let $(A,B)\in \Pi_{\textup{weak}}$ and let 
$G$ be an $(A,B)$-$\lam$-expander such that 
\(
\Delta(G_A \cup G_B)\leq \alpha/\lam\, .
\)
If $(A', B')\in \Pi$ such that $(A',B')\neq(A,B)$, then 
\(
\Delta(G_{A'} \cup G_{B'})> \alpha/\lam\, .
\)
Moreover $(A,B)$ is the unique max cut of $G$. 
\end{lemma}
\begin{proof}
Suppose that $(A', B')$ is a partition distinct from $(A,B)$. Since the partitions are distinct, either $A \cap B'\neq \emptyset$ or $B\cap A'\neq \emptyset$. 
Assume wlog that $B\cap A'\neq \emptyset$.

Suppose first that $|A\cap B'|<\lam n /100$.
By assumption there exists $v\in B\cap A'$. Since $G$ is an $(A,B)$-$\lam$-expander we then have 
\[
d_G(v, A')\geq d_G(v, A) - d_G(v, A\cap B')\geq \lam n/30 - |A\cap B'|\geq \lam n /50\, .
\]
It follows that $\Delta(G_{A'}\cup G_{B'})\geq  \lam n /50>\alpha/\lambda$. 
 Moreover, we note that
\[
d_G(v, B')= d_G(v, B\cap B') + d_G(v, A\cap B') \leq  \alpha/\lam + |A\cap B'|<\lam n /50\, ,
\]
and so $d_G(v, B')<d_G(v, A')$. In particular $(A',B')$ is not a max cut since $(A'\backslash\{v\}, B'\cup \{v\})$ is a larger cut.  

We may therefore assume that $|A\cap B'|\geq \lam n /100$. 
In particular $A\cap B'\neq\emptyset$, and so by an identical argument, we may assume that $|A'\cap B|\geq \lam n /100$ also. 
By symmetry (swapping the roles of $A$ and $B$) we may also assume that $|A \cap A'|\geq \lam n /100$ and $|B\cap B'| \geq \lam n/100$.

Since $(A,B)$ is weakly balanced we have $|B| \geq n/3$. Suppose wlog that $|B\cap B'|\geq |B \cap A'|$ so that in particular $|B\cap B'|\geq n/6$.
Since $G$ is an $(A,B)$-$\lam$-expander we then have
\begin{align}\label{eqABlamExpand}
|E(A\cap B',B\cap B')|\geq \lam |A\cap B'||B\cap B'|/10\, ,
\end{align}
and so there exists $v\in A\cap B'$ such that $d_G(v, B\cap B')\geq \lam |B\cap B'|/10\geq \lam n /60$. It follows that $\Delta(G_{A'}\cup G_{B'})\geq  \lam n /60>\alpha/\lambda$. 

We conclude by showing that again $(A', B')$ is not a max cut. First note that
{\small
\begin{multline}\label{eqABMaxCut}
 |E(A,B)| -|E(A',B')| =\\
 |E(A\cap A', B\cap A')| + |E(A\cap B', B\cap B')| -|E(A\cap A', A\cap B')| -|E(B\cap A', B\cap B')| \, .
\end{multline}}
Since $\Delta(G_A)\leq \alpha/\lam$, we have
\[
|E(A\cap A', A\cap B')| \leq \frac{\alpha}{\lam}|A\cap B'|\, .
\]
It follows from~\eqref{eqABlamExpand} and the bound $|B\cap B'|\geq n/6$ that
\begin{align}\label{eqCrossDisc}
|E(A\cap A', A\cap B')|<\frac{1}{2}|E(A\cap B',B\cap B')| \, .
\end{align}
Since $(A,B)$ is weakly balanced, $|A| \geq n/3$ and so either $|A\cap A'|\geq n/6$, $|A \cap B'|\geq n/6$. 
Suppose first that $|A\cap A'|\geq n/6$, then an argument identical to the above shows that
\[
|E(B\cap A', B\cap B')|<\frac{1}{2}|E(A\cap A', B\cap A')|\, .
\]
Similarly, if $|A \cap B'|\geq n/6$ then 
\[
|E(B\cap A', B\cap B')|<\frac{1}{2} |E(A\cap B', B\cap B')| \, .
\]
In either case, when combined with~\eqref{eqCrossDisc} and~\eqref{eqABMaxCut}, we see that  $|E(A,B)|-|E(A',B')|>0$.
\end{proof}

\section{Strengthening the max degree bound on the defect graph}\label{secMaxDegree}
In this section we prove Proposition~\ref{propMaxDegBootstrap}.
Throughout this section we fix $(A,B)\in \Pi_{\text{weak}}$.
As in the previous section,  we assume that $\lam \geq \omega/\sqrt{n}$ where $\omega>0$ is a sufficiently large absolute constant. 
The main step toward Proposition~\ref{propMaxDegBootstrap}, is to prove a large deviation bound on the maximum degree of a sample from $\nu^{\textup{w}}_{A,B,\lam}$ (defined at~\eqref{eqnuABwdef}). 
Since we will need it later, we do the same for the measure $\nu_{A,B,\lam}$ (defined at~\eqref{eqnuABdef}).
 For $\bm r=(r_A, r_B)\in [0,1)^2$, let $\nu_{\bm r}$ denote the measure on graphs $G\subseteq \binom{A}{2}\cup \binom{B}{2}$ given by
\begin{align}\label{eqnubmrdef}
\nu_{\bm r}(G)\propto \left(\frac{r_A}{1-r_A}\right)^{|G_A|}\left(\frac{r_B}{1-r_B}\right)^{|G_B|}\, ,
\end{align} 
i.e., the distribution of the union of the two independent \ER random graphs $G(A,r_A), G(B,r_B)$.
Given a family of graphs $\cE\subseteq 2^{\binom{A}{2}\cup \binom{B}{2}}$, let $\nu_{\bm r,\cE}$ denote the measure $\nu_{\bm r}$ conditioned on the event $\cE$ that is, 
\begin{align}\label{eqnubmrcEdef}
\nu_{\bm r,\cE}(G)\propto \nu_{\bm r}(G)\mathbf{1}_{G\in \cE}\, .
\end{align}
 Our strategy will be to approximate $\nu_{A,B,\lam},\nu^{\textup{w}}_{A,B,\lam}$ by a perturbation of a measure of the form $\nu_{\bm r,\cE}$
for some choice of $\bm r=(r_A, r_B)$ \footnote{It will always be the case that $r_A=(1+o(1)) q_A$, $r_B=(1+o(1)) q_B$ but the precise choice will vary.} and a family of graphs $\cE$
\footnote{Typically we will take $\cE= \cD_{A,B,\lam}$ or $\cE=\cD^{\textup{w}}_{A,B,\lam}$, but later in the paper we consider other choices of $\cE$.}.
Throughout the paper, we will study various perturbations of measures of the form $\nu_{\bm r,\mathcal E}$ and so the results of this section are stated in greater generality than that needed for our immediate task of understanding $\nu^{\textup{w}}_{A,B,\lam}$. In general, we consider perturbations of the form
\begin{align}\label{eqnurfDef}
\nu^f_{\bm r,\mathcal E}(G)\propto \nu_{\bm r,\mathcal E}(G) e^{f(G)}\, ,
\end{align}
for some $f:\cE\to\R$.
We highlight that if the function $f$ is a linear combination of subgraph counts of $G$ then $\nu^f_{\bm r,\mathcal E}(G)$ is a (conditioned) exponential random graph model.

We will always require $f$ to satisfy a condition of the following type to ensure that the effect of the perturbation can be controlled.
We say that a family of graphs $\cE$ is \emph{downward-closed} if $G\in \cE$ and $F\subseteq G$ implies $F\in \cE$.

\begin{defn}
Let $\mathcal E$ be a downward-closed family of graphs and
let $\delta>0$.
\begin{itemize}
\item We call a function $f : \mathcal E\to\R$ \emph{$\delta$-local} if for all $G\in \mathcal E$ and $F\subseteq G$, we have
\[
|f(G)-f(G\backslash F)| \leq \delta |F| \cdot \max_{H\in \cE}\Delta(H)\, .
\]
\item We call a function $f : \mathcal E\to\R$ \emph{strongly $\delta$-local} if for all $G\in \mathcal E$ and $F\subseteq G$, we have
\[
|f(G)-f(G\backslash F)| \leq \delta |F| \cdot \Delta(G)\, .
\]
\end{itemize}
\end{defn}

Recall the definition of $q_A, q_B$ from Definition~\ref{defsparse} and the definitions of $\cD^{\textup{w}}_{A,B,\lam}, \cD_{A,B,\lam}$ from~\eqref{eqDABwDefAgain}.

\begin{lemma}\label{lemfofNuisLocal}
\,
\begin{enumerate}
\item Let $\cE= \cD^{\textup{w}}_{A,B,\lam}$.
There exists a strongly $(16e^3n\lam^3)$-local $f : \cE\to\R$ such that $\nu^{\textup{w}}_{A,B,\lam}=\nu^f_{\bm r,\mathcal E}$ with $r_A=q_A, r_B=q_B$.
\item  Let $\cE= \cD_{A,B,\lam}$.
There exists a strongly $(16e^3n\lam^3)$-local $f : \cE\to\R$ such that $\nu_{A,B,\lam}=\nu^f_{\bm r,\mathcal E}$ with $r_A=q_A, r_B=q_B$.
\end{enumerate}
\end{lemma}
\begin{proof}
We prove $(1)$. The proof of $(2)$ is similar. 
Recall that we let $G_\boxempty=G_A\boxempty G_B$ and that $\nu^{\textup{w}}_{A,B,\lam}$ is the measure on $ \cD^{\textup{w}}_{A,B,\lam}$ given by
\begin{align*}
\nu^{\textup{w}}_{A,B,\lam}(G)\propto \lam^{|G|}Z_{G_{\boxempty}}(\lam)\, .
\end{align*}
For $G\in \cD^{\textup{w}}_{A,B,\lam}$ we have $\Delta(G)\leq \alpha/\lambda$ by definition, and so $\Delta(G_\boxempty)\leq 2\alpha/\lam$. Since $\alpha=1/(96e^3)$, we may apply  Lemma~\ref{lem:ClusterExpMod} and cluster expand
\[
\log\left( \frac{Z_{G_{\boxempty}}(\lam)}{(1+\lam)^{ab}} \right) = -\lam^2 |G| +  f(G)\, ,
\]
where 
\begin{align}\label{eqflocaldef0}
f(G)= \sum_{\substack{\Gamma\in \cC'(G_\boxempty):\\ |\Gamma|\geq 3}} \phi(\Gamma) \lam^{|\Gamma|}\, ,
\end{align}
and
$\cC'(G_{\boxempty})$ denotes the set of non-constant clusters of $G_{\boxempty}$.
Suppose $F\subseteq G$, then
\begin{align}\label{eqflocalcalc0}
f(G) - f(G\backslash F) = \sum_{\Gamma\in \cC'' : |\Gamma|\geq 3} \phi(\Gamma) \lam^{|\Gamma|}
\end{align}
where 
\[
\cC''=  \cC'(G_\boxempty)\backslash  \cC'((G\backslash F)_{\boxempty})\, .
\]

Now, if $\Gamma\in \cC''$ then $\Gamma$ 
must contain a pair $S=\{(v_1, w), (v_2, w)\}$ (a pair of vertices of $G_{\boxempty}$) such that $\{v_1, v_2\}\in F_A$ or a pair $S=\left\{(v, w_1), (v, w_2)\right\}$ such that $\{w_1, w_2\}\in F_B$. Since there are at most $b|F_A|+a|F_B|\leq n|F|$ such pairs of vertices and $\Delta(G_{\boxempty})\leq 2\Delta(G)$, we have by Lemma~\ref{lemClusterTail} (applied with $k=3$, $t=0$ and $S$, for each of the aforementioned pairs $S$),
\begin{align}\label{eqtailbdgeq3}
\left| \sum_{\Gamma\in \cC'': |\Gamma|\geq 3} \phi(\Gamma) \lam^{|\Gamma|}\right|\leq 
n|F|\cdot (2e)^3\cdot 2\Delta(G)\cdot\lam^3\, .
\end{align}
We conclude from~\eqref{eqflocalcalc0} and~\eqref{eqtailbdgeq3} that $f$ is strongly $(16e^3 n\lam^3)$-local.
\end{proof}

In what follows, given a probability measure $\mu$, we write $\bG\sim \mu$ to denote that $\bG$ is a random sample from $\mu$.
\begin{lemma}\label{lemDegBdLocal}
 Let $\cE\subseteq 2^{\binom{A}{2}\cup \binom{B}{2}}$ be downward closed such that
\[
\max_{H\in \mathcal E}\Delta(H)\leq \alpha/\lam\, .
\]
Let $\delta \leq n\lam^3/(6\alpha)$ and let $f : \cE\to\R$ be strongly $\delta$-local. Let $r_A, r_B\in [0,1)$ be such that $r_A\leq 2q_A$, $r_B\leq 2q_B$.
If $\bG\sim \nu^f_{\bm r,\cE}$, then
\begin{align}\label{eqPGnotinTAB1}
\P(\Delta(\bG)\geq \Delta/2)\leq n^2 e^{-\Delta/2}\, ,
\end{align}
and
\begin{align}\label{eqSTEdgeDeviation1}
\P(|\bG| \geq K/2 ) \leq 2n^2e^{-\Delta/2}\, ,
\end{align}
where $\Delta=\Delta_{A,B,\lam}$ and $K=K_{A,B,\lam}$ are as in Definition~\ref{defsparse}. Moreover~\eqref{eqPGnotinTAB1} and~\eqref{eqSTEdgeDeviation1} hold if instead $\max_{H\in \mathcal E}\Delta(H)\leq \Delta$ and $f$ is $\delta$-local.
\end{lemma}
\begin{proof}
For $v\in A\cup B$ and $j\in \N$, let $\mathcal E(v,j)$ denote the event $\{ d_{\bG}(v) = j = \Delta(\bG) \} $. 
Since $\bG\in \cE$ by definition, $\Delta(\bG)\leq \alpha/\lam$ and so we may assume that $j\leq \alpha/\lam$.
We will show that for $j\geq \Delta/2$ we have $\P(\mathcal E(v,j))\leq e^{-\Delta/2}$ and so~\eqref{eqPGnotinTAB1} follows by a union bound over $v$ and $j$.

Suppose that $v\in A$ and let $E[v]\subseteq \binom{A}{2}$ denote the set of pairs in $A$ containing $v$. Let $\mathbf{G}_v= \bG-E[v]$.
Suppose that $G$ is such that $\P(\bG_v=G)>0$. We then have
\begin{align}\label{eqLocalDegBd}\\
\P( \mathcal E(v,j) | \bG_v=G)&\leq \sum_{\substack{J\subseteq E[v]: |J|=j,\\ G\cup J \in \cE(v,j)}}\left(\frac{r_A}{1-r_A}\right)^j e^{f(G\cup J)-f(G)}
\leq
\sum_{J\subseteq E[v]: |J|=j}\left(\frac{r_A}{1-r_A}\right)^j e^{\delta j^2} 
\end{align}
where for the second inequality we used that $f$ is strongly $\delta$-local.

Suppose first that
\[
\delta j^2 \leq j/10\, .
\]
Letting 
\[
\tilde r_A= \left(\frac{r_A}{1-r_A}\right) e^{1/10}\, ,
\]
we conclude that
\begin{align}\label{eqEvjBound}
\P( \mathcal E(v,j) | \bG_v=G)
\leq 
\binom{a}{j}\tilde r_A^j\leq \left(\frac{ea \tilde r_A}{j} \right)^j\, .
\end{align}
Since $r_A\leq 2q_A$ we have
 $\Delta\geq e^2 a \tilde r_A$. Moreover, $j\geq \Delta/2$ by assumption so the RHS of~\eqref{eqEvjBound} is at most $e^{-\Delta/2}$ as desired. 

Suppose now that $\delta j^2> j/10$. Since $j\leq \alpha/\lam$ we have by the assumption on $\delta$ that
\[
\delta j^2\leq \delta j\alpha /\lam \leq j n\lam^2/6 \leq jb\lam^2/2\, ,
\]
where for the final inequality we used that $(A,B)$ is weakly balanced. 
Returning to~\eqref{eqLocalDegBd} and using that $r_A/(1-r_A)\leq 2q_A/(1-2q_A)\leq  3\lam e^{-b\lam^2}$, we have
\begin{align}\label{eqEvjBound2}
\P( \mathcal E(v,j) | \bG_v=G)
\leq 
\binom{a}{j}\left(3\lam e^{-b\lam^2/2} \right)^j\leq  \left(\frac{3ea \lam e^{-b\lam^2/2} }{j} \right)^j\, .
\end{align}
Recall that by assumption
\[
j>\frac{1}{10\delta}\geq \frac{3\alpha}{5n\lam^3}\geq 3e^2a \lam e^{-b\lam^2/2}\, ,
\]
where the final inequality holds by taking $\omega$ a sufficiently large constant. 
By assumption we also have that $j\geq \Delta/2$ and so the RHS of~\eqref{eqEvjBound2} is at most $e^{-\Delta/2}$ as desired. This concludes the proof of~\eqref{eqPGnotinTAB1}.

 We now turn our attention to \eqref{eqSTEdgeDeviation1}.  Note that
\begin{align}\label{eqdegFirstThenEdge1}
\P(|\mathbf{G}| \geq K )
& \leq  \P(\Delta(\mathbf{G})\geq  \Delta)+ \P(\Delta(\mathbf{G})\leq  \Delta, |\mathbf G|\geq K)\, .
\end{align}
Inequality~\eqref{eqPGnotinTAB1} bounds the first probability on the RHS. We now bound the second probability. 
Suppose first that $qn\geq \log n$ where we recall that $q=\max\{q_A, q_B\}$. Then $\Delta(\mathbf G) \le \Delta$ implies that $|\mathbf G|\leq n\Delta/2=25n^2q<K$ so that  $\P(\Delta(\mathbf{G})\leq  \Delta, |\mathbf G|\geq K)=0$. We may therefore assume that $qn< \log n$.

Fix $G$ such that $\Delta(G)\leq \Delta$. By~\eqref{eqnurfDef}, the definition of  $\nu^f_{\bm r,\cE}$, we have
\begin{align}\label{eqPSTempty}
\nu^f_{\bm r,\cE}(G) \leq \frac{\nu_{\bm r,\mathcal E}(G) e^{f(G)}}{\nu_{\bm r,\mathcal E}(\emptyset) e^{f(\emptyset)}} \leq \left(\frac{r_A}{1-r_A}\right)^{|G_A|}\left(\frac{r_B}{1-r_B}\right)^{|G_B|}e^{\delta |G| \Delta}\, ,
\end{align}
where for the final inequality we used that $f$ is strongly $\delta$-local. 
We conclude that
\begin{align}\label{eqDSTDelta1}
\P(\Delta(\bG)\leq  \Delta, |\bG|\geq K) 
&\leq  \sum_{G: |G|\geq K} \left(\frac{r_A e^{\delta\Delta}}{1-r_A}\right)^{|G_A|}\left(\frac{r_Be^{\delta\Delta}}{1-r_B}\right)^{|G_B|}\, .
\end{align}
 Let
\[
\frac{\hat r_A}{1-\hat r_A}=\frac{r_A e^{\delta\Delta}}{1-r_A} \quad\text{and}\quad \frac{\hat r_B}{1-\hat r_B}=\frac{r_B e^{\delta\Delta}}{1-r_B} \, ,
\]
and note that $\hat r_A=(1+o(1)) r_A$, $\hat r_B=(1+o(1))  r_B$.
Let $\mathbf{G_1}\sim G(A,\hat r_A)$ and  $\mathbf{G_2}\sim G(B,\hat r_B)$, then the RHS of~\eqref{eqDSTDelta1} is equal to
\begin{align*}
 \sum_{G: |G|\geq K} \left(\frac{\hat r_A}{1-\hat r_A}\right)^{|G_A|}\left(\frac{\hat r_B}{1-\hat r_B}\right)^{|G_B|} 
 &= (1-\hat r_A)^{-\binom{a}{2}}(1-\hat r_B)^{-\binom{b}{2}}\cdot \P(|\mathbf{G_1}|+|\mathbf{G_2}|\geq K)\, ,\\
 &\leq e^{n^2q} \cdot \P(|\mathbf{G_1}|+|\mathbf{G_2}|\geq K)\, .
\end{align*}
We now apply the Chernoff bound. Note that $\E(|\mathbf{G_1}|+|\mathbf{G_2}|)\leq n^2 q$ and $K\geq 50n^2q$ so by Lemma~\ref{lemChernoff}
\[
 \P(|\mathbf{G_1}|+|\mathbf{G_2}|\geq K)\leq e^{-K}\, .
\]
Putting everything together we have
\[
\P(\Delta(\mathbf G)<  \Delta, |\mathbf G|\geq K)
 \leq 
\exp\left \{ n^2q -K\right\}\leq e^{-K/2} \leq e^{-\Delta/2}.
\]
Inequality~\eqref{eqSTEdgeDeviation1} now follows from \eqref{eqdegFirstThenEdge1} and \eqref{eqPGnotinTAB1}.

If instead $\max_{H\in \mathcal E}\Delta(H)\leq \Delta$ and $f$ is $\delta$-local (rather than strongly $\delta$-local) then returning to~\eqref{eqLocalDegBd} we have 
\begin{align}
\P( \mathcal E(v,j) | \bG_v=G)&\leq 
\sum_{J\subseteq E[v]: |J|=j}\left(\frac{r_A}{1-r_A}\right)^j e^{\delta j\Delta}\leq \binom{a}{j}\left(\frac{2 r_A}{1-r_A}\right)^j \leq e^{-\Delta/2}\, ,
\end{align}
where for the second inequality we used that $\delta\Delta\leq 2$ for $\omega$ sufficiently large and for the final inequality we used that $j\geq \Delta/2$ by assumption. This establishes~\eqref{eqPGnotinTAB1}. The proof of~\eqref{eqSTEdgeDeviation1} is identical to the one given above. 
\end{proof}

Proposition~\ref{propMaxDegBootstrap} now follows. Recall the definitions of $\mu^{\text{w}}_{A,B,\lam}$ and $\cT_{A,B,\lam}$ from~\eqref{eqmuABdef} and \eqref{eqTABwDefAgain} respectively.
\begin{proof}[Proof of Proposition~\ref{propMaxDegBootstrap}]
Fix $(A,B)\in \Pi_\text{weak}$, let
$\mathbf G\sim \mu^{\text{w}}_{A,B,\lam}$ so that 
$(\mathbf G_A, \mathbf G_B)$ is the defect graph of $\mathbf G$.
 Then $(\mathbf G_A, \mathbf G_B)\sim \nu^{\text{w}}_{A,B,\lam}$ and so by Lemma~\ref{lemfofNuisLocal} and Lemma~\ref{lemDegBdLocal} with $\cE=\cD^{\textup{w}}_{A,B,\lam} $ (recalling that $\alpha=1/(96e^3)$)
\[
\P(\mathbf G \in \cT_{A,B,\lam}) = \P((\mathbf G_A, \mathbf G_B) \in \cD_{A,B,\lam})\geq (1-3n^2e^{-\Delta/2})\, .
\]
By the definition of $\mu^{\text{w}}_{A,B,\lam}$ we then have
\[
1> \P(\mathbf G \in \cT_{A,B,\lam})= \frac{Z_{A,B}(\lam)}{Z_{A,B}^{\text{w}}(\lam)}\geq (1-3n^2e^{-\Delta/2})
\]
as desired.
\end{proof}

\section{Subgraph probabilities in the defect graph}\label{secSubgraphProb}
Given Proposition~\ref{propMaxDegBootstrap}, we now turn to the task of understanding the partition function $Z_{A,B}(\lam)$  defined in~\eqref{eqZABdef}. To this end it will be useful to first study the defect measure $\nu_{A,B,\lam}$ defined in~\eqref{eqnuABdef}. As in the previous section, we take a more general view and study measures of the form $\nu^f_{\bm r,\mathcal \cE}$ defined in~\eqref{eqnurfDef}. 

Given $\bG\sim \nu^f_{\bm r,\mathcal \cE}$, our first goal will be to estimate probabilities of the form $\P(F\subseteq \bG)$ for some fixed, small graph $F$. We will use these estimates to bound statistics related to subgraph counts of $\bG$, such as the variance of the number of edges or $P_2$'s in $\bG$.

Throughout this section we fix $(A,B)\in \Pi_{\text{weak}}$ and let $\Delta=\Delta_{A,B,\lam}$, $K=K_{A,B,\lam}$, and $\cD=\cD_{A,B,\lam}$ as in Definition~\ref{defsparse} and~\eqref{eqDABwDefAgain}. 
Since the calculations  in this section are somewhat technical, we begin with a special case as a warm-up. 

\subsection{Warm-up} Recall from~\eqref{eqGVqpsiDef} that for $r \in (0,1)$, $\psi \in \R$, and a vertex set $V\subseteq [n]$, we let $G(V,r,\psi)$ denote the random graph on $V$ with distribution 
\begin{align}
  \nu_{r,\psi} (G) \propto \left(  \frac{r}{1-r} \right)^{|G|}  e^{\psi P_2(G)} \, ,
  \end{align}
 {conditioned} on the event that $\Delta(G) \le d:=50\max\{rn, \log n\}$ and $G$ is triangle-free. 

\begin{lemma}\label{lemWarmup}
Let $V\subseteq [n]$. Let $\bG\sim G(V,r,\psi)$ where $r=o(1)$ and $\psi d=o(1)$. Let $e\in \binom{V}{2}$ and let $\bH=\bG\backslash e$. Let $H$ be a graph such that $H\cup e$ is triangle-free and $\Delta(H\cup e)\leq d$. Then 
\[
\P(e\in \bG \mid \bH=H) = (1+O(r+\psi d)) r
\]
In particular, for any $S\subseteq \binom{V}{2}$ and any event $\cE$ defined by the presence or absence of edges in $S$, $G(V,r,\psi)$ conditioned on $\cE$  is stochastically dominated by $G(V, r')$  conditioned on $\cE$ for some $r'= (1+O(r+\psi d)) r$.
\end{lemma}
\begin{proof}
Let $\nu= \nu_{A, r, \psi}$.
\begin{align}\label{eqWarmup}
\P(e\in \bG \mid \bH=H) = \frac{ \nu(H \cup e) }{\nu(H)+\nu(H \cup e)} = \frac{\left(\frac{r}{1-r} \right)e^{\psi P_2(H\cup e)-\psi P_2(H)}}{1+\left(\frac{r}{1-r} \right)e^{\psi P_2(H\cup e)-\psi P_2(H)}}\, .
\end{align}
Since $H$ has maximum degree $d$,  $\psi P_2(H\cup e)-\psi P_2(H)=O(\psi d)=o(1)$.
The result follows by observing that the denominator on the RHS of~\eqref{eqWarmup} is $1+O(r)$ and the numerator is $(1+O(r+\psi d)) r$.
\end{proof}

\subsection{Master subgraph probability estimate}
In this section we prove a generalisation of Lemma~\ref{lemWarmup} which we later use to derive subgraph probability estimates in $\nu_{A,B,\lam}$ as well as other consequences. 

Given a collection of triangles and edges $X\subseteq \binom{A}{2}\cup \binom{B}{2}\cup\binom{A}{3}\cup \binom{B}{3}$, let 
{\small
\[
\cD_X:= \left\{G\subseteq \binom{A}{2}\cup\binom{B}{2}: \text{$G$ contains no triangle or edge from $X$ and $\Delta(G)\leq \Delta$, $|G_A|, |G_B|\leq K$}  \right\}\, .
\]}
Note that if $X=\binom{A}{3}\cup \binom{B}{3}$ then $\cD=\cD_X$. The reason for considering $\cD_X$ is that in certain probability estimates, we will successively condition on the absence of edges/triangles (see, e.g., Lemma~\ref{lemJansonERG} below).

As in the previous section, throughout this section we assume that $\lam \geq \omega/\sqrt{n}$ where $\omega>0$ is a sufficiently large absolute constant. 
 It will be useful to note that in this regime,
\begin{align}\label{eqqEstomega}
 q=\max\{q_A, q_B\}\leq \omega e^{-\omega^2/3}n^{-1/2}\, . 
 \end{align}

\begin{lemma}\label{lemfixedconfiglocal}
Let $X\subseteq \binom{A}{2}\cup \binom{B}{2}\cup\binom{A}{3}\cup \binom{B}{3}$.
Let $F\in \cD_X$, $|F|=O(1)$, $\delta\leq n\lam^3/(6\alpha)$ and let $f : \cD_X\to\R$ be $\delta$-local. Let $\bm r$ be such that $r:=\max\{r_A, r_B\}=O(q)$.
If $\bG\sim \nu^f_{\bm r, \cD_X}$, and $\bH=\bG\backslash F$ then
\begin{align}\label{eqlemfixedconfiglocal}
\P(F\subseteq \bG)=\left(1+O\left(n^2\Delta^2\lam^6\right)\right)e^{\E[f(\bH\cup F)-f(\bH)]} r_A^{|F_A|}r_B^{|F_B|}\, .
\end{align}
In particular,
\begin{align}\label{eqlemfixedconfiglocalcor}
\P(F\subseteq \bG)=\left(1+O\left(n\Delta\lam^3\right)\right)r_A^{|F_A|}r_B^{|F_B|}.
\end{align}
\end{lemma}
\begin{proof}
Let $\cH_F$ denote the set of all graphs $H\subseteq \binom{A}{2}\cup \binom{B}{2}$ that are edge-disjoint from $F$ and $H\cup F\in  \cD_X$.
For $H\in \cH_F$, we have 
\begin{align}
\P(F\subseteq \bG \mid \bH=H)
= \frac{\nu^f_{\bm r, \cD_X}(H\cup F)}{\sum_{J\subseteq F}\nu^f_{\bm r, \cD_X}(H\cup J) }
=
\frac{\left(\frac{r_A}{1-r_A}\right)^{|F_A|}\left(\frac{r_B}{1-r_B}\right)^{|F_B|}e^{f(H\cup F)-f(H)}}{\sum_{J\subseteq F}\left(\frac{r_A}{1-r_A}\right)^{|J_A|}\left(\frac{r_B}{1-r_B}\right)^{|J_B|}e^{f(H\cup J)-f(H)}}\, .
\end{align}
Since $f$ is $\delta$-local, $f(H\cup J)-f(H)= O(\Delta \delta)=O(n\Delta\lam^3)=O(1)$ for all $J\subseteq F$. Considering the contribution to the sum in the denominator from $J=\emptyset$ and $J\neq \emptyset$ we see that the denominator is $1+O(r)$. Letting $g(H,F)= f(H\cup F)-f(H)$ we then have
\begin{align}
\P(F\subseteq \bG \mid \bH=H)&= (1+O(r)) r_A^{|F_A|} r_B^{|F_B|}e^{g(H,F)}\, .
\end{align}

If $H\notin  \cH_F$, then $\P(F\subseteq \bG \mid \bH=H)=0$ \footnote{If $\P(\bH=H)=0$ then we define $\P(F\subseteq \bG \mid \bH=H)$ to be $0$.} so that for all $H\subseteq \binom{A}{2}\cup \binom{B}{2}$,
\[
\P(F\subseteq \bG \mid \bH=H)= (1+O(r)) r_A^{|F_A|} r_B^{|F_B|}e^{g(H,F)} \cdot \mathbf 1_{H\in \cH_F}\, ,
\]
and so 
\begin{align}\label{eqFdeltaFirst}
\P(F\subseteq \bG)=(1+O(r)) r_A^{|F_A|}r_B^{|F_B|}\cdot \E\left[e^{g(\bH,F)} \cdot \mathbf 1_{\bH\in \cH_F}\right]\, .
\end{align}
Now since $g(\bH,F)=O(\Delta \delta)=O(n\Delta \lam^3)=O(1)$ we have
\begin{align}\label{eqindBound}
\E\left[e^{g(\bH,F)}\mathbf 1_{\bH\in\cH_F}\right]  = \E\left[e^{g(\bH,F)}\right] + O\left(\P(\bH\notin\cH_F) \right)\, .
\end{align}
Moreover, 
\begin{multline}\label{eqExpFirstOBound}
\E\left[e^{g(\bH,F)}\right] 
= \\
 \E\left[1+g(\bH,F)\right] + O(n^2\Delta^2 \lam^6)
 = e^{ \E[g(\bH,F)] } + O(n^2\Delta^2 \lam^6)
 = \left( 1 + O(n^2\Delta^2 \lam^6) \right) e^{ \E[g(\bH,F)] }\, .
\end{multline}

We now turn to estimating $\P(\bH\notin \cH_F)$. Let $\mathcal{A}_1, \mathcal{A}_2$ denote the collections of all possible edges, copies of $P_2$ in $\bG$ respectively and let
\begin{align}\label{eqp1max}
 p_1=\max_{e\in \mathcal{A}_1}\P(e \subseteq \bG)\, ,
 \end{align}
 and
 \begin{align}\label{eqp2max}
 p_2=\max_{f\in \mathcal{A}_2}\P(f \subseteq \bG)\, .
 \end{align}
Let $t(F)$ denote the number of edges in $e\in \cA_1$ such that $e\cup F$ contains a triangle. 
Now, $\bH\cup F$ contains a triangle in $X$ only if $\bH$ contains one of at most $t(F)$ edges or $O(n)$ copies of $P_2$. By a union bound this occurs with probability at most $O(np_2) + t(F)p_1$. 
If $\Delta(\bH \cup F)>\Delta$ then $\Delta(\bG)>\Delta-O(1)>\Delta/2$ which, by Lemma~\ref{lemDegBdLocal}, occurs with probability at most $n^2 e^{-\Delta/2}$. 
If $|\bH \cup F|>K=K_{A,B,\lam}$ then $|\bG|>K-O(1)>K/2$ which, by Lemma~\ref{lemDegBdLocal}, occurs with probability at most $2n^2 e^{-\Delta/2}$. Finally note that $\bG$ deterministically contains no edge in $X$ and so the same is true of $\bH$.
We conclude that
\[
\P(\bH \notin \cH_F)\leq 3n^2 e^{-\Delta/2} + O(np_2)+t(F)p_1\, .
\]
Combining this with~\eqref{eqFdeltaFirst}~\eqref{eqindBound} and~\eqref{eqExpFirstOBound} we have
\begin{equation}\label{eqFApproxFinal}
\P(F\subseteq \bG)=\left(1+O\left(n^2\Delta^2 \lam^6+r +  n^2 e^{-\Delta/2} + np_2+t(F)p_1\right)\right)r_A^{|F_A|}r_B^{|F_B|}e^{\E[g(\bH,F)]}\, .
\end{equation}
To conclude the proof we will need a rough estimate on $p_1, p_2$. First note that $e^{\E[g(\bH,F)]}=O(1)$. Taking $F\in \cA_1$ to be an edge that witnesses the maximum in~\eqref{eqp1max}, and noting that $t(F)=0$, we have by~\eqref{eqFApproxFinal}
\begin{align}\label{eqp1Approx1}
p_1=O(1+np_2)r\, .
\end{align}
Taking $F\in \cA_2$ to be a copy of $P_2$ that witnesses the maximum in~\eqref{eqp2max} we have
\[
p_2=O(1+np_2+p_1)r^2=O(1+np_2)r^2\, ,
\]
where for the second equality we used~\eqref{eqp1Approx1}. Since $r=O(q)$ we have $nr^2=o_\omega(1)$ by~\eqref{eqqEstomega} and so we conclude that for $\omega$ sufficiently large,
$p_2=O(r^2)$.
Returning to~\eqref{eqp1Approx1} we then have $p_1=O(r)$. 
 Using the previous two estimates in~\eqref{eqFApproxFinal}, noting that $t(F)=O(1)$, and moreover, $n^2e^{-\Delta/2} + nr^2+r=O(n^2\Delta^2\lam^6)$ gives~\eqref{eqlemfixedconfiglocal}.

For~\eqref{eqlemfixedconfiglocalcor} we note that $e^{\E[g(\bH,F)]}=1+O(n\Delta\lam^3)$.
\end{proof}

We have the following immediate corollary of Lemmas~\ref{lemfofNuisLocal} and \ref{lemfixedconfiglocal} (recalling that $\alpha= 1/(96e^3)$).

\begin{cor}\label{corfixedconfig}
Let $F\in \cD$ with $|F|=O(1)$ and let $\bG\sim\nu_{A,B,\lam}$.  Then
\[
\P(F\subseteq \bG) = (1+ O(n \Delta \lam^3))q_A^{|F_A|}q_B^{|F_B|}\, .
\]
\end{cor}
In the following two subsections we record some further consequences of Lemma~\ref{lemfixedconfiglocal}.

\subsection{A refined subgraph probability estimate for $\nu_{A,B,\lam}$.}

Our next goal will be to bootstrap Corollary~\ref{corfixedconfig} to give a more refined estimate on the probability that $F$ is contained in a sample from $\nu_{A,B,\lam}$. First we give a slightly more detailed description of $\nu_{A,B,\lam}$ than that given by Lemma~\ref{lemfofNuisLocal}.
Recall the definition of $q_A', q_B'$ from~\eqref{eqqPrimeDef}.

\begin{lemma}\label{lemfislocal}
There exists an $(11n\lam^3)$-local $f: \cD\to \R$ such that $\nu_{A,B,\lam}= \nu^f_{\bm r,\cD}$ with $r_A=q'_A, r_B=q'_B$.  Moreover
\[
f(G)= P_2(G_\boxempty)\lam^3 + f'(G)
\]
where for $F\subseteq G$,
\begin{align}\label{eq:fPrimeLocal}
|f'(G)-f'(G\backslash F)| = O(n|F|\Delta^2\lam^4)\, .
\end{align}

\end{lemma}
\begin{proof}
Recall that $\nu_{A,B,\lam}$ is the measure on $ \cD$ given by
\begin{align*}
\nu_{A,B,\lam}(G)\propto \lam^{|G|}Z_{G_\boxempty}(\lam)\, .
\end{align*}
By cluster expansion (Lemma~\ref{lem:ClusterExpMod}~ and Corollary~\ref{corclustersimple})
\begin{align}\label{eq:ZGboxExpand}
\log\left( \frac{Z_{G_\boxempty}(\lam)}{(1+\lam)^{ab}} \right) = |G_\boxempty|(-\lam^2 +2\lam^3) +  P_2(G_\boxempty)\lam^3 + \sum_{\substack{\Gamma\in \cC'(G_\boxempty):\\ |\Gamma|\geq 4}} \phi(\Gamma) \lam^{|\Gamma|}\, ,
\end{align}
where
$\cC'(G_\boxempty)$ denotes the set of non-constant clusters of $G_\boxempty$.
Let 
\[
f'(G)= \sum_{\substack{\Gamma\in \cC'(G_\boxempty):\\ |\Gamma|\geq 4}} \phi(\Gamma) \lam^{|\Gamma|}\, .
\]

Let $F\subseteq G$ and let $H=G\backslash F$,
then
\[
f'(G)-f'(H)=
 \sum_{\Gamma\in \cC'': |\Gamma|\geq 4} \phi(\Gamma) \lam^{|\Gamma|}\,,
\]
where 
$\cC''=  \cC'(G_\boxempty)\backslash  \cC'(H_\boxempty)$.

Now if $\Gamma\in \cC''$ then $\Gamma$ 
must contain a pair $S=\left\{(v_1, w), (v_2, w)\right\}$ (a pair of vertices of $G_\boxempty$) such that $\{v_1, v_2\}\in F_A$ or a pair $S=\left\{(v, w_1), (v, w_2)\right\}$ such that $\{w_1, w_2\}\in F_B$. Since there are at most $b|F_A|+a|F_B|\leq n|F|$ such pairs of vertices and $\Delta(G_\boxempty)\leq 2\Delta$, we have by Lemma~\ref{lemClusterTail} (applied with $k=4$, $t=0$ and $S$, for each of the aforementioned pairs $S$),
\begin{align}\label{eqtailbdgeq4}
| f'(G)-f'(H)|=
\left| \sum_{\Gamma\in \cC'': |\Gamma|\geq 4} \phi(\Gamma) \lam^{|\Gamma|}\right|\leq (2e)^4n|F| (2\Delta)^2 \lam^4\leq n|F|\Delta\lam^3
\end{align}
where for the last inequality we used that $\lam\Delta=o_\omega(1)$. The first inequality above establishes~\eqref{eq:fPrimeLocal}.
Next note that 
\begin{align}\label{eq:GboxEdgeExp}
|G_{\boxempty}|=a|G_B| + b|G_A|
\end{align}
and
\begin{align}
P_2(G_\boxempty)= bP_2(G_A) + aP_2(G_B) + 4|G_A||G_B|\, .
\end{align}
Given graphs $H_1, H_2$, let $P_2(H_1, H_2)$ denote the number of copies of $P_2$ in $H_1\cup H_2$ with at least one edge in $H_1$. Then
\begin{multline}
P_2(G_\boxempty)- P_2(H_\boxempty)
= \\
bP_2(F_A, H_A) + aP_2(F_B, H_B) +4(|H_A||F_B| + |H_B||F_A| + |F_A||F_B|)
\leq
10n\Delta |F|\, .\label{eqP2local}
\end{multline}
Letting  $f(G)=P_2(G_\boxempty)\lam^3 +f'(G)$, we conclude that $f$ is $(11n\lam^3)$-local. Moreover, by~\eqref{eq:ZGboxExpand} and~\eqref{eq:GboxEdgeExp} we conclude that $\nu_{A,B,\lam}= \nu^f_{\bm r,\cD}$ with $r_A=q'_A, r_B=q'_B$.
\end{proof}

We can now prove the following refinement of Corollary~\ref{corfixedconfig}.

\begin{cor}\label{corfixedconfigboot}
Let $F\in \cD$ with $|F|=O(1)$ and let $\bG\sim\nu_{A,B,\lam}$.  Then
\begin{multline}
\P(F\subseteq \bG)= \\(1+ O(\lam^6n^2\Delta^2))\left(q'_A e^{2b\lam^3(aq_A+bq_B)}\right)^{|F_A|}\left(q'_B e^{2a\lam^3(aq_A+bq_B)}\right)^{|F_B|}e^{\lam^3 (bP_2(F_A)+aP_2(F_B))}\, .
\end{multline}
\end{cor}
\begin{proof}
By Lemmas~\ref{lemfixedconfiglocal} and~\ref{lemfislocal} we have
\begin{align}\label{eqlemfixedconfiglocal3}
\P(F\subseteq \bG)=\left(1+O\left(n^2\Delta^2\lam^6 \right)\right)e^{\E[f(\bH\cup F)-f(\bH)]} (q_A')^{|F_A|}(q_B')^{|F_B|}\, ,
\end{align} 
where $f$ is as in Lemma~\ref{lemfislocal} and $\bH=\bG\backslash F$.
We turn to estimating the expectation in the exponent. 

 By~\eqref{eqP2local} and the definition of $f$,
\begin{multline}
f(\bH\cup F)-f(\bH) = \\ \lam^3\left[bP_2(F_A, \mathbf H_A) + aP_2(F_B, \mathbf H_B) +4(|\mathbf H_A||F_B| + |\mathbf H_B||F_A| + |F_A||F_B|)\right] + O(n\Delta^2\lam^4)\, .
\end{multline}
By Corollary~\ref{corfixedconfig} we have
\begin{equation}
\E(|\bH_A|)
= (1+ O(n\Delta \lam^3))q_A\left(\binom{a}{2}- |F_A| \right)
=q_Aa^2/2 + O(n\Delta \lam^3\cdot qn^2)\, .
\end{equation}
Suppose now that $\{u,v\}\in F_A$. Each edge of $\bH$ which is incident to either $u$ or $v$ contributes one $P_2$ to the count $P_2(F_A, \bH_A)$. Applying Corollary~\ref{corfixedconfig} and summing these contributions over the edges of $F_A$ yields
\begin{align*}
\E \left(P_2(F_A,\bH_A)\right)
&= 2|F_A|(1+ O(n\Delta \lam^3))q_A(a-O(1)) + P_2(F_A)\\
&=2|F_A|q_Aa + P_2(F_A) + O(n\Delta \lam^3\cdot nq)\, .
\end{align*}
It follows that
\begin{align*}
\E(f(\bH\cup F)-f(\bH))
=& 2\lam^3(b|F_A|+a|F_B|)(aq_A+bq_B)+ b\lam^3P_2(F_A)+ a\lam^3P_2(F_B) \\
& + O(\lam^3+n^3\Delta \lam^6q+n\Delta^2\lam^4)\, .
\end{align*}
The result follows from~\eqref{eqlemfixedconfiglocal3}, since $\lam^3+n^3\Delta \lam^6q+n\Delta^2\lam^4= O(\lam^6n^2\Delta^2)$.

\end{proof}

We state one further corollary that will prove useful in Section~\ref{secFixedM}.
\begin{cor}\label{coredgevar}
Let $\bG\sim\nu_{A,B,\lam}$. 
Then
\[
\var(|\bG|)=O(n^2q + \lam^6n^6\Delta^2q^2)\, .
\]
In particular, if $q=O(n^{-7/8-\eps})$ for some $\eps>0$, then
\[
\var(|\bG|)= O(n^{3/2-\eps})\, .
\]
\end{cor}
\begin{proof}
For $e\in\binom{A}{2}, f\in\binom{B}{2}$, let $X_e, Y_f$ denote the indicators of the events that $e,f$ respectively are  edges of $\bG$.
By Corollary~\ref{corfixedconfig},
\[
\var(X_e)= O(q)\, .
\]
If $e,f\in \binom{A}{2}$ are such that $e\cup f$ forms a copy of $P_2$, then by Corollary~\ref{corfixedconfigboot}
 \begin{align*}
\cov(X_e, X_f)
&= (1+ O(\lam^6n^2\Delta^2))\left(q_A' e^{2b\lam^3(aq_A+bq_B)}\right)^{2}e^{b\lam^3}-\\
&\phantom{=.}
(1+ O(\lam^6n^2\Delta^2))\left(q_A' e^{2b\lam^3(aq_A+bq_B)}\right)^{2}\\
&= O(n\lam^3q^2 +\lam^6n^2\Delta^2q^2 )\, .
\end{align*}
If $e,f\in\binom{A}{2}$ are vertex-disjoint then
\begin{align*}
\cov(X_e, X_f)
&= (1+ O(\lam^6n^2\Delta^2))\left(q_A' e^{2b\lam^3(aq_A+bq_B)}\right)^{2}
-
(1+ O(\lam^6n^2\Delta^2))\left(q_A' e^{2b\lam^3(aq_A+bq_B)}\right)^{2}\\
&= O(\lam^6n^2\Delta^2q^2 )\, .
\end{align*}
Similarly, if $e\in\binom{A}{2}, f\in\binom{B}{2} $, then $\cov(X_e, Y_f)=O(\lam^6n^2\Delta^2q^2 )$.
We conclude that
\[
\var(|\bG|)= O(n^2q + n^3(n\lam^3q^2 +\lam^6n^2\Delta^2q^2) + n^4\cdot\lam^6n^2\Delta^2q^2)
=O(n^2q + \lam^6n^6\Delta^2q^2)\, . \qedhere
\]
\end{proof}

\subsection{Janson's inequality for perturbed measures}\label{secJanson}
A canonical application of Janson's inequality~\cite{janson1987uczak} is to estimate the probability that the \ER random graph is triangle-free. 
As a final application of Lemma~\ref{lemfixedconfiglocal}, we prove an analogous estimate in our setting of locally perturbed measures. 

\begin{lemma}\label{lemJansonERG}
Let $\delta\leq n\lam^3/(6\alpha)$ and let $f : \cD_\emptyset\to\R$ be $\delta$-local. Let $\bm r$ be such that $r:=\max\{r_A, r_B\}=O(q)$.
If $\bG\sim \nu^f_{\bm r, \cD_\emptyset}$, then
\[ \P( \bG \textup{ is triangle-free}) =\exp \left ( - r_A^3\binom{a}{3} - r^3_B\binom{b}{3} + O(n^4\Delta\lam^3q^3) \right) \,. \]
\end{lemma}
\begin{proof}
Let $N= \binom{a}{3}+ \binom{b}{3}$ and let $\{T_1, T_2,\ldots, T_N\}= \binom{A}{3}\cup \binom{B}{3}$. 
Let $A_i$ denote the event that the triangle $T_i$ is contained in $\bG$. Then
\[
\P( \bG \textup{ is triangle-free}) = \P\left (\bigcap_{i=1}^{N}A_i^c \right) = \prod_{i=1}^N  \P\left (A_i^c \, \bigg |\,   \bigcap_{j<i}A_j^c\right) = \prod_{i=1}^N\left[1-  \P\left (A_i \, \bigg |\,   \bigcap_{j<i}A_j^c\right)\right]\, .
\]
Fix $i\in[N]$, let $X=\{T_1, T_2,\ldots, T_{i-1}\}$, and let $\bG'\sim \nu^f_{\bm r, \cD_X}$. Then by Lemma~\ref{lemfixedconfiglocal}
\[
 \P\left (A_i \, \bigg |\,   \bigcap_{j<i}A_j^c\right)= \P(T_i\subseteq \bG') 
 =
 \begin{cases}
 \left(1+O\left(n\Delta\lam^3\right)\right)r_A^{3} &\text{ if $T_i\in \binom{A}{3}$,}\\
\left(1+O\left(n\Delta\lam^3\right)\right)r_B^{3} &\text{ if $T_i\in \binom{B}{3}$.}
 \end{cases}
\]
It follows that
\begin{align*}
\P( \bG \textup{ is triangle-free}) 
&=  \left[1 - r_A^3+O(n\Delta\lam^3q^3))\right]^{\binom{a}{3}}
 \left[1 - r_B^3+O(n\Delta\lam^3q^3))\right]^{\binom{b}{3}}\\
&=
\exp\left\{- (r_A^3+O(n\Delta\lam^3q^3))\binom{a}{3}- (r_B^3+O(n\Delta\lam^3q^3))\binom{b}{3}\right\}\, .
\end{align*}
\end{proof}

We will apply Lemma~\ref{lemJansonERG} in Section~\ref{secSuperCrit}. 
For now it will be useful to note the following corollary which could also be proved using a combination of Janson's inequality and the Harris-FKG inequality. 
Recall the definition of $\nu_{\bm r}$ from~\eqref{eqnubmrdef}.

\begin{cor}\label{corZABJanson}
Let $r_A, r_B\in(0,1)$ be such that $\max\{r_A, r_B\}=O(q)$ and let $\bG\sim \nu_{\bm r}$.
Then 
\begin{align}\label{eqProbGinDEmptyEst}
\P(\bG\in \cD_{\emptyset}) =  1+o(1) .
\end{align}
Moreover,
\begin{align}\label{eqProbGinDEst}
\P(\bG\in \cD) =  \exp \left ( - r_A^3\binom{a}{3} - r_B^3\binom{b}{3} + O(n^4\Delta\lam^3q^3) \right)\, ,
\end{align}
and
\begin{align*}
\sum_{(S,T)\in \cD}\left(\frac{r_A}{1-r_A} \right)^{|S|}\left(\frac{r_B}{1-r_B} \right)^{|T|}
&= \exp\left\{\frac{1}{2}a^2r_A+ \frac{1}{2}b^2r_B+O(nq+n^3q^3) \right\}\, .
\end{align*}
\end{cor}
\begin{proof}
By the definition of $\cD$ and $\cD_{\emptyset}$ we have
\[
\P(\bG\in \cD) = \P(\bG\in \cD_\emptyset)\cdot \P(\bG \text{ is triangle-free}\mid \bG\in \cD_\emptyset)\, .
\]
We estimate the probabilities separately starting with $\P(\bG\in \cD_\emptyset)$.
Note that $\E(|\bG|)=\binom{a}{2}r_A+\binom{b}{2}r_B\leq n^2q\leq K/50$ so that by Chernoff's inequality (Lemma~\ref{lemChernoff})
\[
 \P(|\bG|> K) \leq e^{-K}\leq e^{-\Delta}\, .
\]
For $v\in A\cup B$ we have $\E(d_{\bG}(v))= qn\leq\Delta/50$ and so by Chernoff's inequality and a union bound we have
\[
 \P(\Delta(\bG)> \Delta)\leq ne^{-\Delta}\, .
\]
We conclude that 
\begin{align}\label{eqDemptyEst}
 \P(\bG\in \cD_\emptyset)= 1- O(ne^{-\Delta})
\end{align}
and so~\eqref{eqProbGinDEmptyEst} follows.
Letting $\bG'\sim \nu_{\bm r, \cD_{\emptyset}}$ where $\bm r= (r_A, r_B)$ we have, by Lemma~\ref{lemJansonERG},
\begin{align}
 \P(\bG \text{ is triangle-free}\mid \bG\in \cD_\emptyset) 
 &= \P(\bG' \text{ is triangle-free})\\
 &=
 \exp \left ( - r_A^3\binom{a}{3} - r_B^3\binom{b}{3} + O(n^4\Delta\lam^3q^3) \right)\, .
\end{align}
We note that $ne^{-\Delta}\leq n^{-49}$. On the other hand since $\lam\leq 2 \sqrt{\frac{\log n}{n}}$ (by our assumption at~\eqref{lamAssumption}), we have $q/(1-q)\geq \lam e^{-n\lam^2}=\Omega(n^{-5})$, so that $n^4\Delta\lam^3q^3=\Omega(n^{-15})$, in particular  $ne^{-\Delta}=O(n^4\Delta\lam^3q^3)$. Statement~\eqref{eqProbGinDEst} follows. 

Finally, note that  
\begin{align}
\sum_{(S,T)\in \cD}\left(\frac{r_A}{1-r_A} \right)^{|S|}\left(\frac{r_B}{1-r_B} \right)^{|T|}
&=
(1-r_A)^{-\binom{a}{2}}(1-r_B)^{-\binom{b}{2}}\cdot \P(\bG\in \cD)\\
&= \exp\left\{\frac{1}{2}a^2r_A+ \frac{1}{2}b^2r_B+O(nq+n^2q^2) \right\}\cdot \P(\bG\in \cD)\, ,\label{eqERPGinD}
\end{align}
and $\P(\bG\in \cD)=\exp\{O(n^3q^3)\}$ by~\eqref{eqProbGinDEst}.
\end{proof}

\section{From weak to moderately balanced partitions}
In this section we prove Proposition~\ref{lemZmodsimsum} which allows us to ignore partitions that are not moderately balanced.
Recall from Definition~\ref{defModBalanced} that we call a partition $(A,B)\in \Pi$ $\lam$-moderately balanced if
\[
\big ||A|-|B| \big|\leq M_\lam:=\max\{ne^{-\lam^2n/2}, n^{1/2} \}(\log n)^2  \, , 
\]
and we let $\Pi_{\textup{mod},\lam}$ denote the set of all $\lam$-moderately balanced partitions. 

We say a graph $G\in \cT$ is \emph{captured} by $(A,B)$ if $(G_A, G_B)\in \cD_{A,B,\lam}$. We let $c_{\textup{mod},\lam}(G)$ denote the number of $\lam$-moderately balanced partitions that capture $G$.  Note that
\begin{align}\label{eqmodcapture}
\mu_{\textup{mod},\lam}(G)= \frac{\lam^{|G|}}{Z_{\textup{mod}}(\lam)}\cdot c_{\textup{mod},\lam}(G)\, ,
\end{align} 
where we recall the definition of $\mu_{\textup{mod},\lam}$ from Algorithm~\ref{algMuModStrong}. 

For $G\in \cT$, we let $c_{\textup{weak},\lam}(G)$ denote the number of weakly balanced partitions $(A,B)$ such that $(G_A, G_B)\in \cD^{\textup{w}}_{A,B,\lam}$.

The following lemma is a minor variant of Lemma~\ref{lemWeakUnique} and the proof is the same. 
\begin{lemma}\label{lemCaptureMod}
There exists $\omega>0$ such that if $\lam\geq\omega/\sqrt{n}$ and $\bG\sim \mu_{\textup{mod},\lam}$,  then
\[
\P(c_{\textup{weak},\lam}(\bG)= c_{\textup{mod},\lam}(\bG)=  1)\geq 1- e^{-\lam n/25}\, .
\]
\end{lemma}

\begin{proof}
Clearly $c_{\textup{weak},\lam}(\bG)\geq c_{\textup{mod},\lam}(\bG)\geq  1$ with probability $1$. 
Suppose that $(A,B)$ is chosen at Step 1 in Algorithm~\ref{algMuModStrong} and $(S,T)\in \cD_{A,B,\lam}$ is chosen at Step 2. Since  $ \cD_{A,B,\lam}\subseteq  \cD_{A,B,\lam}^{\textup{w}}$, Lemma~\ref{lemExpanderwhp} shows that $\bG$ is an $(A,B)$-$\lam$-expander with probability at least $1-  e^{-\lam n/25}$. We conclude from Lemma~\ref{lemExpansionCor} that $c_{\textup{weak},\lam}(\bG)\leq 1$ with probability at least $1-  e^{-\lam n/25}$.
\end{proof}

\begin{proof}[Proof of Proposition~\ref{lemZmodsimsum}]
Let 
\[
\hat Z_{\textup{weak}}(\lam) = \sum_{(A,B)\in \Pi_{\textup{weak}}} Z_{A,B}(\lam)\, ,
\]
and note that by Proposition~\ref{propMaxDegBootstrap} and the fact that $\Delta_{A,B,\lam}\geq 50 \log n$ for all $(A,B)\in \Pi_{\textup{weak}}$,
\[
\hat Z_{\textup{weak}}(\lam) = \left(1+O\left(n^{-3}\right)\right)  Z_{\textup{weak}}(\lam) \, .
\]
To establish~\eqref{eqZsimZmod} it therefore suffices to show that
\begin{align}\label{eqZsimZmodhat} 
\hat Z_{\textup{weak}}(\lam)=\left(1 + O\left(n^{-3}\right) \right) Z_{\textup{mod}}(\lam)\, ,
\end{align}
We fix $(A,B)\in \Pi_{\textup{weak}}$.
For $(S,T)\in \cD_{A,B,\lam}$ we apply cluster expansion (Corollary~\ref{corclustersimple}) to conclude that
\begin{align}\label{eqZNSTEst}
\log \left(\frac{Z_{S \boxempty T }(\lam)}{(1+\lam)^{ab}} \right)= -|S\boxempty T|\lam^2 + O(N_{S,T}\Delta^2\lam^3)\, ,
\end{align}
where $\Delta=\Delta_{A,B,\lam}$ and $N_{S,T}$ denotes the number of non-isolated vertices in the graph $S\boxempty T$. 
Since $(S,T)\in \cD_{A,B,\lam}$, $S$ and $T$ are both of size at most $K_{A,B,\lam}=50\max\{n^2q, \log n\}$ by definition, where $q=\max\{q_A,q_B\}$. It follows that $|S\boxempty T|=b|S|+a|T|\leq nK$ and so $S\boxempty T$ has at most
\[N:=\min\{n^2, 2nK\}\] 
non-isolated vertices (recall that $V(S\boxempty T)=A\times B$ which has size $ab<n^2$). 

We conclude from~\eqref{eqZNSTEst} and Corollary~\ref{corZABJanson} that
\begin{multline}\label{claimZABJansonApp}
\frac{Z_{A,B}(\lam)}{(1+\lam)^{ab}}= \\ e^{O(N\Delta^2\lam^3)}\sum_{(S,T)\in \cD_{A,B,\lam}}\left(\frac{q_A}{1-q_A} \right)^{|S|}\left(\frac{q_B}{1-q_B} \right)^{|T|}= e^{O(N\Delta^2\lam^3)}\exp\left\{\frac{1}{2}a^2q_A+ \frac{1}{2}b^2q_B \right\}\, ,
\end{multline}
where for the final equality we used that $nq+n^3q^3=O(N\Delta^2\lam^3)$.

Next we study how the expression in~\eqref{claimZABJansonApp} depends on the degree of imbalance of $(A,B)$.

Let $a=n/2-k$ and $b=n/2+k$, where $k\leq n/20$ since $(A,B)$ is weakly balanced (note that $k$ may be half integral). 
We first consider the case where $\lam^2 k=o(1)$. 
In this case we note that 
\(
q_A=(1+O(q))\lam e^{-b\lam^2},
\)
and similarly for $q_B$ so that 
\begin{align}\label{eqa2b2cancel}
q_Aa^2+q_Bb^2
&= \lam e^{-\lam^2b}a^2 + \lam e^{-\lam^2a}b^2 + O(n^2q^2)\nonumber\\
&= \lam e^{-\lam^2n/2}(n/2)^2
\left[e^{-\lam^2k}\frac{(n/2-k)^2}{(n/2)^2} 
+ e^{\lam^2k}\frac{(n/2+k)^2}{(n/2)^2}   \right] +O(n^2q^2)\nonumber\\
&=\lam e^{-\lam^2n/2}(n/2)^2\left[2+O(\lam^4k^2)\right] +O(n^2q^2)\, ,
\end{align}
where we used that $k/n<\lam^2k$.
It follows from~\eqref{claimZABJansonApp} that
\begin{align}\label{eqZABmodapproximation}
\frac{Z_{A,B}(\lam)}{(1+\lam)^{ab}}= e^{O(N\Delta^2\lam^3)}
\exp\left\{ 
\lam e^{-\lam^2n/2}n^2/4
+ O(\lam^5k^2n^2e^{-\lam^2n/2})
 \right\}\, .
\end{align}
For general $k$ (no longer assuming $\lam^2k=o(1)$), we note that
\begin{align*}
q_Aa^2+q_Bb^2
&= O(n^2q)\, ,
\end{align*}
and so
\[
\frac{Z_{A,B}(\lam)}{(1+\lam)^{ab}}= e^{O(n^2q)}\, .
\]
We then have
\begin{align}\label{eqSecondSum}
\left|\frac{\hat Z_{\textup{weak}}(\lam)}{Z_{\textup{mod}}(\lam)}-1\right|
 \leq &\sum_{ 1/(\lam^2\log n) \leq k \leq n/20} e^{O(n^2q)} \frac{\binom{n}{n/2+k}}{\binom{n}{n/2}} (1+\lam)^{-k^2}\\
& + 
\sum_{ M_\lam \leq k< 1/(\lam^2\log n)} e^{O(N\Delta^2\lam^3)} \frac{\binom{n}{n/2+k}}{\binom{n}{n/2}} e^{O(\lam^5k^2n^2e^{-\lam^2n/2})} (1+\lam)^{-k^2}\,,\nonumber
\end{align}
where we note that $N, \Delta, q$ all implicitly depend on $k$. 

To bound the first sum, we first note that $\binom{n}{n/2+k}\leq \binom{n}{n/2}$.
Moreover for weakly balanced $(A,B)$, $q=\max\{q_A, q_B\}=O(n^{-1/2-\eps})$ for some $\eps>0$ since $\lam\geq c\sqrt{\frac{\log n}{n}}$. It follows that for $1/(\lam^2 \log n)\leq k \leq n/20$ we have $n^2q=o(\lam k^2)$ and so the first sum is bounded above by $n e^{-\tilde \Omega(n^{3/2})}$. 

For the second sum we note that $\lam^5k^2n^2e^{-\lam^2n/2}=o(\lam k^2)$. We claim that  $N\Delta^2\lam^3=o(\lam k^2)$ also. In fact, the definition of $M_\lam$ has been chosen so that this is the case. To see this we first note that if $k< 1/(\lam^2 \log n)=O(n/(\log n)^2)$, then 
\[
q=\max\{q_A, q_B\}=O(\lam e^{-\lam^2n/2})=O\left(\sqrt{\frac{\log n}{n}} e^{-\lam^2n/2}\right)\, .
\]
 We consider two cases depending on the size of $q$. If $nq\leq \log n$ then $\Delta=50\log n$ and $N\leq n^2$. It follows that $N\Delta^2\lam^3=o(\lam k^2)$ for $k\geq M_\lam\geq n^{1/2}(\log n)^2$. If $nq>\log n$, then $\Delta=50nq$ and so $N\Delta^2\lam^3=o(\lam k^2)$ for $k\geq M_\lam\geq ne^{-\lam^2n/2}(\log n)^2$.  It follows that the second sum in~\eqref{eqSecondSum} is bounded above by $ne^{-\lam M_\lam^2/2}=O(n^{-3})$. Statement~\eqref{eqZsimZmod} follows. 
 
 For~\eqref{eqmusimmumod} we apply Pinsker's inequality (Lemma~\ref{lemPinsker}). Let $\bG\sim \mu_{\textup{mod},\lam}$, then  by~\eqref{eqmodcapture}, 
\begin{align}\label{eq:Proof3.3}
D_\text{KL}(\mu_{\textup{mod},\lam} \parallel \mu_{\textup{weak},\lam}) 
&= \E_{\mu_{\textup{mod},\lam}} \log\left(\frac{\mu_{\textup{mod},\lam}(\bG)}{\mu_{\textup{weak},\lam}(\bG)} \right)\\
&=
\E_{\mu_{\textup{mod},\lam}}\log \left(\frac{Z_{\textup{weak}}(\lam)}{Z_{\textup{mod}}(\lam)}
\cdot
\frac
{c_{\textup{mod},\lam}(\bG)}
  {c_{\textup{weak},\lam}(\bG)}\right)\\
&= \E_{\mu_{\textup{mod},\lam}}\log \left(\frac
{c_{\textup{mod},\lam}(\bG)}
  {c_{\textup{weak},\lam}(\bG)}\right) + O\left(n^{-3}\right)\, ,
\end{align}
where for the final equality we used~\eqref{eqZsimZmod}.

To conclude the proof we note that
 $\log \left(\frac
{c_{\textup{mod},\lam}(\bG)}
  {c_{\textup{weak},\lam}(\bG)}\right)=O(n)$ deterministically and by Lemma~\ref{lemCaptureMod} 
\[
\P_{\mu_{\textup{mod},\lam}}(c_{\textup{mod},\lam}(\bG)=c_{\textup{weak},\lam}(\bG)= 1)
\geq 1- e^{-\lam n/25}
\, ,
\]
and so 
\[
 \E_{\mu_{\textup{mod},\lam}}\log \left(\frac
{c_{\textup{mod},\lam}(\bG)}
  {c_{\textup{weak},\lam}(\bG)}\right) = O\left(n  e^{-\lam n/25}\right)=O(n^{-3})\ ,
\]
the result follows. 
\end{proof}

\section{The subcritical defect regime} 
\label{secSubCritRegime}

In this section we prove our main results in the subcritical defect regime: Theorem~\ref{thmGnpSubCritLDprob}, Theorem~\ref{thmGnpBipartite} and Theorem~\ref{thmGnpColoring}. 

We begin with Theorem~\ref{thmGnpSubCritLDprob}. The asymptotics for the probability $G(n,p)$ is triangle-free claimed in Theorem~\ref{thmGnpSubCritLDprob} will be an immediate consequence of the  following asymptotics for the partition function $Z(\lam)$ and the identity $\mathbb{P}_{n,p}( \cT)=(1-p)^{\binom{n}{2}} Z(\lam)$ with $\lam=p/(1-p)$.  

\begin{lemma}\label{lemZapproxER}
Fix $\eps>0$ and let $\lam\geq (1+\eps) \sqrt{\frac{\log n }{n}}$. Then
\begin{align}\label{eqZsimSubcritical}
Z(\lam)\sim
\frac{1}{2}\sqrt{\frac{\pi}{\lam}} \binom{n}{\lfloor n/2\rfloor}(1+\lam)^{n^2/4}
\exp\left\{
\lam e^{-\lam^2n/2+\lam^3n}\frac{n^2}{4} + \lam^5 e^{-\lam^2n}\frac{n^4}{8}
\right\}\, .
\end{align}
\end{lemma}

For this entire section we  fix $\eps>0$ and assume $\lam\geq (1+\eps) \sqrt{\frac{\log n }{n}}$.

To begin with, we fix a $\lam$-moderately balanced partition $(A,B)$ (see Definition~\ref{defVarBalanced}) with $a = |A|$ and $b=|B|$ and study the partition function $Z_{A,B}(\lam)$ and defect distribution $\nu_{A,B,\lam}$ (defined in~\eqref{eqZABdef}, \eqref{eqnuABdef} respectively). We will show that the defect distribution $\nu_{A,B,\lam}$ is within $o(1)$ total variation distance of a suitable \ER measure which will be the main step to proving the approximation to $\mu_\lam$ in Theorem~\ref{thmGnpSubCritLDprob}.

Recall from~\eqref{eqqdefGnp} that we define $q_0/(1-q_0)=\lam e^{-\lam^2n/2}$. Let $\bm q = (q_0,q_0)$ and recall the definition of $\nu_{\bm q}$ from~\eqref{eqnubmrdef}, i.e., the distribution of two independent \ER random graphs on $A$ and $B$ with edge probability $q_0$.

The next lemma provides an asymptotic formula for $Z_{A,B}(\lam)/(1+\lam)^{ab}$ for moderately balanced $(A,B)$. To prove Lemma~\ref{lemZapproxER} we sum this formula over partitions $(A,B)$. Importantly, the asymptotic formula does not depend on the imbalance of the sizes of $A,B$.
\begin{lemma}\label{lemZABapprox1}
If $(A,B)\in \Pi_{\textup{mod},\lam}$, then
\begin{align}
\label{eqZabCase1}
\frac{Z_{A,B}(\lam)}{(1+\lam)^{ab}}\sim  \exp\left\{\lam e^{-\lam^2n/2+\lam^3n} \frac{n^2}{4} + \lambda^5e^{-\lambda^2 n}\frac{n^4}{8}\right\}\, .
\end{align}
Moreover,
\begin{align}\label{eqnuABapprox}
 \| \nu_{A,B,\lam} - \nu_{\bm q} \|_{TV}  = o(1)  \,.
 \end{align}
\end{lemma}

\subsection{A first approximation to $Z_{A,B}$}

A key step toward proving Lemma~\ref{lemZABapprox1} is the following approximation of $Z_{A,B}(\lam)$ which we turn to now. Recall the definitions of $q_A, q_B, q_A', q_B'$ from Definition~\ref{defsparse} and~\eqref{eqqPrimeDef}.

\begin{lemma}\label{lemmaZABabapprox}
If $(A,B)\in \Pi_{\textup{mod},\lam}$,
\begin{align}
\frac{Z_{A,B}(\lam)}{(1+\lam)^{ab}}\sim \exp\left\{\frac{1}{2}q_A'a^2+\frac{1}{2}q_B'b^2 +  \frac{\lam^3}{2}ab\left(aq_A+bq_B\right)^2\right\}\, .
\end{align}
\end{lemma}
\begin{proof}
Note that since $(A,B)$ is $\lam$-moderately balanced, we have  $q=\max\{q_A, q_B\}= O(n^{-(1+\eps)})$.
Let $\cD=\cD_{A,B,\lam}, \Delta=\Delta_{A,B,\lam}$, $K=K_{A,B,\lam}$ as in Definition~\ref{defsparse} and recall that 
\begin{align*}
Z_{A,B}(\lam) = \sum_{(S,T)\in \cD} \lam^{|S|+|T|} Z_{S \boxempty T }(\lam)\, .
\end{align*}
In order to estimate $Z_{A,B}(\lam)$, we begin by estimating the hard-core partition function $Z_{S \boxempty T }(\lam)$ via the  cluster expansion. First, since $(S,T)\in \cD$, the graphs $S, T$ each have maximum degree at most
$\Delta= 50 \max\{qn, \log n\}= 50 \log n$, and
so the graph $S \boxempty T$ has maximum degree at most $2\Delta$. Moreover, as before (see the argument after~\eqref{eqZNSTEst}), $S \boxempty T$ has at most $2nK$ non-isolated vertices. Since $\lam\leq \frac{1}{4e\Delta}$, we conclude from Corollary~\ref{corclustersimple} that
\begin{align}
\label{eqZboxTestimate}
\log\left(\frac{ Z_{S \boxempty T}(\lam)}{(1+\lam)^{ab}} \right) = - |S \boxempty T| \cdot \lambda^2 +  \left(2|S \boxempty T| + P_2(S \boxempty T)\right)\lambda^3 +O(nK\Delta^3\lam^4)\, .
\end{align}
Recalling that $K=50\max\{qn^2, \log n\}$, we have $nK\Delta^3\lam^4=o(1)$ and so
it follows that
\begin{align}\label{eqZABFirstCluster}
\frac{Z_{A,B}(\lam)}{(1+\lam)^{ab}}\sim  \sum_{(S,T)\in\cD} \lambda^{|S| + |T|}e^{- |S \boxempty T| \cdot \lambda^2 +  \left(2|S \boxempty T| + P_2(S \boxempty T)\right)\lambda^3}  \, .
\end{align}
Let $\bm q'=(q_A' ,q_B')$ and recall that $\nu_{\bm q',\cD}$ denotes the measure $\nu_{\bm q'}$ conditioned on the event that $(S,T)\in\cD$. Since $|S\boxempty T| = b|S| + a|T|$, we may rewrite~\eqref{eqZABFirstCluster} as
\begin{align}\label{eqZABexpP2}
\frac{Z_{A,B}(\lam)}{(1+\lam)^{ab}}\sim \E_{\nu_{\bm q',\cD}}\left(e^{P_2(S\boxempty T)\lam^3} \right) \cdot  Z'\, ,
\end{align}
where 
\begin{align}\label{eqZprimeapprox}
Z'=\sum_{(S,T)\in\cD}  \left(\frac{q_A'}{1-q_A'}\right)^{|S|}\left(\frac{q_B'}{1-q_B'}\right)^{|T|}\sim \exp\left\{ \frac{1}{2}q_A'a^2 + \frac{1}{2}q_B'b^2\right\}\, ,
\end{align} 
where  we used Corollary~\ref{corZABJanson} and the fact that $nq=o(1)$  for the asymptotics .

We now turn to estimating the expectation in~\eqref{eqZABexpP2}. 
We apply Lemma~\ref{lemtiltedcumulant} to deduce that
\begin{align}\label{eqTaylorremaindersimple}
 \log \E_{\nu_{\bm q',\cD}}\left(e^{P_2(S\boxempty T)\lam^3} \right) = \lam^3\E_{ \nu_{\bm q',\cD}^f}(P_2(S\boxempty T))\, ,
\end{align}
where 
\[
f(S,T)=\theta\lam^3 P_2(S\boxempty T) \, , \footnote{As usual we identify the pair $(S,T)$ with the graph $S\cup T$ and similarly we use $f(S,T)$ to denote $f(S\cup T)$.}
\]
for some $\theta\in [0,1]$. It is easy to verify that $f$ is $n\lam^3/(6\alpha)$-local (in fact, we have already verified this at~\eqref{eqP2local}). We may therefore apply Lemma~\ref{lemfixedconfiglocal} to calculate the expectation on the RHS of~\eqref{eqTaylorremaindersimple}. Indeed if
$(S,T)\sim \nu_{\bm q',\cD}^f$, then by Lemma~\ref{lemfixedconfiglocal}, if $e_1\in\binom{A}{2}$ and  $e_2\in\binom{B}{2}$,
\[
\P(e_1\in S, e_2\in T)= q_A'q_B'(1+O(n\Delta\lam^3) ) = q_Aq_B(1+O(n\Delta\lam^3) )\, .
\]
Similarly, if $F\subseteq \binom{A}{2}$ is a copy of $P_2$, 
\[
\P(F\subseteq S)= q_A^2(1+O(n\Delta\lam^3) )\, ,
\]
and similarly for $F\subseteq \binom{B}{2}$ a copy of $P_2$ .
Recalling that $P_2(S\boxempty T)=bP_2(S)+aP_2(T)+4|S||T|$ we conclude that
\begin{align}\label{eqP2TiltExp}\\
\lam^3\E_{ \nu_{\bm q',\cD}^f}(P_2(S\boxempty T)) &= (1+O(n\Delta\lam^3) )\lam^3\left[3b\binom{a}{3}q_A^2+ 3a\binom{b}{3}q_B^2+ 4\binom{a}{2}\binom{b}{2}q_Aq_B\right]\, \\
&=
\lam^3\left[3b\binom{a}{3}q_A^2+ 3a\binom{b}{3}q_B^2+ 4\binom{a}{2}\binom{b}{2}q_Aq_B\right] +o(1)\, \\
&=
\lam^3\left[b\frac{a^3}{2}q_A^2+ a\frac{b^3}{2}q_B^2+ \frac{a^2b^2}{2}q_Aq_B\right] +o(1)\, \\
&=
\lam^3\frac{ab}{2}\left(aq_A+bq_B\right)^2+o(1)\, .
\end{align}
The final expression can be arrived at heuristically by noting that the expected degree of a vertex in $S\boxempty T$ is approximately $aq_A+bq_B$. The result follows by combining the above with equation with ~\eqref{eqTaylorremaindersimple},~\eqref{eqZprimeapprox} and~\eqref{eqZABexpP2}.
\end{proof}

\subsection{The dependence of $Z_{A,B}$ on the imbalance of $(A,B)$.}\label{subsec:imbalance}

We now turn to the proof of Lemma~\ref{lemZABapprox1}, the first step of which is to analyze to what extent the expression in Lemma~\ref{lemmaZABabapprox} depends on the imbalance of the partition $(A,B)$.

\begin{proof}[Proof of Lemma~\ref{lemZABapprox1}.]
We note that
\[
q_A= \lam e^{-\lam^2b}(1+O(q)) \text{ and } q_B= \lam e^{-\lam^2a}(1+O(q)) \, ,
\]
and
\[
q_A'= \lam e^{-\lam^2b+2\lam^3b}(1+O(q)) \text{ and } q_B'= \lam e^{-\lam^2a+2\lam^3a}(1+O(q)) \, .
\]
Let $a=n/2-k$ and $b=n/2+k$. Note that $\lam^2 k=o(1)$ since $(A,B)$ is $\lam$-moderately balanced. Then
\begin{align*}
q_A'a^2+q_B'b^2
&= \lam e^{-\lam^2b+2\lam^3b}a^2 + \lam e^{-\lam^2a+2\lam^3a}b^2+o(1)\\
&= \lam e^{-\lam^2n/2+\lam^3n}(n/2)^2
\left[e^{-\lam^2k+2\lam^3k}\frac{(n/2-k)^2}{(n/2)^2} 
+ e^{\lam^2k-2\lam^3k}\frac{(n/2+k)^2}{(n/2)^2}   \right]+o(1)\, \\
&=\lam e^{-\lam^2n/2+\lam^3n}(n/2)^2\left[2+O(\lam^4k^2)\right]+o(1)\, \\
&= n^2 \lam e^{-\lam^2n/2+\lam^3n}/2 + o(1)\, ,
\end{align*}
where for the final equality we used that $k=\tilde O(n^{1/2})$ since $(A,B)$ is $\lam$-moderately balanced. 
We also have
\begin{align*}
\lam^3 & ab  (aq_A+bq_B)^2 \\
&= \lam^3(n/2)^2\left(1-\frac{k^2}{(n/2)^2}\right)
\cdot \left[(n/2)\lam e^{-\lam^2n/2}\left(\frac{n/2-k}{n/2}e^{-k\lam^2} + \frac{n/2+k}{n/2}e^{k\lam^2}  \right)\right]^2+o(\lam^3n^4q^3)\\
&= (n/2)^4\lam^5 e^{-\lam^2n}(4+O(k^2\lam^4))+o(1)\,  \\
&= n^4 \lam^5 e^{-\lam^2n}/4 +o(1)\, .
\end{align*}
We therefore have by Lemma~\ref{lemmaZABabapprox},
\begin{align}\label{eqZABimbalance}
\frac{Z_{A,B}(\lam)}{(1+\lam)^{ab}}&\sim 
\exp\left\{
\lam e^{-\lam^2n/2+\lam^3n}\frac{n^2}{4} + \lam^5 e^{-\lam^2n}\frac{n^4}{8}
\right\}\, .
\end{align}
This concludes the proof of~\eqref{eqZabCase1}. 
We now prove~\eqref{eqnuABapprox}.

First we show that $D_\text{KL}(\nu_{\bm q', \cD} \parallel \nu_{A,B,\lam})=o(1)$. 
For $(S,T)\in\cD_{A,B,\lam}$ we have, by~\eqref{eqZboxTestimate} and~\eqref{eqqPrimeDef}, the definition of $q_A', q_B'$,
\[
\frac{\nu_{\bm q',\cD}(S,T)}{\nu_{A,B,\lam}(S,T)}
\sim \frac{Z_{A,B}(\lam)}{(1+\lam)^{ab}Z' }e^{-\lam^3P_2(S\boxempty T)}\, .
\]
We then have
\begin{align*}
D_\text{KL}(\nu_{\bm q',\cD} \parallel \nu_{A,B,\lam})
&=\log\left( \frac{Z_{A,B}(\lam)}{(1+\lam)^{ab}Z' } \right) - \lam^3 \E_{\nu_{\bm q',\cD}}(P_2(S\boxempty T))+o(1)\, \\
&=  \frac{\lam^3}{2}ab\left(aq_A+bq_B\right)^2 -  \lam^3 \E_{\nu_{\bm q',\cD}}(P_2(S\boxempty T)) +o(1)\, \\
&=o(1)\, ,
\end{align*}
where for the penultimate equality we used~Lemma~\ref{lemmaZABabapprox} and \eqref{eqZprimeapprox}. The final equality follows by~\eqref{eqP2TiltExp} (applied with $f=0$).

We conclude from Pinsker's inequality (Lemma~\ref{lemPinsker}) that $\|\nu_{\bm q', \cD} -\nu_{A,B,\lam} \|_{TV}=o(1)$.

Letting $\bG\sim\nu_{\bm q'}$, we have
\[
\|\nu_{\bm q',\cD} - \nu_{\bm q'}\|_{TV}= \P(\bG \notin \cD)=o(1)\, ,
\]
where for the final equality we used Corollary~\ref{corZABJanson}.
Finally we note that
\[
 \|  \nu_{\bm q'}- \nu_{\bm q}  \|_{TV} \leq \|G(A, q'_A)- G(A,q)\|_{TV} + \|G(B, q'_B)- G(B,q)\|_{TV}\, .
\]
To show $\|G(A, q'_A)- G(A,q)\|_{TV}=o(1)$ we note that conditioned on the number of edges, the distributions of $G(A, q'_A), G(A,q)$ are identical and so it suffices to show that the total variation distance between the number of edges in $G(A, q'_A), G(A,q)$ is $o(1)$. This follows by observing that the distributions of the number of edges are binomial and the difference in their means is $(q-q_A')\binom{a}{2}$ which is $o(1)$ times the standard deviation $(q(1-q)\binom{a}{2})^{1/2}$. Similarly $\|G(B, q'_B)- G(B,q)\|_{TV} =o(1)$.
Statement~\eqref{eqnuABapprox} now follows by the triangle inequality. 
\end{proof}

We can now  prove Lemma~\ref{lemZapproxER}. Recall that we call a partition $(A,B)$ strongly balanced if $
\big||A|-|B| \big|\leq 10(n\log n)^{1/4}$.  Recall the definition of $\mu_{\textup{strong},\lam}$ from Algorithm~\ref{algMuModStrong}.  Recall that  a graph $G\in \cT$ is \emph{captured} by $(A,B)$ if $(G_A, G_B)\in \cD_{A,B,\lam}$ and we let $c_{\textup{mod},\lam}(G)$ denote the number of $\lam$-moderately balanced partitions that capture $G$.  Let $c_{\textup{strong},\lam}(G)$ denote the number of strongly balanced partitions that capture $G$ and note that
\begin{align}\label{eqstrongcapture}
\mu_{\textup{strong},\lam}(G)= \frac{\lam^{|G|}}{Z_{\textup{strong}}(\lam)}\cdot c_{\textup{strong},\lam}(G)\, .
\end{align} 
We record the following lemma whose proof is identical to that of Lemma~\ref{lemCaptureMod}. 
\begin{lemma}\label{lemCaptureStrong}
Let $\bG\sim \mu_{\textup{strong},\lam}$. We have,
\[
\P(c_{\textup{strong},\lam}(\bG) = c_{\textup{mod},\lam}(\bG)=  1)\geq 1- e^{-\lam n/25}\, .
\]
\end{lemma}

We now prove Proposition~\ref{propGroundStateStrong} in the subcritical defect regime. 
\begin{lemma}\label{lemZstrongsimsum}
Fix $\eps>0$ and let $\lam\geq (1+\eps) \sqrt{\frac{\log n }{n}}$. Then
\begin{align}\label{eqZstrongsimsum}
Z_{\textup{mod}}(\lam)\sim Z_{\textup{strong}}(\lam)\, 
\end{align}
and
\begin{align}\label{eqMustrongMulam}
 \| \mu_{\textup{mod},\lam} - \mu_{\textup{strong},\lam}  \|_{TV}  = O(n^{-3/2}) \,.
 \end{align}
\end{lemma}
\begin{proof}
Let  $M=5(n\log n)^{1/4}$. By~\eqref{eqZABimbalance} and the fact that $\Pi_{\textup{strong}} \subset \Pi_{\textup{mod},\lam}$, we have
\[
\left|\frac{ Z_{\textup{mod}}(\lam)}{ Z_{\textup{strong}}(\lam)}-1\right|
 \leq 
(1+o(1))  \sum_{k\geq M} \frac{\binom{n}{\lfloor n/2 \rfloor+k}}{\binom{n}{\lfloor n/2 \rfloor}} (1+\lam)^{-k^2}.
\]
Noting that $\binom{n}{\lfloor n/2 \rfloor+k}\leq \binom{n}{\lfloor n/2 \rfloor}$, the RHS is bounded above by
\begin{align}\label{eqGaussiantail}
 \sum_{k\geq M} e^{-\lam k^2/2} \leq \int_{M-1}^\infty e^{-\lam x^2/2}\, dx \leq \frac{1}{\lam(M-1)} e^{-(M-1)^2\lam/2}=O(n^{-3})\, ,
\end{align}
 where for the second inequality we used the standard integral estimate $\int_{t}^\infty e^{-ax^2}\, dx\leq e^{-at^2}/(2at)$ for $a,t>0$. Statement~\eqref{eqZstrongsimsum} follows.
 The proof of~\eqref{eqMustrongMulam} is identical to the proof of~\eqref{eqmusimmumod} (carried out at ~\eqref{eq:Proof3.3}) except that we use Lemma~\ref{lemCaptureStrong} in place of Lemma~\ref{lemCaptureMod}.
\end{proof}

We note that by Proposition~\ref{lemZerothApprox}, and Propositions~\ref{propGroundStateRefinedAlt},~\ref{lemZmodsimsum}, we obtain Corollary~\ref{corZmodsimsum} in the subcritical defect regime. 
\begin{cor}\label{corZmodsimsumRest}
Fix $\eps>0$ and let $\lam\geq (1+\eps) \sqrt{\frac{\log n }{n}}$,
\begin{align*}
Z(\lam)\sim Z_{\textup{strong}}(\lam)\, ,
\end{align*}
and
\begin{align*}
\|\mu_{\lam} - \mu_{\textup{strong},\lam}\| _{TV}=o(1)\, .
\end{align*}
\end{cor}

We  now prove Lemma~\ref{lemZapproxER}.
\begin{proof}[Proof of Lemma~\ref{lemZapproxER}]
Returning to~\eqref{eqZABimbalance}, we see that for $(A,B)$ strongly balanced we have
\begin{align}\label{eqZABstrong}
\frac{Z_{A,B}(\lam)}{(1+\lam)^{ab}}&\sim 
\exp\left\{
\lam e^{-\lam^2n/2+\lam^3n}\frac{n^2}{4} + \lam^5 e^{-\lam^2n}\frac{n^4}{8}
\right\}=:f(\lam, n)\, .
\end{align}
Letting  $M=5(n\log n)^{1/4}$ as before, it follows from Corollary~\ref{corZmodsimsumRest} that 
\begin{align}\label{eqZGaussSum}
Z(\lam)&\sim
(1+\lam)^{n^2/4}f(\lam,n)
\sum_{-M\leq k\leq M}\frac{1}{2}\binom{n}{\lfloor n/2 \rfloor+k} (1+\lam)^{-k^2}\\
&\sim \frac{1}{2} \binom{n}{\lfloor n/2 \rfloor}(1+\lam)^{n^2/4}f(\lam,n)
\sum_{-M\leq k\leq M}(1+\lam)^{-k^2}\, .
\end{align}
We note that
\[
 \sum_{-M\leq k\leq M}(1+\lam)^{-k^2}= \int_{-M}^M (1+\lam)^{-x^2} \, dx +O(1)\, ,
\]
and estimating as in~\eqref{eqGaussiantail} we have
\begin{equation}\label{eqGaussiantail2}
\int_{-M}^M (1+\lam)^{-x^2} = \int_{-\infty}^\infty (1+\lam)^{-x^2} +o(1)= \sqrt{\frac{\pi}{\log(1+\lam)}} +o(1)= (1+o(1))\sqrt{\frac{\pi}{\lam}}\, .
\end{equation}
Returning to~\eqref{eqZGaussSum} we conclude that
\begin{align}\label{eqZsimflam}
Z(\lam)
\sim
\frac{1}{2}\sqrt{\frac{\pi}{\lam}} \binom{n}{\lfloor n/2 \rfloor}(1+\lam)^{n^2/4}f(\lam,n)\, . 
\end{align}
\end{proof}

We now prove Theorem~\ref{thmGnpSubCritLDprob}.
\begin{proof}[Proof of Theorem~\ref{thmGnpSubCritLDprob}]
We recall the identity~\eqref{eqpviaZ}:
\begin{align*}
\mathbb{P}_{n,p} (\cT )= (1-p)^{\binom{n}{2}} Z\left(\frac{p}{1-p}\right) \,.
\end{align*}
We note that $\frac{p}{1-p}\geq p\geq  (1+\eps) \sqrt{\frac{\log n } {n}} $.
The first statement of Theorem~\ref{thmGnpSubCritLDprob} now follows from Lemma~\ref{lemZapproxER} (with $\lam=p/(1-p)$).

It remains to show that  $\| \mu_{\lam} - \mu_{\lam,1} \| _{TV} = o(1) \,.$
By Corollary~\ref{corZmodsimsumRest} it suffices to show that $\|  \mu_{\textup{strong},\lam}- \mu_{\lam,1}  \|_{TV}  = o(1)$.
Let $\bm \pi_0, \bm \pi_1$ denote the partitions selected at Step 1 in Algorithms~\ref{algMuModStrong} and~\ref{algMulam1} respectively.
Given $\pi\in \Pi$, let $\mu_{\textup{strong},\lam}^{\pi}, \mu_{\lam,1}^{\pi}$ denote the measures $\mu_{\textup{strong},\lam}, \mu_{\lam,1}$
conditioned on the events $\bm \pi_0=\pi, \bm \pi_1=\pi$ respectively.

\begin{claim}\label{claimTVconditioned}
With $\pi\sim\bm\pi_0$,
\begin{align}
\|  \mu_{\textup{strong},\lam}- \mu_{\lam,1}  \|_{TV}\leq \E_\pi \|  \mu_{\textup{strong},\lam}^{\pi}- \mu_{\lam,1}^{\pi}  \|_{TV} + \|\bm\pi_0-\bm\pi_1\|_{TV}\, .
\end{align}

\end{claim}
\begin{proof}
For $G\in \cT$ we have 
\begin{align}\label{eqMu1MuStrongTv}
\left|\mu_{\textup{strong},\lam}(G)- \mu_{\lam,1}(G) \right| 
= \left| \sum_{\pi\in \Pi} \left[\mu^\pi_{\textup{strong},\lam}(G)\P(\bm\pi_0=\pi)-   \mu^\pi_{\lam,1}(G)\P(\bm\pi_1=\pi) \right]\right| \, ,
\end{align}
where we set $\mu^\pi_{\textup{strong},\lam}(G)=0$ if $\pi$ is not strongly balanced. The RHS of~\eqref{eqMu1MuStrongTv} is at most
\begin{align*}
 \sum_{\pi\in \Pi} \left|\mu^\pi_{\textup{strong},\lam}(G) - \mu^\pi_{\lam,1}(G) \right|\P(\bm\pi_0=\pi)+  
  \sum_{\pi\in \Pi} \mu^\pi_{\lam,1}(G)|\P(\bm\pi_1=\pi) -  \P(\bm\pi_0=\pi)| \, .
\end{align*}
Summing over $G\in \cT$ proves the claim.
\end{proof}

First  we show
\begin{align}\label{eqpi0pi1TV}
\|\bm\pi_0-\bm\pi_1\|_{TV} = o(1)\, .
\end{align}

If $\pi=(A,B)$  is strongly balanced with $a=\lfloor n/2\rfloor+t, b=\lceil n/2\rceil-t$,
(so in particular $t=O((n\log n)^{1/4})$) we have by Corollary~\ref{corZmodsimsumRest},~\eqref{eqZABstrong} and~\eqref{eqZsimflam} that

\begin{align}\label{eqPpi0}
\P(\bm\pi_0=\pi)= \frac{Z_{A,B}(\lam)}{Z_{\textup{strong}}(\lam)}\sim\frac{ (1+\lam)^{-t^2}}{\frac{1}{2}\sqrt{\frac{\pi}{\lam}} \binom{n}{\lfloor n/2 \rfloor}}\, .
\end{align}
Moreover, using~\eqref{eqGaussiantail} and~\eqref{eqGaussiantail2},
\begin{align}\label{eqPpi1}
\P(\bm\pi_1=\pi)= \frac{(1+\lam)^{-t^2}}{\frac{1}{2}\sqrt{\frac{\pi}{\lam}}\binom{n}{\lfloor n/2 \rfloor+t}} \sim \frac{ (1+\lam)^{-t^2}}{\frac{1}{2}\sqrt{\frac{\pi}{\lam}} \binom{n}{\lfloor n/2 \rfloor}}\, .
\end{align}
Again estimating as in~\eqref{eqGaussiantail}, the probability that $\bm\pi_1$ is not strongly balanced is $o(1)$. Statement~\eqref{eqpi0pi1TV} now follows.
Finally we fix $\pi=(A,B)$ strongly balanced and show that
\[
 \|  \mu_{\textup{strong},\lam}^{\pi}- \mu_{\lam,1}^{\pi}  \|_{TV}=o(1)\, ,
\]
which will complete the proof by Claim~\ref{claimTVconditioned}.
Given $(S,T)\in \cD_{A,B,\lam}$, let $\mu_{\textup{strong},\lam}^{\pi,S,T}, \mu_{\lam,1}^{\pi,S,T}  $ denote the measures  $\mu_{\textup{strong},\lam}, \mu_{\lam,1}$ conditioned on the event that $\pi$ is chosen at Step 1 and $(S,T)$ is chosen at Step 2 in Algorithms~\ref{algMuModStrong} and \ref{algMulam1} respectively.
First note that $\mu_{\textup{strong},\lam}^{\pi,S,T}, \mu_{\lam,1}^{\pi,S,T}  $ are identically distributed (they are both the union of $S\cup T$ with a crossing graph whose distribution is the hard-core model on $S\boxempty T$ at $\lam$).
Let $\nu'_{A,B,\lam}$ denote the measure associated to the random graph in Step 2 of Algorithm~\ref{algMulam1}, i.e., the union of two independent samples from $G(A,q_0), G(B,q_0)$ where we output the empty graph if the graph contains a triangle. 

 It follows, by an argument identical to the proof of Claim~\ref{claimTVconditioned}, that 
\[
 \|  \mu_{\textup{strong},\lam}^{\pi}- \mu_{\lam,1}^{\pi}  \|_{TV}\leq \|  \nu_{A,B,\lam}- \nu'_{A,B,\lam} \|_{TV}\, .
\]
By Lemma~\ref{lemZABapprox1},  $\| \nu_{A,B,\lam} - \nu_{\bm q} \|_{TV}  = o(1) $ and
\begin{align}\label{eqtriangleTV}
\|\nu_{\bm q} -  \nu'_{A,B,\lam} \|_{TV} = \frac{1}{2} \nu_{\bm q}(\text{$G$ contains a triangle})=o(1)
\end{align}
by a union bound. We conclude that $\|\nu_{A,B,\lam}- \nu'_{A,B,\lam}  \|_{TV}=o(1) $ completing the proof. 
\end{proof}

We can now deduce Theorem~\ref{thmGnpBipartite} from Theorem~\ref{thmGnpSubCritLDprob}.
\begin{proof}[Proof of Theorem~\ref{thmGnpBipartite}]
Set $\lam=p/(1-p)$, let $\bG\sim \mu_{\lam,1}$ and let $X_1$ be the minimum number of edges whose removal makes $\bG$ bipartite. 
Recall that $q_0/(1-q_0)= \lam e^{-\lam^2n/2} $ and note that $q_0=O(n^{-1-\eps})$.
By Theorem~\ref{thmGnpSubCritLDprob}, it suffices to show that $\|X_1-\hat X\|_{TV}=o(1)$ where $\hat X \sim \text{Bin}(\lfloor n^2/4\rfloor, q_0)$. 
Let $(A,B)$ denote the partition chosen at Step 1 in Algorithm~\ref{algMulam1} and let  $( S,  T)$ denote the set of edges chosen at Step 2.
 By Lemmas~\ref{lemExpanderwhp} and~\ref{lemExpansionCor}, $(A,B)$ is the unique max cut of $\bG$ whp. 
 In particular, $X_1= |S|+ | T|$ whp. It follows that
\[
\|X_1-(| S|+ | T|)\|_{TV} = o(1)\, .
\]
Let $N=\binom{a}{2}+\binom{b}{2}$ and 
recall that $S,T$ are two independent $G(A,q), G(B, q)$ random graphs where we output the empty graph if $S$ or $T$ contains a triangle. 
Letting $X_2 \sim \text{Bin}(N, q)$, it follows that 
\[
\||S|+|T|-X_2\|_{TV}=\frac{1}{2}\P(S\cup T \text{ contains a triangle})=o(1)
\]
by a union bound. It therefore suffices to show that $\|X_2-\hat X\|_{TV}=o(1)$.

 Since $(A,B)$ is strongly balanced we have $N=n^2/4-k$ where $k=\tilde O(n^{1/2})$. We couple $X_2$ and $\hat X$ via the natural coupling of $ \text{Bin}(N, q)$ and $ \text{Bin}(\lfloor n^2/4\rfloor, q)$ and write $\hat X = | S|+ | T| +Z$ where $Z\sim  \text{Bin}(k, q)$. By the coupling inequality~Lemma~\ref{lemCoupling} we then have
 \[
 \|\hat X-(| S|+ | T|)\|_{TV}\leq \P(Z> 0)\leq q_0k =o(1)\, .
 \]

For the second part of the theorem, fix $t \in \R$ and let $p =   \sqrt{3+ \frac{\log \log n}{\log n} - \frac{t}{\log n}  } \sqrt{ \frac{\log n}{n}}$.  By the above, it is enough to show that 
\[ \lim_{n \to \infty} \P \left[  \hat X =0 \right] =   \exp \left(  - \frac{\sqrt{3}}{4} e^{t/2} \right ) \,. \] 
In this regime, $n^2 q_0 = \Theta(1)$ and $q_0 \sim pe^{-p^2n/2}$,  and so $\hat X$ converges in distribution to a Poisson with mean $\tau = \lim_{n \to \infty} \frac{n^2 pe^{-p^2n/2}}{4}$.  From here  a calculation  shows that $\lim_{n \to \infty} \frac{n^2 pe^{-p^2n/2}}{4}  = \frac{\sqrt{3}}{4} e^{t/2}$.
\end{proof}

Next we prove Theorem~\ref{thmGnpColoring} on the chromatic number. 
\begin{proof}[Proof of Theorem~\ref{thmGnpColoring}]
Given a partition $\pi\in \Pi$, let $\mu^{\pi}_{\lam,1}$ denote the measure $\mu_{\lam,1}$ conditioned on the event that $\pi$ is chosen at Step 1 in Algorithm~\ref{algMulam1}.
By Theorem~\ref{thmGnpSubCritLDprob} it suffices to fix a strongly balanced partition $\pi=(A,B)$ and prove the result for $G\sim\mu^{\pi}_{\lam,1}$.  Let $S,T, \Ec$ denote the edges in  $A$, in $B$, and across the partition respectively.

 The fact that $\chi(G) \ge 3$ whp if  $(1+\eps)  \sqrt{\frac{\log n } {n}} \le\lam \le (\sqrt{3} -\eps)  \sqrt{\frac{\log n } {n}}  $ follows from Theorem~\ref{thmGnpBipartite}.
 
 Next we show that if $\lam \ge (1+\eps)  \sqrt{\frac{\log n } {n}}  $ then $\chi(G) \le 4$ whp.   Since $q_0 =o(n^{-1})$ in this regime and $S,T$ are distributed as \ER random graphs $G(A,q_0), G(B,q_0)$, whp the graphs $(A,S)$ and $(B,T)$ are forests and thus $2$-colorable.  We can then color $G$ with $4$ colors by assigning disjoint sets of $2$ colors to the vertices of $A$ and $B$ respectively.

 Next we show that when $\lam \ge (\sqrt{2}+\eps)  \sqrt{\frac{\log n } {n}}$, $\chi(G) \le 3$ whp.
  To  use only $3$ colors, we want to assign, say, colors red and green to $A$, and blue and green to $B$.  In this regime $q_0=o(n^{-3/2})$ and so whp all edges of $S$ and $T$ are isolated edges.  To color the graph, we  assign red to all isolated vertices of $(A,S)$, blue to all isolated vertices of $(B,T)$, and color the endpoints of edges in $S$ red and green and the endpoints of the edges in $T$ blue and green.  We call such a coloring a `green edge coloring'.   Note that for each edge  there is a choice of two colorings based on which endpoint receives green.  Fix some canonical ordering on all $n$ vertices to determine an anchor for each edge (the earlier vertex), and then call a red-green or blue-green coloring of the endpoints of a given edge `positive' if the anchor is green and `negative' otherwise.  
  
  If there were no crossing edges, any assignment of positive or negative colorings to the edges in $S,T$ would result in a proper green edge coloring, but there may be crossing edges connecting edges in $S$ to edges in $T$, and  edges between green vertices are not allowed.  We  create a graph $G_{\mathrm{col}}$ in which the edges of $S,T$ are nodes and two nodes are connected by an edge in $G_{\mathrm{col}}$ for each crossing edge joining the corresponding edges. We claim  that if $G_{\mathrm{col}}$ has no cycles or multiple edges, then there is a proper green edge coloring of $G$.  To see this, choose a green edge coloring as follows: for each component of $G_{\mathrm{col}}$, pick an arbitrary node (edge of $S,T$) and color it with (say) its positive coloring; since there are no multiple edges, any node it is connected to can still be colored either positively or negatively (or perhaps both).  We can continue coloring the edges by exploring the components of  $G_{\mathrm{col}}$ in this way and will not reach a contradiction since there are no cycles.  Finally, to see that whp $G_{\mathrm{col}}$ contains no multiple edges or cycles, note that whp $|S| + |T| = O(n^{1/2-\del})$ for some fixed $\del = \del(\eps)>0$. Therefore the expected number of multiple edges is $O( \lam^2 n^{1/2-\del}) = o(1)$. Further, given $S,T$ the graph $G_{\mathrm{col}}$ is stochastically dominated by an \ER random graph on the node set with edge probability $4 \lam$ ($4$ for the possible crossing edges connected an edge in $S$ with an edge in $T$).  Since the number of nodes times the edge probability is $o(1)$, whp there are no cycles.

  Finally we show that when $ (1+\eps)  \sqrt{\frac{\log n } {n}} \le \lam \le (\sqrt{2}-\eps)  \sqrt{\frac{\log n } {n}}$, $\chi(G) \ge 4$ whp.  Lemma~\ref{lemExpanderwhp} shows that whp over the choice of crossing edges $\Ec$, $G$ is an $(A,B)$-$\lam$-expander. In particular, for all pairs of sets of vertices $X \subset A, Y \subset B$ so that $|X| , |Y| = 10 \lam n$, we have $\Ec \cap (X \times Y) \neq \emptyset$.  In this regime of $\lam$ we have $q_0=O(n^{-1-\eps})$ and $q_0=\Omega(n^{-3/2+\eps})$ and so whp both $S$ and $T$ are forests of size $\Omega(n^{1/2+\eps})$ with maximum degree $O_\eps(1)$. Due to the maximum degree bound, any independent set in $S$ has size at most $(1-\Omega_\eps(1))|V(S)|$. Indeed given a connected component $C\subseteq S$ and an independent set $I$ in $C$, the number of edges between $I$ and $V(C)\backslash I$ is at least $|I|$, and at most $O_\eps(1)|V(C)\backslash I|$, so that $|I|\leq (1-\Omega_\eps(1))|V(C)|$. It follows that in any proper $3$-coloring of $S$, there must be two color classes of size  $\Omega_\eps(1)|V(S)|=\Omega_\eps(1)|S|$. Similarly for $T$. 
   In particular in a proper $3$-coloring of $G$, there must be a common color appearing on at least $\Omega_\eps(n^{1/2+\eps})\geq 10\lam n$ vertices on each side. But by the expansion property whp there is an edge between these sets of vertices and so the coloring cannot be proper. 
\end{proof}

\begin{remark}
The arguments in the proof of Theorem~\ref{thmGnpColoring} are  rough, and a much more precise understanding of the transition between $3$- and $4$-colorability can likely be obtained.  In particular, we conjecture that the threshold for the existence of a `green tree coloring' (in which tree components of $S,T$ are properly colored red-green and blue-green respectively) marks the threshold for $3$-colorability and that the scaling window for the existence of a green tree coloring can be completely determined by analyzing a random (bipartite) 2-SAT formula obtained by the constraints imposed on the tree colorings by crossing edges.  The analysis of the scaling window then could be done by adapting the methods from~\cite{bollobas2001scaling}.
\end{remark}

\section{Critical and supercritical defect regimes} 
\label{secSuperCrit}
In this section we prove our main results in the critical and supercritical defect regimes: Theorems~\ref{thmGnpStructureSuper},~\ref{thm4colorGnp},~\ref{thmGnpGiant} and~\ref{thmGenDef}.

We begin with Theorem~\ref{thmGnpStructureSuper}. As in Section~\ref{secSubCritRegime} we reformulate an asymptotic formula for $\mathbb{P}_{n,p}( \cT)$ in terms of an asymptotic formula for the partition function $Z(\lam)$. Recall the definitions of $q_0,q_1,q_2$ from~\eqref{eqqdef},~\eqref{eqq0q1Def} and \eqref{eqq2Def} now considered as functions of $\lambda$. Let
\begin{align}\label{flamnsuper}
f(\lam,n):=
(1-q_2)^{-n^2/4+n/2} &  \exp\left\{  \frac{1}{64} \lam^6 n^5 q_0^2 - \frac{1}{64}\lam^6 n^6 q_0^3 -\frac{1}{24}n^3q_0^3 \right\}
\times\\
& \exp\left\{\frac{1}{64}\lam^4n^4q_0^2-\frac{1}{6}\lam^4n^5q_0^3-\frac{1}{2}\lam^4n^4q_0^2\right\}\, .
\end{align}

\begin{lemma}
\label{lemSuperMainLem}
If $\lam \ge \frac{13}{14} \sqrt{\frac{\log n}{n}}$, then
\begin{align*}
Z(\lam) \sim   \frac{1}{2}\sqrt{\frac{\pi}{\lam}} \binom{n}{\lfloor n/2 \rfloor}(1+\lam)^{n^2/4}f(\lam,n) \,.
\end{align*}
\end{lemma}
For the remainder of this section, we assume that $\lam \ge \frac{13}{14} \sqrt{\frac{\log n}{n}}$.

Our strategy for proving Lemma~\ref{lemSuperMainLem} will follow similar lines to the arguments of Section~\ref{secSubCritRegime}, only now the calculations are significantly more involved. The source of the additional complication stems from the fact that we now need to take into account further terms in the cluster expansion (see Corollary~\ref{corclustersimple2} below) at the step of equation~\eqref{eqZboxTestimate} in order to obtain an asymptotic formula for the partition function $Z_{A,B}(\lam)$ and to obtain an accurate enough approximation of the measures $\mu_{A,B,\lam}$ and $\nu_{A,B,\lam}$ (defined at \eqref{eqmuABdef}, \eqref{eqnuABdef}).

To begin with, we fix a $\lam$-moderately balanced partition (see Definition~\ref{defVarBalanced}) $(A,B)$ with $a = |A|$ and $b=|B|$. We will show that the defect distribution $\nu_{A,B,\lam}$ is within $o(1)$ total variation distance of a suitable conditioned exponential random graph measure which will be the main step to proving the approximation to $\mu_\lam$ in Theorem~\ref{thmGnpStructureSuper}. Recall the definition of the exponential random graph $G(V,q,\psi)$ from~\eqref{eqGVqpsiDef}.

\begin{lemma}\label{lemZABapprox1super}
If $(A,B)\in \Pi_{\textup{mod},\lam}$, then
\begin{align}\label{eqZabCase1super}
\frac{Z_{A,B}(\lam)}{(1+\lam)^{ab}}\sim f(\lam,n) \cdot \exp\left\{n^2q\lam^4(a-b)^2\right\}\, ,
\end{align}
Moreover, if $(A,B)\in \Pi_{\textup{strong}}$, then
\begin{align}\label{eqnuABapproxsuper}
 \| \nu_{A,B,\lam} -G(A, q_2, \psi) \times G(B, q_2, \psi) \|_{TV}  = o(1)  \,.
 \end{align}
 \end{lemma}

\subsection{A first approximation to $Z_{A,B}$}

A key step toward proving Lemma~\ref{lemZABapprox1super} is an intermediate approximation of $Z_{A,B}(\lam)$ analogous to Lemma~\ref{lemmaZABabapprox}. To state the result we need a few more definitions. Recall the definitions of $q_A'$ and $q_B'$ from~\eqref{eqqprimesuperdef}.
Then define 
\begin{align}\label{eqmurhodef}\\
\mu_A= \binom{a}{2}q'_Ae^{2\lam^3b(aq_A+bq_B)} \text{ and }
\mu_B= \binom{b}{2}q'_Be^{2\lam^3a(aq_A+bq_B)}\, .
\end{align}
The quantities $\mu_A, \mu_B$ will serve as approximations to the expected number of edges appearing inside $A,B$ (respectively) in a sample from $\nu_{A,B,\lam}$.  We then let
\begin{align}\label{eqqdoubleprimesuperdef}
\frac{q_A''}{1-q_A''}= \frac{q_A'}{1-q_A'}e^{4\mu_B\lam^3} \text{ and }\frac{q_B''}{1-q_B''}= \frac{q_B'}{1-q_B'}e^{4\mu_A\lam^3}\, .
\end{align}
We highlight that $q_A\sim q_A'\sim q_A''\sim q_B\sim q_B'\sim q_B''$, and that since $\lam \ge \frac{13}{14} \sqrt{\frac{\log n}{n}}$ and $(A,B)\in \Pi_{\text{mod},\lam}$ we have
\[
q=\max\{q_A, q_B\}=o(n^{-13/14})\, .
\]

\begin{lemma}\label{lemmaZABabapproxsuper}
If $(A,B)\in \Pi_{\textup{mod},\lam}$,
{\small
\begin{align}
\frac{Z_{A,B}(\lam)}{(1+\lam)^{ab}}\sim
&(1-q_A'')^{-\binom{a}{2}}  \exp\left\{ \frac{1}{2}\lam^3a^3bq_A''^2 + \frac{1}{4} \lam^6 a^3b^2 q_A^2 + \frac{3}{2}\lam^6 a^4b^2 q_A^3 -\frac{1}{6}a^3q_A^3 \right\}
\times\\
&
(1-q_B'')^{-\binom{b}{2}}    \exp\left\{ \frac{1}{2}\lam^3b^3aq_B''^2 + \frac{1}{4} \lam^6 b^3a^2 q_B^2 + \frac{3}{2}\lam^6 b^4a^2 q_B^3 -\frac{1}{6}b^3q_A^3 \right\}
\times\\
&\phantom{(1-q_B'')^{-\binom{b}{2}} } \exp\left\{-4\lam^3\mu_A\mu_B
+
\lam^4 ab\left(\frac{1}{4}abq_Aq_B-  \frac{2}{3}(aq_A+bq_B)^3 - 2(aq_A+bq_B)^2 \right)\right\}\, .
\end{align}}
\end{lemma}

The derivation of Lemma~\ref{lemZABapprox1super} from Lemma~\ref{lemmaZABabapproxsuper} is very similar to the derivation of Lemma~\ref{lemZABapprox1} from Lemma~\ref{lemmaZABabapprox}. In particular, we show that  we can replace all instances of $a,b$ (both implicit and explicit) with $n/2$ on the RHS of the above asymptotic formula whilst incurring only a $1+o(1)$ multiplicative error. Since the calculations are similar to those of the previous section (only now more tedious), we defer the proof of Lemma~\ref{lemZABapprox1super} to Appendix~\ref{secAppCalc}. 

As usual we let $\cD=\cD_{A,B,\lam}$. As in the proof of Lemma~\ref{lemmaZABabapprox}, a key step toward proving Lemma~\ref{lemmaZABabapproxsuper} is to estimate an expectation $E=\E_{\nu_{\bm q', \cD}}\exp\left\{\lam^3P_2(S\boxempty T)\right\}$. In the previous section, we approximated $\log E$ (via~\ref{lemtiltedcumulant}) by the expectation of $\lam ^3 P_2(S\boxempty T)$ (with respect to a tilted measure). In our regime of $\lambda$ this is no longer possible since now the \emph{variance} of $P_2(S\boxempty T)$ can also make a significant contribution to $E$. This makes the estimation of $E$  more delicate, and we isolate this estimate in the following lemma. 
We let
\begin{align}\label{eq:ZprimeDefSuper}
Z'=\sum_{(S,T)\in\cD}  \left(\frac{q_A'}{1-q_A'}\right)^{|S|}\left(\frac{q_B'}{1-q_B'}\right)^{|T|}\, ,
\end{align}
the normalising constant associated to the measure $\nu_{\bm q', \cD}$.

\begin{lemma}\label{lem:P2MomentSuper}
{
\begin{align}
 Z'\cdot\E_{\nu_{\bm q', \cD}}\left(e^{\lam^3P_2(S\boxempty T)}
 \right) &\sim
(1-q_A'')^{-\binom{a}{2}}  \exp\left\{ \frac{1}{2}\lam^3a^3bq_A''^2 + \frac{1}{4} \lam^6 a^3b^2 q_A^2 + \frac{3}{2}\lam^6 a^4b^2 q_A^3 -\frac{1}{6}a^3q_A^3 \right\}\\
&
\times(1-q_B'')^{-\binom{b}{2}}  \exp\left\{ \frac{1}{2}\lam^3b^3aq_B''^2 + \frac{1}{4} \lam^6 b^3a^2 q_B^2 + \frac{3}{2}\lam^6 b^4a^2 q_B^3 -\frac{1}{6}b^3q_A^3 \right\}\\
& 
\times\exp\{-4\lam^3\mu_A\mu_B\}\, .
  \end{align}}
\end{lemma}
\begin{proof}
Define the `centered' random variables
\begin{align}\label{eq:CenterST}
     s= |S| - \mu_A,\quad t= |T| - \mu_B\, .
\end{align}
We then have
$P_2(S\boxempty T) = bP_2(S) + aP_2(T) +4(st+|S|\mu_B+|T|\mu_A -\mu_A\mu_B)$.
Letting
\[
h(S,T) =\lam^3( bP_2(S) + aP_2(T) +4|S|\mu_B+4|T|\mu_A )\, ,  
\]
we have
\begin{align}\label{eqZP30}
 \E_{\nu_{\bm q', \cD}}\left(e^{\lam^3P_2(S\boxempty T)}\right) &= 
 e^{-4\lam^3\mu_A\mu_B}
 \E_{\nu^h_{\bm q', \cD}}\left(e^{4\lam^3st}\right)\cdot  \E_{\nu_{\bm q', \cD}}\left(e^{h(S,T)}\right)\, .
\end{align}
We now estimate the two expectations in the expression on the RHS.
The advantage of centering $S$ and $T$ as in~\eqref{eq:CenterST} is that  $\E_{\nu_{\bm q', \cD}}\left(e^{h(S,T)}\right)$ factorizes as a product of expectations of independent random variables that depend on $S$ and $T$ respectively. Moreover we will show that
\begin{align}\label{eqlognuh0}
 \E_{\nu^h_{\bm q', \cD}}\left(e^{4\lam^3st}\right)=1+o(1)\, .
\end{align}
Let us first establish~\eqref{eqlognuh0}.

By Lemma~\ref{lemtiltedcumulant} there exists $\theta\in[0,1]$ such that 
\begin{align}\label{eqnuhcumulant}
\log  \E_{\nu^h_{\bm q', \cD}}\left(e^{4\lam^3st}\right) =  4\lam^3\E_{\nu^j_{\bm q', \cD}}\left(st\right)\, ,
\end{align}
where 
\begin{align}\label{eqjdef}
j(S,T) = h(S,T) + 4\theta\lam^3st\, .
\end{align}

\begin{claim}\label{claimprobrefine}
Let $F\in \cD$ such that $|F|=O(1)$. If $\bG\sim\nu^j_{\bm q', \cD}$, then
{
\begin{align}
\P(F\subseteq \bG)
=& (1+ O(n^2\Delta^2 \lam^6))\left(q'_Ae^{2\lam^3b(aq_A+bq_B)}\right)^{|F_A|}\left(q'_Be^{2\lam^3a(aq_A+bq_B)}\right)^{|F_B|}\times\\
&\exp\left\{\lam^3 (bP_2(F_A)+aP_2(F_B))\right\}\, .
\end{align}}
\end{claim}
Note that the probability estimate of Claim~\ref{claimprobrefine} is independent of $\theta$. We defer the proof of Claim~\ref{claimprobrefine} to Appendix~\ref{secAppSubCalc} since it follows similar lines to the proof of Corollary~\ref{corfixedconfigboot}. 

With $\bG\sim\nu^j_{\bm q', \cD}$, 
we note that by Claim~\ref{claimprobrefine}, if $e_1\in\binom{A}{2}, e_2\in\binom{B}{2}$, then
\[
\P(e_1, e_2\in \bG)= (1+ O(n^2\Delta^2 \lam^6))q'_Ae^{2\lam^3b(aq_A+bq_B)}\cdot q'_Be^{2\lam^3a(aq_A+bq_B)}\, .
\]
It follows that 
\begin{align*}
 \E_{\nu^j_{\bm q', \cD}}(s  t)
 &=  \E_{\nu^j_{\bm q', \cD}}(|S||T|) -  \E_{\nu^j_{\bm q', \cD}}(|S|)\mu_B-  \E_{\nu^j_{\bm q', \cD}}(|T|)\mu_A + \mu_A\mu_B\\
 &=\binom{a}{2}\binom{b}{2}q'_Ae^{2\lam^3b(aq_A+bq_B)}\cdot q'_Be^{2\lam^3a(aq_A+bq_B)}\cdot O(n^2\Delta^2 \lam^6)= o(\lam^{-3})\, ,
\end{align*}
and so, returning to~\eqref{eqnuhcumulant},  we see that ~\eqref{eqlognuh0} holds. 

Returning to~\eqref{eqZP30} we now  estimate $ \E_{\nu_{\bm q', \cD}}\left(e^{h(S,T)} \right)$. 
Recall the definition of $q_A'', q_B''$ at~\eqref{eqqdoubleprimesuperdef} and let $\bm q''_A=(q''_A,0), \bm q''_B=(0,q''_B)$. Let 
\begin{align}\label{eqZAdprimedef}
Z_A''=\sum_{S\subseteq \binom{A}{2}: S\in \cD}\left(\frac{q_A''}{1-q_A''}\right)^{|S|}\, 
\end{align}
denote the normalizing constant associated to $\nu_{\bm q''_A, \cD}$ and define $Z_B''$ similarly. 
We then have
\begin{align}\label{eqhatZP3}
 Z'\cdot\E_{\nu_{\bm q', \cD}}\left(e^{h(S,T)} \right)
= Z_A''Z_B''\cdot \E_{\nu_{\bm q''_A, \cD}} \left(e^{\lam^3b P_2(S)}\right)
 \cdot \E_{\nu_{\bm q''_B, \cD}}\left( e^{\lam^3aP_2(T)}\right)\, .
\end{align}

Combining this with~\eqref{eqlognuh0} and \eqref{eqZP30} we conclude that
\begin{align}\label{eq:P2momentforTV}\\
 Z'\cdot\E_{\nu_{\bm q', \cD}}\left(e^{\lam^3P_2(S\boxempty T)}\right)
 \sim  e^{-4\lam^3\mu_A\mu_B} \cdot 
 {Z_A''Z_B''}\cdot \E_{\nu_{\bm q''_A, \cD}} \left(e^{\lam^3b P_2(S)}\right)
 \cdot \E_{\nu_{\bm q''_B, \cD}}\left( e^{\lam^3aP_2(T)}\right)\, .
\end{align}

Next we estimate $ Z_A''\cdot\E_{\nu_{\bm q''_A, \cD}}\left( e^{ \lam^3b P_2(S)}\right)$.
\begin{claim}\label{claimERGApartfun}
With $\psi=\lam^3b$,
\begin{align}
 Z_A''\cdot \E_{\nu_{\bm q''_A, \cD}}\left( e^{\psi P_2(S)}\right)\sim (1-q_A'')^{-\binom{a}{2}}  \exp\left\{ \frac{1}{2}\psi a^3q_A''^2 + \frac{1}{4} \psi^2 a^3 q_A^2 +\frac{3}{2} \psi^2 a^4 q_A^3 -\frac{1}{6}a^3q_A^3 \right\}\, .
\end{align}
\end{claim}
\begin{proof}
Recall that
\[
\cD_\emptyset:= \left\{G\subseteq \binom{A}{2}\cup\binom{B}{2}: \Delta(G)\leq \Delta, |G_A|, |G_B|\leq K  \right\}\, ,
\]
where $\Delta=\Delta_{A,B,\lam}, K=K_{A,B,\lam}$ are as in Definition~\ref{defsparse}. In particular $\cD$ is the set of triangle-free graphs in $\cD_\emptyset$.
By no longer conditioning on triangle-freeness, 
we are able to get a more precise understanding of the measure $\nu_{\bm q''_A, \cD_\emptyset}$ than that of $\nu_{\bm q''_A, \cD}$ (see Claim~\ref{claimprobrefine2} below). 
We therefore relate $ \E_{\nu_{\bm q''_A, \cD}}\left( e^{\psi  P_2(S)}\right)$ to the expectation $ \E_{\nu_{\bm q''_A, \cD_\emptyset}}\left( e^{\psi  P_2(S)}\right)$ by using our version of Janson's inequality (Lemma~\ref{lemJansonERG}). To this end we consider the tilted measure $\nu^k_{\bm q''_A, \cD_\emptyset}$ where $k(S)=\psi  P_2(S)$ and let $\mathbf G\sim \nu^k_{\bm q''_A, \cD_\emptyset}$. We then have
\begin{align}\label{eqDtoDempty}
Z_A'' \cdot \E_{\nu_{\bm q''_A, \cD}}\left( e^{\psi  P_2(S)}\right) = Z_{A,\emptyset}'' \cdot \E_{\nu_{\bm q''_A, \cD_\emptyset}}\left( e^{\psi  P_2(S)}\right) \P(\bG \text{ triangle-free})\, ,
\end{align}
where
\[
Z_{A,\emptyset}''=\sum_{S\subseteq \binom{A}{2}: S\in \cD_\emptyset}\left(\frac{q_A''}{1-q_A''}\right)^{|S|}\, .
\]

Since $\psi \leq n\lam^3$, $k$ is $(2n\lam^3)$-local, and so we may apply Lemma~\ref{lemJansonERG}, obtaining
\begin{align}\label{eqJansonERG}
\P(\bG \text{ triangle-free})\sim \exp\left\{-\binom{a}{3}q_A''^3 + O(n^4\Delta\lam^3q^3)\right\}\sim \exp\left\{-\frac{1}{6}a^3q_A^3\right\}\, .
\end{align}
Moreover, by ~\eqref{eqProbGinDEmptyEst} of Corollary~\ref{corZABJanson},
\begin{align}\label{eqZAemptyset}
Z_{A,\emptyset}''\sim (1-q_A'')^{-\binom{a}{2}}\, .
\end{align}

We now turn to estimating $\E_{\nu_{\bm q''_A, \cD_\emptyset}}\left( e^{\psi  P_2(S)}\right) $.  Given $\theta>0$, we abuse notation slightly and define the measure $\nu^{\theta}_{\bm q''_A, \cD_\emptyset}$ via
\[
\nu^{\theta}_{\bm q''_A, \cD_\emptyset}(S)\propto  \nu_{\bm q''_A, \cD_\emptyset}(S) e^{\theta \psi  P_2(S)}\, .
\]
We apply Lemma~\ref{lemtiltedcumulant} to deduce that
\begin{align}\label{eqcumulantbarnu}
\log\E_{\nu_{\bm q''_A, \cD_\emptyset}} e^{\psi  P_2(S)} = \psi  \cdot  \E_{\nu_{\bm q''_A, \cD_\emptyset}}(P_2(S)) + \frac{\psi ^2}{2}\var_{\nu^{\theta}_{\bm q''_A, \cD_\emptyset}}(P_2(S))\, 
\end{align}
for some $\theta\in [0,1]$.

\begin{claim}\label{claimprobrefine2}
Let $\theta\in[0,1]$ and let $\bG\sim\nu^{\theta}_{\bm q''_A, \cD_\emptyset}$. 
Let $F\subseteq\binom{A}{2}$ such that $|F|=O(1)$, then
\begin{align}
\P(F\subseteq \bG)=
(1+O(\psi ^3\Delta^3) ) q_A''^{|F|} 
\left(1+ \theta\psi  \E( P_2(\bG,F) )+ \frac{\theta^2\psi ^2}{2}  \E(P_2(\bG,F)^2)\right)\, .
\end{align}
In particular,
\begin{align*}
\P(F\subseteq \bG)=
(1+O(\psi \Delta) ) q_A''^{|F|} \, .
\end{align*}
\end{claim}
Again, as the proof of this claim follows similar lines as those of Corollary~\ref{corfixedconfigboot} and Claim~\ref{claimprobrefine}, we defer it to Appendix~\ref{secAppSubCalc}.

Returning to~\eqref{eqcumulantbarnu}, we estimate  $\E_{\nu_{\bm q''_A, \cD_\emptyset}}(P_2(S))$. Let $\bG\sim\nu_{\bm q''_A, \cD_\emptyset}$.
Let $F\subseteq \binom{A}{2}$ be a copy of $P_2$ then by the above claim with $\theta=0$,
\[
\P(F\subseteq \bG)= (1+O(\psi ^3\Delta^3) )q_A''^2\, .
\]
Noting that $\psi =\tilde O(n^{-1/2}), q_A=o(n^{-13/14})$, it follows that
\begin{align}\label{eqexpbarnu}
\psi  \E_{\nu_{\bm q''_A, \cD_\emptyset}}(P_2(S)) = 3\psi \binom{a}{3}(1+O(\psi ^3\Delta^3) )q_A''^2 = \frac{1}{2}\psi  a^3q_A''^2 + o(1)\, .
\end{align}
We now estimate $\var_{\nu_{\bm q''_A, \cD_\emptyset}^\theta}(P_2(S))$.
Let $\bG\sim \nu_{\bm q''_A, \cD_\emptyset}^\theta$. 
Let $\{F_1, \ldots, F_{m}\}$, $m=3\binom{a}{3}$, be the collection of potential copies of $P_2$ in $\bG$.
We have
\[
\var(P_2(S)) = \sum_i \var(\mathbf 1_{F_i\subseteq \bG}) + 2\sum_{\substack{\{F_i, F_j\}:\\ F_i\neq F_j}} \cov(\mathbf 1_{F_i\subseteq \bG}, \mathbf 1_{F_j\subseteq \bG})\, ,
\]
where the variances and covariances are with respect to the measure ${\nu_{\bm q''_A, \cD_\emptyset}^\theta}$.
By Claim~\ref{claimprobrefine2},
\[
\psi ^2\sum_i \var(\mathbf 1_{F_i\subseteq \bG})=(1+O(\psi  \Delta))\cdot 3\psi ^2\binom{a}{3}q_A''^2 =  3\psi ^2\binom{a}{3}q_A''^2 + o(1)\, .
\] 
By Claim~\ref{claimprobrefine2} we also have
\[
\psi ^2\sum_{\substack{\{F_i, F_j\}:\\ |F_i\cap F_j|=1}} \cov(\mathbf 1_{F_i\subseteq \bG}, \mathbf 1_{F_j\subseteq \bG})=(1+O(\psi  \Delta))\cdot 36\psi ^2\binom{a}{4}q_A''^3 =  36\psi ^2\binom{a}{4}q_A''^3 + o(1)\, 
\] 
and
\[
\psi ^2\sum_{\substack{\{F_i, F_j\}:\\ |V(F_i)\cap V(F_j)|=1}} \cov(\mathbf 1_{F_i\subseteq \bG}, \mathbf 1_{F_j\subseteq \bG})=\psi ^2\cdot O(n^5q^4\cdot\psi  \Delta)=o(1)\, .
\] 
It remains to estimate the contribution to the variance of $P_2(S)$ from vertex-disjoint pairs $F_i, F_j$.

Let $F_1, F_2\subseteq\binom{A}{2}$ be two vertex-disjoint copies of $P_2$ in $\binom{A}{2}$
then 
\[
P_2(\bG, F_1\cup F_2)= P_2(\bG, F_1)+ P_2(\bG, F_2)\, .
\]
It follows from Claim~\ref{claimprobrefine2} that 
\begin{align}\label{eqcovF1F2disjoin}
\cov(\mathbf1_{F_1\subseteq\bG}, \mathbf1_{F_2\subseteq\bG})= 
q_A''^4\frac{\theta^2\psi ^2}{2}\cdot\cov( P_2(\bG,F_1), P_2(\bG,F_2) ) + O(\psi ^3\Delta^3q^4)\, .
\end{align}
Suppose now that $e_1, e_2\in \binom{A}{2}\backslash (F_1\cup F_2)$ are such that $e_i$ forms a copy of $P_2$ with an edge of $F_i$ for $i=1,2$. If $e_1\neq e_2$, by Claim~\ref{claimprobrefine2}
\[
 \cov(\mathbf 1_{e_1\subseteq \bG}, \mathbf 1_{e_2\subseteq \bG})
 =O(\psi  \Delta q^2)
\]
and there are $O(n^2)$ such pairs $\{e_1, e_2\}$. 
If $e_1=e_2$ then
\[
 \cov(\mathbf 1_{e_1\subseteq \bG}, \mathbf 1_{e_2\subseteq \bG})
 =O(q)
\]
and there are $O(1)$ choices for $e_1$ since it has to join a vertex in $F_1$ to a vertex in $F_2$. We conclude that
\[
\cov( P_2(\bG,F_1), P_2(\bG,F_2) )= O(n^2\psi  \Delta q^2+q) = O(n^2\psi  \Delta q^2) \, .
\]
Returning to~\eqref{eqcovF1F2disjoin}, we conclude that
\[
\cov(\mathbf1_{F_1\subseteq\bG}, \mathbf1_{F_2\subseteq\bG})= O(\psi ^3\Delta^3q^4+n^2\psi ^3 \Delta q^6)= O(\psi ^3\Delta^3q^4)\, ,
\]
and so

\[
\psi ^2\sum_{\substack{\{F_i, F_j\}:\\ V(F_i)\cap V(F_j)=\emptyset}} \cov(\mathbf 1_{F_i\subseteq \bG}, \mathbf 1_{F_j\subseteq \bG})=\psi ^2n^6\cdot O(\psi ^3\Delta^3q^4)=o(1). \footnote{It is here that we need the assumption $\lam\geq \frac{13}{14}\sqrt{\frac{\log n}{n}}$.}
\] 
Putting everything together we have
\begin{align*}
\psi ^2\var_{\nu_{\bm q''_A, \cD_\emptyset}^\theta}(P_2(S))
&=
\psi ^2\sum_i \var(\mathbf 1_{F_i\subseteq \bG}) + 2 \psi ^2\sum_{\substack{\{F_i, F_j\}:\\ F_i\neq F_j}} \cov(\mathbf 1_{F_i\subseteq \bG}, \mathbf 1_{F_j\subseteq \bG})\\
&=
3\psi ^2\binom{a}{3}q_A''^2 + 72\psi ^2\binom{a}{4}q_A''^3 +o(1)\\
&=
\frac{1}{2} \psi ^2 a^3 q_A^2 + 3\psi ^2 a^4 q_A^3  + o(1)\, .
\end{align*}
Returning to~\eqref{eqcumulantbarnu} and recalling~\eqref{eqexpbarnu}, we conclude that
\begin{align}
\E_{\nu_{\bm q''_A, \cD_\emptyset}} e^{\psi  P_2(S)} \sim \exp\left\{ \frac{1}{2}\psi  a^3q_A''^2 + \frac{1}{4} \psi ^2 a^3 q_A^2 + \frac{3}{2}\psi ^2 a^4 q_A^3  \right\}\, .
\end{align}
Combining this with \eqref{eqDtoDempty}, \eqref{eqJansonERG} and \eqref{eqZAemptyset} completes the proof of Claim~\ref{claimERGApartfun}.
\end{proof}
Lemma~\ref{lem:P2MomentSuper} now follows from 
Claims~\ref{claimERGApartfun} and~\eqref{eq:P2momentforTV}
\end{proof}

We now turn to the proof of Lemma~\ref{lemmaZABabapproxsuper}. We will need the following refinement of Corollary~\ref{corclustersimple} whose proof we defer to Appendix~\ref{secPinnedCluster}. We let $P_3, S_3, C_4$ denote the path, star and cycle on $4$ vertices respectively. Recall that given graphs $H, G$, we let $H(G)$ denote the number of (not necessarily induced) copies of $H$ in $G$. 
\begin{cor}\label{corclustersimple2}
Let $G$ be a triangle-free graph with $n$ vertices, and maximum degree $\Delta$. Then for $\lam\leq \frac{1}{4e\Delta}$,
\begin{multline}
\log\left( \frac{Z_G(\lam)}{(1+\lam)^n} \right)= -|G|\lam^2 + \left(P_2(G)+2|G|\right)\lam^3 -
(P_3(G)+S_3(G)-C_4(G)+4P_2(G)+7|G|/2)\lam^4\\
 +
 O(n\Delta^4\lam^5)\, .
\end{multline}
\end{cor}
In addition to the relations $|S\boxempty T|= b|S|=a|T|$, $P_2(S\boxempty T) = bP_2(S) + aP_2(T) + 4|S||T|$ which we have already encountered, we will also need the following relations and include a proof in Appendix~\ref{secPinnedCluster}.
\begin{lemma}\label{lemSubgCart}
For $S\subseteq \binom{A}{2}$, $T\subseteq \binom{B}{2}$ such that $S\cup T$ is triangle-free,
\begin{align*}
P_3(S\boxempty T) &= bP_3(S) + aP_3(T) + 6P_2(S)|T|+ 6|S|P_2(T)\, , \\
S_3(S\boxempty T) &= bS_3(S) + aS_3(T) + 2P_2(S)|T|+ 2|S|P_2(T)\, , \\
C_4(S\boxempty T) &= bC_4(S) + aC_4(T) + |S||T|\, .
\end{align*}
\end{lemma}

With these preliminaries in hand, we prove Lemma~\ref{lemmaZABabapproxsuper}.
\begin{proof}[Proof of Lemma~\ref{lemmaZABabapproxsuper}]
In order to estimate $Z_{A,B}(\lam)$, we begin by estimating the hard-core  partition function $Z_{S \boxempty T }(\lam)$ via the  cluster expansion. First we note that since $(S,T)\in \cD$, both $S$ and $T$ have maximum degree at most
\[
\Delta=50 \max\{qn,\log n\}= o(n^{1/14})  \, ,
\] 
and so the graph $S \boxempty T$ has maximum degree at most $2\Delta=o(n^{1/14})$.  Since $\lam\leq \frac{1}{8e\Delta}$, we conclude from Corollary~\ref{corclustersimple2} that
\begin{multline}
\label{eqZboxTestimateSuper}
\log\left(\frac{ Z_{S \boxempty T}(\lam)}{(1+\lam)^{ab}} \right)
 = - |S \boxempty T| \cdot \lambda^2 +  \left(2|S \boxempty T| + P_2(S \boxempty T)\right)\lambda^3 \\
- 
(P_3(S \boxempty T)+S_3(S \boxempty T)-C_4(S \boxempty T)+4P_2(S \boxempty T)+7|S \boxempty T|/2)\lam^4
 +o(1)\, .
\end{multline}
We note that
$|S\boxempty T| = b|S| + a|T|$, and so 
\begin{equation}\label{eqZABbarnuexp}
\frac{Z_{A,B}(\lam)}{(1+\lam)^{ab}}= \sum_{(S,T)\in \cD}\frac{ Z_{S \boxempty T}(\lam)}{(1+\lam)^{ab}}  \sim 
Z'\cdot\E_{\nu_{\bm q', \cD}}\left(e^{P_2(S\boxempty T)\lam^3
-(P_3(S \boxempty T)+S_3(S \boxempty T)-C_4(S \boxempty T)+4P_2(S \boxempty T))\lam^4}
 \right),
\end{equation}
where $\bm q'=(q_A', q_B')$ and $Z'$ is as in~\eqref{eq:ZprimeDefSuper}.
We turn our attention to understanding the expectation on the RHS of~\eqref{eqZABbarnuexp}. Letting
\[
f(S,T)=\lam^3 P_2(S\boxempty T)\, ,
\]
the expectation on the RHS of~\eqref{eqZABbarnuexp} is equal to 
\begin{align}\label{eqExpectationProd}
 \E_{\nu^f_{\bm q', \cD}}\left(e^{
-(P_3(S \boxempty T)+S_3(S \boxempty T)-C_4(S \boxempty T)+4P_2(S\boxempty T))\lam^4}
 \right) 
 \cdot 
 \E_{\nu_{\bm q', \cD}}\left(e^{\lam^3P_2(S\boxempty T)}
 \right)\, .
\end{align}
We estimated the rightmost expectation in Lemma~\ref{lem:P2MomentSuper}. 
We now estimate the leftmost expectation.
\begin{claim}\label{claim:P3S3}
\begin{align}
 \E_{\nu^f_{\bm q', \cD}}&\left(e^{
-(P_3(S \boxempty T)+S_3(S \boxempty T)-C_4(S \boxempty T)+4P_2(S \boxempty T))\lam^4}
 \right)\\ 
 &\phantom{==}\sim
 \exp\left[\lam^4 ab\left(\frac{1}{4}ab q_Aq_B - \frac{2}{3}(aq_A+bq_B)^3 - 2(aq_A+bq_B)^2 \right)\right]\, .
 \end{align}
\end{claim}
\begin{proof}
By Lemma~\ref{lemtiltedcumulant}
\begin{multline}
\log\,  \E_{\nu^f_{\bm q', \cD}}\left(e^{
-(P_3(S \boxempty T)+S_3(S \boxempty T)-C_4(S \boxempty T)+4P_2(S \boxempty T))\lam^4}
 \right) \\
= -\lam^4 \E_{\nu^g_{\bm q', \cD}}\left(
P_3(S \boxempty T)+S_3(S \boxempty T)-C_4(S \boxempty T)+4P_2(S \boxempty T)
 \right) \, ,
\end{multline}
where 
\[
g(S,T)=f(S,T)-\theta\cdot(P_3(S \boxempty T)+S_3(S \boxempty T)-C_4(S \boxempty T)+4P_2(S \boxempty T))\lam^4\, ,
\]
and $\theta\in [0,1]$. 
We calculate the expectation on the RHS of the above. 
One easily verifies that $g$ is $n\lam^3/(6\alpha)$-local and so by Lemma~\ref{lemSubgCart} and \eqref{eqlemfixedconfiglocalcor} of Lemma~\ref{lemfixedconfiglocal},
{\Small
\begin{align}
&\lam^4  \E_{\nu^g_{\bm q', \cD}}(
P_3(S \boxempty T))  \\
&=
\lam^4(1+O(n\Delta \lam^3))
\left[
b\cdot12\binom{a}{4}q_A'^3 
+
a\cdot12\binom{b}{4}q_B'^3
+
6\cdot 3\binom{a}{3}\binom{b}{2}q_A'^2q_B'
+
6\cdot 3\binom{b}{3}\binom{a}{2}q_B'^2q_A'
\right] \\
&=
\lam^4
\left[
\frac{1}{2}ba^4q_A^3 
+
\frac{1}{2}ab^4q_B^3 
+
\frac{3}{2}a^3b^2q_A^2q_B
+
\frac{3}{2}b^3a^2q_B^2q_A
\right] +o(1)\\
&=
\frac{1}{2}\lam^4ab(aq_A+bq_B)^3
 +o(1)\, .
\end{align}}
The final expression can be arrived at  heuristically by noting that the product graph $S\boxempty T$ has $ab$ vertices and the expected degree of any vertex is approximately $aq_A+bq_B$.
Similarly
\begin{align}
\lam^4 \E_{\nu^g_{\bm q', \cD}}(
S_3(S \boxempty T)) 
&=
\frac{1}{6}\lam^4ab(aq_A+bq_B)^3+o(1)\, ,\\
\lam^4 \E_{\nu^g_{\bm q', \cD}}(
C_4(S \boxempty T)) &= \frac{1}{4}\lam^4a^2b^2 q_Aq_B + o(1)\, , \\
4\lam^4 \E_{\nu^g_{\bm q', \cD}}(
P_2(S \boxempty T)) &= 2ab(aq_A+bq_B)^2 +o(1)\, .
\end{align}
The claim follows. 
\end{proof}
Lemma~\ref{lemmaZABabapproxsuper} follows by combining ~\eqref{eqZABbarnuexp} and \eqref{eqExpectationProd} with Lemma~\ref{lem:P2MomentSuper} and Claim~\ref{claim:P3S3}.
\end{proof}

As mentioned at the start of this section, the proof of Lemma~\ref{lemZABapprox1super} is deferred to Appendix~\ref{secAppCalc}.

We now prove Proposition~\ref{propGroundStateStrong}. The proof is a minor variant of that of Lemma~\ref{lemZstrongsimsum}.
\begin{proof}[Proof of Proposition~\ref{propGroundStateStrong}.]
Let  $M:=5(n\log n)^{1/4}$. By~Lemma~\ref{lemZABapprox1super},
\[
\left|\frac{ Z_{\textup{mod}}(\lam)}{ Z_{\textup{strong}}(\lam)}-1\right|
 \leq 
(1+o(1))  \sum_{k\geq M} \frac{\binom{n}{\lfloor n/2 \rfloor+k}}{\binom{n}{\lfloor n/2 \rfloor}} e^{O(n^2q\lam^4k^2)} (1+\lam)^{-k^2}.
\]
Noting that $n^2q\lam^4k^2=o(\lam k^2)$   the RHS is bounded above by
\begin{align}\label{eqGausstail}
 \sum_{k\geq M} e^{-\lam k^2/2} \leq \int_{M-1}^\infty e^{-\lam x^2/2}\, dx \leq \frac{1}{\lam(M-1)} e^{-(M-1)^2\lam/2}=O(n^{-3})\, ,
\end{align}
 where for the second inequality we used the standard integral estimate $\int_{t}^\infty e^{-ax^2}\, dx\leq e^{-at^2}/(2at)$ for $a,t>0$. This proves~\eqref{eqStrongtoMod}. The proof of~\eqref{eqStrongtoModTV} is identical to the proof of \eqref{eqmusimmumod} except that we use Lemma~\ref{lemCaptureStrong} in place of Lemma~\ref{lemCaptureMod}.
\end{proof}

Recall that Corollary~\ref{corZmodsimsum} now follows from Proposition~\ref{lemZerothApprox} and Propositions~\ref{propGroundStateRefinedAlt},~\ref{lemZmodsimsum} and~\ref{propGroundStateStrong}.
Lemma~\ref{lemSuperMainLem} then follows from Corollary~\ref{corZmodsimsum} in precisely the same way that Lemma~\ref{lemZapproxER} followed from Corollary~\ref{corZmodsimsumRest}.

The proof of Theorem~\ref{thmGnpStructureSuper} is now identical to the proof of Theorem~\ref{thmGnpSubCritLDprob} in Section~\ref{subsec:imbalance}.

\subsection{Chromatic number}
In this section we prove Theorem~\ref{thm4colorGnp} which states that for $\eps \in (0,1/14]$ and  
$p \sim (1-\eps)\sqrt{\frac{\log n}{n}}$, if $G$ is sampled from $G(n,p)$ conditioned on $\cT$, then the independence number of $G$ is $o(n)$ whp. In particular, the chromatic number of $G$ is $\omega(n)$ whp.

\begin{proof}[Proof of Theorem~\ref{thm4colorGnp}]
Set $\lam=p/(1-p)\sim  (1-\eps)\sqrt{\frac{\log n}{n}}$ and fix $\pi=(A,B)$ strongly balanced. Let $\mu_{\lam,2}^\pi$ denote the measure $\mu_{\lam,2}$ conditioned on the event that $\pi$ is chosen at Step 1 in Algorithm~\ref{algMulam2}. Let $\bG\sim \mu^\pi_{\lam,2}$. By Theorem~\ref{thmGnpStructureSuper} it suffices to show that $\mathbb{P}[\alpha(\bG) = o(n)] = 1-o(1)$.
For this, it will suffice to show that  $\mathbb{P}[\alpha(\bG_A) = o(n)] = 1-o(1)$ and $\mathbb{P}[\alpha(\bG_B) = o(n)] = 1-o(1)$.

Note that $\bG_A$ is distributed according to the  random graph $G(A, q_2, \psi)$. 
 We fix $U\subseteq A$ and estimate the probability that $U$ is an independent set in $\bG_A$. Let $\binom{U}{2}=\{e_1, \ldots, e_{N}\}$ where $N=\binom{|U|}{2}$. Let $E_i$ denote the event that $e_i$ is an edge of $\bG_A$, then
{\small
\[
\P( U \textup{ is an independent set in }\bG) = \P\left (\bigcap_{i=1}^{N}E_i^c \right) = \prod_{i=1}^N  \P\left (E_i^c \, \bigg |\,   \bigcap_{j<i}E_j^c\right) = \prod_{i=1}^N\left[1-  \P\left (E_i \, \bigg |\,   \bigcap_{j<i}E_j^c\right)\right]\, .
\]}
Fix $i\in[N]$. Then by Lemma~\ref{lemWarmup}
\[
 \P\left (E_i \, \bigg |\,   \bigcap_{j<i}E_j^c\right)
 =
 \left(1+O\left(n\Delta\lam^3\right)\right)q_2 \geq q_2/2 \,.
\]
It follows that
\begin{align*}
\P( U \textup{ is an independent set in }\bG_A)
\leq  (1-q_2/2)^{\binom{|U|}{2}}
\leq 
\exp\left\{- \frac{q_2}{2}{\binom{|U|}{2}}\right\}\, .
\end{align*}
Note that since $\lam\sim  (1-\eps)\sqrt{\frac{\log n}{n}}$ we have $q_2\geq n^{-1+\eps/2} $.
Let $k=n^{1-\eps/4}$.
We conclude by a union bound that
\[
\mathbb{P}[\alpha(\bG_A) \geq k] \leq \binom{a}{k}\exp\left\{- \frac{q_2}{2}{\binom{k}{2}}\right\}\leq \exp\left\{ k \log (ea/k)- \frac{q_2}{2}{\binom{k}{2}}\right\}=o(1)\, .
\]
We conclude that $\mathbb{P}(\alpha(\bG_A) < n^{1-\eps/4})=1-o(1)$ and similarly for $\bG_B$ concluding the proof. 
\end{proof}

\subsection{A sandwiching theorem}

Although the distribution of the defect edges is not that of a pair of \ER random graphs in the supercritical defect regime, the distribution of edges is sandwiched between two \ER random graphs (conditioned on triangle-freeness) with edge probabilities that differ by a small amount.  Recall the definitions of $q_2, \psi$ from~~\eqref{eqq2Def} and\eqref{eqPsiDef}. For $A\subseteq [n]$ and $q\in(0,1)$ we let $G(A,q|\cT)$ denote the \ER
graph $G(A,q)$ conditioned on triangle-freeness.

\begin{prop}
\label{propSandwhich}
Suppose $\lam \geq \frac{13}{14} \sqrt{\frac{\log n}{n}}$.  Let $q_{\ell} =q_2 (1- n^{-2/5}) $ and $q_u =q_2 (1+ n^{-2/5})$.   Then for a vertex set $A\subseteq [n]$, there is a coupling of the distributions $G(A,q_\ell  |  \cT), G(A,q_2,\psi), G(A,q_u  |  \cT)$ so that with probability $1-o(1)$,
\[ G(A,q_{\ell}  |  \cT) \subseteq G(A,q_2, \psi) \subseteq G(A,q_u  |  \cT) \,. \]
\end{prop}
\begin{proof}
Let $a=|A|$.  We construct the coupling as follows.  Order the $\binom {a}{2}$ possible edges arbitrarily $e_1, \dots, e_{\binom{a}{2}}$.  Let $X^\ell_i, X_i, X^u_i$ be the indicator random variables that $e_i$ is present in $G(A,q_\ell  |  \cT), G(A,q_2,\psi),G(A,q_u  | \cT)$ respectively.  

Let $E_i=\{e_j : X_j=1, j\leq i\}$ and define $E_i^\ell, E_i^u$ similarly.  Select  iid $U[0,1]$ random variables $U_1, \dots , U_{\binom{a}{2}}$.   For $i=0, \dots, \binom{a}{2}-1$, we set $X_{i+1} =1$ if $U_{i+1} \le \P[e_{i+1} \in  G(A,q_2, \psi)  |  E_i]$ and $0$ otherwise; and likewise with $X^\ell_{i+1}$ and $X^u_{i+1}$; in particular, we use the same uniform random variable for each of the three processes.  Clearly the coupling produces faithful copies of $G(A,q_\ell  |  \cT), G(A,q_2,\psi),G(A,q_u  |  \cT)$.  

We now argue about containment.   
We will show that  $G(A,q_2, \psi) \subseteq G(A,q_u  |  \cT)$ whp; the proof that $G(A,q_\ell  |  \cT) \subseteq G(A,q_2, \psi)$ whp is similar and we omit it.   By a union bound it suffices to show that $\P(X_i =1 \wedge X^u_i =0 ) = o(n^{-2})$. 

We say that an edge $e_i$ is \emph{blocked} by a set $E$ of edges if $e_i\cup E$ contains a triangle.  Let $B_i, B_i^u$ denote the event that $e_i$ is blocked by the final graphs $E:=E_{{\binom{a}{2}}}, E^u:=E_{{\binom{a}{2}}}^u$ respectively.
 Noting that $X_i=1$ only if $e_i$ is not blocked by $E$ we write 
\begin{align}\label{eqXiXiuTP}
\P(X_i =1 \wedge X^u_i =0 ) 
&
= \P(X_i =1 \wedge X^u_i =0 \mid \bar B_i \wedge \bar B_i^u )\P( \bar B_i \wedge \bar B_i^u)\\
&\phantom{=} +   \P(X_i =1 \wedge X^u_i =0 \mid \bar B_i \wedge  B_i^u )\P( \bar B_i \wedge  B_i^u)\, ,
\end{align}
where $\bar B_i $ denotes the complement of the event $B_i $.  
\begin{claim}\label{claimWedge}
\[
\P(X_i =1 \wedge X^u_i =0 \mid \bar B_i \wedge \bar B_i^u ) \leq 2ne^{-d/3}\, ,
\]
where $d:=50\max\{q_u n, \log n\}$.
\end{claim}
\begin{proof}
Let $\cE_i$ denote the event that neither $E_{i-1}$ nor $E^u_{i-1}$ block $e_i$, and \\ $\max\{\Delta(E^u_{i-1}), \Delta(E_{i-1})\}\leq d/3$. Note that 
\begin{multline}\label{eqBothUnb}
\P(X_i =1 \wedge X^u_i =0 \mid \bar B_i \wedge \bar B_i^u ) 
= \P(X_i =1 \wedge X^u_i =0 \mid \bar B_i \wedge \bar B_i^u \wedge \cE_i)\P(\cE_i \mid \bar B_i \wedge \bar B_i^u)\\
\phantom{=}+ \P(X_i =1 \wedge X^u_i =0 \mid \bar B_i \wedge \bar B_i^u \wedge \bar\cE_i)\P(\bar\cE_i \mid \bar B_i \wedge \bar B_i^u)\, .
\end{multline}
Observe that 
\[
\P(X_i =1 \wedge X^u_i =0 \mid \bar B_i \wedge \bar B_i^u \wedge \cE_i)
\leq 
\frac{\P(X_i =1 \wedge X^u_i =0 \mid  \cE_i)}{\P(\bar B_i \wedge \bar B_i^u \mid  \cE_i)}\, .
\]
We will show that the numerator on the RHS is $0$. By the definition of the coupling it suffices to show that for $F, F^u$ such that
$\P(E_{i-1}=F \mid \cE_i)>0$ and $\P(E^u_{i-1}=F^u \mid \cE_i)>0$ we have 
\begin{align}\label{eqThreshcomp}
    \P(X_i=1 \mid E_{i-1}=F) < \P(X^u_i=1 \mid E^u_{i-1}=F^u)\, .
\end{align}
We begin by estimating the LHS. 
\begin{align}\label{eqUthreshold}
    \P(X_i=1 \mid E_{i-1}=F) = \P(X_i=1 \mid E_{i-1}=F, \bar B_i)\P(\bar B_i \mid E_{i-1}=F)\, .
\end{align}
We estimate the two probabilities on the RHS. First note that since $\P(E_{i-1}=F \mid \cE_i)>0$ by assumption, $F$ does not block $e_i$ and $\Delta(F)\leq d/3$. There are therefore at most $2d/3$ single edges and at most  $n$ pairs of edges whose addition to $F$ could block $e_i$. We conclude from Lemma~\ref{lemWarmup} that 
\begin{align}\label{eqBlockhalf}
\P(\bar B_i \mid E_{i-1}=F)\geq 1- (nq_u^2 + 2q_ud/3) \geq 1-q_ud\, .
\end{align}
Let $\cH_i$ denote the set of all possible realisations $H$ of $E\backslash e_i$ that do not block $e_i$. Let $\cH_i'\subseteq \cH_i$ denote the subset of graphs that satisfy $\Delta(H)\leq 2d/3$.
\begin{align}\label{eqXicondUnB}
 \P(X_i=1 \mid E_{i-1}=F, \bar B_i)= \sum_{H\in \cH_i}\P(X_i=1 \mid E\backslash{e_i}=H)\P(E\backslash{e_i}=H \mid E_{i-1}=F, \bar B_i)\, .
\end{align}
We will split the above sum according to whether $H\in \cH_i'$ or not. If $H\in\cH_i'$, then by Lemma~\ref{lemWarmup}, noting that $q_2=o(\psi d)$,
\begin{align}\label{eqXiCondFull}
\P(X_i=1 \mid E\backslash{e_i}=H)= (1+O(\psi d))q_2
\end{align}
Suppose now that $H\in \cH_i\backslash \cH_i'$ so that in particular $\Delta(H)>2d/3$. We have 
\begin{align}\label{eqHnotinHi}\\
    \P(E\backslash{e_i}=H \mid E_{i-1}=F, \bar B_i)= \frac{\P(E\backslash{e_i}=H,  \bar B_i \mid E_{i-1}=F)}{\P(\bar B_i \mid E_{i-1}=F)}\leq \frac{\P(\Delta(E)>2d/3 \mid E_{i-1}=F)}{\P(\bar B_i \mid E_{i-1}=F)}\, .
\end{align}
If $\Delta(E)> 2d/3$ then since $\Delta(F)\leq d/3$ there exists a vertex $v\in V$ with at least $2d/3-d/3=d/3$ incident edges that do not belong to $F$. By Lemma~\ref{lemWarmup}, the probability of this occurring is at most the probability a binomial $\bin(n, q_u)$ random variable is at least $d/3=(50/3)\max\{q_u n, \log n\}$. By Chernoff's inequality (Lemma~\ref{lemChernoff}), this occurs with probability at most $e^{-d/3}$. By a union bound over $v\in V$ we have
\[
\P(\Delta(E)>2d/3 \mid E_{i-1}=F)\leq ne^{-d/3}\, .
\]
Recalling~\eqref{eqBlockhalf} and 
returning to~\eqref{eqHnotinHi}, we conclude that
\[
\P(E\backslash{e_i}=H \mid E_{i-1}=F, \bar B_i)\leq 2ne^{-d/3}\, .
\]
Combining this fact, with~\eqref{eqXiCondFull} and splitting the sum according to whether $H\in \cH_i'$ or not we have 
\[
 \P(X_i=1 \mid E_{i-1}=F, \bar B_i) = (1+O(\psi d))q_2(1-O(ne^{-d/3})) + O(ne^{-d/3}) = (1+O(\psi d))q_2\, .
\]
Recalling~\eqref{eqBlockhalf} again and 
returning to~\eqref{eqUthreshold} we have
\[
 \P(X_i=1 \mid E_{i-1}=F)= (1+O(\psi d))q_2\, .
\]
An identical argument shows that 
\[
\P(X^u_i=1 \mid E^u_{i-1}=F)= (1+O(\psi d))q_u\, .
\]
Recalling that $q_u =q_2 (1+ n^{-2/5}) $ and $\psi d = o(n^{-2/5})$, we see that the inequality at~\eqref{eqThreshcomp} holds. We conclude that
\begin{align}\label{eqBothUnbEi}
    \P(X_i =1 \wedge X^u_i =0 \mid \bar B_i \wedge \bar B_i^u \wedge \cE_i)=0\, .
\end{align}

Returning to~\eqref{eqBothUnb} we turn to estimating $\P(\bar\cE_i \mid \bar B_i \wedge \bar B_i^u)$.
For this, we note that under the event $\bar B_i \wedge \bar B_i^u$, the only way for the event $\bar\cE_i$ to occur is if $\max\{\Delta(E^u_{i-1}), \Delta(E_{i-1})\}\ge d/3$. Arguing as above (i.e., applying Lemma~\ref{eqWarmup}, Chernoff's inequality  and a union bound) this occurs with probability at most $2n e^{-d/3}$.
Therefore 
\[
\P(\bar\cE_i \mid \bar B_i \wedge \bar B_i^u)\leq 2n e^{-d/3}\, .
\]
Combining this with~\eqref{eqBothUnb} and~\eqref{eqBothUnbEi} completes the proof of the claim.
\end{proof}

We now return to ~\eqref{eqXiXiuTP} and bound $\P( \bar B_i \wedge  B_i^u)$. For the event $\bar B_i \wedge  B_i^u$ to occur, there must exist $j,k$ such that $e_i, e_j, e_k$ forms a triangle, $X_j^u=1, X_k^u=1$ and $\{X_j=0$ or $X_k=0\}$. By a union bound
\[
\P( \bar B_i \wedge  B_i^u)\leq 2n\cdot  \P(X_j^u=1, X_k^u=1, X_j=0).
\]
Let $\cE_{j}$ denote the event that neither $E_{j-1}$ nor $E^u_{j-1}$ block $e_j$ and $\max\{\Delta(E^u_{j-1}), \Delta(E_{j-1})\}\leq d/3$.
\begin{align}
\P(X_j^u=1, X_k^u=1, X_j=0)
&\leq \P(X_j^u=1, X_k^u=1, X_j=0\mid \cE_j)\P(\cE_j)\\
&\phantom{=}+  \P(X_j^u=1, X_k^u=1, X_j=0\mid \bar\cE_j)\P(\bar\cE_j)\, .
\end{align}
It is simple to bound the terms on the RHS by Lemma~\ref{lemWarmup}. 
We bound $\P(\bar\cE_j)=O(nq_u^2)$, $ \P(X_j^u=1, X_k^u=1, X_j=0\mid \bar\cE_j)=O(q_u^2)$, $\P(X_j^u=1, X_k^u=1, X_j=0\mid \cE_j) = O(q_u^2\psi d)$ so that
\[
\P(X_j^u=1, X_k^u=1, X_j=0)\leq O(n q_u^4 +q_u^2\psi d) = O(q_u^2\psi d)\, .
\]
It follows that
\[
\P( \bar B_i \wedge  B_i^u)=O(nq_u^2\psi d)
\]
Finally note that $\P(X_i =1 \wedge X^u_i =0 \mid \bar B_i \wedge  B_i^u )=O(q_u)$ and so by~\eqref{eqXiXiuTP} and Claim~\ref{claimWedge}
\[
\P(X_i =1 \wedge X^u_i =0 ) = O(2ne^{-d/3} + nq_u^3\psi d  ) =o(n^{-2})\, . \qedhere
\]
\end{proof}

\subsection{Emergence of the giant  defect component and connectivity}

 We now prove Theorem~\ref{thmGnpGiant}.   Note that from~\eqref{eqqdef},\eqref{eqq0q1Def},\eqref{eqq2Def}, we have 
 
\[ q_0 = q_2(1+ O(\mu \lam^3 + 
\lam^3 n))  = q_2(1+ o(n^{-2/5}) ) \,.\]
Let $q_u$ and $q_\ell$ be as in Proposition~\ref{propSandwhich}.
If $\lam$ is such that $q_0 = \frac{2}{n} \pm \frac{\omega(n)}{n^{4/3}}$ with $\omega(n) \gg 1$, then we have, with $a= |A|$, and using the fact that $a=n/2+\tilde O(n^{1/4})$ since $(A,B)$ is strongly balanced,
\begin{align*}
    q_u &=   \left( \frac{2}{n} \pm \frac{\omega(n)}{n^{4/3}} \right) (1+ o(n^{-2/5})) =  \frac{1}{a} \pm (2^{-4/3}+o(1)) \frac{\omega(n)}{a^{4/3}} \\
    q_\ell &= \left( \frac{2}{n} \pm \frac{\omega(n)}{n^{4/3}} \right) (1+ o(n^{-2/5})) =  \frac{1}{a} \pm (2^{-4/3}+o(1)) \frac{\omega(n)}{a^{4/3}}  \,.
\end{align*}
Similarly, when $q_0 = \frac{2}{n} + \frac{\omega}{n^{4/3}}$ with $\omega \in \R$ constant, then we have
\begin{align*}
    q_u &=   \left( \frac{2}{n} + \frac{\omega}{n^{4/3}} \right) (1+ o(n^{-2/5})) =  \frac{1}{a} +  \frac{2^{-4/3}\omega + o(1)}{a^{4/3}} \\
    q_\ell &= \left( \frac{2}{n} + \frac{\omega}{n^{4/3}} \right) (1+ o(n^{-2/5})) =  \frac{1}{a} + \frac{2^{-4/3}\omega +o(1)}{a^{4/3}}  \,.
\end{align*}

For an \ER random graph of constant average degree, the probability of having a triangle is bounded away from $1$, and so any property that holds with probability $1-o(1)$ continues to hold with probability $1-o(1)$ after conditioning on triangle-freeness.   In particular, classic results on the giant component phase transition in random graphs (e.g.,~\cite{bollobas1998random}) and the estimates on $q_\ell, q_u$ above tell us that:
\begin{itemize}
\item  If   $q_0 = \frac{2}{n} - \frac{\omega(n)}{n^{4/3}}$ with $1 \ll \omega(n) \ll n^{1/3}$ then in   both $G(A, q_u)$ and  $G(A, q_\ell)$, whp the  largest connected component is of size $\Theta(n^{2/3} \omega^{-2} \log \omega ) $.
\item If   $q_0 = \frac{2}{n} + \frac{\omega}{n^{4/3}}$ with $\omega$ constant (positive or negative) then in both $G(A, q_u)$ and  $G(A, q_\ell)$, whp the  largest connected component is of size $\Theta(n^{2/3})$.
    \item If   $q_0 = \frac{2}{n} + \frac{\omega(n)}{n^{4/3}}$ with $1 \ll \omega(n) \ll n^{1/3}$ then in both $G(A, q_u)$ and  $G(A, q_\ell)$, whp the  largest connected component is of size $(2+o(1)) \cdot \omega \cdot a^{2/3}=(2+o(1)) \cdot \omega \cdot (n/2)^{2/3}$.
    \item If   $q_0 = \frac{c}{n}$ with $c>2$ fixed, then in both $G(A, q_u)$ and  $G(A, q_\ell)$, whp the  largest connected component is of size $\Theta(n)$.
\end{itemize}
Under the coupling of Proposition~\ref{propSandwhich}, whp the size of the largest component of $G(A,q_2,\psi)$ is bounded between the size of the largest component of $G(A,q_\ell| \cT)$ and that of $G(A,q_u| \cT)$.  The first four statements of Theorem~\ref{thmGnpGiant} then follow. 

Now fix $\eps>0$ and suppose $\lam$ is such that $q_0 = (1+\eps) \frac{2 \log n}{n}$.  By the same argument as above, we have that $q_\ell = (1 + \eps +o(1)) \frac{ \log a}{a}$.  Via Proposition~\ref{propSandwhich}, to show that $G(A,q_2, \psi)$ is connected whp it suffices to show that $G(A,q_\ell|\cT)$ is connected whp.  In this range of $q_{\ell}$ the probability of triangle-freeness in $G(A,q_\ell)$ is $o(1)$ and so we need to be a little careful about the conditioning. 

To prove $G(A,q_\ell|\cT)$ is connected whp when $q_\ell = (1 + \eps +o(1)) \frac{ \log a}{a}$ we bound  the expected number of non-trivial cuts with no edges (if the graph is disconnected there must be at least one such cut).  Call this expectation $\E Y$. 

We will use the fact that for any set of edges $B$, and any edge $e \notin B$,  $\P(e \in G|B \cap G = \emptyset) = q_\ell (1+ O(n q_{\ell}^2))$. Let $\bG$ be distributed as $G(A,q_\ell|\cT)$ conditioned on $B \cap  G = \emptyset$, and let $\bH =\bG \setminus e$.  Then if $H \cup e$ contains no triangles, $\P(e \in \bG| \bH=H) =q_\ell$; on the other hand, by stochastic domination of $\bG$ by $G(A, q_\ell)$, $\P(\bH \cup e \text{ triangle-free}) =1+O(n^2 q_\ell)$, and the fact follows.

We then have
{\small
\begin{align*}
    \E Y &\le  \sum_{k=1}^{\lceil a/2 \rceil} \binom{a}{k} \left( 1 -q_\ell +O(nq_\ell^3)   \right)^{k (a-k)} \\
    &\le \sum_{k=1}^{\lceil a/2 \rceil} \binom{a}{k} \left( 1 -\frac{(1+\eps') \log a}{a}  \right) ^{k (a-k)}  \text{ for some fixed }\eps'>0 \\
    &\le \sum_{k=1}^{\lfloor \sqrt{n}/\log n \rfloor} \left( \frac{a e  }{ k}   \right)^k \left( 1 -\frac{(1+\eps') \log a}{a}  \right) ^{k a - n/(\log n)^2} +   2\sum_{k=\lfloor \sqrt{n}/\log n \rfloor}^{\infty} \left( \frac{e a }{ k}   \right)^k \left( 1 -\frac{(1+\eps') \log a}{a}  \right) ^{k a /2}\\
    &\le 2  \sum_{k=1}^{\infty} \left( \frac{ae  }{ k}   \right)^k a^{-(1+\eps') k}  + 2 \sum_{k=\lfloor \sqrt{n}/\log n \rfloor}^\infty (6 \sqrt{a} \log n)^k a^{-(1+\eps') k/2}  \\
    &= o(1) \,.
\end{align*}}
Thus whp $G(A, q_\ell | \cT)$ is connected and thus so is $G(A,q_2, \psi)$.  

For the other side, $G(A,q_u|\cT)$ is stochastically dominated by $G(A,q_u)$; when $q_0 = (1-\eps) \frac{2 \log n}{n}$ and so $q_u = (1- \eps +o(1)) \frac{\log a}{a}$, $G(A,q_u)$ is disconnected whp; thus so is $G(A,q_u|\cT)$.  Then by Proposition~\ref{propSandwhich}, $G(A,q_2,\psi)$ is disconnected whp.  This proves the  last statement of Theorem~\ref{thmGnpGiant}.
 
\section{Results for $\cT(n,m)$}
\label{secFixedM}

In this section we transfer our results from $G(n,p)$ conditioned on triangle-freeness to the uniform distribution on $\cT(n,m)$ and prove the results of Section~\ref{secIntro}.

Recall the identity
\begin{align*}
|\cT(n,m)| =   \frac{ Z(\lam)}{\lam^m} \cdot \mu_{\lam} ( \{|G| = m \}) \,. 
\end{align*}
This reduces the determination of the asymptotics of $|\cT(n,m)|$ to the asymptotics of $Z(\lam)$ and  $\mu_{\lam} ( \{|G| = m \})$ for some choice of $\lam$.
Our main tool will be to apply a local central limit theorem for the hard-core model after conditioning on $(A,B)$ and $(S,T)$. 
Recall from~\eqref{lamAssumption} that we assume throughout that $m  \le  \frac{1}{2}n^{3/2}\sqrt{\log n}$ since larger densities are covered by~\cite{osthus2003densities}.

Let us recall the parameters defined at~\eqref{eqlam0def} and \eqref{eqpdef} in the introduction, $\lam_0=\frac{4m}{n^2}$ and  
\begin{align}\label{eqpdefagain}
\lam= \lam(m)= \lam_0+\lam_0^2 + (n\lam_0^2-1)\lam_0e^{-\lam_0^2n/2}\, .
\end{align}
Throughout this section $\lam=\lam(m)$ as above. As we will see below, $\lam$ is chosen so that the typical number of edges in a sample from $\mu_\lam$ is close to $m$.

We begin by recalling Algorithm~\ref{algMum1} from Section~\ref{secIntro}. Recall also that $q_0/(1-q_0)= \lam e^{-\lam^2n/2}$.

\begin{algorithm}[ht]
  \caption{The distribution $\mu_{m,1}$}
\begin{enumerate}[leftmargin=*]
\item Choose a random partition $(A,B)$ according to $\theta_\lam$. 
\item Choose  defect edges $S \subseteq \binom{A}{2}$, $T\subseteq \binom{B}{2}$ according to independent realizations of $G(A,q_0)$ and $G(B,q_0)$ respectively. If $S\cup T$ contains a triangle or if $|S| + |T| >m$, output an arbitrary graph $G_0 \in \cT(n,m)$. Otherwise proceed to the next step. 
\item Choose $\Ec \subseteq A\times B$ as a uniformly random independent set of size $m- |S| -|T|$ from the graph $S \boxempty T$.
\item Output $S \cup T \cup \Ec$.
\end{enumerate}
\end{algorithm} 

Recall that Theorem~\ref{thmGnmSubRestated} states that for $m \ge \frac{1+\eps}{4} n^{3/2} \sqrt{\log n}$,  the distribution $\mu_{m,1}$ is at total variation distance  $o(1)$ to the uniform distribution on $\cT(n,m)$. This is a convenient restatement of Theorem~\ref{thmSubcriticalfixedEdgeDist}.
Similarly, it will be convenient to restate Theorem~\ref{thmP2ERG} algorithmically. To this end, we define the measure $\mu_{m,2}$ below. Recall first the definitions of $q_2, \psi$ from~~\eqref{eqq2Def} and\eqref{eqPsiDef}.

\begin{algorithm}[ht]
  \caption{The distribution $\mu_{m,2}$ \label{algMum2}}
\begin{enumerate}[leftmargin=*]
\item Choose a random partition $(A,B)$ according to $\theta_\lam$. 
\item Choose  defect edges $S \subseteq \binom{A}{2}$, $T\subseteq \binom{B}{2}$ according to independent realizations of $G(A,q_2, \psi)$ and $G(B,q_2,\psi)$ respectively. 
\item Choose $\Ec \subseteq A\times B$ as a uniformly random independent set of size $m- |S| -|T|$ from the graph $S \boxempty T$.
\item Output $S \cup T \cup \Ec$.
\end{enumerate}
\end{algorithm} 

We then have the following reformulation of Theorem~\ref{thmP2ERG}. Recall that $\mu_m$ denotes the uniform distribution on $\cT(n,m)$.
\begin{theorem}
\label{thmGnmSuperRestated}
If  $m \ge  \frac{13}{56} n^{3/2} \sqrt{\log n}$, then
\[
\|\mu_m - \mu_{m,2}\|_{TV}=o(1)\, .
\]
\end{theorem}
Our first step towards proving Theorem~\ref{thmGnmSubRestated} and Theorem~\ref{thmGnmSuperRestated} is to  approximate $\mu_m$ by the intermediate measure $\mu_{\textup{strong},m}$ (the analogue of $\mu_{\textup{strong},\lam}$ from Algorithm~\ref{algMuModStrong}) defined below. 
 
\begin{algorithm}[h]
  \caption{The distribution $\mu_{\textup{strong},m}$ \label{algMuStrongm}}
\begin{enumerate}[leftmargin=*]
\item Choose $(A,B)\in \Pi_{\textup{strong}}$ with probability proportional to $Z_{A,B}(\lam)$.
\item Choose $(S,T)\in \cD_{A,B,\lam}$ from the distribution $\nu_{A,B,\lam}$.
\item Choose $\Ec \subseteq A\times B$ as a uniformly random independent set of size $m- |S| -|T|$ from the graph $S \boxempty T$.
\item Output $S \cup T \cup \Ec$.
\end{enumerate}
\end{algorithm}

Our main goal of this section is to prove the following analogue of Corollary~\ref{corZmodsimsum}. Define 
\[
\cL(n,m) = \{G\in \cL(n,\lam): |G|=m\}\, .
\]
\begin{theorem}\label{thmstatesum}
Let $m \ge \frac{13}{56} n^{3/2} \sqrt{\log n}$. Then
\begin{align}\label{eqTnmSimZ}
|\cT(n,m)|\sim|\cL(n,m)|\sim \frac{1}{\lam^m n \sqrt{\pi\lam/2}} \cdot Z(\lam) \, .
\end{align}
Moreover, 
\begin{align}\label{eqFixedEdgeTV}
\| \mu_{m}-\mu_{\textup{strong},m}\|_{TV}=o(1)\, .
\end{align}
\end{theorem}

We will then show that $\mu_{m,1}, \mu_{m,2}$ are close to $\mu_{\textup{strong},m}$ in the relevant ranges of $m$.

\subsection{Proof of Theorem~\ref{thmstatesum}}
Recall that, given a partition $(A,B)$ of $[n]$, we can describe the measure $\mu_{A,B,\lam}$ (defined at~\ref{eqmuABdef}) via the following process:

\begin{algorithm}[h]
  \caption{Alternative description of $\mu_{A,B,\lam}$ \label{algMuAB}}
    \begin{enumerate}[leftmargin=*]
\item Choose $(S,T)\in \cD_{A,B,\lam}$ according to $\nu_{A,B,\lam}$.
\item Choose $\Ec \subseteq A\times B$ according to the hard-core measure on $S\boxempty T$ at activity $\lam$.
\end{enumerate}
\end{algorithm}
The proof of Theorem~\ref{thmstatesum} will have two main steps. The first is showing that a sample from the measure $\mu_{A,B,\lam}$ has exactly $m$ edges with good probability; this is done by using the specific choice of $\lam$ at~\eqref{eqpdefagain} and  by showing that the variance of the number of defect edges chosen at Step 1 of Algorithm~\ref{algMuAB} is small compared to the variance of the number of crossing edges selected at Step 2.  The second step is showing that we do not overcount graphs: a typical sample from $\mu_{\textup{strong},m}$ is captured by  a single partition $(A,B)$.

The following lemma elucidates  the choice of $\lam$ in~\eqref{eqpdefagain}.
\begin{lemma}\label{lemlamshift}
Let $m \ge \frac{13}{56} n^{3/2} \sqrt{\log n}$. 
Let $(A,B)\in \Pi_{\textup{strong}}$ and let $\bG\sim \mu_{A,B,\lam}$.
Then
\[
\big |\E  |\mathbf G|-m \big|=o(\sqrt{m})\, .
\]
\end{lemma}
\begin{proof}
Let $(\mathbf S, \mathbf T)\sim\nu_{A,B,\lam}$ be the  set of defect edges chosen at Step 1 of Algorithm~\ref{algMuAB}. By Corollary~\ref{corfixedconfig}
\[
\E |\mathbf S|= \binom{a}{2}q_A(1+O(n\Delta\lam^3)) = \binom{a}{2}q_A + O(n^3q\Delta\lam^3)\, ,
\]
and similarly $\E |\mathbf T|= \binom{b}{2}q_B + O(n^2\Delta^2\lam^3)$.
Let $\Ec$ denote the set of crossing edges chosen at Step 2 of Algorithm~\ref{algMuAB}.
By Corollary~\ref{corclustersimple}
\begin{align}\label{eqcrossingexp}
\E|\Ec|= \frac{\lam}{1+\lam}ab-2(b\E|\mathbf S|+ a\E|\mathbf T|)\lam^2 + O(n^2\Delta^2\lam^3)\, .
\end{align}
Let $a=n/2-k$ and $b=n/2+k$, where $k\leq 5(n\log n)^{1/4}$ since $(A,B)$ is strongly balanced. We then have
\begin{align*}
\E|\mathbf G|& =
\E |\mathbf S|+\E |\mathbf T|+\E|\Ec|\\
&=
  \binom{a}{2}q_A(1-2b\lam^2) + \binom{b}{2}q_B(1-2a\lam^2) + \frac{\lam}{1+\lam}ab + O(n^3\Delta^2\lam^5)\\
  &=
 (1-n\lam^2)\left[ \frac{1}{2}a^2q_A +  \frac{1}{2}b^2q_B \right] + \frac{\lam}{1+\lam}\frac{n^2}{4} + O(n^3\Delta^2\lam^5 + n^2qk\lam^2 + k^2\lam)\\
 &=
  (1-n\lam^2)\left[ \frac{1}{2}a^2q_A +  \frac{1}{2}b^2q_B \right] + \frac{\lam}{1+\lam}\frac{n^2}{4} + o(\sqrt{m})\, .
 \end{align*}
For the final equality we used that $k^2\lam=o(\sqrt{m})$ and $n^2qk\lam^2=o(\sqrt{m})$ since $q=o(n^{-13/14})$. 
As in~\eqref{eqa2b2cancel} we have
\[
a^2q_A+b^2q_B= \lam e^{-\lam^2n/2}(n/2)^2\left[2+O(\lam^4k^2)\right]=\lam e^{-\lam^2n/2}n^2/2 +o(\sqrt{m}/(n\lam^2-1)) \, .
\]
It follows that
\begin{align*}
\E |\mathbf G|& = (\lam-\lam^2)\frac{n^2}{4} -(n\lam^2-1)\lam e^{-\lam^2n/2}\frac{n^2}{4}+o(\sqrt{m})\, .
\end{align*}
It therefore suffices to show that 
\[
 (\lam-\lam^2) -(n\lam^2-1)\lam e^{-\lam^2n/2}= \frac{4m}{n^2}+ o(\sqrt{m}/n^2)\, .
\]
Let 
\[
\delta= \lam_0+(n\lam_0^2-1)e^{-\lam^2_0n/2}
\]
so that $\lam=\lam_0(1+\delta)$.
Our task is then to show that
\[
\lam_0\delta-\lam_0^2(1+\delta)^2-(n\lam_0^2(1+\delta)^2-1)\lam_0(1+\delta)e^{-\lam^2_0(1+\delta)^2n/2}= o(\sqrt{m}/n^2)\, .
\]
Since $\delta= o(n^{-1/2+1/14})$, the LHS is equal to  
\[
\lam_0\delta-\lam_0^2-(n\lam_0^2-1)\lam_0e^{-\lam^2_0n/2} + o(\sqrt{m}/n^2)
\]
which by the definition of $\delta$ is equal to $o(\sqrt{m}/n^2)$ as desired. 
\end{proof}
Using the local CLT for the low-density hard-core model (Proposition~\ref{thmhcLCLT}) we prove the following. 
\begin{lemma}\label{lemLCLTcross}
Let $m \ge \frac{13}{56} n^{3/2} \sqrt{\log n}$ and 
let  $(A,B)\in \Pi_{\textup{strong}}$.
Given $S\subseteq\binom{A}{2}, T\subseteq\binom{B}{2} $, let $\Ec=\Ec(S,T)\subseteq A\times B$ denote a random sample from the hard-core measure on $ S\boxempty  T$ at activity $\lam$. Then
\begin{align}\label{eqLCLTcross}
 \P[|\Ec(S,T)|=m-|S|-|T|]\sim\frac{1}{n\sqrt{\pi\lam/2}} \, 
\end{align}
for $\nu_{A,B,\lam}$-almost all $(S,T)$.
Moreover 
\begin{align}\label{eqcrossuniform}
\P\left[|\Ec(S,T)|= m-|S|-|T| \right]= O\left(\frac{1}{n\sqrt{\lam}}\right)
\end{align}
uniformly over all $(S,T)\in \cD_{A,B,\lam}$. 
\end{lemma}
\begin{proof}
We first note that for $(S,T)\in \cD_{A,B,\lam}$, by Corollary~\ref{corclustersimple},
\begin{align}\label{eqcrossingexpfixed}
\E|\Ec(S,T)| &= \frac{\lam}{1+\lam}ab-2(b |S|+ a |T|)\lam^2 + O(n^2\Delta^2\lam^3)\, ,
\end{align}
and 
\begin{align}\label{eqGcvarest}
\var |\Ec(S,T)| \sim ab \lam \sim \lam n^2/4\sim m\, .
\end{align}
Statement~\eqref{eqcrossuniform} follows immediately from~\eqref{eqGcvarest} and Proposition~\ref{thmhcLCLT}.

Now suppose $(\mathbf S, \mathbf T)\sim \nu_{A,B,\lam}$. By Chebyshev's inequality and Corollary~\ref{coredgevar}, noting that $q=o(n^{-13/14})$, there exists $\eps>0$ such that 
\[
|S| + |T| = \E|\mathbf S| + \E|\mathbf T| + O(n^{3/4-\eps})\, ,
\]
for $\nu_{A,B,\lam}$-almost all $(S,T)$.
Combining this with~\eqref{eqcrossingexp} and~\eqref{eqcrossingexpfixed} we conclude that
\[
\E|\Ec( S,  T)| = \E| \Ec(\mathbf S,  \mathbf T)| + o(\sqrt{m})
\]
for $\nu_{A,B,\lam}$-almost all $(S,T)$.
 By Lemma~\ref{lemlamshift} we then have 
\[
|S| + |T| +  \E| \Ec( S,  T)| = \E|\mathbf S| + \E|\mathbf T|  + \E | \Ec(\mathbf S,\mathbf T)| +o(\sqrt{m})= m + o(\sqrt{m})\, ,
\]
for $\nu_{A,B,\lam}$-almost all $(S,T)$.

Statement~\eqref{eqLCLTcross} now follows from~\eqref{eqGcvarest} and Proposition~\ref{thmhcLCLT}. \end{proof}

We record the following analogue of Lemma~\ref{lemCaptureMod}. We recall that for a graph $G$, $c_{\textup{strong},\lam}(G)$ denotes the number of strongly balanced partitions $(A,B)$ that capture $G$ that is, $(G_A, G_B)\in \cD_{A,B,\lam}$.
\begin{lemma}\label{lemCaptureModfixed}
Let $m \ge \frac{13}{56} n^{3/2} \sqrt{\log n}$ and
let $\bG\sim \mu_{\textup{strong},m}$. We have,
\[
\P(c_{\textup{strong},\lam}(\bG)= 1)= 1-o(1)\, .
\]
Moreover $\bG$ is an $(A,B)$-$\lam$-expander whp where $(A,B)$ is the partition chosen at Step 1 of Algorithm~\ref{algMuStrongm}. In particular, $(A,B)$ is the unique max cut of $\bG$ whp.
\end{lemma}
\begin{proof}
Suppose that $(A,B)$ is chosen at Step 1 of Algorithm~\ref{algMuStrongm} and $(S,T)\in \cD_{A,B,\lam}$ is chosen at Step 2. Let $\Ec'\subseteq A\times B$ be a sample from the hard-core measure on $S\boxempty T$ at activity $\lam$.

By Lemma~\ref{lemExpanderwhp}, 
\[
\P(([n], \Ec') \textup{ is not an $(A,B)$-$\lam$-expander} )\leq e^{-\lam n/25}\, .
\]
Let $\Ec\subseteq A\times B$ be the set chosen at Step 3 in Algorithm~\ref{algMuStrongm}. Then 
\begin{align}
\P(([n], \Ec) &\textup{ is not an $(A,B)$-$\lam$-expander} )
 \\&= 
\P\left(([n], \Ec') \textup{ is not an $(A,B)$-$\lam$-expander} \Big | |\Ec'|=m-|S|-|T|\right)\\
&\leq \frac{e^{-\lam n/25}}{\P( |\Ec'|=m-|S|-|T|)}\, .
\end{align}
By Lemma~\ref{lemLCLTcross} we have 
\begin{align*}
 \P[|\Ec'|=m-|S|-|T|]\sim\frac{1}{n\sqrt{\pi\lam/2}} 
\end{align*}
whp over the choice of $(S,T)$ in Step 2 in Algorithm~\ref{algMuStrongm}.
We conclude that 
\begin{align}
\P(([n], \Ec) \textup{ is not an $(A,B)$-$\lam$-expander} ) =o(1)\, ,
\end{align}
whp over the choice of $(S,T)$.
Thus $\bG$ is an $(A,B)$-$\lam$-expander whp (wrt $ \mu_{\textup{strong},m}$).
The result follows from Lemma~\ref{lemExpansionCor}.
\end{proof}

\begin{lemma}\label{lemhatmuGapprox}
Let $m \ge \frac{13}{56} n^{3/2} \sqrt{\log n}$.
 Then for $\mu_{\textup{strong},m}$-almost all $G\in \cT(n,m)$,
\begin{align}\label{eqhatmumodaa}
\mu_{\textup{strong},m}(G)\sim \frac{\lam^m}{Z(\lam)} n\sqrt{\pi\lam/2}\, .
\end{align}
Moreover if $G\in \cT(n,m)$ is such that $\mu_{\textup{strong},m}(G)>0$, then
\begin{align}\label{eqhatmumodlb}
\mu_{\textup{strong},m}(G)=  \frac{\lam^m}{Z(\lam)}\cdot \Omega\left(n\sqrt{\lam}\right)\, .
\end{align}
\end{lemma}
\begin{proof}
Suppose that $(A,B)$ is strongly balanced and captures $G$. Given that $(A,B)$ is selected at Step 1 of Algorithm~\ref{algMuStrongm}, the probability that we output $G$ is
\[
P_{A,B}:=\frac{\lam^{|G_A|+|G_B|}Z_{G_A\boxempty G_B}(\lam)}{Z_{A,B}(\lam)} \cdot \frac{1}{i_{m-|G_A|-|G_B|}(G_A\boxempty G_B)}\, ,
\]
where we use $i_k(H)$ to denote the number of independent sets of size $k$ in a graph $H$. 
Let $\Ec$ denote a random sample from the hard-core model on the graph $G_A\boxempty G_B$ at activity $\lam$. Then 
\[
\P(|\Ec| = m-|G_A|-|G_B| )= \frac{\lam^{m-|G_A|-|G_B| }i_{m-|G_A|-|G_B|}(G_A\boxempty G_B)}{Z_{G_A\boxempty G_B}(\lam)}\, ,
\] 
so that 
\[
P_{A,B}:=\frac{\lam^m}{Z_{A,B}(\lam)} \cdot \frac{1}{\P(|\Ec| = m-|G_A|-|G_B| )}\, .
\]
Letting $C_{\textup{strong},\lam}(G)$ denote the set of strongly balanced $(A,B)$ that capture $G$, we then have
\[
\mu_{\textup{strong},m}(G)= \frac{\lam^m}{Z_{\textup{strong}}(\lam)}\sum_{(A,B)\in C_{\textup{strong},\lam}(G)} \frac{1}{\P(|\Ec|=m-|G_A|-|G_B|)} \, .
\]
Lemma~\ref{lemCaptureModfixed} tells us that $c_{\textup{strong},\lam}(G)=|C_{\textup{strong},\lam}(G)|=1$ for $\mu_{\textup{strong},m}$-almost all $G$. 
Statement~\eqref{eqhatmumodaa} now follows from~\eqref{eqLCLTcross} and Corollary~\ref{corZmodsimsum} which states that $Z(\lam)\sim Z_{\textup{strong}}(\lam)$.
Statement~\eqref{eqhatmumodlb} follows from~\eqref{eqcrossuniform} and Corollary~\ref{corZmodsimsum}.
\end{proof}

For $(A,B)\in\Pi_{\textup{strong}}$, define 
\begin{align*}
\cT_{A,B}(n,m):=\left\{G\in \cT(n,m) : (G_A, G_B)\in \cD_{A,B,\lam} \right\}\, ,
\end{align*}
where $\lam$ is as in~\eqref{eqpdefagain}.
\begin{theorem}\label{thmABlclt}
Let $m \ge \frac{13}{56} n^{3/2} \sqrt{\log n}$ and let $(A,B)\in\Pi_{\textup{strong}}$.
Then
\[
|\cT_{A,B}(n,m)|\sim\frac{Z_{A,B}(\lam)}{\lam^m n \sqrt{\pi\lam/2}}\, .
\]
\end{theorem}
\begin{proof}
Let $\bG\sim\mu_{A,B,\lam}$.
We have 
\[
|\cT_{A,B}(n,m)|= \frac{Z_{A,B}(\lam)}{\lam^m}\pr(|\mathbf G|=m)\, .
\]
It therefore suffices to show that 
\begin{align}\label{eqLCLTGm}
\pr(|\mathbf G|=m)\sim\frac{1}{n\sqrt{\pi\lam/2}}\, .
\end{align}
Let $(\mathbf S, \mathbf T)\sim\nu_{A,B,\lam}$ denote the random sets of edges selected at Step 1 in Algorithm~\ref{algMuAB} and let $\Ec$ denote the random set of edges selected at Step 2. Then
\begin{align}
\P(|\mathbf G|=m)
&=
\sum_{(S,T)\in \cD_{A,B,\lam}}\P((\mathbf S, \mathbf T)=(S,T))\cdot 
\P\left[|\Ec|= m-|S|-|T| \right]\label{eqprobhitm}
\end{align}
Statement~\eqref{eqLCLTGm} now follows from Lemma~\ref{lemLCLTcross}.
\end{proof}
We are now in a position to prove Theorem~\ref{thmstatesum}.
\begin{proof}[Proof of Theorem~\ref{thmstatesum}]
We begin by showing that $|\cT(n,m)|\sim|\cL(n,m)|$. By definition, $\cL(n,m)$ is the set of $G\in \cT(n,m)$ such that $G$ admits a weakly balanced, dominating cut (see Definition~\ref{def:dominating}) of size $\geq m - 2\delta \lambda n^2$. Since $\lam \geq \lam_0=4m/n^2$, we have $m - 2\delta \lambda n^2\leq (1-\delta)m$. The fact that $|\cT(n,m)|\sim|\cL(n,m)|$ now follows from Theorem~\ref{thm:luczakrefined} which asserts that almost all $G\in\cT(n,m)$ admit a weakly balanced, dominating cut of size at least $(1-\delta) m$.

Our next goal is to show that 
\begin{align}\label{eqTnmSimSum}
|\cL(n,m)|\sim \sum_{(A,B)\in \Pi_{\textup{strong}}} |\cT_{A,B}(n,m)|\, .
\end{align}
Statement~\eqref{eqTnmSimZ} will then follow from Theorem~\ref{thmABlclt} and Corollary~\ref{corZmodsimsum}.

Let  $\cL=\cL(n,\lam)$. Then
\[
\mu_{\cL,\lam}(|G|=m)= \frac{\lam^m|\cL(n,m)|}{Z(\cL,\lam)}\, .
\]
On the other hand by~\eqref{eqmodcapture},
\begin{align*}
\mu_{\textup{strong},\lam}(|G|=m)
&= \frac{\lam^m}{Z_{\textup{strong}}(\lam)} \cdot \sum_{G\in \cT(n,m)}c_{\textup{strong},\lam}(G)\\
& =   \frac{\lam^m}{Z_{\textup{strong}}(\lam)} \cdot  \sum_{(A,B)\in \Pi_{\textup{strong}}} |\cT_{A,B}(n,m)|\, .
\end{align*}
We note also that by~\eqref{eqLCLTGm}
\begin{align}\label{eqeqLCLThatmu}
\mu_{\textup{strong},\lam}(|G|=m)\sim \frac{1}{n\sqrt{\pi\lam/2}}\, .
\end{align}

By Propositions~\ref{propGroundStateRefinedAlt},~\ref{lemZmodsimsum} and~\ref{propGroundStateStrong}, 
\[
\left|\mu_{\cL,\lam}(|G|=m)-\mu_{\textup{strong},\lam}(|G|=m)\right|= O\left(n^{-3/2}\right)\,
\]
and so
\[
\left|\frac{Z_{\textup{strong}}(\lam)}{Z(\cL,\lam)}\cdot\frac{|\cL(n,m)|}{ \sum_{(A,B)\in \Pi_{\textup{strong}}} |\cT_{A,B}(n,m)|}-  1\right|= \frac{1}{\mu_{\textup{strong},\lam}(|G|=m)}\cdot O\left(n^{-3/2}\right)=o(1)\, ,
\]
where for the last equality we used~\eqref{eqeqLCLThatmu}.
Statement~\eqref{eqTnmSimSum} follows since $Z_{\textup{strong}}(\lam)\sim Z(\cL,\lam)$ (again by Propositions~\ref{propGroundStateRefinedAlt},~\ref{lemZmodsimsum} and~\ref{propGroundStateStrong}). We now turn our attention to~\eqref{eqFixedEdgeTV}. 
By~\eqref{eqTnmSimZ}  
\begin{align*}
\| \mu_{m}-\mu_{\textup{strong},m}\|_{TV}
&=
\sum_{G : \mu_{\textup{strong},m}(G)> \mu_{m}(G)} \mu_{\textup{strong},m}(G) -  \mu_{m}(G)\\
&=
\sum_{G : \mu_{\textup{strong},m}(G)> \mu_{m}(G)} \mu_{\textup{strong},m}(G)\left(1 - \frac{1}{|\cT(n,m)| \cdot \mu_{\textup{strong},m}(G)} \right)\\
&=
\sum_{G : \mu_{\textup{strong},m}(G)> \mu_{m}(G)} \mu_{\textup{strong},m}(G)\left(1 - \frac{(1+o(1))\lam^m n \sqrt{\pi\lam/2}}{Z(\lam)\cdot \mu_{\textup{strong},m}(G)} \right)\, .
\end{align*}
Statement~\eqref{eqFixedEdgeTV} now follows from  Lemma~\ref{lemhatmuGapprox}. 
\end{proof}

\subsection{Proof of theorems from Section~\ref{secIntro}}

Theorems~\ref{thmSubcriticalfixedEdgeDist}~\ref{thmGnmSubRestated},~\ref{thmNumberM},~\ref{thmP2ERG} and~\ref{thmNumberMsuper} now follow easily.

\begin{proof}[Proof of Theorems~\ref{thmNumberM} and~\ref{thmNumberMsuper}]
First fix $\eps>0$ and let $m \ge \frac{1+\eps}{4} n^{3/2} \sqrt{\log n}$. Let $\lam=\lam(m)$ be as in~\eqref{eqpdefagain}.  Since $
\lam\geq (1+\eps)\sqrt{{\log n}/{n}}$,  Theorem~\ref{thmNumberM} follows from~\eqref{eqTnmSimZ} and Theorem~\ref{lemZapproxER}. Next let $m \ge \frac{13}{56} n^{3/2} \sqrt{\log n}$ and let $\lam=\lam(m)$ be as in~\eqref{eqpdefagain}.  Since $
\lam\geq \frac{13}{14}\sqrt{{\log n}/{n}}$,  Theorem~\ref{thmNumberMsuper} follows from~\eqref{eqTnmSimZ} and Lemma~\ref{lemSuperMainLem}.
\end{proof}

\begin{proof}[Proof of Theorems~\ref{thmSubcriticalfixedEdgeDist},~\ref{thmGnmSubRestated} and~\ref{thmP2ERG}]
First we prove Theorem~\ref{thmGnmSubRestated} (and therefore also Theorem~\ref{thmSubcriticalfixedEdgeDist}). By Theorem~\ref{thmstatesum} it suffices to show that 
\begin{align}\label{eqhatmutostrong}
\|\mu_{m,1} - \mu_{\textup{strong},m}\|_{TV}=o(1)\, .
\end{align}
Let $\bm \pi_0, \bm \pi_1$ denote the partitions selected at Step 1 in Algorithms~\ref{algMum1} and~\ref{algMuStrongm} respectively.
 Given $\pi\in \Pi$, let $ \mu^\pi_{\textup{strong},m},  \mu^\pi_{m,1}$ denote the measures $\mu_{\textup{strong},m}, \mu_{m,1}$
conditioned on the events $\bm \pi_0=\pi, \bm \pi_1=\pi$ respectively. 
By the proof of Claim~\ref{claimTVconditioned}, if $\pi\sim\bm\pi_0$, then
\begin{align}\label{eqclaimTVagain}
\|\mu_{m,1} - \mu_{\textup{strong},m}\|_{TV} \leq \E_\pi \| \mu^\pi_{m,1} - \mu^\pi_{\textup{strong},m}\|_{TV} + \|\bm\pi_0-\bm\pi_1\|_{TV}\, .
\end{align}
The proof that $\|\bm\pi_0-\bm\pi_1\|_{TV}=o(1)$ follows from~\eqref{eqPpi0}, \eqref{eqPpi1} which hold equally well in this context. 
Let $\nu'_{A,B,\lam}$ denote the measure associated to the random graph in Step 2 of Algorithm~\ref{algMum1}, i.e., the union of two independent samples from $G(A,q_0), G(B,q_0)$ where we output the empty graph if the graph contains a triangle or has more than $m$ edges. 
Note that if $\pi$ is strongly balanced, then
\begin{align}\label{eqhatupistrong}
\| \mu^\pi_{m,1} - \mu^\pi_{\textup{strong},m}\|_{TV} \leq \|  \nu_{A,B,\lam}- \nu'_{A,B,\lam}  \|_{TV} \, .
\end{align}
By Lemma~\ref{lemZABapprox1},  $\| \nu_{A,B,\lam} - \nu_{\bm q} \|_{TV}  = o(1) $ and
\[
\|\nu_{\bm q} -  \nu'_{A,B,\lam} \|_{TV} = \frac{1}{2} \nu_{\bm q}(\{G: |G|>m \text{ or $G$ contains a triangle}\})=o(1)
\]
by Markov's inequality and a union bound. We conclude that $\|\nu_{A,B,\lam}- \nu'_{A,B,\lam}  \|_{TV}=o(1) $ and so~\eqref{eqhatupistrong} and~\eqref{eqclaimTVagain} give~\eqref{eqhatmutostrong}. This concludes the proof of Theorem~\ref{thmGnmSubRestated} and hence also Theorem~\ref{thmSubcriticalfixedEdgeDist}.  The proof of Theorem~\ref{thmP2ERG} follows the same lines.
 \end{proof}
Finally, we prove the structural results of Theorems~\ref{thmDistToBipartite}, \ref{thm3color}, \ref{thm4color}, \ref{thmGiantComponent}.  Any structural result that just involves defect edges will follow immediately from the corresponding result in Section~\ref{secERresults} on $G(n,p)$ since Theorem~\ref{thmstatesum}  shows that for the choice of $\lam=\lam(m)$ the distribution of defect edges in $\mu_\lam$ and $\mu_m$ coincide up to $o(1)$ total variation distance.  For structural results involving crossing edges, we note that the relevant property of crossing edges, namely the expansion property captured in Lemma~\ref{lemCaptureModfixed},  holds whp in $\mu_m$ as well.
\begin{proof}[Proof of Theorems~\ref{thmDistToBipartite}, ~\ref{thm3color}] 
The proof of Theorem~\ref{thmDistToBipartite} is the same as the proof of Theorem~\ref{thmGnpBipartite} where now we apply Lemma~\ref{lemCaptureModfixed} in place of Lemmas~\ref{lemExpanderwhp} and~\ref{lemExpansionCor}.

The proof of Theorem~\ref{thm3color} is the same as the proof of Theorem~\ref{thmGnpColoring} where we apply Lemma~\ref{lemCaptureModfixed} in place of  Lemma~\ref{lemExpanderwhp}.
\end{proof}

\begin{proof}[Proof of Theorems~\ref{thm4color} and \ref{thmGiantComponent}]
These theorems only concern properties of the defect edges and so follow from Theorem~\ref{thm4colorGnp} and Theorem~\ref{thmGnpGiant}.
\end{proof}

\section{The first approximation}
\label{secMorrisOPT}

In this section we prove Proposition~\ref{propGroundStateRefinedAlt}. The proof will follow a modification of the strategy of~\cite{balogh2016typical} specialized to triangle-free graphs. 

Recall that we call a partition $(A,B)$ of the set $[n]$ \emph{weakly balanced} if $\big||A|-|B| \big|\leq n/10$. Moreover, we call a cut $(A,B)$ of a graph $G$ \emph{dominating} if 
\[
d_G(v,B)\geq d_G(v,A) \text{ for all } v\in A \,,
\]
and similarly with $A,B$ swapped.

Recall that $\alpha=1/(96e^3)$. Before we proceed, we fix some constants that obey the following chain of dependencies:
\begin{align}\label{eq:ChainDep}
\frac{1}{\omega}\ll\frac{1}{C}\ll \delta \ll \tau \ll \theta\ll \beta \ll \alpha \, .
\end{align}
Here we use the $\ll$ notation informally.
For concreteness, we note that the following choices suffice: $\beta$ satisfies $\beta\log(e/\beta)= \alpha/11$ and $\theta = e^{-100/\beta}$, $\tau = \left(\theta/10\right)^8$, $60\delta \log(e/\delta)= \tau$.
We then choose $C=C(\delta)$ as in Theorem~\ref{thm:luczakrefined}. Finally we pick $\omega=\max\{\sqrt{\alpha/(4\delta)}, 20\beta^{-2/3}, 50C\}$.
Throughout this section we assume 
\[
\lam\geq \frac{\omega}{\sqrt{n}}\, ,
\]
and set $\cL=\cL(n,\lam)$ (as defined at~\eqref{eqcLdef}).
We begin with a proof of Proposition~\ref{lemZerothApprox} which states that $
Z(\lam)\sim Z(\cL, \lam),
$
and $
\|\mu_{\lam} - \mu_{\cL,\lam}\|=o(1)
$.

\begin{proof}[Proof of Proposition~\ref{lemZerothApprox}]
We first aim to show that the dominant contribution to  $Z(\lam)$ comes from graphs $G\in\cT$ such that $|G|$ is within a constant factor of $\lam n^2$. We will then use Theorem~\ref{thm:luczakrefined} to conclude that almost all of these graphs belong to $\mathcal{L}$.

We begin by noting the crude estimate 
\begin{equation}\label{eqn:crudelb}
 Z(\mathcal{L}, \lam) \ge (1 + \lambda)^{\lfloor n^2/4\rfloor} \geq \exp\left\{\frac{\lambda n^2}{8}\right\}
\end{equation}
obtained by counting just the bipartite graphs with a fixed bipartition $(A,B)$ such that $|A| = \lceil\frac{n}{2}\rceil$ and $|B| = \lfloor\frac{n}{2}\rfloor$.

With this estimate in hand, let us account for the weight of all graphs having fewer than $\lambda\binom{n}{2}/15$ edges. Letting $\cT_1=\{G\in \mathcal{T} : |G| \leq \lambda\binom{n}{2}/15\}$, we have
\begin{align*}
Z(\cT_1, \lam)
  \leq \sum_{j=0}^{\lambda\binom{n}{2}/15} \binom{\binom{n}{2}}{j}\lam^j
\leq \sum_{j=0}^{\lambda\binom{n}{2}/15} \left(e\binom{n}{2}\lam/j \right)^j\, .
\end{align*}
The largest term in the above sum is at $j=\lambda\binom{n}{2}/15$ and so the RHS is at most
\begin{align*}
n^2 \cdot \left(15e\right)^{\lambda\binom{n}{2}/15} 
\leq \exp\{0.124 \cdot \lambda n^2\}\, .
\end{align*}
Using~\eqref{eqn:crudelb}, we get
\[
Z(\cT_1, \lam) \leq \exp\left\{-0.01\lambda n^2\right\} \cdot Z(\mathcal{L}, \lam).
\] 
Now let us estimate the weight from all graphs with many edges. Let $\cT_2=\{G\in \mathcal{T} : |G| \geq 3\lambda\binom{n}{2}\}$. We have
\begin{align*}
Z(\cT_2, \lam)
  \leq \sum_{j\geq 3 \lambda\binom{n}{2}}\binom{\binom{n}{2}}{j}\lam^j
\leq \sum_{j\geq 3 \lambda\binom{n}{2}} \left(e\binom{n}{2}\lam/j \right)^j\, .
\end{align*}
The largest term in the above sum is at $j=3\lambda\binom{n}{2}$ and so the RHS is at most
\begin{align*}
n^2 \cdot \left(e/3\right)^{3\lambda\binom{n}{2}}=o(1)\, .
\end{align*}
Now let $\mathcal{L}(t)$ denote the set of all $G\in \cT$ with $t$ edges that admit a weakly balanced, dominating cut of size at least $(1-\delta)t$. Since $\omega\geq 50C$ we have $\lam\binom{n}{2}/15\geq Cn^{3/2}$.  By Theorem~\ref{thm:luczakrefined}, we have that if $t\geq\lam\binom{n}{2}/15$, then $|\cT(n,t)|\sim |\mathcal{L}(t)|$. Moreover, if $t\leq 3\lambda\binom{n}{2}$, then $\delta t\leq 2\delta \lam n^2$ so that $\mathcal{L}(t)\subseteq \mathcal{L}$. It follows that
\begin{align*}
Z(\lam)  = (1 + o(1)) \sum_{t = \lambda\binom{n}{2}/15}^{3\lambda\binom{n}{2}}|\cT(n,t)|\lam^t
 = (1 + o(1))\sum_{t = \lambda\binom{n}{2}/15}^{3\lambda\binom{n}{2}}|\mathcal{L}(t)|\lam^t
 \leq (1 + o(1))Z(\mathcal{L}, \lam)\, .
\end{align*}
Since also $Z(\mathcal{L}, \lam)\leq Z(\lam)$, we have $Z(\lam)\sim Z(\cL, \lam)$ as desired.
To conclude the proof note that 
\begin{align*}
\|\mu_\lam - \mu_{\mathcal{L},\lam}\|_{TV}
&=
\sum_{G :\mu_{\mathcal{L},\lam}(G)> \mu_\lam (G)}\mu_{\mathcal{L},\lam}(G)- \mu_\lam (G)
=\sum_{G\in \mathcal{L}}\frac{\lam^{|G|}}{Z(\mathcal L, \lam)}\left(1-\frac{Z(\mathcal L, \lam)}{Z( \lam)}\right)=o(1)\, .
\end{align*}
\end{proof}

Henceforth we fix a weakly balanced partition $(A,B)$ and let 
\begin{align*}
\mathcal{L}_{A,B}:= \left\{G\in \mathcal L: (A,B) \text{ is a dominating  cut of $G$ with  $\leq 2\delta \lambda n^2$ defect edges}\right\}\, .
\end{align*}
We can now  state the main step toward the proof of Proposition~\ref{propGroundStateRefinedAlt}. Recall that 
\(
Z_{A,B}^{\textup{w}}(\lam)= Z( \cT_{A,B,\lam}^{\textup{w}}, \lam).
\)

\begin{lemma}\label{lemOPT}
\begin{align*}
Z(\mathcal{L}_{A,B}, \lam)= \left(1+O\left(e^{-\sqrt{n}}\right)\right) Z_{A,B}^{\textup{w}}(\lam)\, .
\end{align*}
\end{lemma}
We prove the lemma in two parts: first we show that $Z(\mathcal{L}_{A,B}, \lam)\geq \left(1+O\left(e^{-\sqrt{n}}\right)\right) Z_{A,B}^{\textup{w}}(\lam)$, which we refer to as the `lower bound' of Lemma~\ref{lemOPT}.

\begin{proof}[Proof of the lower bound of Lemma~\ref{lemOPT}]
Let $G\sim \mu_{A,B,\lam}^{\textup{w}}$.
Suppose that $G\in \cT_{A,B,\lam}^{\textup{w}}\backslash \mathcal L_{A,B}$. Since $\Delta(G_A\cup G_B)\leq {\alpha}/{\lam}$, the number of defect edges of $G$ wrt $(A,B)$ is at most $n\alpha/(2\lam)\leq 2\delta \lambda n^2$ (since $\omega^2\geq \alpha/(4\delta)$). Since $G\notin \mathcal L_{A,B}$ we conclude that $(A,B)$ is not a dominating cut of $G$. We may therefore assume wlog that there exists $v\in A$ such that
\[
d_G(v,B)<d_G(v,A)\leq \alpha/\lam<\lam n/30\, .
\]
We conclude that $G$ is not an $(A,B)$-$\lam$-expander (see Definition~\ref{defABExpand}).
It follows from Lemma~\ref{lemExpanderwhp} that
\[
\P(G\in \cT_{A,B,\lam}^{\textup{w}}\backslash \mathcal L_{A,B})\leq e^{-\lam n /25}\, ,
\]
or in other words,
\[
Z(\cT_{A,B,\lam}^{\textup{w}}\backslash \mathcal L_{A,B}, \lam) \le e^{-\lam n /25}\cdot Z_{A,B}^{\textup{w}}(\lam)\, .
\]
The result follows. 
\end{proof}

\subsection{The upper bound of Lemma~\ref{lemOPT}}

Before we proceed let us set up some notation. Recall that
\[
\mathcal{D} = \mathcal{D}^{\textup{w}}_{A,B,\lam} =\{(G_A, G_B): G\in  \cT^{\textup{w}}_{A,B,\lam}\}\, .
\]
For $F \in \mathcal{D}$, let
\[		
 \mathcal{T}(F) = \mathcal{T}_{A,B}(F) : = \left\{G\in \mathcal{T} : G_A\cup G_B = F \right\}\, .
\]
In particular
\[
\cT_{A,B,\lam}^{\textup{w}}= \bigcup_{F\in \mathcal D} \mathcal{T}(F)\, . \footnote{Recall that we identify the pair $(G_A, G_B)$ with the graph $G_A\cup G_B$.}
\]

Let $\mathcal F \supseteq \mathcal{D}$ denote the set of all graphs $F\subseteq \binom{A}{2}\cup \binom{B}{2}$ such that $|F|\leq2\delta \lambda n^2$. Given $F\in \cF$, let 
\[
\mathcal{L}(F)=\mathcal{L}_{A,B}(F)=\left\{G\in \mathcal{L}_{A,B}:  G_A\cup G_B  =F\right\}. 
\]
Let $D=\frac{\alpha}{\lam}$.

Given a graph $F \subseteq \binom{A}{2} \cup \binom{B}{2}$,
\begin{itemize}
\item Let $U(F)$ be some (arbitrarily chosen) edge-maximal subgraph of $F$ (on the same vertex set as $F$) with maximum degree at most $D$.
\item Let $X(F)$ be the set of all $v$ whose degree in $U(F)$ is $D$.
\item Let $H(F) \subseteq X(F)$ denote the set of vertices $v$ in $X(F)$ whose degree in $F$ is at least $\beta\lambda n$.
\item Let $T(F)$ denote the graph $F$ with all edges incident to $X(F)$ removed.
\end{itemize}
We note that $\beta \lam n$ is significantly larger than $D$ and so we think of vertices in $H(F)$ as vertices of `high degree'. 
Note also that if $\{u,v\}$ is an edge of $T(F)$, then the degrees of both $u$ and $v$ in $U(F)$ are less than $D$,  otherwise one of $u,v$ would belong to $X$. By edge maximality of $U(F)$, we then have $\{u,v\}\in U(F)$ and so 
\begin{align}\label{eq:TFinUF}
T(F)\subseteq U(F)\in \mathcal{D}\, .
\end{align}
It will also be useful to note that by the above definitions we have
\begin{align}\label{eq:UTX}
|U(F)|\geq |T(F)|+|X(F)|D/2\, ,
\end{align}
and
\begin{align}\label{eq:HF}
 |F| \geq |H(F)|\cdot \beta\lam n/2\, .
\end{align}

Now, for an integer $t$, let $\mathcal{F}_t$ be the subfamily of $\mathcal{F}$ consisting of graphs with exactly $t$ edges. Let $T\in \mathcal{D}$ be a fixed graph with at most $t$ edges. For integers $x$ and $h$, let $\mathcal{F}_t(T, x,h)$ be the set of $F\in \mathcal{F}_t$  such that there exists sets $H, X\subset [n]$ with $|H|=h, |X|=x$ and $H\subseteq X$ such that:
\begin{enumerate}
\item Deleting all edges incident to $X$ from $F$ results in the graph $T$. \label{itemF-X}
\item $\textup{deg}_F(v)<\beta\lambda n$ for every $v\in X\backslash H$.
\end{enumerate}
Moreover let $\mathcal{F}'_t(T, x,h)$ be the set of $F\in \mathcal{F}_t$ such that $T(F)=T$, $|X(F)|=x$, $|H(F)|=h$. We note that $\mathcal{F}'_t(T, x,h)\subseteq \mathcal{F}_t(T, x,h)$ and that for any $F\in \cF_t$, we have $F\in \cF'_t(T(F), |X(F)|, |H(F)|)$.

\begin{lemma}\label{lemweightedTxh}
For $T\in \cD$ and non-negative integers $t, x$ and $h$,
\begin{align}\label{eqweightedTxh}
\lam^t|\mathcal{F}_t(T, x,h)|\leq \lam^{|T|} e^{x\omega^2D/10}e^{2\lambda n h}\, .
\end{align}
\end{lemma}
\begin{proof}
We prove the lemma by induction on $x$. For the base case suppose that $x=0$. Then if $F\in \mathcal{F}_t(T, x,h)$ we must have $F=T$ and $t=|T|$ by item~\eqref{itemF-X} in the definition of $\mathcal{F}_t(T, x,h)$ and so $|\mathcal{F}_t(T, x,h)|\leq1$ and~\eqref{eqweightedTxh} is easily seen to hold.

Assume now that $x\geq 1$. Given $F\in \mathcal{F}_t(T, x,h)$, we fix $X$ and $H$ as in the definition of $\mathcal{F}_t(T, x,h)$ and pick an arbitrary $v\in X$. Let $d=\text{deg}_F(v)$ and let $F'$ denote the graph $F$ with all edges incident to $v$ deleted. Note that 
$F'$ lies in $\mathcal{F}_{t-d}(T, x-1,h) \cup \mathcal{F}_{t-d}(T, x-1,h-1)$. Moreover, if $d\geq \beta \lam n$ then we must have that $v\in H$ and so $F'\in \cF_{t-d}(T, x-1,h-1)$. It follows that 
\begin{multline}\label{eqindstep}
\lam^t|\cF_{t}(T, x,h)| \\ \leq \sum_{d=0}^{\beta \lambda n} n\binom{n}{d}\lam^d\cdot \lam^{t-d}|\cF_{t-d}(T, x-1,h)| +
\sum_{d=0}^{n} n\binom{n}{d} \lam^d\cdot \lam^{t-d}|\cF_{t-d}(T, x-1,h-1)|\, .
\end{multline}
Note that
\[
\sum_{d=0}^{\beta \lambda n} n\binom{n}{d}\lam^d
 \leq n^2 \left(\frac{e}{\beta} \right)^{\beta\lambda n}
\leq n^2 e^{\beta\log(e/\beta)\lambda n}
\leq  \frac{1}{2}e^{\omega^2D/10}\, ,
\]
since $\beta\log(e/\beta)\leq \alpha /11$.
Note also that
\[
\sum_{d=0}^{n} n\binom{n}{d}\lam^d=n(1+\lam)^n \leq \frac{1}{2}e^{2n\lam}\, .
\]
The lemma now follows from~\eqref{eqindstep} and the inductive hypothesis.
\end{proof}

Recall that for a graph $F\subseteq \binom{A}{2}\cup\binom{B}{2}$, we write $F_\boxempty$ to denote $F_A\boxempty F_B$.
\begin{lemma}\label{lemZFT}
For $T\in \mathcal{D}$, non-negative integers $t, x$ and $h$ and $F\in \cF'_t(T,x,h)$
we have
\[
Z_{F_\boxempty}\leq Z_{T_\boxempty}\cdot e^{-xDn\lam^2/10}\, .
\]
\end{lemma}
\begin{proof}
Recall that we let $U=U(F)$ denote an edge-maximal subgraph of $F$ with maximum degree at most $D$. Since $U\subseteq F$ we trivially have that 
\begin{align}\label{eqZFZU}
Z_{F_\boxempty}\leq Z_{U_\boxempty}.
\end{align}
Now, since $U$ has maximum degree $D$, $U_\boxempty$ has maximum degree at most $2D$.  Moreover, $\lam\leq 1/(8eD)$ (since $\alpha= 1/(96e^3)$) and so we may apply Lemma~\ref{lemClusterTail} and cluster expand
\[
\log Z_{U_\boxempty}(\lam)= \sum_{\Gamma\in \cC(U_\boxempty)} \phi(\Gamma) \lam^{|\Gamma|} \,.
\]
Now, by the definition of $\cF'_t(T,x,h)$ we have $T=T(F)$ and so by~\eqref{eq:TFinUF} we have $T_\boxempty\subseteq U_\boxempty$. We may therefore cluster expand $\log Z_{T_\boxempty}$ similarly. Letting 
\[
\cC_0= \cC(U_\boxempty)\backslash  \cC(T_\boxempty)\]
 we deduce that 
\begin{align*}
\log Z_{U_\boxempty}- \log Z_{T\boxempty}=  \sum_{\Gamma\in \cC_0} \phi(\Gamma) \lam^{|\Gamma|} 
&= -(|U_\boxempty|- |T_\boxempty|)\lam^2 +  \sum_{\Gamma\in \cC_0: |\Gamma|\geq 3} \phi(\Gamma) \lam^{|\Gamma|}.
\end{align*}
First we note that by~\eqref{eq:UTX},
\[
|U|\geq |T|+xD/2 \,,
\]
and so 
\[
|U_\boxempty|- |T_\boxempty|\geq \frac{xD}{2}\min\{a,b\}\geq  \frac{xDn}{5}\, 
\]
since $A,B$ is weakly balanced.
Now, if $\Gamma\in \cC_0$ then $\Gamma$ 
must contain a vertex of $(u,w)\in V(U_\boxempty)=A\times B$ such that either $u\in X$ or $w\in X$. Since there are at most $xn$ such vertices  we have by Lemma~\ref{lemClusterTail}
\footnote{applied with $k=3$, $|S|=1$, $\Delta=D=\frac{\alpha}{\lam}$, noting that $\lam<\frac{1}{4e\Delta}$ since $\alpha=\frac{1}{96e^3}$.},
\[
\left| \sum_{\Gamma\in \cC_0: |\Gamma|\geq 3} \phi(\Gamma) \lam^{|\Gamma|}\right| 
\leq xn\cdot (2e)^3D^2\lam^3 \leq  xnD\lam^2/10
\]
where for the final inequality we recalled that $\alpha=D\lam= 1/(96e^3)$. 
We conclude, using~\eqref{eqZFZU}, that
\[
\log (Z_{F_\boxempty}/Z_{T_\boxempty})\leq \log (Z_{U_\boxempty}/Z_{T_\boxempty})\leq -xDn\lam^2/10\, ,
\]
completing the proof.
\end{proof}

For the upper bound of Lemma~\ref{lemOPT}, our task is to upper bound the total weight of graphs in $\mathcal L_{A,B} \backslash\cT_{A,B,\lam}^{\textup{w}}$. We follow~\cite{balogh2016typical} and separate graphs according to whether their defect graph has `many or few high-degree vertices'. More precisely, let
\[
\cF_L=\{F: F\in \cF'_t(T, x, h) \text{ for some $T\in\mathcal{D}$ and $t,x,h\geq 0$ where $h\leq \alpha x/16$ and $x\geq1$}\}\, ,
\]
\[
\cF_H=\{F: F\in \cF'_t(T, x, h) \text{ for some $T\in \mathcal{D}$ and $t,x,h\geq 0$ where $h> \alpha x/16$}\}\, ,
\]
and let
\[
\mathcal{L}_{L}=\{G: G\in \mathcal{L}(F) \text{ for some } F\in \cF_L \}\, ,
\]
\[
\mathcal{L}_{H}=\{G: G\in \mathcal{L}(F) \text{ for some } F\in \cF_H \}\, .
\]
We will see that $\mathcal L_{A,B} \backslash\cT_{A,B,\lam}^{\textup{w}}\subseteq \mathcal L_L\cup  \mathcal L_H$ and so it will suffice to upper bound $Z( \mathcal{L}_{L},\lam), Z( \mathcal{L}_{H},\lam)$ separately. We refer to this as the `low-degree case' and `high-degree case' respectively.

\subsubsection{The low-degree case.}

\begin{lemma}\label{lemlowdegree}
\[
Z( \mathcal{L}_{L},\lam) \leq e^{-2\sqrt{n}}\cdot Z_{A,B}^{\textup{w}}(\lam)\, .
\]
\end{lemma}
\begin{proof}
By the definition of $\cF_L$ we have
\begin{align*}
Z( \mathcal{L}_{L},\lam)=\sum_{F\in \cF_L}\lam^{|F|} Z_{F_\boxempty} 
&
\leq \sum_{T\in\mathcal{D}} \sum_{\substack{t,x,h: \\ x\geq 1,\,  h\leq \alpha x/16}}\sum_{F\in \cF'_t(T,x,h)}\lam^{t} Z_{F_\boxempty}\\
&
\leq \sum_{T\in\mathcal{D}}  \sum_{\substack{t,x,h: \\ x\geq 1,\,  h\leq \alpha x/16}} \sum_{F\in \cF_t(T,x,h)}\lam^{t} Z_{T_\boxempty} \cdot e^{-xDn\lam^2/3} \\
&
\leq 
\sum_{T\in\mathcal{D}}  \lam^{|T|} Z_{T_\boxempty}  \sum_{\substack{t,x,h: \\ x\geq 1,\,  h\leq \alpha x/16}} e^{xD \omega^2/10}e^{2\lambda n h}e^{-xDn\lam^2/3}\, ,
\end{align*}
where for the second inequality we used Lemma~\ref{lemZFT} and for the final inequality we used Lemma~\ref{lemweightedTxh}.
Finally we note that
\[
\sum_{\substack{t,x,h: \\ x\geq 1,\,  h\leq \alpha x/16}} e^{xD\omega^2/10}e^{2\lambda n h}e^{-xDn\lam^2/3}
\leq 
\sum_{\substack{t,x,h: \\ x\geq 1,\,  h\leq \alpha x/16}} e^{-x\alpha \omega\sqrt{n}/8}\leq e^{-2\sqrt{n}}
\]
and 
\[
\sum_{T\in\mathcal{D}}  \lam^{|T|} Z_{T_\boxempty}= Z_{A,B}^{\textup{w}}(\lam)
\]
completing the lemma.
\end{proof}

\subsubsection{The high-degree case.}
Let $G\in \mathcal{L}_H\subseteq \cL_{A,B}$, and recall that $(A,B)$ is a dominating cut for $G$. 
We begin with the following observation. Fix $F\in \cF_H$ such that $G\in \mathcal{L}(F)$.
Since $(A,B)$ is a dominating cut, if $v\in A$ satisfies $\deg_F(v)= \deg_G(v,A)\geq d$ for some $d$,  then we must have $\deg_G(v,B)\geq d$ also. 

Let
\[
\mathcal{B}_G(v)= (N_G(v)\cap A) \times (N_G(v)\cap B), \text{ \, and\,  } \mathcal{B}_G= \bigcup_{v\in V(G)}\mathcal{B}_G(v)\, .
\]
Since $G$ is triangle-free, $G$ cannot contain any edge from the `blocked' set $\mathcal{B}_G$. 

Let 
\[
D^\ast= \frac{\beta\lambda n}{2}.
\]
We borrow the following lemma from~\cite{balogh2016typical} (see~\cite[Claim 7.6]{balogh2016typical}). Since our notation is a little different, we include the short proof for completeness. 
\begin{lemma}\label{lemoutnbr}
Let $F\in  \cF'_t(T, x, h)$. There is a subset $H'(F)\subseteq H(F)$ such that $H'\subseteq C$ where $C\in\{A,B\}$, 
\[
|H'(F)|=k:=\left\lceil \frac{h}{4}\right\rceil\, ,
\]
and
\[
\deg_F(v, C\backslash H')\geq D^\ast \text{\, \,  for every\, \, } v\in H'\, .
\]
\end{lemma}
\begin{proof}
Since $F$ has $h$ vertices of degree at least $\beta \lam n$, we may assume wlog that $A$ contains at least $h/2$ such vertices. $A$ can then be partitioned into two sets $A_1, A_2$ so that $\deg_F(v, A_1)\geq \deg_F(v, A_2)$ for each $v\in A_2$ and $\deg_F(v, A_2)\geq \deg_F(v, A_1)$ for each $v\in A_1$ (e.g., by taking $(A_1, A_2)$ to be a max cut in $F[A]$). Either $A_1$ or $A_2$ contains a set $A'$ of at least $h/4$ vertices with degree at least $\beta \lam n$ in $F$. Let $H'$ be an arbitrary $k$-element subset of $A'$.
\end{proof}

Now, for each $F\in  \cF'_t(T, x, h)$, we pick an arbitrary $H'(F)$ as in the above lemma.
Next, given a graph $G\in  \mathcal{L}(F)$, for every $v\in H'(F)$, let $W_A(v), W_B(v)$ be a canonically chosen $D^\ast$-element subset of $N_G(v)\cap A, N_G(v)\cap B$ respectively. Given such choice, consider the graph $\mathcal B'_G=\mathcal B'$ defined by 
\begin{align}\label{eqBprimeDef}
\mathcal B'= \bigcup_{v\in H'} W_A(v) \times W_B(v)\, .
\end{align}
Note that $\mathcal B'\subseteq \mathcal B_G$.

First, we take care of the case when $|\mathcal{B}'| \geq \tau \cdot ab$ where $\tau$ is as in~\eqref{eq:ChainDep}. Let
\[
 \mathcal{L}^0_{H}= \left\{G\in  \mathcal{L}_{H}:|\mathcal{B}_G'|\geq \tau\cdot ab\right\}\, .
\] 

\begin{lemma}\label{lemManyCrossing}
\[
Z( \mathcal{L}^0_{H},\lam)  \leq e^{-2\sqrt{n}}\cdot Z_{A,B}^{\textup{w}}(\lam)\, .
\]
\end{lemma}

\begin{proof}
Fix $t,h\geq0$.
We first bound the contribution to the sum from those $G$ belonging to the set
\[
\mathcal Q_{t,h}=\{G\in \mathcal{L}^0_{H} : \text{ there exists } F, T, x \text{ s.t. } F\in \cF'_t(T,x,h), G\in \mathcal L(F) \}\, .
\]
 We do so by first choosing the vertices of $H'(F)$ (given by Lemma~\ref{lemoutnbr}) and then choosing the neighbourhoods $W_A(v), W_B(v)$ for each $v\in H'(F)$ in such a way that $|\cB'|\geq \tau ab$ (with $\cB'$ as defined at~\eqref{eqBprimeDef}). Note that there are at most 
\[
\left(n\binom{a}{D^*}\binom{b}{D^*}\right)^k \leq n^k\binom{n}{2D^*}^k 
\]
choices for these vertex sets where $k=\lceil h/4 \rceil$. We note that in this process,  we fix $2kD^\ast$ edges of $G$, $kD^\ast$ of which are defect edges (i.e., they belong to $G_A\cup G_B$).  The remaining $t' := t - kD^*$ defect edges can then be chosen in at most $\binom{n^2}{t'}$ possible ways. We then note that there are $ab- |\cB'|\leq (1-\tau)ab$ available edges to include from $A\times B$. 
 It follows that 
\begin{align*}
Z( \mathcal Q_{t,h}, \lam)
& \leq n^k\binom{n}{2D^*}^k \lam^{2kD^\ast}\binom{n^2}{t'}\lam^{t'}(1 + \lambda)^{(1-\tau)ab} \\
& \leq n^k\left(\frac{e n\lambda}{2D^*}\right)^{2kD^*} \left(\frac{en^2\lambda}{t'}\right)^{t'} (1 + \lambda)^{(1-\tau)ab}  \\
& \leq n^k\left(\frac{e }{\beta}\right)^{kD^\ast}   \left(\frac{en^2\lambda}{3\delta \lam n^2-kD^\ast}\right)^{3\delta \lam n^2-kD^\ast} (1 + \lambda)^{(1-\tau)ab} \, ,
\end{align*}
where for the final inequality we used that if $G\in \mathcal Q_{t,h}\subseteq \mathcal {L}$, then we must have $t\leq 2\delta \lam n^2$ by the definition of $\mathcal{L}$ so we certainly have $t'\leq 3\delta \lam n^2-kD^\ast$. On the other hand by~\eqref{eq:HF} we have $2\delta \lam n^2\geq t\geq hD^\ast\geq kD^\ast$ and so $3\delta \lam n^2-kD^\ast\geq \delta \lam n^2$. We conclude that 
\begin{align*}
Z( \mathcal Q_{t,h}, \lam)
& \leq
n^k\left(\frac{e }{\beta}\right)^{kD^\ast}   \left(\frac{e}{\delta}\right)^{3\delta \lam n^2-kD^\ast} (1 + \lambda)^{(1-\tau)ab}\\
&= (1 + \lambda)^{ab}\exp\left\{k\log n - \log(\beta/\delta)kD^\ast + 3\delta\log(e/\delta)\lam n^2-\tau ab \log(1+\lam)\right\}\\
&\leq (1 + \lambda)^{ab} \exp\left\{-\tau n^2\lam/20 \right\}
\end{align*}
where for the final inequality we used that $\delta<\beta/4$, $D^\ast \geq \log n$, $ab\geq n^2/5$, $\log(1+\lam)\geq \lam/2$ and $\tau\geq 60\delta\log(e/\delta)$. 

Summing over all possible values of $t$ and $h$ we conclude that 
\[
Z( \mathcal{L}^0_{H},\lam) \leq n^3 (1 + \lambda)^{ab} \exp\left\{-\tau n^2\lam/20 \right\}\, .
\]
The proof is completed by observing that $Z_{A,B}^{\textup{w}}(\lam) \geq (1 + \lambda)^{ab}$.
\end{proof}

We consider two further cases depending on the size of $\mathcal{B}_G'$. Let
\[
 \mathcal{L}^1_{H}= \left\{G\in  \mathcal{L}_{H}: |H(G)| \frac{(D^\ast)^2}{16}\leq |\mathcal{B}_G'| < \tau ab \right\}\, ,
\] 
and
\[
 \mathcal{L}^2_{H}= \left\{G\in  \mathcal{L}_{H}: |\mathcal{B}_G'| <\min\left\{ \tau ab, |H(G)| \frac{(D^\ast)^2}{16} \right\}\right\}\, ,
\] 
and note that $ \mathcal{L}_{H}= \mathcal{L}^0_{H}\cup  \mathcal{L}^1_{H}\cup  \mathcal{L}^2_{H}$.
\begin{lemma}\label{lemhighdegree}
\[
Z( \mathcal{L}^1_{H},\lam) \leq e^{-2\sqrt{n}}\cdot Z_{A,B}^{\textup{w}}(\lam)\, .
\]
\end{lemma}
\begin{proof}
In order to construct a graph $G\in \mathcal{L}^1_{H}$, we first fix $T\in \cD$ and $h, x, t\geq 0$ and construct a graph $G\in  \mathcal{L}^1_{H}\cap \mathcal{L}(F)$ where $F\in \cF'_t(T, x,h)$. As in the proof of Lemma~\ref{lemManyCrossing} we first choose the set $H'(F)$ of $k$ vertices from either $A$ or $B$ and for each $v\in H'$ choose the sets $W_A(v), W_B(v)$ of size $D^\ast$ each. The number of ways to choose these sets is at most
\[
\left(n\binom{a}{D^*}\binom{b}{D^*}\right)^k \leq n^k\binom{n}{2D^*}^k 
\]
After these are fixed, we choose the remaining $t-kD^\ast-|T|$ edges of $F$. Let $F'$ denote the graph formed by taking the union of $T$ and these $t-kD^\ast-|T|$ edges and note that $F'\in \cF_{t-kD^{\ast}}(T, x, h)$. 

We then choose the remaining edges of $G$ from $A \times B$. Suppose without loss of generality that $H'\subseteq A$. Let 
\[
\mathcal O=\bigcup_{v\in H'}\left(\{v\}\times W_B(v)\right)\subseteq A\times B\, ,
\]
and recall that 
\[
\mathcal B'= \bigcup_{v\in H'} \left( W_A(v) \times W_B(v)\right)\, .
\]
Choosing the remaining edges of $G$ from $A\times B$ amounts to choosing an independent set in the subgraph of $ F_\boxempty$ obtained by deleting the vertices of 
\[
X:=\mathcal B'\cup \mathcal O\cup N_{F_\boxempty}(\mathcal O)\, .
\]
Indeed, we have already forced the elements of $\mathcal O$ to be in $G$ and the elements of $\mathcal B'$ are blocked, hence not in $G$. 
 Letting $T'$ denote the subgraph of $T_\boxempty$ obtained by deleting the vertices of $X$ we have $T'\subseteq F'_\boxempty$ and so
\[
Z_{F'_\boxempty}(\lam)\leq Z_{T'}(\lam)  \, .
\] 
We now compare $Z_{T_\boxempty}$ and $Z_{T'}$ via cluster expansion. Note that since $T$ has maximum degree $D$, $T_\boxempty$ has maximum degree at most $2D$. 
Letting
\[
\cC_0= \cC(T_\boxempty)\backslash \cC(T')\, ,
\]
we have
\[
\log Z_{T_\boxempty} - \log Z_{T'}=  \sum_{\Gamma\in \cC_0} \phi(\Gamma) \lam^{|\Gamma|} = |X|\lam + \sum_{\Gamma\in \cC_0: |\Gamma|\geq 2} \phi(\Gamma) \lam^{|\Gamma|}\, .
\]
Note that if $\Gamma\in \cC_0$ then $v\in\Gamma$ for some $v\in X$. It follows by Lemma~\ref{lemClusterTail}\footnote{applied with $k=2$, $|S|=1$, $\Delta=2D=\frac{2\alpha}{\lam}$, noting that $\lam<\frac{1}{4e\Delta}$. } that 
\[
\left |  \sum_{\Gamma\in \cC_0: |\Gamma|\geq 2} \phi(\Gamma) \lam^{|\Gamma|} \right|\leq (2e)^2|X| D\lam^2\, .
\]
We conclude that
\[
Z_{T'}/Z_{T_\boxempty}\leq e^{-|X|\lam/2}\leq e^{-h\lam (D^\ast)^2/32}
\]
where we used that $\alpha=D\lam=1/(96e^3)$ for the first inequality and $X\supseteq \mathcal B'$,  $|\mathcal B'|\geq h(D^\ast)^2/16$ for the second.

Putting everything together we have 
\begin{align*}
Z( \mathcal{L}^1_{H},\lam)
&\leq  \sum_{T \in \mathcal{D}} \sum_{\substack{t,x,h\geq 0: \\ h> \alpha x/16}}n^k\binom{n}{2D^\ast}^k\lam^{2D^\ast k} \sum_{F\in \cF_{t-kD^\ast}(T,x,h)}\lam^{t-kD^\ast} Z_{T_\boxempty} e^{-h\lam (D^\ast)^2/32} \, .
\end{align*}

By Lemma~\ref{lemweightedTxh},
\begin{align*}
\lam^{t-kD^\ast}|\mathcal{F}_{t-kD^\ast}(T, x,h)|\leq \lam^{|T|} e^{x\omega^2D/10}e^{2\lambda n h}
\leq \lam^{|T|} e^{4h\lambda n}
\end{align*}
since that $h> \alpha x/16$.
Note also that
\begin{align}\label{eqwchoice}
n^k\binom{n}{2D^\ast}^k \lam^{2D^\ast k} \leq n^h \left(\frac{en\lam}{2D^\ast} \right)^{2D^\ast k}
\leq e^{\beta\log(e/\beta)h\lambda n/4 +h\log n}\leq  e^{h\lambda n} \, ,
\end{align}
since $\beta\log(e/\beta)\leq 3$ and $\log n\leq \lam n/4$.

It follows that
\begin{align*}
Z( \mathcal{L}^1_{H},\lam)
&\leq  
\sum_{T \in \mathcal{D}} \lam^{|T|}Z_{T_\boxempty} \sum_{\substack{t,x,h\geq 0: \\ h> \alpha x/16}}\exp\{5h\lambda n -h\lam (D^\ast)^2/32 \}\\
&\leq  
Z_{A,B}^{\textup{w}}(\lam) \sum_{\substack{t,x,h\geq 0: \\ h> \alpha x/16}}\exp\{-h\beta^2\omega^{3} \sqrt{n}  / 150\}\\
&\leq e^{-\beta^2\omega^3\sqrt{n}/200}\cdot  Z_{A,B}^{\textup{w}}(\lam)\, .
\end{align*}
\end{proof}

Finally we bound the contribution from $ \mathcal{L}^2_{H}$.
\begin{lemma}\label{lemsmallcrossgraph}
\[
Z( \mathcal{L}^2_{H},\lam) \leq e^{-2\sqrt{n}}\cdot Z_{A,B}^{\textup{w}}(\lam)\, .
\]
\end{lemma}
The idea will be that for a random choice of the sets $W_A(v), W_B(v)$, the resulting set $\mathcal{B}'$ is typically a constant factor larger than $|H(G)| \frac{(D^\ast)^2}{16}$. This is a consequence of the following lemma which is a special case of \cite[Lemma 3.6]{balogh2016typical}

\begin{lemma}\label{lemcrossingdeviation}
Suppose that $\mathcal B\subseteq A\times B$ satisfies
\[
|\mathcal B|\leq \tau ab
\]
and that $W_A\subseteq A, W_B\subseteq B$ are independent uniformly chosen subsets of size $D^\ast$. 
Then 
\[
\P\left(|\mathcal B\cap (W_A \times W_B)|> \frac{1}{2} (D^\ast)^2\right)\leq \theta^{D^\ast}\,,
\]
where $\theta = 10\tau^{1/8}$.
\end{lemma}

We record the following corollary. 
\begin{cor}\label{corcrossingdeviation}
For each $i\in \{1,\ldots, k\}$, choose subsets $W_A^i\subseteq A$, $W_B^i\subseteq B$ of size $D^\ast$ independently and uniformly at random. Let 
\[
\mathcal B:= \bigcup_{i=1}^k W_A^{i}\times W_B^{i}\, .
\]
Then 
\[
\P \left( |\mathcal B|< \min\left\{\tau ab,\frac{k(D^\ast)^2}{4} \right\}\right) \leq 2^k \theta^{kD^\ast /2}\, .
\]
\end{cor}
\begin{proof}
For $\ell\in \{1,\ldots, k\}$, let 
\[
\mathcal B_\ell= \bigcup_{i=1}^\ell W_A^i\times W_B^i\, ,
\]
so in particular $\mathcal B=\mathcal B_k$.
We say index $\ell$ is \emph{useful} if
\[
\left |\mathcal B_{\ell-1} \cap (W_A^\ell \times W_B^\ell ) \right|\leq \frac{1}{2}(D^\ast)^2
\]
or in other words
\[
|\mathcal B_{\ell}|- |\mathcal B_{\ell-1}|  \geq \frac{1}{2}(D^\ast)^2\, .
\]
If at least half of the indices $\ell\in  \{1,\ldots, k\} $ are useful then
\[
|\mathcal B|\geq \frac{k}{2}\cdot \frac{(D^\ast)^2 }{2}\, .
\]
It follows that if $|\mathcal B|< k(D^\ast)^2/4$, then at most half of the indices $\ell\in  \{1,\ldots, k\} $ are useful. Let $\cE$ denote the event that $\leq k/2$  indices $\ell\in  \{1,\ldots, k\} $ are useful and that $ |\mathcal B|< \tau ab$. It suffices to bound $\P(\cE)$. For $I\subseteq [k]$, let $\mathcal{Q}_{I}$ denote the event that all the indices in $I$ are not useful. By a union bound, we then have 
\begin{align}\label{eqUsefulSteps}
\mathbb{P}(\mathcal{E}) \leq \sum_{\substack{I \subset [k]:\\|I| > k/2}} \P(\mathcal{Q}_{I} \cap \{ |\mathcal B|< \tau ab \})\, .
\end{align}
Fix $I=\{i_1,\ldots, i_j \}\subseteq [k]$ and $\ell\in[j]$. Note that for any realisation of $\mathcal B_{i_\ell-1}$ such that $| \mathcal B_{i_\ell-1}|< \tau ab$, the probability that index $i_\ell$ is not useful is at most $\theta^{D^\ast}$ by Lemma~\ref{lemcrossingdeviation}. It follows that
\[
\P(\mathcal{Q}_{I} \cap \{ |\mathcal B|< \tau ab \})\leq \theta^{|I|D^\ast}
\]
and so by~\eqref{eqUsefulSteps},
\[
\mathbb{P}(\mathcal{E}) \leq 2^k  \theta^{kD^\ast/2}\, .
\]
\end{proof}

\begin{proof}[Proof of Lemma~\ref{lemsmallcrossgraph}]
As in the proof of Lemma~\ref{lemhighdegree}
in order to construct a graph $G\in \mathcal{L}^2_{H}$, we first fix $T\in \cD$ and $h, x, t\geq 0$ and construct a graph $G\in  \mathcal{L}^2_{H}\cap \mathcal{L}(F)$ where $F\in \cF'_t(T, x,h)$. 

First choose the set $H'(F)$ of $k$ vertices (there are $\leq n^k$ choices for these) and for each $v\in H'(F)$ choose the sets $W_A(v), W_B(v)$ of size $D^\ast$ each in such a way that
\[
\left| \bigcup_{v\in H'} W_A(v) \times W_B(v)\right| \leq h \frac{(D^{\ast})^2}{16}\, .
\]
By Corollary~\ref{corcrossingdeviation},
the number of ways to choose sets $W_A(v), W_B(v)$ for each $v\in H'(F)$ is at most
\[
\left(\binom{a}{D^\ast}\binom{b}{D^\ast}\right)^k 2^k \theta^{D^\ast k/2}\leq \binom{n}{2D^\ast}^k 2^k \theta^{D^\ast k/2}\, .
\]

After these are fixed, we choose the remaining $t-kD^\ast-|T|$ edges of $F$. Let $F'$ denote the graph formed by taking the union of $T$ and these $t-kD^\ast-|T|$ edges and note that $F'\subseteq \cF_{t-kD^{\ast}}(T, x, h)$. We then choose the remaining edges of $G$ from $A \times B$. We conclude that
\begin{align*}
Z( \mathcal{L}^2_{H},\lam)
&\leq  \sum_{T \in \mathcal{D}} \sum_{\substack{t,x,h\geq 0: \\ h> \alpha x/16}}n^k\binom{n}{2D^\ast}^k2^k \theta^{D^\ast k/2}\lam^{2D^\ast k} \sum_{F\in \cF_{t-kD^\ast}(T,x,h)}\lam^{t-kD^\ast} Z_{F_\boxempty}\, .
\end{align*}
Now, by Lemma~\ref{lemweightedTxh} and the bounds $Z_{F_\boxempty}\leq Z_{T_\boxempty}$ and $h> \alpha x/16$ we have 
\[
\sum_{F\in \cF_{t-kD^\ast}(T,x,h)}\lam^{t-kD^\ast} Z_{F_\boxempty} \leq  \lam^{|T|}Z_{T_\boxempty}e^{4h\lambda n}\, .
\]
Moreover, bounding as in~\eqref{eqwchoice} we have
\begin{align*}
n^k\binom{n}{2D^\ast}^k2^k \theta^{D^\ast k/2}\lam^{2D^\ast k}  
&\leq
\exp\left\{\log(2n)k +\beta \log(e/\beta)\lam n k -\log(1/\theta)\beta\lam nk/4\right\}\\
&\leq
\exp\left\{-\log(1/\theta)\beta\lam nk/5\right\}\\
&\leq
e^{-5h\lam n}
\end{align*}
where for the second inequality we used $k\geq h/4$ and $\theta=e^{-100/\beta}$.
It follows that
\begin{align*}
Z( \mathcal{L}^2_{H},\lam)
&\leq  
\sum_{T \in \mathcal{D}} \lam^{|T|}Z_{T_\boxempty} \sum_{\substack{t,x,h\geq0: \\ h> \alpha x/16}}\exp\{-h\lambda n \}
\leq e^{-\omega\sqrt{n}/2}\cdot  Z_{A,B}^{\textup{w}}(\lam)\,. \qedhere
\end{align*}
\end{proof}

We have now collected all the necessary ingredients to complete the proof of Lemma~\ref{lemOPT}.
\begin{proof}[Proof of the upper bound of Lemma~\ref{lemOPT}]
Let $G\in \mathcal L_{A,B} \backslash\cT_{A,B,\lam}^{\textup{w}}$ so that $\Delta(G_A\cup G_B)>\alpha/\lam$. 
Let $F=G_A\cup G_B$. Recall that $U(F)$ is an edge-maximal subgraph of $F$ with maximum degree at most $D$, and $X(F)$ is the set of all vertices whose degree in $U(F)$ is $D$. Since $\Delta(F)>D$, we must then have $|X(F)|\geq 1$, else we could add an edge to $U(F)$ without violating the degree bound, contradicting the maximality of $U(F)$. It follows that $F\in \cF'_t(T, x,h)$ for some $T\in \cD$ and $t,x,h$ where $x\geq 1$. We therefore have that 
\(
F\in \cF_L \cup \cF_H \,,
\)
i.e.,
\[
\mathcal L_{A,B} \backslash\cT_{A,B,\lam}^{\textup{w}}\subseteq \mathcal L_L\cup  \mathcal L_H=  \mathcal L_L\cup \mathcal L_H^0 \cup \mathcal L_H^1 \cup \mathcal L_H^2\, .
 \]

By Lemmas~\ref{lemlowdegree},~\ref{lemManyCrossing},~\ref{lemhighdegree}, and~\ref{lemsmallcrossgraph},
\begin{align*}
Z(\mathcal L_{A,B}\backslash \cT_{A,B,\lam}^{\textup{w}}, \lam)
\leq {Z(\mathcal{L}_{L},\lam)+Z( \mathcal{L}^0_{H},\lam)+Z( \mathcal{L}^1_{H},\lam)+Z( \mathcal{L}^2_{H},\lam)}
\leq e^{-\sqrt{n}}Z_{A,B}^{\textup{w}}(\lam)\, .
\end{align*}
The result follows. 
\end{proof}

\subsection{Proof of Proposition~\ref{propGroundStateRefinedAlt}}
The fact that a graph $G$ drawn according to $\mu_{\textup{weak},\lam}$ has a unique weakly balanced max cut whose defect graph has maximum degree at most $\alpha/\lam$ follows immediately from Lemma~\ref{lemWeakUnique}.

We first prove
\begin{align}\label{eqGroundStateRefinedRep}
 Z(\cL, \lam) = \left(1+O\left(e^{-\sqrt{n}} \right) \right) Z_{\textup{weak}}(\lam)\, .
\end{align}
First observe that by Lemma~\ref{lemOPT},
\[
Z(\cL,\lam)
\leq  \sum_{(A,B)\in\Pi_{\textup{weak}}}Z( \mathcal{L}_{A,B},\lam)
= \left(1+O\left(e^{-\sqrt{n}}\right)\right)   \sum_{(A,B)\in\Pi_{\textup{weak}}}  Z_{A,B}^{\textup{w}}(\lam)\, .
\]

Let $U_{A,B,\lam}$ denote the set of all $G\in \cT$ such that $\Delta(G_A\cup G_B)\leq \alpha/\lam$ and $(A,B)$ is the unique max cut of $G$. By Lemma~\ref{lemWeakUnique}
\[
Z( \cT_{A,B,\lam}^{\textup{w}} \cap U_{A,B,\lam},\lam) =  \left(1+O\left(e^{-\sqrt{n}}\right)\right)  Z_{A,B}^{\textup{w}}(\lam)\, .
\]
We note that if $G\in \cT_{A,B,\lam}^{\textup{w}} \cap U_{A,B,\lam}$, then $(A,B)$ is a dominating cut for $G$ (since $(A,B)$ is a max cut) and $|G_A \cup G_B|\leq n\cdot \alpha/(2\lam)\leq 2\delta \lam n^2$ and so $G\in \mathcal{L}$. By Lemma~\ref{lemWeakUnique}, we also know that $G\notin \cT_{A',B',\lam}^{\textup{w}} $ for all weakly balanced partitions $(A',B')$ distinct from $(A,B)$. We conclude that
\[
Z(\mathcal{L},\lam)
\geq  \sum_{(A,B)\in\Pi_{\textup{weak}}} Z( \cT_{A,B,\lam}^{\textup{w}} \cap U_{A,B,\lam},\lam)
= \left(1+O\left(e^{-\sqrt{n}}\right)\right)   \sum_{(A,B)\in\Pi_{\textup{weak}}}  Z_{A,B}^{\textup{w}}(\lam)\, .
\]
The estimate~\eqref{eqGroundStateRefinedRep} follows.

To conclude the proof, note that
{\small \begin{align}
\|\mu_{\cL,\lam} - \mu_{\textup{weak},\lam}\|_{\textup{TV}}
&= \sum_{G: \mu_{\textup{weak},\lam}(G)>\mu_{\cL,\lam}(G)} \mu_{\textup{weak},\lam}(G)-\mu_{\cL,\lam}(G)\\
&= \sum_{G: \mu_{\textup{weak},\lam}(G)>\mu_{\cL,\lam}(G)}c_{\textup{weak},\lam}(G)\frac{\lam^{|G|}}{Z_{\textup{weak}}(\lam)}-\frac{\lam^{|G|}}{Z(\cL,\lam)}\mathbf 1_{G\in \cL} \\
&\leq  \mu_{\textup{weak},\lam}(c_{\textup{weak},\lam}(G)>1) +  \mu_{\textup{weak},\lam}(G\notin \cL)  +\sum_{G\in\cL:  c_{\textup{weak},\lam}(G)= 1}
\left|
\frac{\lam^{|G|}}{Z_{\textup{weak}}(\lam)}-
\frac{\lam^{|G|}}{Z(\cL,\lam)}\right|\\
&\leq
O\left(e^{-\sqrt{n}} \right) +
\sum_{G\in\cL:  c_{\textup{weak},\lam}(G)= 1}\frac{\lam^{|G|}}{Z_{\textup{weak}}(\lam)}
\left|
1-
 \left(1+O\left(e^{-\sqrt{n}} \right) \right) \right|\\
&\leq
O\left(e^{-\sqrt{n}} \right)\, .
\end{align}}
For the first inequality we used Lemma~\ref{lemWeakUnique} and for the second inequality we used \eqref{eqGroundStateRefinedRep}.

\section*{Acknowledgments}

We thank Bob Krueger, Corrine Yap, and the anonymous referees for their very helpful feedback on this paper. MJ is supported by a UKRI Future Leaders Fellowship MR/W007320/2. WP supported in part by NSF grant DMS-2348743.

\appendix
\section{Pinned cluster expansions}
\label{secPinnedCluster}

Given a graph $H$, let $T(H)$ denote the set of labelled spanning trees of $H$. Moreover we let $T_k$ denote the set of all labelled trees on vertex set $\{1,\ldots, k\}$.  The following is the tree--graph bound of Penrose.  See~\cite[Section 4]{faris2010combinatorics} for a detailed discussion.  
\begin{lemma}[\cite{penrose1967convergence}] \label{lempenrosetree}
Given a graph $H$, 
\[
\left| \sum_{\substack{A \subseteq E(H)\\ \textup{spanning, connected}}}  (-1)^{|A|}\right| \leq |T(H)|\, .
\]
\end{lemma}

With this we prove Lemma~\ref{lemClusterTail}.
\begin{proof}[Proof of Lemma~\ref{lemClusterTail}]
Fix $k\geq |S|$. We will show that
\begin{align}\label{fixedclusterphibd}
 \sum_ {\substack{\Gamma: \Gamma\supseteq S, \\ |\Gamma|= k}} |\phi(\Gamma)|  \leq e^k k^{|S|-2}(\Delta+1)^{k-|S|}\, . 
\end{align}
Given a cluster $\Gamma$ of size $k$, we identify the vertex set of $H_\Gamma$ with $\{1,\ldots, k\}$.
By Lemma~\ref{lempenrosetree}, we have
\begin{align}
\nonumber
\sum_ {\substack{\Gamma: \Gamma\supseteq S, \\ |\Gamma|= k}} |\phi(\Gamma)|
 &=
  \frac{1}{k!} \sum_ {\substack{\Gamma: \Gamma\supseteq S, \\ |\Gamma|= k}}
\left|  \sum_{\substack{A \subseteq E(H_\Gamma)\\ \textup{spanning, connected}}}  (-1)^{|A|}
  \right|\\\nonumber
   &\leq
  \frac{1}{k!} \sum_ {\substack{\Gamma: \Gamma\supseteq S, \\ |\Gamma|= k}}
\sum_{T\in T_k}\mathbf 1_{T\in T(H_\Gamma)} \\
 &\leq
  \frac{1}{k!} 
\sum_{T\in T_k}\sum_ {\substack{\Gamma: \Gamma\supseteq S, \\ |\Gamma|= k}}
\mathbf 1_{T\in T(H_\Gamma)} \, .\label{eqpenroseapp}
 \end{align}
 Fix $T\in T_k$.
 We will construct a cluster $\Gamma$ such that $|\Gamma|=k$, $S\subseteq \Gamma$ and $T\in T(H_\Gamma)$ iteratively as follows. First we select one of the $\binom{k}{|S|}|S|!$ ways to place the elements of $S$ in the tuple $\Gamma$. Now suppose we have filled coordinates $i_1, \ldots, i_j$ of $\Gamma$ with vertices $v_{i_1}, \ldots, v_{i_j}\in V(G)$ respectively (where $j\geq |S|$). There exists $r\in [k]\backslash \{i_1, \ldots, i_j\}$ such that $r$ is adjacent to one of $\{i_1, \ldots, i_j\}$ in the graph $T$. Without loss of generality assume it is $i_1$. We then must select $v_r\in V(G)$ such that either $v_r=v_{i_1}$ or $v_r$ is adjacent to $v_{i_1}$ in $G$ and place it in the $r$th coordinate of $\Gamma$. There are at most $\Delta+1$ choices for such a $v_r$ in $V(G)$. Continuing iteratively we see that
 \[
 \sum_ {\substack{\Gamma: \Gamma\supseteq S, \\ |\Gamma|= k}}
\mathbf 1_{T\in T(H_\Gamma)} \leq \binom{k}{|S|}|S|! (\Delta+1)^{k-|S|}\, .
 \]
 By Cayley's formula $|T_k|=k^{k-2}$ and so we conclude from~\eqref{eqpenroseapp}
 \[
 \sum_ {\substack{\Gamma: \Gamma\supseteq S, \\ |\Gamma|= k}} |\phi(\Gamma)|
 \leq 
  \frac{1}{k!} k^{k-2} \binom{k}{|S|}|S|! (\Delta+1)^{k-|S|}
  \leq e^k k^{|S|-2} (\Delta+1)^{k-|S|}\, ,
 \]
 where for the final inequality we used that $k!\geq(k/e)^k$ for all $k\geq 1$.
This establishes~\eqref{fixedclusterphibd}.
 Recalling that $\lam_{\textup{max}}:= \max_{v\in V(G)}|\boldsymbol{\lam}(v)|$, we conclude that
 \begin{align*}
\left|  \sum_ {\substack{\Gamma: \Gamma\supseteq S, \\ |\Gamma|\geq k}} |\Gamma|^t \phi(\Gamma) \prod_{v\in\Gamma} \boldsymbol{\lam}(v)\right| 
\leq 
(\Delta+1)^{-|S|}\sum_{j\geq k} j^{|S|-2+t} \left(e(\Delta+1)\lam_{\textup{max}}\right)^{j} = O_{k,t}\left( \Delta^{k-|S|}\lam_{\textup{max}}^k \right)\, ,
\end{align*}
 where for the final inequality we used that $\lam_{\textup{max}}<1/(4e\Delta)\leq1/(2e(\Delta+1))$. 
Finally if $|S|\in \{1,2\}$ then we may use the explicit upper bound
  \begin{align*}
\left|  \sum_ {\substack{\Gamma: \Gamma\supseteq S, \\ |\Gamma|\geq k}} \phi(\Gamma) \prod_{v\in\Gamma} \boldsymbol{\lam}(v)\right| 
&\leq 
2e^k(\Delta+1)^{k-|S|}\lam_{\textup{max}}^k
\leq 
(2e)^k\Delta^{k-|S|}\lam_{\textup{max}}^k
\, .  \qedhere
\end{align*}
\end{proof}
\begin{proof}[Proof of Corollary~\ref{corclustersimple2}]
The proof is the same as that of~\eqref{eqNonConstExpand} in Corollary~\ref{corclustersimple} with some additional calculations that we detail now. 
Recall that $\cC'_k$ denotes the set of non-constant clusters of $G$ of size $k$. Since $G$ is triangle-free we have
\[
\cC'_4=\{(v_1,v_2,v_3, v_4): G[\{v_1,v_2,v_3, v_4\}]\cong K_2, P_2, P_3, S_3 \text{ or }C_4\}\, .
\]
If $\Gamma=(v_1, v_2, v_3, v_4)\in \cC'_4$ such that $\{v_1,v_2,v_3, v_4\}=\{u,v\}$ for some edge $\{u,v\}$ of $G$, then either
\begin{enumerate}[label=(\roman*)]
\item $\Gamma$ has two coordinates equal to $u$ and two coordinates equal to $v$. \label{itemGam22}
\item $\Gamma$ has one coordinate equal to $u$ and three coordinates equal to $v$ or vice versa. \label{itemGam13}
\end{enumerate}
In either case $H_\Gamma\cong K_4$, a clique on $4$ vertices and so $\phi(\Gamma) = -1/4$. \footnote{Note that $\phi(\Gamma)=\frac{1}{|\Gamma|!}(-1)^{|\Gamma|+1}T_{H_\Gamma}(1,0)$ where $T_G$ is the Tutte polynomial of a graph $G$.}
There are $\binom{4}{2}|G|$ clusters of the type in case~\ref{itemGam22} and there are $2\cdot 4|G|$ clusters of the type in case~\ref{itemGam13} and so $14|G|$ clusters of type \ref{itemGam22} and \ref{itemGam13} in total. 

If $\Gamma=(v_1, v_2, v_3, v_4)\in \cC'_4$ such that $\{v_1,v_2,v_3, v_4\}=\{u,v,w\}$ for some  $\{u,v,w\}$ where $G[\{u,v,w\}]\cong P_2$ and $v$ has degree $2$ in $G[\{u,v,w\}]$ then either
\begin{enumerate}[label=(\roman*)]
\item $\Gamma$ has two coordinates equal to $v$. \label{itemGamv2}
\item $\Gamma$ has two coordinates equal to $u$ or two coordinates equal to $w$. \label{itemGamuw2}
\end{enumerate}
In case~\ref{itemGamv2}, $H_\Gamma$ is isomorphic to the unique graph on $4$ vertices and $5$ edges (a cycle of length $4$ with a chord) and so $\phi(\Gamma) = -1/6$. 
In case~\ref{itemGamuw2}, $H_\Gamma$ is isomorphic to a triangle with a pendant edge and so $\phi(\Gamma) = -1/12$. 
There are $2\cdot\binom{4}{2}P_2(G)$ clusters of the type in case~\ref{itemGamv2} and there are $2\cdot 2\cdot\binom{4}{2}P_2(G)$ clusters of the type in case~\ref{itemGamuw2}.

 If $\Gamma=(v_1, v_2, v_3, v_4)\in \cC'_4$ such that $G[\{v_1,v_2,v_3, v_4\}]\cong P_3$ then $H_\Gamma\cong P_3$, $\phi(\Gamma) = -1/4!$ and there are $4!\cdot P^{\text{ind}}_3(G)$ such clusters, where $P^{\text{ind}}_3(G)$ denotes the number of induced copies of $P_3$ in $G$.  
 
  If $\Gamma=(v_1, v_2, v_3, v_4)\in \cC'_4$ such that $G[\{v_1,v_2,v_3, v_4\}]\cong S_3$ then $H_\Gamma\cong P_3$, $\phi(\Gamma) = -1/4!$ and there are $4!\cdot S_3(G)$ such clusters. 
  
  Finally if $\Gamma=(v_1, v_2, v_3, v_4)\in \cC'_4$ such that $G[\{v_1,v_2,v_3, v_4\}]\cong C_4$ then $H_\Gamma\cong C_4$, $\phi(\Gamma) = -1/8$ and there are $4!\cdot C_4(G)$ such clusters. 
  
  Putting everything together, we conclude that 
  \[
  \sum_{\Gamma\in \cC'_4}\phi(\Gamma)\lam^{|\Gamma|}=-\lam^4\left (P^{\text{ind}}_3(G)+S_3(G)+3C_4(G)+4P_2(G)+7|G|/2\right)\, .
  \]
  The corollary follows by noting that
  \[
  P^{\text{ind}}_3(G) = P_3(G) - 4C_4(G)
  \]
since $G$ is triangle-free. 
\end{proof}

\begin{proof}[Proof of Lemma~\ref{lemSubgCart}]
Given graphs $F_1, F_2$ we say that $H\subseteq F_1 \boxempty F_2$ is a \emph{transversal} subgraph if $H$ is connected and for all $u\in V(F_1)$, there exists $w\in V(F_2)$ such that $(u,w)\in V(H)$ and for all $v\in V(F_2)$, there exists $w\in V(F_1)$ such that $(w,v)\in V(H)$. In other words, viewing $V(F_1 \boxempty F_2)= V(F_1)\times V(F_2)$ as a grid, the vertex set of $H$ hits every row and every column of the grid. 

Given a graph $H$, we let $H^\ast(F_1 \boxempty F_2)$ denote the number of transversal copies of $H$ in $F_1 \boxempty F_2$ and  let 
\[
\mathcal X(H)= \{(F_1, F_2): F_1 \boxempty F_2 \text{ contains a transversal copy of }H \text{ and } F_1\cup F_2 \text{ is triangle-free} \}\, .
\]
Since every copy of $H$ in $S\boxempty T$ may be uniquely identified with a transversal subgraph of $F_1\boxempty F_2$ for some \emph{induced} copy of $F_1$ in $S$ and $F_2$ in $T$,
we have the relation
\[
H(S\boxempty T) = \sum_{(F_1, F_2)\in \mathcal X(H)} F_1^{\text{ind}}(S) F_2^{\text{ind}}(T) H^\ast(F_1 \boxempty F_2)\, .
\]
The lemma now follows by noting the following (we let $\boldsymbol{\cdot}$ denote the graph consisting of a single vertex) :
\begin{align*}
\mathcal X(P_3) &= \{(P_3, \boldsymbol{\cdot}), (\boldsymbol{\cdot}, P_3), (C_4, \boldsymbol{\cdot}), (\boldsymbol{\cdot}, C_4), (P_2, K_2), (K_2, P_2), (K_2, K_2) \}\, , \\
\mathcal X(S_3)& = \{(S_3, \boldsymbol{\cdot}), (\boldsymbol{\cdot}, S_3), (P_2, K_2), (K_2, P_2) \}\, , \\
\mathcal X(C_4)& = \{(C_4, \boldsymbol{\cdot}), (\boldsymbol{\cdot}, C_4), (K_2, K_2) \}\, .
\end{align*} 
Moreover, $H^\ast(H \boxempty \boldsymbol{\cdot})=1$ for all $H$ and 
\[
 P_3^*(C_4 \boxempty \boldsymbol{\cdot})= 4, \quad P_3^*(P_2 \boxempty K_2) = 6, \quad P_3^*(K_2 \boxempty K_2) = 4, \quad S_3^*(P_2 \boxempty K_2) = 2,  \quad C_4^*(K_2 \boxempty K_2) = 1\,. 
\]
Finally note that since $S$ is triangle-free, $F^{\text{ind}}(S)=F(S)$ for $F\in\{C_4, S_3, P_2, K_2, \boldsymbol{\cdot}\}$ and
\[
P_3(S)=P_3^{\text{ind}}(S)+4 C_4^{\text{ind}}(S)\, ,
\]
and similarly for $S$ replaced with $T$.
\end{proof}

\section{Quasirandomness for the hard-core model}\label{secHCQuasi}

\begin{proof}[Proof of Lemma~\ref{lemhcquasirandom}]
Let $t\in \R$ and define  $\boldsymbol{\lam}_t: V(G)\to\R$ by setting $\boldsymbol{\lam}_t(v)=\lam e^t$ for $v\in U$ and $\boldsymbol{\lam}_t(v)=\lam$ for $v\in U^c$. 
Let $\mathbf I$ be a random sample from the hard-core model on $G$ at activity $\lam$ and note that we can write the moment-generating function of $|\mathbf I \cap U|$
\[
\E\left[ e^{t| \mathbf I \cap U|} \right] = \sum_{I \in \cI(G)} \frac{\lam^{|I|}}{Z_G(\lam)}e^{t| I \cap U|}
= \frac{Z_G(\boldsymbol{\lam}_t)}{Z_G(\lam)}   \, .
\]
Suppose now that $t\leq 1$.
Since $\lam\leq1/(16e^2\Delta)$, we have that $\boldsymbol{\lam}_t(v)\leq 1/(16e\Delta)$ for all $v\in V(G)$ and so we may analyze the ratio $Z_G(\boldsymbol{\lam}_t)/Z_G(\lam)$ via the cluster expansion (and in particular apply Lemma~\ref{lemClusterTail}). 
Indeed we have
\begin{align*}
\log \left(\frac{Z_G(\boldsymbol{\lam}_t)}{Z_G(\lam)}\right) 
&=  \sum_{\Gamma\in \cC(G)}\phi(\Gamma)\prod_{v\in \Gamma} \boldsymbol{\lam}_t(v) -  \sum_{\Gamma\in \cC(G)}\phi(\Gamma)\lam^{|\Gamma|}\\
&=
 \sum_{ \Gamma\sim U}\phi(\Gamma)\prod_{v\in \Gamma} \boldsymbol{\lam}_t(v) -  \sum_{\Gamma\sim U}\phi(\Gamma)\lam^{|\Gamma|}\\
 &=
 (e^t-1) \lam |U| + \sum_{ \Gamma\sim U, |\Gamma|\geq 2}\phi(\Gamma)\prod_{v\in \Gamma} \boldsymbol{\lam}_t(v) -  \sum_{\Gamma\sim U, |\Gamma|\geq 2}\phi(\Gamma)\lam^{|\Gamma|} \, ,
\end{align*}
where we write $\Gamma\sim U$ to mean that $\Gamma$ contains at least one element of $U$. 
By Lemma~\ref{lemClusterTail} (applied with $k=2$ and $S=\{u\}$ for each $u\in U$ and using the explicit bound of~\eqref{eqpinnedexplicittail}) we have
\begin{align*}
\left| \sum_{ \Gamma\sim U, |\Gamma|\geq 2}\phi(\Gamma)\prod_{v\in \Gamma} \boldsymbol{\lam}_t(v)\right| + \left|  \sum_{\Gamma\sim U, |\Gamma|\geq 2}\phi(\Gamma)\lam^{|\Gamma|}\right| 
&\leq 
(2e)^2\Delta\lam^2(\max\{e^{2t},1\}+1)|U| \\
&\leq
\frac{1}{4}(\max\{e^{2t},1\}+1)|U|\lam\, ,
\end{align*}
where for the last inequality we used that $\lam\leq1/(16e^2\Delta)$.
It follows that
\begin{align}\label{eqmomentgenbd}
\E\left[ e^{t| \mathbf I \cap U|} \right]\leq \exp\left\{\left(e^t-1 +  \frac{1}{4}(\max\{e^{2t},1\}+1)\right)|U|\lam \right\}\, .
\end{align}
By~\eqref{eqmomentgenbd} with $t=1$ and Markov's inequality, we then have
\[
\P \left (|\mathbf I \cap U|\geq 5 |U| \lam \right)= \P \left (e^{|\mathbf I \cap U|}\geq  e^{5|U| \lam} \right)
\leq
\E\left( e^{|\mathbf I \cap U|} \right) e^{-5|U| \lam} \leq e^{-|U| \lam}\, .
\]
Similarly, by~\eqref{eqmomentgenbd} with $t=-\log(10)$ we have
\begin{multline}
\P \left (|\mathbf I \cap U|\leq |U| \lam/10 \right) \\= \P \left (e^{-\log(10)|\mathbf I \cap U|}\geq  e^{-\log(10)|U| \lam/10} \right)
\leq
\E\left( e^{-\log(10)|\mathbf I \cap U|} \right) e^{\log(10)|U| \lam/10} \leq e^{-|U| \lam/8}\, . \qedhere
\end{multline}
\end{proof}

\section{Local Central Limit Theorem for the hard-core model}
\label{secLCLTproof}

The proof of  Proposition~\ref{thmhcLCLT}  is similar to those for CLT's and local CLT's in~\cite{dobrushin1977central,jenssen2020independent,jain2021approximate,jenssen2022independent}, but here we allow for a growing sequence of maximum degree bounds and use Lemma~\ref{lemClusterTail} in a crucial way.
\begin{proof}[Proof of Proposition~\ref{thmhcLCLT}]
The condition $\lam_n\Delta_n \to 0$ ensures that the cluster expansion converges and allows us to apply Lemma~\ref{lemClusterTail}.  We will use the fact that the cumulants of $X_n$ can be written as cluster expansions.  This allows us to estimate the mean and variance of $X_n$ and to prove a CLT (as in~\cite{jenssen2020independent}).  Let $\kappa_k(X_n)$ denote the $k$th cumulant of $X_n$.  Then under the stated conditions we have an expression for $\kappa_k(X_n)$ as a convergent cluster expansion:
\[ \kappa_k(X_n) =   \sum_{\Gamma}|\Gamma| ^k  \phi (\Gamma)  \lam^{|\Gamma|}  \,.   \]

In particular, using Lemma~\ref{lemClusterTail}, we have for each fixed $k \ge 1$, 
\begin{equation}
\label{eqKkasCLuster}
\kappa_k(X_n) = \lam n + O( n \Delta \lam^2) = \lam n (1+ O(\lam \Delta)) = \lam n (1+o(1)) \,. 
\end{equation}
Applying~\eqref{eqKkasCLuster} with $k=1,2$ we obtain $\E X_n  \sim \lam n$ and $\var(X_n) \sim \lam n$.  To prove a CLT for $X_n$, let $\overline X_n = (X_n - \E X_n)/\sqrt{\var(X_n)}$.  We must show that for fixed $k \ge 3$, $\kappa_k(\overline X_n) \to 0$.  Since cumulants for $k \ge 2$ satisfy $\kappa_k(aX+b) = a^k \kappa_k(X)$, it is enough to show $\kappa_k(X_n) = o( (\lam n)^{k/2})$. This follows since $\kappa_k(X_n) \sim \lam n $ and $\lam n \to \infty$ by assumption.

Next let $\phi_{X_n}(t) = \E e^{it X_n} $ denote the characteristic function of $X_n$.   Then following the proof of~\cite[Lemma 22]{jenssen2022independent} verbatim, we obtain
\[ |\phi_{X_n}(t)|  \le \exp \left[  - (1+o(1)) \frac{t^2}{5}  \lam n    \right ] \,. \]
With this estimate and the CLT for $X_n$ from above, the local CLT follows exactly as in the proof of~\cite[Theorem 20]{jenssen2022independent}.
\end{proof}

\section{Subgraph probabilities in defect distributions}\label{secAppSubCalc}
In this section we prove Claims~\ref{claimprobrefine} and~\ref{claimprobrefine2} from Section~\ref{secSuperCrit}.
\begin{proof}[Proof of Claim~\ref{claimprobrefine}]
Recall that 
\[
j(S,T)= \lam^3( bP_2(S) + aP_2(T) +4|S|\mu_B+4|T|\mu_A  +4\theta st )\, . 
\]
One easily verifies that $j$ is $(n\lam^3/(6\alpha))$-local and so 
by Lemma~\ref{lemfixedconfiglocal}
\begin{align}\label{eqlemfixedconfiglocal4}
\P(F\subseteq \bG)=\left(1+O\left(n^2\Delta^2\lam^6\right)\right)e^{\E[j(\bH\cup F)-j(\bH)]} (q_A')^{|F_A|}(q_B')^{|F_B|}\, ,
\end{align} 
where $\bH=\bG\backslash F$.
We now estimate the expectation in the exponent. 
\begin{align*}
j(\bH\cup F)-j(\bH) = &\lam^3\left[bP_2(F_A, \mathbf H_A) + aP_2(F_B, \mathbf H_B) +4(\mu_A|F_B| + \mu_B|F_A|)\right]\\
&+
4\lam^3 \theta \left[ |F_A|(|\mathbf H_B|-\mu_B) +  |F_B|(|\mathbf H_A|-\mu_A) +|F_A||F_B|\right]\, ,
\end{align*}
where we recall that $P_2(F_A, \mathbf H_A)$ denotes the number of copies of $P_2$ in $F_A\cup \mathbf H_A$ with at least one edge in $F_A$.

By Lemma~\ref{lemfixedconfiglocal},
\begin{align*}
\E(|\bH_A|)
&= (1+ O(n\Delta \lam^3))q'_A\left(\binom{a}{2}- |F_A| \right)
= (1+ O(n\Delta \lam^3))q'_A\binom{a}{2}\, ,
\end{align*}
since $|F|=O(1)$.
Recalling that 
\[
\mu_A= \binom{a}{2}q'_Ae^{2\lam^3b(aq_A+bq_B)} =  (1+ O(n\Delta \lam^3))q'_A\binom{a}{2}
\]
we conclude that
\[
\E(|\bH_A|) - \mu_A = O(n\Delta \lam^3\cdot qn^2)\, .
\]
Suppose now that $\{u,v\}$ is an edge of $F_A$. Each edge of $\bH$ which is incident to either $u$ or $v$ contributes one $P_2$ to the count $P_2(F_A, \bH_A)$. Applying Lemma~\ref{lemfixedconfiglocal} and summing these contributions over the edges of $F_A$ yields
\begin{align*}
\E \left(P_2(F_A,\bH_A)\right)
&= 2|F_A|(1+ O(n\Delta \lam^3))q_A(a-O(1)) + P_2(F_A)\\
&=2|F_A|q_Aa + P_2(F_A) + O(n\Delta \lam^3\cdot nq)\, .
\end{align*}
It follows, noting $\mu_A=(1+ O(n\Delta \lam^3))q_Aa^2/2$, that
\begin{align*}
\E(j(\bH\cup F)-j(\bH))
=& 2\lam^3(b|F_A|+a|F_B|)(aq_A+bq_B)+ b\lam^3P_2(F_A)+ a\lam^3P_2(F_B)\\
&+ O(n^3\Delta \lam^6q+n\Delta^2\lam^4)\, .
\end{align*}
The result follows from~\eqref{eqlemfixedconfiglocal4}, noting that $n^3\Delta \lam^6q+n\Delta^2\lam^4= O(n^2\Delta^2\lam^6)$.
\end{proof}

\begin{proof}[Proof of Claim~\ref{claimprobrefine2}]
Let $\cH_F$ denote the set of all graphs $H\subseteq \binom{A}{2}$ that are edge-disjoint from $F$ and $\Delta(H\cup F)\leq \Delta$. Let $\bH=\bG\backslash F$.
For $H\subseteq \binom{A}{2}$, we have
\begin{align}\label{eqPFGH3}\\
\P(F\subseteq \bG \mid \bH=H)
= \frac{\nu_{\bm q''_A, \cD_\emptyset}^\theta(H\cup F) }{\sum_{F'\subseteq F}\nu_{\bm q''_A, \cD_\emptyset}^\theta(H\cup F') }
=\frac{ \left(\frac{q''_A}{1-q''_A}\right)^{|F|} e^{\theta\psi P_2(F,H)}}
{\sum_{J\subseteq F} \left(\frac{q''_A}{1-q''_A}\right)^{|J|}e^{\theta\psi P_2(J,H)}}\cdot \mathbf 1_{H\in \cH_F}\, .
\end{align}

We note that
\[
e^{\theta\psi P_2(F,H)}= 1+ \theta\psi P_2(F,H) + \frac{\theta^2\psi^2}{2}  P_2(F,H)^2 + O(\psi^3\Delta^3)\, .
\] 
Moreover the denominator in~\eqref{eqPFGH3} is equal to
\begin{align}
\sum_{J\subseteq F} \left(\frac{q''_A}{1-q''_A}\right)^{|J|}(1+ O(\psi \Delta\mathbf 1_{J\neq \emptyset})) 
=
(1-q_A'')^{-|F|} + O(\psi \Delta q)
=(1+O(\psi \Delta q) ) (1-q_A'')^{-|F|} 
\end{align}
since $\sum_{J\subseteq F, J\neq\emptyset} (q''_A/(1-q''_A))^{|J|}=O(q)$ and $P_2(J,H)=O(\Delta)$ for $J\subseteq F$.

It follows that
\[
\P(F\subseteq \bG \mid \bH=H)=(1+O(\psi^3\Delta^3 + \psi\Delta q)) q_A''^{|F|} 
\left(1+ \theta\psi P_2(F,H) + \frac{\theta^2\psi^2}{2}  P_2(F,H)^2)\right)\cdot \mathbf 1_{H\in \cH_F}\, ,
\]
and so, since $\psi\Delta q=O(\psi^3\Delta^3)$, 
\begin{align}
\P(F\subseteq \bG)
&= 
(1+O(\psi^3\Delta^3)) 
q_A''^{|F|} \E\left[
\left(1+ \theta\psi P_2(\bH,F) + \frac{\theta^2\psi^2}{2}  P_2(\bH,F)^2\right) \mathbf1_{\bH\in\cH}\right]\, ,\nonumber\\
&=(1+O(\psi^3\Delta^3)) 
q_A''^{|F|}\left( \E\left[
1+ \theta\psi P_2(\bH,F) + \frac{\theta^2\psi^2}{2}  P_2(\bH,F)^2)\right] + O(\P(\bH\notin\cH_F))\right)\, ,
\label{eqPFinGPrecise}
\end{align}
since $\theta\psi P_2(\bH, F)=O(1)$.

We now turn to bounding $\P(\bH\notin\cH_F)$.
With a view to apply Lemma~\ref{lemDegBdLocal} ,
first note that $\nu_{\bm q''_A, \cD_\emptyset}^\theta$ may be identified with the measure $\nu^f_{\bm r, \cD_{\emptyset}}$ where $\bm r=(q_A'', 0)$, and we recall that
\[
 \cD_{\emptyset}=\left\{G\subseteq \binom{A}{2}\cup \binom{B}{2}: \Delta(G)\leq \Delta, |G_A|, |G_B|\leq K \right\}\, ,
\]
and $f:\cD_{\emptyset}\to \R$ is given by $f(S)= \theta \psi P_2(S)$.
Recall
 from~\eqref{eqPsiDef} that $\psi\leq n\lam^3$ and so $f$ is $(2n\lam^3)$-local.
 Moreover, $r_A=q_A''\leq 2q_A$. Now note that if $\Delta(\bH \cup F)>\Delta$ then $\Delta(\bH)>\Delta-O(1)>\Delta/2$ which, by Lemma~\ref{lemDegBdLocal}, occurs with probability at most $n^2 e^{-\Delta/2}$. 

We deduce that $\P(\bH\notin\cH_F)=O(n^2e^{-\Delta/2})= O(\psi^3\Delta^3)$ and so returning to~\eqref{eqPFinGPrecise} we have
\begin{align*}
\P(F\subseteq \bG)&=
(1+O(\psi^3\Delta^3) q_A''^{|F|} 
\left(1+ \theta\psi \E( P_2(\bH,F) )+ \frac{\theta^2\psi^2}{2}  \E(P_2(\bH,F)^2)\right)\, .
\end{align*}
Finally we note that $P_2(\bH,F)=P_2(\bG,F)$.
\end{proof}

\section{Quantifying the effect of unbalanced partitions}\label{secAppCalc}

In this section we prove Lemma~\ref{lemZABapprox1super}.
Fixing $(A,B)\in \Pi_{\textup{mod},\lam}$,
our strategy will be to consider the expression in Lemma~\ref{lemmaZABabapproxsuper} to the same expression for a perfectly balanced partition 
(i.e., the expression with all instances of $a,b$ replaced by $n/2$). 
Suppose then than $a=n/2-k, b=n/2+k$. We have that $k=o(n^{1/2+1/14})$ since $\lam\geq \frac{13}{14}\sqrt{\frac{\log n}{n}}$ and $(A,B)$ is $\lam$-moderately balanced. 
We consider the parameters $q_A, q_A', q_A'', \mu_A$ and compare them to the expressions obtained by replacing $a,b$ by $n/2$ that is $q_0, q_1, q_2, \mu$.
It will be useful to first record some estimates. Recall that
\[
q\sim\max\{q_A, q_B\}=o(n^{-13/14})\, .
\]
We note also that 
\begin{align}\label{eqqAfirstapprox}
q_A=\lam e^{-\lam^2b}(1+O(q))\, ,
\end{align}
and so
\begin{align}\label{eqqAq0Est}
\frac{q_A}{q_0}= e^{-\lam^2k}(1+O(q)).
\end{align}
Similarly, letting 
\[f(\lam):= -\lam^2+2\lam^3-7\lam^4/2\, ,\] we have
\begin{align}\label{eqqAPrimeq1Est}
\frac{q'_A}{q_1}= e^{f(\lam)k}(1+O(q)).
\end{align}
Next we compare $\mu_A$ to $\mu$. First note that for $\ell\in\{2,3\}$ we have
\begin{align}\label{eqBinomComp}
\frac{\binom{a}{\ell}}{\binom{n/2}{\ell}} = (1+O(1/n))\left(1-\frac{2k}{n} \right)^\ell\, ,
\end{align}
and so by~\eqref{eqqAPrimeq1Est},
\begin{align}\label{eqmuAmuEst}
\frac{\mu_A}{\mu}=\frac{\binom{a}{2}q_A'e^{2\lam^3a(aq_A+bq_B)}}{\binom{n/2}{2}q_1e^{\lam^3n^2q_0}}=(1+O(q))\left(1-\frac{2k}{n}\right)^2e^{f(\lam)k + 2\lam^3a(aq_A+bq_B) - \lam^3n^2q_0}\, .
\end{align}
With the exponent of the above expression in mind we note that  by~\eqref{eqqAq0Est}
\begin{align}
2a^2q_A - n^2q_0/2 =\frac{n^2q_0}{2}\left[ \left(1-\frac{2k}{n} \right)^2e^{-\lam^2k}(1+O(q)) -1 \right] = \frac{n^2q_0}{2}\left[-\frac{4k}{n}-\lam^2k + O(\lam^4k^2 + q) \right] 
\end{align}
where for the second inequality we used that $k/n=O(\lam^2k)$ and $\lam^2k=o(1)$. 
Similarly
\begin{align}
2abq_B - n^2q_0/2=\frac{n^2q_0}{2}\left[\lam^2k + O(\lam^4k^2 + q) \right] \, .
\end{align}
Returning to~\eqref{eqmuAmuEst} we conclude that
\begin{align}\label{eqmuAmuEst2}
\frac{\mu_A}{\mu}
&=(1+O(q + \lam^4k^2))\left(1-\frac{2k}{n}\right)^2e^{(f(\lam)-2nq_0\lam^3)k +O(n^2q\lam^7k^2+n^2q^2\lam^3)}\\
&=(1+O(q + \lam^4k^2))\left(1-\frac{4k}{n}\right)e^{(f(\lam)-2nq_0\lam^3)k}\, .
\end{align}

Next we compare $q''_A$ and $q_2$. For this comparison it will be important not to incur a $(1+O(q))$ multiplicative error as we did above since such an error is non-negligible when considering the expression $(1-q_A'')^{-\binom{a}{2}}$ (see Claim~\ref{claimqAdprime} below). Instead we compare $r_A:=q''_A/(1-q''_A)$ and $r=q_2/(1-q_2)$.
\begin{align}\label{eqrArEst}
\frac{r_A}{r} 
= \exp\left\{f(\lam)k+4\lam^3(\mu_B-\mu)\right\}\, .
\end{align}
By the analogue of \eqref{eqmuAmuEst} for $\mu_B$ we have
\begin{align}\label{eqmuBminusmu}
\mu_B-\mu = \mu\left( \frac{4k}{n}- (f(\lam)-2nq_0\lam^3)k +O(q+\lam^4k^2)\right)\, .
\end{align}

Returning to~\eqref{eqrArEst} we conclude,  that
\[
\frac{r_A}{r} 
= (1+O(\lam^3\mu(q+\lam^4k^2)))\exp\left\{\left[f(\lam)+4\lam^3 \mu\left( \frac{4}{n}- f(\lam)+2nq_0\lam^3\right)\right]k\right\}\, .
\]
The precise form of what appears in the exponent will not be important and so, noting that $\mu=O(n^2q)$, we write
\begin{align}\label{eqrArEstFinal}
\frac{r_A}{r} 
= (1+O(n^2q^2\lam^3+n^2q\lam^7k^2))\exp\left\{g(n,\lam)k\right\}\, ,
\end{align}
where we simply record that $g(n,\lam)=O(\lam^2)$.
In particular, we have
\begin{align}\label{eqqAdPrimeq2Est2}
\frac{q''_A}{q_2} 
=  (1+O(q+n^2q\lam^7k^2))\exp\left\{g(n,\lam)k\right\}\, ,
\end{align}
but we reiterate that it will be important to have the added accuracy of the estimate~\eqref{eqrArEstFinal}.

With these estimates in hand, we return to the expression in Lemma~\ref{lemmaZABabapproxsuper}.
\begin{claim}\label{claimqAdprime}
\[
(1-q_A'')^{-\binom{a}{2}} (1-q_B'')^{-\binom{b}{2}}\sim (1-q_2)^{-\binom{n/2}{2}}\exp\left\{ O(n^2q\lam^4k^2) \right\}\, .
\]
\end{claim}
\begin{proof}
First note that 
\begin{align}\label{eqqAdPrimeTaylor}
(1-q_A'')^{-\binom{a}{2}} (1-q_B'')^{-\binom{b}{2}} &= (1+r_A)^{\binom{a}{2}} (1+r_B)^{\binom{b}{2}}\\
&\sim \exp\left\{ (r_A-r_A^2/2)\binom{a}{2}+(r_B-r_B^2/2)\binom{b}{2}\right\}\, .
\end{align}
By~\eqref{eqrArEstFinal}, letting $g=g(n,\lam)$, 
\begin{align}
\frac{r_A\binom{a}{2}}{r\binom{n/2}{2}}
&= (1+O(n^2q^2\lam^3+n^2q\lam^7k^2))e^{gk}\left(1-\frac{2k}{n} \right)\left(1-\frac{2k}{n-2} \right)\\
&=(1+O(n^2q^2\lam^3+k^2/n^2))e^{gk}\left(1-\frac{2k}{n}-\frac{2k}{n-2} \right)\, .
\end{align}
We conclude from the analogous expression for $B$ (obtained by replacing $k$ by $-k$) that 
\[
\frac{r_A\binom{a}{2}+r_B\binom{b}{2}}{r\binom{n/2}{2}} = 2 + O(n^2q^2\lam^3+\lam^4k^2)\, ,
\]
since the terms that are linear in $k$ cancel.
Similarly 
\[
\frac{r^2_A\binom{a}{2}+r^2_B\binom{b}{2}}{r^2\binom{n/2}{2}} = 2 + O(n^2q^2\lam^3+\lam^4k^2)\, .
\]
Returning to~\eqref{eqqAdPrimeTaylor} we then have
\[
(1-q_A'')^{-\binom{a}{2}} (1-q_B'')^{-\binom{b}{2}}\sim  \exp\left\{ 2(r-r^2/2)\binom{n/2}{2} + O(n^4q^3\lam^3+n^2q\lam^4k^2)\right\}\, .
\]
The claim follows by noting that $n^4q^3\lam^3=o(1)$.
\end{proof}

\begin{claim}
\[
\exp\left\{ \frac{1}{2} \lam^3 a^3bq_A''^2 + \frac{1}{2} \lam^3 b^3aq_B''^2\right\}\sim \exp\left\{  \lam^3 \left(\frac{n}{2}\right)^4q_1^2(1+\lam^3n^2q_0) +O(n^2q\lam^4k^2) \right\}\, .
\]
\end{claim}
\begin{proof}

By~\eqref{eqqAdPrimeq2Est2} we have
\begin{align}
\frac{ a^3bq_A''^2}{\left(\frac{n}{2}\right)^4q^2_2} 
&= (1+O(q+\lam^4k^2))\left(1-\frac{4k}{n}\right)e^{2g(n,\lam)k}\, .
\end{align}
By the analogous expression with $A,B$ swapped (thus replacing $k$ with $-k$) and recalling that $g(n,\lam)=O(\lam^2)$,
we conclude that 
\begin{align}
\frac{  \tfrac{1}{2}a^3bq_A''^2 + \tfrac{1}{2} b^3aq_B''^2}{\left(\frac{n}{2}\right)^4q^2_2} 
&= 1 + O(q+\lam^4k^2).
\end{align}
Noting that $\lam^3 n^4q^3=o(1)$ and $ n^4q^2\lam^7k^2=O(n^2q\lam^4k^2)$ we conclude that
\[
\exp\left\{ \frac{1}{2} \lam^3 a^3bq_A''^2 + \frac{1}{2} \lam^3 b^3aq_B''^2\right\}\sim \exp\left\{  \lam^3 \left(\frac{n}{2}\right)^4q_2^2 +O(n^2q\lam^4k^2) \right\}\, .
\]
The claim follows by noting that
\begin{align}
 \lam^3 \left(\frac{n}{2}\right)^4q_2^2
 &=
 \lam^3 \left(\frac{n}{2}\right)^4q_1^2e^{8\lam^3\mu}+o(1)\\
 &=
  \lam^3 \left(\frac{n}{2}\right)^4q_1^2(1+8\lam^3\mu)+o(1)\\
  &= \lam^3 \left(\frac{n}{2}\right)^4q_1^2(1+\lam^3n^2q_0) +o(1)\, .
\end{align}
\end{proof}

\begin{claim}
\[
\exp\left\{-4\lam^3\mu_A\mu_B \right\}
\sim 
\exp\left\{ -  \lam^3 \left(\frac{n}{2}\right)^4q_1^2(1+{2\lam^3n^2q_0}) +O(n^2q\lam^4k^2) \right\}\, .
\]
\end{claim}
\begin{proof}
By~\eqref{eqmuAmuEst2}
\begin{align}
\frac{\mu_A\mu_B}{\mu^2}
&=(1+O(q + \lam^4k^2)).
\end{align}
Since $\lam^3\mu^2q=o(1)$ (since $\mu=O(n^2q)$) and $\lam^7\mu^2k^2=O(n^2q\lam^4k^2)$ 
we conclude that
\[
\exp\left\{-4\lam^3\mu_A\mu_B \right\}
\sim 
\exp\left\{ -4\lam^3\mu^2+ O(n^2q\lam^4k^2)  \right\}\, .
\]
Now, 
\begin{align}
4\lam^3\mu^2
= \lam^3 \left(\frac{n}{2}\right)^4q_1^2e^{2\lam^3n^2q_0} +o(1)
=
 \lam^3 \left(\frac{n}{2}\right)^4q_1^2(1+{2\lam^3n^2q_0}) +o(1)\, .
\end{align}
\end{proof}

By the previous two claims we have

\begin{align}
\exp\left\{ \frac{1}{2} \lam^3 a^3bq_A''^2 + \frac{1}{2} \lam^3 b^3aq_B''^2-4\lam^3\mu_A\mu_B\right\}
&\sim \exp\left\{  -\lam^3 \left(\frac{n}{2}\right)^4q_1^2\cdot\lam^3n^2q_0 +O(n^2q\lam^4k^2) \right\}\\
&\sim
\exp\left\{  -\frac{1}{16}\lam^6 n^6q_0^3 +O(n^2q\lam^4k^2) \right\}
\end{align}

Similar calculations show that if $r_1, r_2, \ell_1, \ell_2\in \Z$ are fixed then
\begin{align}
\frac{  \tfrac{1}{2}a^{r_1}b^{\ell_1}q_A^{r_2}q_B^{\ell_2} + \tfrac{1}{2}b^{r_1}a^{\ell_1}q_B^{r_2}q_A^{\ell_2}}{\left(\frac{n}{2}\right)^{r_1+\ell_1}q_0^{r_2+\ell_2}} 
&= 1 + O(q+\lam^4k^2)\, ,
\end{align}
and so
\[
\exp\left\{ \frac{1}{4}\lam^6 a^3b^2 q_A^2 + \frac{1}{4} \lam^6 b^3a^2q_B^2\right\}
\sim 
\exp\left\{ \frac{1}{2} \lam^6 \left( \frac{n}{2}\right)^5q_0^2 +O(n^2q\lam^4k^2)\right\}\, .
\]

\[
\exp\left\{\frac{3}{2}\lam^6a^4b^2q_A^3 +\frac{3}{2}\lam^6b^4a^2q_B^3 \right\}
\sim 
\exp\left\{3\lam^6\left( \frac{n}{2}\right)^6q_0^3+O(n^2q\lam^4k^2)\right\}\, .
\]

\[
\exp\left\{-\frac{1}{6}a^3q_A^3 -\frac{1}{6}b^3q_B^3 \right\}
\sim 
\exp\left\{-\frac{1}{3} \left( \frac{n}{2}\right)^3q_0^3+ O(n^2q\lam^4k^2)\right\}\, .
\]
and
\begin{align}
\exp&\left\{
\lam^4 ab\left(\frac{1}{4}abq_Aq_B-  \frac{2}{3}(aq_A+bq_B)^3 - 2(aq_A+bq_B)^2 \right)
\right\}\\
&\sim 
\exp\left\{
\lam^4 \left( \frac{n}{2}\right)^2\left(\frac{1}{4}\left( \frac{n}{2}\right)^2q_0^2-  \frac{2}{3}(nq_0)^3 - 2(nq_0)^2 \right)+O(n^2q\lam^4k^2)
\right\} \,.
\end{align}
Putting everything together yields~\eqref{eqZabCase1super}.

We now turn to~\eqref{eqnuABapproxsuper}. Suppose now that $(A,B)\in \Pi_{\textup{strong}}$ so that $k\leq 10(n \log n)^{1/4}$.
Let $\bm q''=(q_A'', q_B'')$, let $\psi_A=\lam^3b$, $\psi_B=\lam^3a$ and $\Psi:\cD\to\R$ be such that $\Psi(S,T)= \psi_A P_2(S) + \psi_B P_2(T)$. The measure
$\nu_{\bm q'',\cD}^\Psi$ is the measure associated to the random graph $ G(A,q_A'', \psi_A) \times G(B,q_B'', \psi_B)$. We first show that $D_\text{KL}( \nu_{\bm q'',\cD}^\Psi\parallel \nu_{A,B,\lam})=o(1)$. Let $(S,T)\in\cD$.
Recall that
\[
\nu_{\bm q'',\cD}^\Psi(S,T) = \frac{\left( \frac{q_A''}{1-q_A''}\right)^{|S|}\left( \frac{q_A''}{1-q_A''}\right)^{|T|}e^{\psi_AP_2(S)+\psi_AP_2(S)}}{Z_A''\E_{\nu_{\bm q''_A, \cD}}  e^{\psi_AP_2(S)}\cdot Z_B''\E_{\nu_{\bm q''_B, \cD}} e^{\psi_AP_2(T)}}
\]
where $Z_A''$ is as in~\eqref{eqZAdprimedef} and $\bm q''_A=(q''_A,0), \bm q''_B=(0,q''_B)$.

 By~\eqref{eqZboxTestimateSuper}
\begin{align}
\log \left( \frac{\nu_{\bm q'',\cD}^\Psi(S,T)}{\nu_{A,B,\lam}(S,T)} \right) &=
4\lam^3\mu_B|S| +
4\lam^3\mu_A |T| - 4\lam^3|S||T|\\ 
&+(P_3(S \boxempty T)+S_3(S \boxempty T)-C_4(S \boxempty T)+4P_2(S \boxempty T))\lam^4\\
&+\log \left(\frac{Z_{A,B}}{(1+\lam)^{ab}} \right) 
-\log\left( 
Z_A''\E_{\nu_{\bm q''_A, \cD}}  e^{\psi_AP_2(S)}\cdot Z_B''\E_{\nu_{\bm q''_B, \cD}} e^{\psi_AP_2(T)}
\right)
+o(1)\, .
\end{align}

On the other hand, by~\eqref{eqZABbarnuexp},~\eqref{eqExpectationProd}, Claim~\ref{claim:P3S3} and \eqref{eq:P2momentforTV} we have
\begin{align}
\log \left(\frac{Z_{A,B}}{(1+\lam)^{ab}} \right) 
&= -4\lam^3\mu_A\mu_B+\lam^4 ab\left(\frac{1}{4}ab q_Aq_B - \frac{2}{3}(aq_A+bq_B)^3 - 2(aq_A+bq_B)^2 \right)\\
&\phantom{=} +\log \left(
Z_A''\E_{\nu_{\bm q''_A, \cD}}  e^{\psi_AP_2(S)}\cdot Z_B''\E_{\nu_{\bm q''_B, \cD}} e^{\psi_AP_2(T)}
\right)
+o(1)\, .
\end{align}
Thus
\begin{align}
\log \left( \frac{\nu_{\bm q'',\cD}^\Psi(S,T)}{\nu_{A,B,\lam}(S,T)} \right) 
&=
4\lam^3\mu_B|S| +
4\lam^3\mu_A |T| - 4\lam^3|S||T|-4\lam^3\mu_A\mu_B\\ 
&+(P_3(S \boxempty T)+S_3(S \boxempty T)-C_4(S \boxempty T)+4P_2(S \boxempty T))\lam^4\\
&+\lam^4 ab\left(\frac{1}{4}ab q_Aq_B - \frac{2}{3}(aq_A+bq_B)^3 - 2(aq_A+bq_B)^2 \right)
+o(1)\, .
\end{align}
Calculating as in the proof of Claim~\ref{claim:P3S3}, we have
\begin{align}
\E_{\nu_{\bm q'',\cD}^\Psi}&\left[
(P_3(S \boxempty T)+S_3(S \boxempty T)-C_4(S \boxempty T)+4P_2(S \boxempty T))\lam^4
\right]\\
&=
-\lam^4 ab\left(\frac{1}{4}ab q_Aq_B - \frac{2}{3}(aq_A+bq_B)^3 - 2(aq_A+bq_B)^2 \right)+o(1)\, .
\end{align}
Moreover, by an application of~\eqref{eqlemfixedconfiglocal} of  Lemma~\ref{lemfixedconfiglocal} 
we have
\[
\E_{\nu_{\bm q'',\cD}^\Psi}
\left[ 
4\lam^3\mu_B|S| +
4\lam^3\mu_A |T| - 4\lam^3|S||T|-4\lam^3\mu_A\mu_B\right]
= o(1)\, .
\]

It follows that
\[
D_\text{KL}( \nu_{\bm q'',\cD}^\Psi\parallel \nu_{A,B,\lam})= \E_{\nu_{\bm q'',\cD}^\Psi}\log \left( \frac{\nu_{\bm q'',\cD}^\Psi(S,T)}{\nu_{A,B,\lam}(S,T)} \right) =o(1)\, .
\]
Let $\nu_{V,q,\psi}$ denote the measure associated to the random graph $G(V,q, \psi)$ so that $ \nu_{\bm q'',\cD}^\Psi = \nu_{A,q_A'',\psi_A}\times \nu_{B,q_B'',\psi_B}$. 
We note that
\[
D_\text{KL}( \nu_{\bm q'',\cD}\parallel  \nu_{A,q_2,\psi}\times \nu_{B,q_2,\psi}) = 
D_\text{KL}(  \nu_{A,q_A'',\psi_A}\parallel  \nu_{A,q_2,\psi}) + D_\text{KL}(  \nu_{B,q_B'',\psi_B}\parallel  \nu_{B,q_2,\psi}) \, .
\]
We now show that the RHS is $o(1)$ thereby completing the proof. By symmetry it suffices to show that $D_\text{KL}(  \nu_{A,q_A'',\psi_A}\parallel  \nu_{A,q_2,\psi})=o(1)$. Recall that $r_A=q_A''/(1-q_A'')$, $r=q_2/(1-q_2)$ and $\psi_A=\lam^3b, \psi=\lam^3n/2$. Let 
\[
\Xi_A = \sum_{S\subseteq A : S\in \cD} r_A^{|S|}e^{\psi_AP_2(S)} \quad \text{ and }\quad \Xi = \sum_{S\subseteq A : S\in \cD} r^{|S|}e^{\psi P_2(S)} \, .
\]
Then 
\begin{align}
D_\text{KL}(  \nu_{A,q_A'',\psi_A}\parallel  \nu_{A,q_2,\psi})
& = \E_{ \nu_{A,q_A'',\psi_A}}\log\left(\frac{r_A^{|S|}e^{\psi_AP_2(S)}}{\Xi_A} \cdot \frac{\Xi}{r^{|S|}e^{\psi P_2(S)}} \right)\\
&= \log(\Xi/\Xi_A) + \log(r_A/r) \E_{ \nu_{A,q_A'',\psi_A}}|S| + (\psi_A-\psi) \E_{ \nu_{A,q_A'',\psi_A}}P_2(S)\, . \label{eqDKLfinalEst}
\end{align}
First we estimate $\log(\Xi/\Xi_A)$.
By (the proof of) Claim~\ref{claimERGApartfun}
\[
\Xi_A\sim (1+r_A)^{\binom{a}{2}}   \exp\left\{ \frac{1}{2}\psi_Aa^3q_A''^2 + \frac{1}{4} \psi_A^2 a^3 q_A^2 +\frac{3}{2} \psi_A^2 a^4 q_A^3 -\frac{1}{6}a^3q_A^3 \right\}\, ,
\]
and
\[
\Xi\sim (1+r)^{\binom{a}{2}}  \exp\left\{ \frac{1}{2}\psi a^3q^2_2 + \frac{1}{4} \psi^2 a^3 q_0^2 + \frac{3}{2}\psi^2 a^4 q_0^3 -\frac{1}{6}a^3q_0^3 \right\}\, .
\]
Note that by~\eqref{eqrArEstFinal},
\begin{align}
\binom{a}{2}\log \left(\frac{1+r}{1+r_A}\right)
&
= \binom{a}{2}\left(r-r_A -r^2/2+r_A^2/2 \right) +o(1)\\
&=  \binom{a}{2}\left(-r_Agk +r_A^2gk \right) +O(n^2q\lam^4k^2)+o(1)\\
&=- \binom{a}{2} r_Agk+o(1)\, ,
\end{align}
where for the last inequality we recall that $g=g(n,\lam)=O(\lam^2), k=\tilde O(n^{1/4})$ and $r,q=o(n^{-13/14})$.
By \eqref{eqqAdPrimeq2Est2} 
\begin{align}
\psi a^3q^2_2 - \psi_A a^3q_A''^2= -\psi a^3q^2_2 (2gk +O(q+\lam^4k^2+k/n) )=o(1)\, .
\end{align}
Similarly  $\psi^2 a^3 q_0^2 -  \psi^2_A a^3q_A^2=o(1)$,  $\psi^2 a^4 q_0^3 -  \psi^2_A a^4q_A^3=o(1)$ and $a^3q_0^3- a^3q_A^3=o(1)$.
We conclude that
\begin{align}\label{eqlogXiXiA}
\log(\Xi/\Xi_A) = - \binom{a}{2} r_Agk+o(1)\, .
\end{align}
Returning to~\eqref{eqDKLfinalEst} we next estimate $ \log(r_A/r) \E_{ \nu_{A,q_A'',\psi_A}}|S|$.
First note that by~\eqref{eqrArEstFinal},  
\[
\log(r_A/r)= gk + O(\lam^4k^2 + n^2q^2\lam^3)\, .
\]
By Lemma~\ref{lemfixedconfiglocal}
\[
 \E_{ \nu_{A,q_A'',\psi_A}}|S| = \binom{a}{2}q''_A(1+O(n\Delta \lam^3))=\binom{a}{2}r_A(1+O(n\Delta \lam^3))\, ,
\]
where for the last inequality we used that $q''_A=r_A(1+O(q))$ and $q=O(n\Delta \lam^3)  $ (recall that $\Delta=50\max\{qn, \log n\}$). 
It follows that 
\begin{align}\label{eqlogrAr}
 \log(r_A/r) \E_{ \nu_{A,q_A'',\psi_A}}|S| = \binom{a}{2} r_Agk+o(1)\, .
\end{align}
Finally by Lemma~\ref{lemfixedconfiglocal} we have 
\[
 (\psi_A-\psi) \E_{ \nu_{A,q_A'',\psi_A}}P_2(S) = \psi \cdot O(k/n)\cdot O(n^3q^2)=o(1)\, .
\]
Combining this with~\eqref{eqlogrAr},~\eqref{eqlogXiXiA} and~\eqref{eqDKLfinalEst} completes the proof.


\begin{thebibliography}{10}

\bibitem{alon2014counting}
N.~Alon, J.~Balogh, R.~Morris, and W.~Samotij.
\newblock Counting sum-free sets in {A}belian groups.
\newblock {\em Israel Journal of Mathematics}, 199(1):309--344, 2014.

\bibitem{alon2014refinement}
N.~Alon, J.~Balogh, R.~Morris, and W.~Samotij.
\newblock A refinement of the {C}ameron--{E}rd{\H{o}}s conjecture.
\newblock {\em Proceedings of the London Mathematical Society}, 108(1):44--72, 2014.

\bibitem{balogh2017number}
J.~Balogh, H.~Liu, and M.~Sharifzadeh.
\newblock The number of subsets of integers with no k-term arithmetic progression.
\newblock {\em International Mathematics Research Notices}, 2017(20):6168--6186, 2017.

\bibitem{balogh2015independent}
J.~Balogh, R.~Morris, and W.~Samotij.
\newblock Independent sets in hypergraphs.
\newblock {\em Journal of the American Mathematical Society}, 28(3):669--709, 2015.

\bibitem{balogh2018method}
J.~Balogh, R.~Morris, and W.~Samotij.
\newblock The method of hypergraph containers.
\newblock In {\em Proceedings of the International Congress of Mathematicians: Rio de Janeiro 2018}, pages 3059--3092. World Scientific, 2018.

\bibitem{balogh2016typical}
J.~Balogh, R.~Morris, W.~Samotij, and L.~Warnke.
\newblock The typical structure of sparse ${K}_{r+1}$-free graphs.
\newblock {\em Transactions of the American Mathematical Society}, 368(9):6439--6485, 2016.

\bibitem{barvinok2015computing}
A.~Barvinok.
\newblock Computing the partition function for cliques in a graph.
\newblock {\em Theory of Computing}, 11(13):339--355, 2015.

\bibitem{barvinok2017combinatorics}
A.~Barvinok.
\newblock {\em Combinatorics and Complexity of Partition Functions}, volume~30 of {\em Algorithms and Combinatorics}.
\newblock Springer, 2017.

\bibitem{bhamidi2008mixing}
S.~Bhamidi, G.~Bresler, and A.~Sly.
\newblock Mixing time of exponential random graphs.
\newblock In {\em 2008 49th Annual IEEE Symposium on Foundations of Computer Science}, pages 803--812. IEEE, 2008.

\bibitem{bollobas1984evolution}
B.~Bollob{\'a}s.
\newblock The evolution of random graphs.
\newblock {\em Transactions of the American Mathematical Society}, 286(1):257--274, 1984.

\bibitem{bollobas1998random}
B.~Bollob{\'a}s.
\newblock {\em Random graphs}.
\newblock Springer, 1998.

\bibitem{bollobas2001scaling}
B.~Bollob{\'a}s, C.~Borgs, J.~T. Chayes, J.~H. Kim, and D.~B. Wilson.
\newblock The scaling window of the 2-{SAT} transition.
\newblock {\em Random Structures \& Algorithms}, 18(3):201--256, 2001.

\bibitem{borgs2001birth}
C.~Borgs, J.~T. Chayes, H.~Kesten, and J.~Spencer.
\newblock The birth of the infinite cluster: finite-size scaling in percolation.
\newblock {\em Communications in Mathematical Physics}, 224:153--204, 2001.

\bibitem{borgs2012tight}
C.~Borgs, J.~T. Chayes, and P.~Tetali.
\newblock Tight bounds for mixing of the {S}wendsen--{W}ang algorithm at the {P}otts transition point.
\newblock {\em Probability Theory and Related Fields}, 152(3-4):509--557, 2012.

\bibitem{borgs1989unified}
C.~Borgs and J.~Z. Imbrie.
\newblock A unified approach to phase diagrams in field theory and statistical mechanics.
\newblock {\em Communications in mathematical physics}, 123(2):305--328, 1989.

\bibitem{brydges1984short}
D.~C. Brydges.
\newblock A short course on cluster expansions.
\newblock {\em Les Houches}, (PART I), 1984.

\bibitem{chatterjee2013estimating}
S.~Chatterjee and P.~Diaconis.
\newblock Estimating and understanding exponential random graph models.
\newblock {\em The Annals of Statistics}, 41(5):2428--2461, 2013.

\bibitem{csiszar2011information}
I.~Csisz{\'a}r and J.~K{\"o}rner.
\newblock {\em Information theory: coding theorems for discrete memoryless systems}.
\newblock Cambridge University Press, 2011.

\bibitem{dobrushin1977central}
R.~Dobrushin and B.~Tirozzi.
\newblock The central limit theorem and the problem of equivalence of ensembles.
\newblock {\em Communications in Mathematical Physics}, 54:173--192, 1977.

\bibitem{duminil2017lectures}
H.~Duminil-Copin.
\newblock Lectures on the {I}sing and {P}otts models on the hypercubic lattice.
\newblock In {\em PIMS-CRM Summer School in Probability}, pages 35--161. Springer, 2017.

\bibitem{erdosasymptotic}
P.~Erd\H{o}s, D.~Kleitman, and B.~Rothschild.
\newblock Asymptotic enumeration of {$K_n$}-free graphs.
\newblock {\em Colloquio Internazionale sulle Teorie Combinatorie (Rome, 1973)}, (17):19--27, 1973.

\bibitem{erdos1960evolution}
P.~Erd{\H{o}}s and A.~R{\'e}nyi.
\newblock On the evolution of random graphs.
\newblock {\em Publ. Math. Inst. Hung. Acad. Sci}, 5(1):17--60, 1960.

\bibitem{faris2010combinatorics}
W.~G. Faris.
\newblock Combinatorics and cluster expansions.
\newblock {\em Probability Surveys}, 7:157--206, 2010.

\bibitem{FernandezProcacci}
R.~Fern{\'a}ndez and A.~Procacci.
\newblock Cluster expansion for abstract polymer models. new bounds from an old approach.
\newblock {\em Communications in Mathematical Physics}, 274(1):123--140, 2007.

\bibitem{frank1986markov}
O.~Frank and D.~Strauss.
\newblock Markov graphs.
\newblock {\em Journal of the American Statistical Association}, 81(395):832--842, 1986.

\bibitem{galvin2022zeroes}
D.~Galvin, G.~McKinley, W.~Perkins, M.~Sarantis, and P.~Tetali.
\newblock On the zeroes of hypergraph independence polynomials.
\newblock {\em Combinatorics, Probability and Computing}, 33(1):65--84, 2024.

\bibitem{green2004cameron}
B.~Green.
\newblock The {C}ameron--{E}rd{\H{o}}s conjecture.
\newblock {\em Bulletin of the London Mathematical Society}, 36(6):769--778, 2004.

\bibitem{helmuth2023finite}
T.~Helmuth, M.~Jenssen, and W.~Perkins.
\newblock Finite-size scaling, phase coexistence, and algorithms for the random cluster model on random graphs.
\newblock In {\em Annales de l'Institut Henri Poincare (B) Probabilites et statistiques}, volume~59, pages 817--848. Institut Henri Poincar{\'e}, 2023.

\bibitem{HelmuthAlgorithmic2}
T.~Helmuth, W.~Perkins, and G.~Regts.
\newblock Algorithmic {P}irogov-{S}inai theory.
\newblock {\em Probability Theory and Related Fields}, 176:851--895, 2020.

\bibitem{jain2021approximate}
V.~Jain, W.~Perkins, A.~Sah, and M.~Sawhney.
\newblock Approximate counting and sampling via local central limit theorems.
\newblock In {\em Proceedings of the 54th Annual ACM SIGACT Symposium on Theory of Computing (STOC)}, pages 1473--1486, 2022.

\bibitem{janson1987uczak}
S.~Janson, T.~{\L}uczak, and A.~Ruci\'nski.
\newblock An exponential bound for the probability of nonexistence of a specified subgraph in a random graph.
\newblock {\em Random graphs' 87, Proceedings, Pozn\'an, 1987, eds. M. Karon\'nski, J. Jaworski and A.Ruci\'nski}, pages 73--87, 1987.

\bibitem{janson1988exponential}
S.~Janson, T.~Luczak, and A.~Rucinski.
\newblock {\em An exponential bound for the probability of nonexistence of a specified subgraph in a random graph}.
\newblock Institute for Mathematics and its Applications (USA), 1988.

\bibitem{jenssen2020homomorphisms}
M.~Jenssen and P.~Keevash.
\newblock Homomorphisms from the torus.
\newblock {\em Advances in Mathematics}, 430:109212, 2023.

\bibitem{JKP2}
M.~Jenssen, P.~Keevash, and W.~Perkins.
\newblock Algorithms for \#{BIS}-hard problems on expander graphs.
\newblock {\em SIAM Journal on Computing}, 49(4):681--710, 2020.

\bibitem{jenssen2020independent}
M.~Jenssen and W.~Perkins.
\newblock Independent sets in the hypercube revisited.
\newblock {\em Journal of the London Mathematical Society}, 102(2):645--669, 2020.

\bibitem{jenssen2022independent}
M.~Jenssen, W.~Perkins, and A.~Potukuchi.
\newblock Independent sets of a given size and structure in the hypercube.
\newblock {\em Combinatorics, Probability and Computing}, 31(4):702--720, 2022.

\bibitem{jenssen2024lower}
M.~Jenssen, W.~Perkins, A.~Potukuchi, and M.~Simkin.
\newblock Lower tails for triangles inside the critical window.
\newblock {\em arXiv preprint arXiv:2411.18563}, 2024.

\bibitem{jenssen2024sampling}
M.~Jenssen, W.~Perkins, A.~Potukuchi, and M.~Simkin.
\newblock Sampling, counting, and large deviations for triangle-free graphs near the critical density.
\newblock In {\em 2024 IEEE 65th Annual Symposium on Foundations of Computer Science (FOCS)}, pages 151--165. IEEE, 2024.

\bibitem{kolaitis1987k}
P.~G. Kolaitis, H.-J. Pr{\"o}mel, and B.~L. Rothschild.
\newblock ${K}_{\ell+1}$-free graphs: asymptotic structure and a 0-1 law.
\newblock {\em Transactions of the American Mathematical Society}, 303(2):637--671, 1987.

\bibitem{kotecky1986cluster}
R.~Koteck\'{y} and D.~Preiss.
\newblock Cluster expansion for abstract polymer models.
\newblock {\em Communications in Mathematical Physics}, 103(3):491--498, 1986.

\bibitem{laanait1991interfaces}
L.~Laanait, A.~Messager, S.~Miracle-Sol{\'e}, J.~Ruiz, and S.~Shlosman.
\newblock Interfaces in the {P}otts model {I}: {P}irogov-{S}inai theory of the {F}ortuin-{K}asteleyn representation.
\newblock {\em Communications in Mathematical Physics}, 140(1):81--91, 1991.

\bibitem{lev2002cameron}
V.~F. Lev and T.~Schoen.
\newblock Cameron-{E}rd{\H{o}}s modulo a prime.
\newblock {\em Finite Fields and Their Applications}, 8(1):108--119, 2002.

\bibitem{levin2017markov}
D.~A. Levin and Y.~Peres.
\newblock {\em Markov chains and mixing times}, volume 107.
\newblock American Mathematical Soc., 2017.

\bibitem{luczak2000triangle}
T.~{\L}uczak.
\newblock On triangle-free random graphs.
\newblock {\em Random Structures \& Algorithms}, 16(3):260--276, 2000.

\bibitem{mantel1907problem}
W.~Mantel.
\newblock Problem 28.
\newblock {\em Wiskundige Opgaven}, 10(60-61):320, 1907.

\bibitem{mitzenmacher2017probability}
M.~Mitzenmacher and E.~Upfal.
\newblock {\em Probability and computing: Randomization and probabilistic techniques in algorithms and data analysis}.
\newblock Cambridge university press, 2017.

\bibitem{morris2016number}
R.~Morris and D.~Saxton.
\newblock The number of ${C}_{2\ell}$-free graphs.
\newblock {\em Advances in Mathematics}, 298:534--580, 2016.

\bibitem{mousset2020probability}
F.~Mousset, A.~Noever, K.~Panagiotou, and W.~Samotij.
\newblock On the probability of nonexistence in binomial subsets.
\newblock {\em The Annals of Probability}, 48(1):493--525, 2020.

\bibitem{osthus2001almost}
D.~Osthus, H.~J. Pr{\"o}mel, and A.~Taraz.
\newblock Almost all graphs with high girth and suitable density have high chromatic number.
\newblock {\em Journal of Graph Theory}, 37(4):220--226, 2001.

\bibitem{osthus2003densities}
D.~Osthus, H.~J. Pr{\"o}mel, and A.~Taraz.
\newblock For which densities are random triangle-free graphs almost surely bipartite?
\newblock {\em Combinatorica}, 23(1):105--150, 2003.

\bibitem{patel2016deterministic}
V.~Patel and G.~Regts.
\newblock Deterministic polynomial-time approximation algorithms for partition functions and graph polynomials.
\newblock {\em SIAM Journal on Computing}, 46(6):1893--1919, 2017.

\bibitem{peierls1936ising}
R.~Peierls.
\newblock On {I}sing's model of ferromagnetism.
\newblock {\em Mathematical Proceedings of the Cambridge Philosophical Society}, 32(3):477--481, 1936.

\bibitem{penrose1963convergence}
O.~Penrose.
\newblock Convergence of fugacity expansions for fluids and lattice gases.
\newblock {\em Journal of Mathematical Physics}, 4(10):1312--1320, 1963.

\bibitem{penrose1967convergence}
O.~Penrose.
\newblock Convergence of fugacity expansions for classical systems.
\newblock {\em Statistical mechanics: foundations and applications}, page 101, 1967.

\bibitem{pirogov1975phase}
S.~A. Pirogov and Y.~G. Sinai.
\newblock Phase diagrams of classical lattice systems.
\newblock {\em Theoretical and Mathematical Physics}, 25(3):1185--1192, 1975.

\bibitem{promel1996counting}
H.~J. Pr{\"o}mel and A.~Steger.
\newblock Counting {H}-free graphs.
\newblock {\em Discrete Mathematics}, 154(1-3):311--315, 1996.

\bibitem{promel1996asymptotic}
H.~J. Pr{\"o}mel and A.~Steger.
\newblock On the asymptotic structure of sparse triangle free graphs.
\newblock {\em Journal of Graph Theory}, 21(2):137--151, 1996.

\bibitem{promel2001random}
H.~J. Pr{\"o}mel and A.~Taraz.
\newblock Random graphs, random triangle-free graphs, and random partial orders.
\newblock In {\em Computational Discrete Mathematics}, pages 98--118. Springer, 2001.

\bibitem{radin2013phase}
C.~Radin and M.~Yin.
\newblock Phase transitions in exponential random graphs.
\newblock {\em The Annals of Applied Probability}, pages 2458--2471, 2013.

\bibitem{regts2015zero}
G.~Regts.
\newblock Zero-free regions of partition functions with applications to algorithms and graph limits.
\newblock {\em Combinatorica}, pages 1--29, 2015.

\bibitem{robins2007introduction}
G.~Robins, P.~Pattison, Y.~Kalish, and D.~Lusher.
\newblock An introduction to exponential random graph ($p^*$) models for social networks.
\newblock {\em Social networks}, 29(2):173--191, 2007.

\bibitem{ruelle1963correlation}
D.~Ruelle.
\newblock Correlation functions of classical gases.
\newblock {\em Annals of Physics}, 25(1):109--120, 1963.

\bibitem{sapozhenko2008cameron}
A.~A. Sapozhenko.
\newblock The {C}ameron--{E}rd{\H{o}}s conjecture.
\newblock {\em Discrete Mathematics}, 308(19):4361--4369, 2008.

\bibitem{saxton2015hypergraph}
D.~Saxton and A.~Thomason.
\newblock Hypergraph containers.
\newblock {\em Inventiones mathematicae}, 201(3):925--992, 2015.

\bibitem{scott2005repulsive}
A.~D. Scott and A.~D. Sokal.
\newblock The repulsive lattice gas, the independent-set polynomial, and the {L}ov{\'a}sz local lemma.
\newblock {\em Journal of Statistical Physics}, 118(5-6):1151--1261, 2005.

\bibitem{stark2018probability}
D.~Stark and N.~Wormald.
\newblock The probability of non-existence of a subgraph in a moderately sparse random graph.
\newblock {\em Combinatorics, Probability and Computing}, 27(4):672--715, 2018.

\bibitem{steger2005evolution}
A.~Steger.
\newblock On the evolution of triangle-free graphs.
\newblock {\em Combinatorics, Probability and Computing}, 14(1-2):211--224, 2005.

\bibitem{turan1941external}
P.~Tur{\'a}n.
\newblock On an external problem in graph theory.
\newblock {\em Mat. Fiz. Lapok}, 48:436--452, 1941.

\bibitem{wasserman1996logit}
S.~Wasserman and P.~Pattison.
\newblock Logit models and logistic regressions for social networks: {I}. an introduction to {M}arkov graphs and $p^*$.
\newblock {\em Psychometrika}, 61(3):401--425, 1996.

\bibitem{wormald1996perturbation}
N.~C. Wormald.
\newblock The perturbation method and triangle-free random graphs.
\newblock {\em Random Structures \& Algorithms}, 9(1-2):253--270, 1996.

\bibitem{zhang2023hypergraph}
S.~Zhang.
\newblock Hypergraph independence polynomials with a zero close to the origin.
\newblock {\em Combinatorics, Probability and Computing}, pages 1--5, 2023.

\end{thebibliography}
\end{document}